\documentclass[a4paper,10pt]{amsart}
\usepackage[utf8]{inputenc}
\usepackage{amssymb,amsmath}
\usepackage[all]{xy}

\newtheorem{theorem}{Theorem}[section]
\newtheorem{definition}[theorem]{Definition}
\newtheorem{example}[theorem]{Example}
\newtheorem{proposition}[theorem]{Proposition}
\newtheorem{lemma}[theorem]{Lemma}
\newtheorem{remark}[theorem]{Remark}
\newtheorem{conjecture}[theorem]{Conjecture}
\newtheorem{corollary}[theorem]{Corollary}

\renewcommand{\AA}{\mathbb{A}}
\newcommand{\altA}{\mathcal{A}}

\newcommand{\CC}{\mathbb{C}}

\newcommand{\EE}{\mathcal{E}}
\newcommand{\FF}{\mathcal{F}}

\newcommand{\GG}{\mathbb{G}}
\newcommand{\altG}{\mathcal{G}}
\newcommand{\kk}{\Bbbk}
\newcommand{\KK}{\mathbb{K}}
\newcommand{\LL}{\mathbb{L}}
\newcommand{\MM}{M}
\newcommand{\NN}{\mathbb{N}}
\newcommand{\OO}{\mathcal{O}}
\newcommand{\PP}{\mathbb{P}}
\newcommand{\PPP}{P}

\newcommand{\QQ}{\mathbb{Q}}

\newcommand{\altT}{\mathcal{T}}

\newcommand{\WW}{\mathcal{W}}

\newcommand{\Xst}{\mathfrak{X}}
\newcommand{\Yst}{\mathfrak{Y}}
\newcommand{\ZZ}{\mathbb{Z}}
\newcommand{\Zst}{\mathfrak{Z}}
\newcommand{\Mst}{\mathfrak{M}}
\newcommand{\Msp}{\mathcal{M}}

\newcommand{\unit}{1}
\newcommand{\ICS}{\mathcal{IC}}
\newcommand{\DTS}{\mathcal{DT}}
\newcommand{\Part}{\mathcal{P}}
\newcommand{\AAA}{{\mathbb{A}^1}}
\newcommand{\muu}{{\hat{\mu}}}
\newcommand{\fst}{\mathfrak{f}}
\newcommand{\gst}{\mathfrak{g}}
\newcommand{\SSm}{\mathfrak{S}}
\newcommand{\PPro}{\mathfrak{P}}
\newcommand{\Lin}{\mathcal{L}}
\newcommand{\Ab}{\mathcal{A}}
\newcommand{\mm}{\mathfrak{m}}

\DeclareMathOperator{\Hom}{Hom}
\DeclareMathOperator{\End}{End}
\DeclareMathOperator{\Ext}{Ext}

\DeclareMathOperator{\Vect}{Vect}
\DeclareMathOperator{\rep}{-mod}
\DeclareMathOperator{\Ka}{K}
\DeclareMathOperator{\Aut}{Aut}
\DeclareMathOperator{\codim}{codim}
\DeclareMathOperator{\Perv}{Perv}
\DeclareMathOperator{\MHM}{MHM}
\DeclareMathOperator{\rat}{rat}
\DeclareMathOperator{\IC}{IC}
\DeclareMathOperator{\DT}{DT}
\DeclareMathOperator{\Sym}{Sym}

\DeclareMathOperator{\Alt}{Alt}

\DeclareMathOperator{\Sch}{Sch}
\DeclareMathOperator{\Spec}{Spec}
\DeclareMathOperator{\Gl}{GL}
\DeclareMathOperator{\id}{id}

\DeclareMathOperator{\Mod}{-mod}
\DeclareMathOperator{\Tr}{Tr}

\DeclareMathOperator{\Crit}{Crit}
\DeclareMathOperator{\Pic}{Pic}

\DeclareMathOperator{\Con}{Con}
\DeclareMathOperator{\Char}{char}
\DeclareMathOperator{\St}{ASt}
\DeclareMathOperator{\QSt}{QSt}
\DeclareMathOperator{\rk}{rk}
\DeclareMathOperator{\Ho}{H}

\DeclareMathOperator{\pr}{pr}
\DeclareMathOperator{\SF}{SF}
\DeclareMathOperator{\Th}{Th}
\DeclareMathOperator{\Sm}{\mathcal{S}}
\DeclareMathOperator{\Bl}{Bl}
\DeclareMathOperator{\CTh}{CaTh}
\DeclareMathOperator{\DM}{DM}

\DeclareMathOperator{\Gr}{Gr}
\DeclareMathOperator{\cl}{cl}
\DeclareMathOperator{\cone}{Cone}
\DeclareMathOperator{\Pro}{\mathcal{P}}
\DeclareMathOperator{\obj}{Obj}
\DeclareMathOperator{\gr}{gr}
\DeclareMathOperator{\Pot}{Pot}

\DeclareMathOperator{\Coh}{Coh}
\DeclareMathOperator{\Proj}{Proj}
\DeclareMathOperator{\rad}{rad}
\DeclareMathOperator{\Ta}{T}
\DeclareMathOperator{\Cat}{\mathcal{C}}
\DeclareMathOperator{\Mat}{Mat}

\title[DT-theory for categories of dimension one with potential]{Donaldson--Thomas theory for categories of homological dimension one with potential}
\author{Ben Davison \and Sven Meinhardt}

\begin{document}

\begin{abstract}
The aim of the paper is twofold. Firstly, we give an axiomatic presentation of Donaldson--Thomas theory for categories of homological dimension at most one with potential. In particular, we provide rigorous proofs of all standard results concerning the integration map, wall-crossing, PT--DT correspondence, etc.\ following Kontsevich and Soibelman. We also show the equivalence of their approach and the one given by Joyce and Song. Secondly, we relate Donaldson--Thomas functions for such a category with arbitrary potential to those with zero potential under some mild conditions. As a result of this, we obtain a geometric interpretation of Donaldson--Thomas functions in all known realizations, i.e.\ mixed Hodge modules, perverse sheaves and constructible functions.   	
\end{abstract}

\maketitle

\tableofcontents

\section{Introduction}

In analogy to the classical Casson invariant, R.\ Thomas invented a holomorphic Casson invariant in his PhD-thesis \cite{Thomas1} by constructing a perfect obstruction theory on the moduli space of Gieseker stable sheaves on a compact Calabi--Yau 3-fold. If there are no strictly semistable sheaves of a particular Chern character, he ``integrated'' the resulting degree zero cycle associated to the perfect obstruction theory to define his invariant which is nowadays known as the ($\ZZ$-valued) Donaldson--Thomas invariant. \\
A couple of years later, K.\ Behrend proved in \cite{Behrend} that the integral over the degree zero cycle of a symmetric perfect obstruction theory can also be written as the integral over a $\ZZ$-valued constructible function, the so-called Behrend function, with respect to the compactly supported Euler characteristic. Writing the Donaldson--Thomas invariant in this way, its motivic nature became more transparent which was the starting point for generalizing the invariant to cases in which strictly semistable objects are present. \\
There were essentially two independent approaches aiming at generalizing Thomas's invariant to all semistable sheaves. One was developed by M.\ Kontsevich and Y.\ Soibelman in two papers \cite{KS1}, \cite{KS2}. A survey of the theory can be found in \cite{KS3}. Their project to develop a motivic theory of all reasonable Calabi--Yau 3-categories is quite ambitious and involves a lot of technical material which was often only sketched in their papers. \\
A second approach is due to D.\ Joyce, who developed in a series of papers \cite{JoyceI}, \cite{JoyceCF}, \cite{JoyceII},   \cite{JoyceIII}, \cite{JoyceMF}, \cite{JoyceIV} and \cite{JoyceDT} (in collaboration with Y.\ Song) a framework to give a rigorous definition of $\ZZ$-valued Donaldson--Thomas invariants.  A survey of the theory can be found in \cite{JoyceDTS}.\\
It turns out that both approaches not only apply to sheaves on (compact) Calabi--Yau 3-folds, but also to representations of quivers with potential. In fact, the theory is much easier to handle in this case as many technical difficulties disappear, namely the need of derived algebraic geometry and orientation data. On the other hand, quivers with potential describe the ``local'' structure of any polystable object in a Calabi--Yau 3-category. Hence, understanding this class of examples is a big step toward more general Calabi--Yau 3-categories, such as for instance the derived category of coherent sheaves on a compact Calabi--Yau 3-manifold. \\
The case of quivers with zero potential has been studied quite intensively by M.\ Reineke in a series of papers \cite{Reineke1}, \cite{Reineke_HN} ,\cite{Reineke_counting}, \cite{Reineke2}, \cite{Reineke3}, \cite{Reineke4} even before Kontsevich, Soibelman and Joyce revolutionized Donaldson--Thomas theory. Indeed, for zero potential, Donaldson--Thomas theory is very much related to classical work going back to Ringel, Hall and many other mathematicians. The complete understanding of Donaldson--Thomas theory for quivers with zero potential has been achieved by M.\ Reineke and the second author in \cite{MeinhardtReineke}. It turns out that all results for quivers with zero potential can be generalized to arbitrary abelian categories ``of homological dimension at most one''. This will include sheaves on smooth projective curves but also on smooth projective surfaces under special conditions. It will also cover the case of representations or sheaves satisfying some locally closed condition. For example, representations of cycles in the quiver should act invertibly or sheaves should not meet a certain divisor. All this material can be found in \cite{Meinhardt4}.\\

The aim of the present paper is to extend the results in \cite{Meinhardt4} to categories of homological dimension at most one with arbitrary potential, generalizing the case of quiver with potential. We provide a rigorous approach to Donaldson--Thomas theory  including a geometric interpretation of Donaldson--Thomas invariants in reasonable realizations. We will mostly follow the approach suggested by Kontsevich and Soibelman, but also show the equivalence of this version with the approach developed by Joyce if both theories can be applied. Our contribution to the field is twofold. Firstly, we give rigorous proofs in an axiomatic framework of many of the facts in (motivic) Donaldson--Thomas theory which are folklore to the experts but never written down properly, hoping to close the gap in the literature. Many of the proofs were already sketched by Kontsevich and Soibelman, and we will mostly follow their ideas. However, the idea of looking at a Donaldson--Thomas ``functions'' extending the Behrend function instead of just ``numbers'' was motivated by the paper \cite{JoyceDT} of Joyce and Song, where the concept was already sketched. We believe that Donaldson--Thomas ``functions'' are more fundamental than invariants, and by developing this idea rigorously, we were finally able to relate 
Donaldson--Thomas functions for categories of homological dimension at most one with arbitrary potential to those defined for the same category but with zero potential. Using the results of \cite{Meinhardt4}, we can give a geometric interpretation of Donaldson--Thomas functions in all known realizations. This is the second contribution of the paper to Donaldson--Thomas theory. \\

Let us sketch the idea how to tackle potentials. In fact, it turns out that the potential will only play a minor part and does not enter the formalism in a substantial way. It will just provide examples of a much more general formalism which we explain now. The central object of interest is that of a ($\lambda$-)ring theory which is a structure putting the following two example into an axiomatic framework. If you are familiar with constructible function, you know that such functions can be pulled back along morphisms and one can even define a push-forward along morphisms by integration along the fibers with respect to the (compactly supported) Euler characteristic. Moreover, given  constructible function on $X$ and $Y$, one can form their ``exterior product'' producing a constructible function on $X\times Y$. Here is an example which is more complicated but shares the same properties. Instead of constructible functions on $X$, we consider constructible sheaves or more general complexes of sheaves on $X$ with constructible cohomology. In fact, we are only interested in their classes in the associated Grothendieck group. Such (complexes of) sheaves can be pulled back, and we can also form the (derived) push-forward (with compact support) and exterior products. Thus, a ring theory is a rule associating an abelian group to every scheme $X$ along with pull-backs, push-forwards and exterior products. In fact, there is also a $\lambda$-ring structure involved which can be seen at best in the case of constructible sheaves. There is an obvious morphism $\oplus:\Sym(X)\times \Sym(X)\to \Sym(X)$ turning the disjoint union $\Sym(X)$ of all symmetric powers of $X$ into a commutative monoid and given by concatenating two unordered tuples of points in $X$. Pushing down the exterior product of two complexes of  sheaves on $\Sym(X)$ along $\oplus$, we end up with a symmetric monoidal tensor product on the derived category of complexes of sheaves on $\Sym(X)$ with constructible cohomology. As the product preserves the perverse t-structure, there is a standard way to construct a $\lambda$-ring structure on the associated Grothendieck group, and a $\lambda$-ring theory will also formalize this. Given a complex of sheaves on $X$ with constructible cohomology and a point $x\in X$, we can take the alternating sum of dimensions of the stalk of the complex at $x$ and obtain a constructible function by varying $x$. This construction descends to the Grothendieck group of sheaves providing us with a first example of a morphism between theories. There are also theories which do not come from a ``theory of something'' but just exist by formal arguments doing some abstract non-sense. Nevertheless, these theories play an important role as we will see shortly. At this point we should mention two important things. Firstly, pull-backs and push-forwards might not exist for all morphisms and we have to specify two classes $\Sm$ (for pull-back) and $\Pro$ (for push-forwards) for which these operations are defined. To make the hole story work, the pair $(\Sm,\Pro)$ has to satisfy some properties. In the case of constructible functions, the push-forward is for example only defined for finite type morphisms. Secondly, we can replace schemes $X$ by schemes $X\xrightarrow{f}\MM$ over a fixed base leading us to theories over $\MM$. To generalize the previous discussion, $\MM$ should be a commutative monoid in the category of schemes over some fixed ground field $\kk$. So far we only considered the case $\MM=\Spec\kk$ and this is exactly the situation considered (implicitly) in \cite{Meinhardt4}. Another important case is $\MM=(\AA^1,+,0)$, because vanishing cycles can be interpreted as morphisms from a particular ``abstract nonsense'' theory to the theory of interest, e.g.\ constructible functions or complexes of sheaves with constructible cohomology. Let us point out, that both theories, the abstract one and the one of interest were originally defined over $\Spec\kk$. The abstract one comes  actually from the initial object in the category of ring theories over 
$\Spec \kk$. But it is possible to pull-back theories along monoid homomorphisms, e.g. $\AA^1\to \Spec\kk$. The vanishing cycle interpreted as morphism of theories does only exist on the pull-backs to $\AA^1$, in other words, the pull-back functor from theories over $\Spec\kk$ to theories over $\AA^1$ is not surjective on morphisms, i.e.\ not full. \\
In fact, the situation we are interested in, is the case where $\MM=\Msp$ is the moduli space of objects in our category $\Ab$ of homological dimension one. Taking direct sums of objects, turns $\Msp$ into a commutative monoid. Using a monoid homomorphism $W:\Msp\to \AA^1$, this is our version of a potential, we can pull back vanishing cycles discussed earlier to obtain examples of the following situation: a $\lambda$-ring theory $R$ over $\Msp$ along with a morphism $\phi:\underline{\ZZ}(Sm^{proj})\longrightarrow R$ from the pull-back $\underline{\ZZ}(Sm^{proj})$ of the initial ring-theory over $\Spec\kk$ along $\Msp\to \Spec\kk$ to $R$. This kind of morphism can be interpreted as a very general $R$-valued vanishing cycle for regular functions with values in $\Msp$ instead of $\AA^1$. Given such a situation, we will define   
Donaldson--Thomas ``functions'' $\DTS(\Ab,\phi)$ in $R(\Msp\xrightarrow{\id}\Msp)$ which should be interpreted as some very general  function on $\Msp$. In the case of constructible functions, this is indeed just a constructible function on $\Msp$ extending the Behrend function on some singular subscheme depending on the potential. In the case of constructible sheaves, it is the class of some complex of sheaves in the associated Grothendieck group. \\
Even though this machinery looks rather abstract, it provides us with the correct formalism to prove our main result in very few lines.
\begin{theorem}[Existence and Uniqueness] 
Assume that $\phi$ has a factorization $\phi:\underline{\ZZ}(Sm^{proj})\xrightarrow{\unit} \underline{\Ka}_0(\Sch^{ft})[\LL^{-1/2}] \xrightarrow{\phi} R$ for some morphism $\phi$ of $\lambda$-ring theories\footnote{We apologize for the confusing notation. But in practice it will be clear which $\phi$ to take.}. Then, 
\[ \DTS(\Ab,\phi) = \phi_{\id} \big(\DTS(\Ab,mot)\big). \]
is a Donaldson--Thomas function and $\DTS(\Ab,\phi)_d$ is uniquely determined up to an element annihilated by $[\PP^{gcd(d)-1}]$. In particular, its image under the map $R(\id)\longrightarrow R(\id)[[\PP^n]^{-1}\mid n\in \NN]$ is unique. 
\end{theorem}
Let us make some comments on this statement. Firstly, $\underline{\Ka}_0(\Sch^{ft})[\LL^{-1/2}]$ is another theory over $\Spec\kk$ associating to every scheme (over $\kk$) its naive Grothendieck ring of varieties extended by an inverse of the square root of the motive $\LL$ of $\AAA$, and $\underline{\ZZ}(Sm^{proj})\xrightarrow{\unit} \underline{\Ka}_0(\Sch^{ft})[\LL^{-1/2}]$ is the (pull-back of the) unique morphism of theories over $\kk$ using that $\underline{\ZZ}(\Sm^{proj})$ is the pull-back of the initial object. The function $\DTS(\Ab,mot)\in \underline{\Ka}_0(\Sch^{ft}_\Msp)[\LL^{-1/2}]$ is the one studied extensively in the paper \cite{Meinhardt4}. Hence, the theorem tells us how the reduce the complicated case of very general vanishing cycles to the case of no potentials and vanishing cycles. Secondly, the Theorem is a statement about existence of Donaldson--Thomas functions, which does not follow from its definition. This existence is just the generalized version of the famous integrality conjecture which is therefore proven in this paper. Notice, that the theorem applies to all classical vanishing cycles for constructible functions, constructible sheaves and mixed Hodge modules. It would also apply to any reasonable categorification of the motivic vanishing cycle $\phi^{mot}$ introduced by Denef and Loeser (cf.\ \cite{DenefLoeser1}, \cite{DenefLoeser2}, \cite{DenefLoeser3} and \cite{Heinloth1}). However, the motivic vanishing cycle $\phi^{mot}$ itself  might not fulfill the assumptions of the theorem as it is not clear at the moment whether or not $\phi^{mot}$ commutes with $\lambda$-operations. In particular, the integrality conjecture, while proven in any classical realization, is still open in the context of naive motives. \\
The following theorem is very similar to the previous one and provides an even better characterization of $\DTS(\Ab,\phi)$ in all classical examples using the theory of mixed Hodge modules.
\begin{theorem}[Geometric interpretation] 
Assume that $\oplus:\Msp\times \Msp\longrightarrow\Msp$ is a finite morphism. If  $\phi$ has a factorization $\phi:\underline{\ZZ}(Sm^{proj})\xrightarrow{\unit} \underline{\Ka}_0(\MHM)[\LL^{-1/2}] \xrightarrow{\phi} R$ for some morphism $\phi$ of $\lambda$-ring theories. Then, 
\[ \DTS(\Ab,\phi)_d = \begin{cases} \phi_{\iota_d} \big(\cl(\ICS^{mhm}_{\Msp_d})\big) & \mbox{if } \Msp^s_d\not=\emptyset, \\ 0 &\mbox{else.} \end{cases} \]
Here, $\cl(\ICS^{mhm}_{\Msp_d})$ denotes the class of the (normalized) intersection complex $\ICS^{mhm}_{\Msp_d}\in \MHM(\Msp_d)[\QQ(1/2)]$ of the singular space $\Msp_d\subset \Msp$, and $\iota_d$ denotes the embedding.
\end{theorem}
The finiteness assumption on $\oplus$ is fulfilled in all examples. If  $\phi$ is the pull-back of a classical vanishing cycle, considered as a morphism of theories over $\AA^1$, along a potential $W:\Msp\to \AA^1$, then $\DTS(\Ab,\phi)$ is just the classical vanishing cycle functor applied to the intersection complex of the trivial local system on the smooth space $\Msp^s$ of simple objects in $\Ab$. This result has a categorification which along with many other interesting material will be developed in \cite{DavisonMeinhardt4}.\\

The paper is organized as follows. In section 2 we give a short survey of the content of \cite{Meinhardt4}, mostly to fix the notation used throughout the paper. We will also introduce two sorts of potentials, the version seen above and another version which is closer to the definition for quivers. The latter gives rise to a potential in the former sense.\\
Starting with section 3, we introduce the axiomatic framework used throughout the paper. The category of ($\lambda$-)ring $(\Sm,\Pro)$-theories is the object of interest in section 3, and we already provided a motivation and the basic idea. At the end of this section we give a short introduction to a possible categorification  which will be the central subject of \cite{DavisonMeinhardt4}. For the moment it will provide us with many examples.\\
Section 4 gives an alternative definition of Donaldson--Thomas functions using framed objects in $\Ab$. This approach has two advantages. Firstly, it avoids the notion of stacks allowing us to work within the framework introduced in section 3. Secondly, it requires the weakest assumption on our $\lambda$-ring theory $R$. The price of this generality is the absence of any Ringel--Hall algebra formalism and wall-crossing formulas. Moreover, the definition depends on a choice of a fiber functor $\omega$ for $\Ab$ as introduced in section 2. However, by our main theorem, the Donaldson--Thomas functions are independent of this choice whenever the theorem applies. We will also give a prospect of a categorification to which we come back in \cite{DavisonMeinhardt4}.\\
In section 5 we generalize the framework of theories to Artin stacks. We will show that under good conditions, a theory on Artin stacks or slightly more specific on quotient stacks is uniquely determined by their restriction to schemes. The interesting part of this correspondence is the converse point of few allowing us to determine which theories on schemes have a (unique) extension to Artin or quotient stacks. \\
Using the framework of theories on quotient stacks developed in section 5, we can give the standard definition of Donaldson--Thomas functions following the idea of Kontsevich and Soibelman in section 6. We will introduce Ringel--Hall algebras and give a rigorous proof that some integration map involving our general vanishing cycle $\phi$ is an algebra homomorphism under the assumption that some non-linear integral identity holds. Moreover, we show that the non-linear integral identity holds for quiver with potential or in all classical realizations for arbitrary categories $\Ab$. Having this at hand, we prove the PT--DT correspondence saying that the functions defined in section 4 agree with the one of section 6 if both formalisms apply. We will also discuss the wall-crossing formula and give a rigorous proof that the approach taken by Joyce and Song to define Donaldson--Thomas functions will give the same result. \\
We close the paper by giving two appendices. The first one contains a detailed proof of the linear integral identity which will also enter the proof of the  non-linear version. The proof was already given in \cite{KS2} and we only fill in some details and provide the generalization needed in this paper. The second quite extensive appendix gives some background on $\lambda$-rings which occur throughout the paper. The aim of this part is twofold. Firstly, we need to fix the notation. Secondly, we discuss in full detail the process of adjoining elements which to our knowledge has never been written down properly but plays an important role in Donaldson--Thomas theory. \\

We apologize to the reader for writing such a lengthy paper. The authors felt the need of writing down all the details which are usually sketch or only conjectured in papers dealing with Donaldson--Thomas invariants. We could not cover all material, but we hope that the present text in combination with \cite{Meinhardt4} and \cite{DavisonMeinhardt4} provides a solid basis for learning Donaldson--Thomas theory. \\

\textbf{Acknowledgments.} The authors would like to thank  Michel van Garrel and Bumsig Kim for giving us the opportunity to visit KIAS, where the central ideas of this paper were born, mostly during a delicious social dinner. During the writing of this paper, Ben Davison was a postdoc at EPFL, supported by the Advanced Grant ``Arithmetic and physics of Higgs moduli spaces'' No. 320593 of the European Research Council. Sven Meinhardt wants to thank Markus Reineke for providing a wonderful atmosphere to complete this research.

\section{Moduli of objects in abelian categories} 

In this section we recall the main results of \cite{Meinhardt4} which provide the background of this paper. We start by listing the assumptions we want to impose on our categories and their moduli functors. We proceed by recalling the properties of moduli stacks and spaces which will be a consequence of our assumptions.  A couple of examples will be given to illustrate the theory. We also discuss framed objects and their relation to the unframed ones. Finally, we introduce potentials in this general context generalizing potentials for quivers. 

\subsection{The general setting}

We wish to talk about moduli of objects in a given $\kk$-linear abelian category $\Ab_\kk$. However, there is no unique way to form a moduli functor, and part of our data will be a choice of such a moduli functor. For good reasons we also need morphisms between ``families of objects'' which are not necessarily isomorphisms. For example, we want to talk about families of short exact sequences or of subobjects which are obviously related to short exact sequences. Therefore, our moduli functor $\Ab$ will not only take values in groupoids but in exact categories. The reason why we restrict ourselves to exact and not to abelian categories is very simple. If we think of vector bundles as being ``good'' families of vector spaces, then it is easy to see that vector bundles on a given scheme do not form an abelian category but an exact one. In fact, this is more than just an example. There is a deeper connection between families of objects in $\Ab_\kk$ and good families of vector spaces, i.e.\ vector bundles. Following the Tannakian approach we have a good reason to assume that our objects in $\Ab_\kk$ can be seen as vector spaces with an addition structure. This assumption is usually fulfilled in examples, but in practice it is hard to understand its meaning. Given a smooth compact curve $X$ over $\kk$ with an ample line bundle $\OO_X(1)$, consider the category of semistable coherent sheaves of slope $\mu\in (-\infty,+\infty]$ on $X$. It is hard to believe but nevertheless true that such a sheaf $E$  is completely described by some additional structure on $\Ho^0(X,E(m))$ for some sufficiently high twist of $E$ with $\OO_X(m)$. Nevertheless, this picture has a beautiful consequence. Fixing an isomorphism $\psi:\Ho^0(X,E(m))\xrightarrow{\sim} \kk^d$, where $d=P(m)$ for the Hilbert polynomial $P$ of $E$ with respect to $\OO_X(1)$, we can describe the moduli of semistable sheaves  with fixed Hilbert polynomial $P$ as the space $X_d$ of all possible (depending on our setting) additional structures on $\kk^d$ modulo a $\Gl(d)$-action given by changing $\psi$. Hence, the stack of all coherent sheaves as above is a quotient stack $X_d/\Gl(d)$. The technical ingredient to obtain this beautiful picture  is a functorial choice $\omega$ of (families of) vector spaces for all (families of) objects. This allows us to construct a compatible family of quotient stacks for all types of diagrams in $\Ab_\kk$, making it much easier to deal with all these Artin stacks. The final piece of data we want to require is a functorial connection between the tangent bundle of our moduli stacks and the stack of short exact sequences, i.e.\ extensions. It is the latter which is of great importance in Donaldon--Thomas theory and we need to get some control over it. This is essentially done by identifying it with a direct summand of a tangent bundle (restricted to some subspace). Technically, we achieve this by means of an isomorphism $p$ between two functors. We should mention that only $\Ab$ is the object  we want to deal with, whereas $\omega$ and $p$ are just some technical tools allowing us to prove some deep results. At the end of the day, all results should be independent of the choice of $\omega$ or $p$.  
Let us fill in the details and have a look at the set of axioms we want to impose. Examples will be given afterwards.\\
\\
\textbf{(1) Existence of a moduli theory:} We fix a contravariant (pseudo)functor $\Ab:S\mapsto \Ab_S$ from the category of $\kk$-schemes $S$ to the category of essentially small exact categories  using the shorthand $\Ab_R:=\Ab_{\Spec R}$ for any commutative $\kk$-algebra $R$, satisfying the usual gluing axioms of a stack for the big Zariski topology\footnote{One can also take the big \'{e}tale or any other Grothendieck topology for which the prestack of vector bundles is a stack. We haven taken the big Zariski topology for simplicity.}. We assume that pull-backs and push-outs of short exact sequences exist in $\Ab_S$ so that $\Ext^1_{\Ab_S}(E,F)$ is a well-defined group for every pair $E,F\in \Ab_S$. Moreover, for every $\kk$-scheme $S$ and every pair $E,F\in\Ab_S$ the groups $\Hom_{\Ab_S}(E,F)$ and $\Ext^1_{\Ab_S}(E,F)$ should have an enrichment to a finitely generated $\OO_S(S)$-module such that composition is $\OO_S$-bilinear. The action of $f\in \OO_S(S)$  on $\Ext^1_{\Ab_S}(E_1,E_2)$ coincides with the pull-back along  $f\in \End_{\Ab_S}(E_1)$ or the push-out of a short exact sequence along $f\in \End_{\Ab_S}(E_2)$. If $\KK\supset \kk$ is a field extension, $\Ab_\KK$ should be abelian. \\
\\
\textbf{(2) Existence of a fiber functor:} There is a faithful exact natural transformation  $\omega:\Ab \to \Vect^I$, where $\Vect^I$ is the (pseudo)functor associating to every $\kk$-scheme the exact category of $I$-graded vector bundles on $S$ with finite total rank on each connected component of $S$, where $I$ is any not necessarily finite set.  Moreover, as $\OO_S(S)$ acts on $\Hom$- and $\Ext^1$-groups in $\Ab_S$ as well as in $\Vect^I_S$, we assume that $\omega_S$ is $\OO_S(S)$-linear.  We also require that every morphism $f:E\to E''$ in $\Ab_S$ with $\omega_S(f):\omega_S(E)\to \omega_S(E'')$ fitting into an exact sequence $0 \to V' \to\omega_S(E) \xrightarrow{\omega_S(f)} \omega_S(E'')\to 0$ in $\Vect^I_S$ can be completed to an exact sequence $0\to E' \to E \xrightarrow{f} E''\to 0$ in $\Ab_S$ which implies  $V'\cong \omega_S(E')$ as $\omega_S$ is exact.\\ 
\\
\textbf{(3) Existence of good moduli stacks:} Denoting by $\Mst_d(S)\subset \Ab_S$ the subcategory consisting of objects $E\in \Ab_S$ such that $\omega_S(E)$ has multi rank $d\in \NN^{\oplus I}$ and isomorphisms between them, we get a morphism $\omega|_{\Mst_d}:\Mst_d \to \Spec\kk/G_d$ of stacks with $G_d=\prod_{i\in I}\Gl(d_i)$. We require that $\omega|_{\Mst_d}$ is representable for all $d\in \NN^{\oplus I}$. \\
\\
In particular, $\Mst=X_d/G_d$ is a quotient stack, where $X_d=\Mst_d\times_{\Spec\kk/G_d}\Spec\kk$ is the algebraic space with $S$-points given by pairs $(E,\psi)$ with $E\in \Ab_S$ of rank $d$ and $\psi:\omega_S(E)\xrightarrow{\sim} \OO_S^d$ an isomorphism of locally free sheaves. Throughout the paper we use the notation $\OO_S^d=\bigoplus_{i\in I}\OO_S^{\oplus d_i}$. The next assumption is about GIT-quotients. Given a $\kk$-scheme $Y$ of finite type over $\kk$ with an action of a reductive algebraic group $G$ and a linearization $\Lin$ of the $G$-action on some line bundle $\Lin$, we define the integer $\mu^\Lin(y,\lambda)$ depending on a point $y\in Y$ with residue field $\kk(y)$ and a one-parameter subgroup, 1-PS for short, $\lambda:\GG_m\to G$ defined over some algebraic extension $K\supset \kk(y)$ as follows. If the limit $y_0^\lambda=\lim\limits_{z\to 0}\lambda(z)y$ does not exist, we put $\mu^\Lin(y,\lambda)=+\infty$. Otherwise, we get an induced $\GG_m$-action in the fiber $\Lin_{y^\lambda_0}$ of $\Lin$ over $y_0^\lambda$ given by $z\mapsto z^r$ and we put $\mu^{\Lin}(y,\lambda)=-r$. The following result is a generalization of classical results of Mumford (in the projective case) and King (in the affine case).
\begin{theorem}[\cite{GHH}, Theorem 3.3]
Let $Y\to \Spec A$ be a projective morphism to some affine scheme of finite type over $\kk$ such that $\Lin$ is relatively ample, i.e.\ some power of $\Lin$ provides a closed embedding of $Y$ into $\PP^N_A$ compatible with the projection to $\Spec A$. Then, $y\in Y$ is semistable if and only if $\mu^{\Lin}(y,\lambda)\ge 0$ for all 1-PS defined over some algebraic extension of $\kk(y)$. Moreover, $y$ is stable if and only if $\mu^{\Lin}(y,\lambda)>0$.  
\end{theorem}
Let us come to our next assumption. Notice that every 1-PS $\lambda:\GG_m\to G_d$ induces a descending filtration $W^n\subset \KK^d$ given by the sum of all eigenspaces of weight $\ge n$ for the induced $\GG_m$-action. Conversely, every subspace $W\subset \KK^d$ and every $n\in \ZZ$ define a (non-unique) 1-PS $\lambda^{W,n}$ defined over $\KK$ such that $W^{n+1}=0, W^n=W$ and $W^{n-1}=\KK^d$ for the induced filtration on $\KK^d$. \\
\\

\textbf{(4) Existence of good GIT-quotients:} Assume that  every $X_d$ has an embedding into the open subscheme  $\overline{X}_d^{ss}$ of semistable points inside some $\kk$-scheme $\overline{X}_d$ with $G_d$-action and 
$G_d$-linearization $\Lin_d$. Moreover, there should be a projective morphism $f_d:\overline{X}_d\to \Spec A_d$ to some affine scheme of finite type over $\kk$ such that $\Lin_d$ is ample with respect to $f_d$. Let $\tilde{p}_d:\overline{X}^{ss}\to \overline{X}^{ss}/\!\!/G_d$ be the uniform categorical quotient (see \cite{MumfordGIT}, Theorem 1.10). We also require $X_d=\tilde{p}_d^{-1}(\Msp_d)$ for some locally closed subscheme $\Msp_d\subset \overline{X}^{ss}_d/\!\!/G_d$ which should be open if $\Char\kk>0$. Furthermore, for every 1-PS in $G_d$ and any $x\in \overline{X}_d$ we require 
\begin{equation} \label{eq7} \mu^{\Lin_d}(x,\lambda)=\sum_{n\in \ZZ} \mu^{\Lin_d}(x,\lambda^{W^n,n}). \end{equation}
Moreover, for $x\in X_d$ represented by $(E,\psi)$ defined over $\KK$ the conditions 
\begin{enumerate}
 \item[(a)] the limit point $x_0^\lambda$ exists and is in $X_d$,
 \item[(b)] there are subobjects $E^n\subset E$ with $\omega_\KK(E^n)=\psi(W^n)$ for all $n\in \ZZ$,
 \item[(c)] the equation $\mu^{\Lin_d}(x,\lambda)=0$ holds,
\end{enumerate}
should be equivalent. In this case $x_0^\lambda$ is given by the associated graded (trivialized) object $(\gr E^\bullet=\oplus_{n\in\ZZ} E^n/E^{n+1},\gr \psi^\bullet)$.\\
\\
In this case $\Msp_d$ is a categorical quotient of $X_d$ and corepresents $\Mst_d$. Hence, it does neither depend on the choice of $\overline{X}_d$ and $\Lin_d$ nor on the choice of $\omega$. However, $\Lin_d$ descends to an ample line bundle on $\Msp_d$ making it into a quasiprojective variety. Also, $X_d$ is a quasiprojective variety and not just an algebraic space. It also follows that each of the conditions (a), (b) and (c) are equivalent to (d) saying that $x_0^\lambda$ exists in $\overline{X}_d^{ss}$.   \\
The following result has been proven in \cite{Meinhardt4}. Recall that an object in $\Ab_\KK$ is called absolutely simple if it remains simple in $\Ab_{\KK'}$ for every algebraic extension $\KK'\supset \KK$. A direct sum of absolutely simple objects is called absolutely semisimple. 
\begin{lemma}
 The $\KK$-points in $\Msp$ are in bijection with the absolutely semisimple objects in $\Ab_\KK$. The $\KK$-points in $(X_d\cap \overline{X}_d^{st})/\!\!/G_d=:\Msp^s$ are in bijection with the absolutely simple objects in $\Ab_\KK$. 
\end{lemma}
By assumption on $\omega$ we also get a faithful functor from the stack of short exact sequences $0\to E_1\to E_2\to E_3\to 0$  in $\Ab_\kk$ to the stack of short exact sequences in $\Vect^I_\kk$. Fixing the dimensions $d=\dim_\kk \omega_\kk(E_1)$ and $d'=\dim_\kk\omega_\kk(E_3)$ of the outer terms, the latter is given by $\Spec\kk/G_{d,d'}$, where $G_{d,d'}$ is the stabilizer of the subspace $\kk^d\subset \kk^d\oplus \kk^{d'}=\kk^{d+d'}$. In particular, we get a commutative diagram
\[ \xymatrix {  & X_{d,d} \ar[dl]_{\hat{\pi}_1\times \hat{\pi}_3} \ar[dr]^{\hat{\pi}_2} \ar[dd]^{\rho_{d,d'}} & \\ X_d\times X_{d'} \ar[dd]_{\rho_d\times \rho_{d'}} & & X_{d+d'} \ar[dd]^{\rho_{d+d'}}\\
 & \mathfrak{E}xact_{d,d'} \ar[dl]_{\pi_1\times \pi_3} \ar[dr]^{\pi_2} & \\ \Mst_d\times \Mst_{d'} & & \Mst_{d+d'} } 
\]
with $X_{d,d'}=\mathfrak{E}xact_{d,d'} \times_{\Spec\kk/G_{d,d'}} \Spec\kk$ being the set valued functor mapping $S$ to the set of equivalence classes of tuples $(0\to E_1\to E_2\to E_3\to 0,\psi_1,\psi_2,\psi_3)$ consisting of a short exact sequence in $\Ab_S$ and an isomorphism
\[ \xymatrix { 0 \ar[r] & \omega_S(E_1)\ar[d]^{\psi_1}_\wr \ar[r] & \omega_S(E_2)\ar[d]^{\psi_2}_\wr \ar[r] & \omega_S(E_3)\ar[d]^{\psi_3}_\wr \ar[r] & 0 \\ 
 0 \ar[r] & \OO_S^d \ar[r] & \OO_S^{d} \oplus \OO_S^{d'} \ar[r] & \OO_S^{d'} \ar[r] & 0 } 
\]
to the trivial sequence of $I$-graded vector bundles on $S$. Two tuples as above are equivalent  if there is a (unique) isomorphism of the underlying short exact sequences in $\Ab_S$ compatible with the trivializations.  The maps $\pi_i$ map a short exact sequence to their $i$-th entry and similarly for $\hat{\pi}_i$. 
\begin{lemma} \label{finite_sum}
 If $X_d=\overline{X}_d^{ss}$ for all $d\in \NN^{\oplus I}$, then the morphism $\oplus:\Msp\times\Msp\to \Msp$ defined by taking direct sums is a finite morphism. 
\end{lemma}
For a generalization of this one should have a look at Example \ref{finite_direct_sum}.\\
\\
\textbf{(5) The universal Grassmannian is proper:} The map $\pi_2:\mathfrak{E}xact \to \Mst$ is representable and proper. \\

Because of this condition $Y_{d,d'}=\mathfrak{E}xact_{d,d'}\times_{\Mst_{d+d'}}X_{d+d'}\cong X_{d,d'}\times_{G_{d,d'}}G_{d+d'}$ is an algebraic space which embeds into $X_{d+d'}\times \Gr_{d}^{d+d'}$ as a closed subspace. Hence, $Y_{d,d'}$ is a quasiprojective variety containing $X_{d,d'}$ as a closed subfunctor. It will follow from the next assumption that $X_{d,d'}$ is also a scheme, hence a closed subscheme of $Y_{d,d'}$. An $S$-point of $Y_{d,d'}$ is a  short exact sequence $0\to E_1\to E_2\to E_3\to 0$ in $\Ab_S$ together with an isomorphism $\psi_2:\omega_S(E_2)\xrightarrow{\sim} \OO_S^{d+d'}$. Having this description at hand, the maps $X_{d,d'}\to Y_{d,d'}$, $Y_{d,d'}\to X_{d+d'}$ and $Y_{d,d'}\to \Gr_{d}^{d+d'}$ should be clear. Let us mention that the fiber of $\pi_2$ over $E\in \Ab_\KK$ is the proper scheme of subobjects of $E$ in $\Ab_\KK$. Therefore, $\pi_2$ is indeed the universal Grassmannian parameterizing subobjects. \\

The next assumption is important to describe the map $\hat{\pi}_1\times \hat{\pi}_3:X_{d,d'}\longrightarrow X_d\times X_{d'}$.
For any $\kk$-algebra $R$, an $R$-object in $\Ab_S$ is an object $E\in \Ab_S$ together with a $\kk$-algebra homomorphism $\nu:R\to \End_{\Ab_S}(E)$ inducing an $R$-action on $E$. The definition of a homomorphism between $R$-objects in $\Ab_S$ should be clear giving rise to a category of $R$-objects in $\Ab_S$. The functor $\omega_S$ lifts to a functor $\omega_S^R$ from the category of $R$-objects in $\Ab_S$ to the category of $R\otimes_\kk \OO_S$-modules on $S$ which are locally free as $\OO_S$-modules. We denote with $\Ab_S^R$ the full subcategory of $R$-objects $(E,\nu)$ in $\Ab_S$ such that $\omega_S^R(E,\nu)$ is even locally free as an $R \otimes_\kk \OO_S$-module. Obviously, we can pull-back $R$-objects in $\Ab_S$ along any morphism $S'\to S$. \\

\textbf{(6) Existence of a good deformation theory:} For every morphism $R\to R'$ of local Artin $\kk$-algebras the functor $R'\otimes_R(-)$ from locally free $R\otimes_\kk\OO_S$-modules to locally free $R'\otimes_\kk\OO_S$-modules  lifts to a functor $R'\otimes_R(-):\Ab_S^R \to \Ab_S^{R'}$ making it into a bifunctor on $\{\kk\mbox{-schemes}\}^{op}\times \{\mbox{local Artin }\kk\mbox{-algebras}\}$. We assume the existence of  an equivalence $p_S^R:\Ab_{\Spec(R)\otimes S}\xrightarrow{\sim} \Ab_S^R$ of exact categories for every $\kk$-scheme $S$ and every local Artin $\kk$-algebra $R$ such that $p_{(-)}^{(-)}:\Ab_{\Spec(-)\times (-)} \longrightarrow \Ab_{(-)}^{(-)}$ is an equivalence of bifunctors. This equivalence should be compatible with the functors to locally free $R\otimes_\kk\OO_S$-modules, were we identify each locally free $\OO_{\Spec(R)\times S}$ module with its push-down onto $S$. \\

Here are a couple of consequences.
\begin{proposition}
 The map $\hat{\pi}_1\times\hat{\pi}_3:X_{d,d'}\longrightarrow X_d\times X_{d'}$ is an abelian cone, i.e.\ the relative spectrum  $\Spec \Sym \mathcal{G}$ of the symmetric algebra generated by some coherent sheaf $\mathcal{G}$ on $X_d\times X_{d'}$. It can be identified with a direct summand of the tangent cone $TX_{d+d'}=\Spec \Sym \Omega^1_{X_{d+d'}}$ restricted to $X_d\times X_{d'}$ which embeds into $X_{d+d'}$ by taking direct sums. Moreover, $X_{d,d'}$ is normal to that embedding. Conversely, the restriction of $X_{d,d}\longrightarrow X_d\times X_d$ to the diagonal $\Delta_{X_d}$ is the tangent cone $TX_d$. The fiber $F$ of $\hat{\pi}_1\times \hat{\pi}_3$ over any point $((E_1,\psi_1),(E_3,\psi_3))\in X_d\times X_{d'}$ defined over $\KK$ fits canonically into the following exact sequence of $\KK$-vector spaces 
 \[ 0 \longrightarrow \Hom_{\Ab_\KK}(E_3,E_1)\xrightarrow{\,\omega_S\,} \Hom_\KK(\KK^{d'},\KK^{d}) \longrightarrow F\longrightarrow \Ext^1_{\Ab_\KK}(E_3,E_1) \longrightarrow 0.\]
\end{proposition}

\begin{lemma}
 The space $X_{d,d'}$ is a quasiprojective scheme and $\hat{\pi}_2$ is a closed embedding. 
\end{lemma}
\begin{lemma}
 We have $\Ab_{\bar{\kk}}=\varinjlim_{\kk\subset\KK} \Ab_\KK$ and $\Ka_0(\Ab_{\bar{\kk}})=\varinjlim_{\kk\subset \KK} \Ka_0(\Ab_\KK)$, where the limit is taken over all finite extensions of $\kk$. 
\end{lemma}

\begin{example}\rm
 Consider a small $\kk$-linear category $A_\kk$ with finite presentation. That means, $A_\kk$ is the quotient of the $\kk$-linear path category $\kk Q$ of a quiver $Q$ with $Q_0$ being the set $I$ of objects in $A_\kk$ and such that the set $Q(i,j)$ of arrows from $i$ to $j$ is finite for every pair $(i,j)$. Moreover, there should be a finite set $r(i,j)\subset \Hom_{\kk Q}(i,j)$ of ``relations'' for every pair $(i,j)$ such that the kernel of $\kk Q \twoheadrightarrow A_\kk$ is generated (as a two-sided ideal of $\kk Q$) by $r:=\sqcup_{i,j\in I} r(i,j)$. We could form the quiver $Q^r$ of relations with $Q^r_0=I$ and $Q^r(i,j)=r(i,j)$ and a canonical $\kk$-linear functor $\kk Q^r \to \kk Q$. Then, the kernel of $\kk Q \to A_\kk$ is just the two-sided ideal generated by the image of $\kk Q^r\to \kk Q$. Let $\Ab_S=A_\kk\Mod_S$ be the category  of $\kk$-linear functors $V:A_\kk\to \Vect_S$ with finite total rank. Put $\omega_S(V)=(V(i))_{i\in I}$, and let $p$ be defined by the push-forward along the affine morphism $\Spec R\times S \to S$. Then $(\Ab,\omega,p)$ satisfies all properties (1)--(6). We can even modify this example by introducing a King stability condition given by a ``universal character'' $\theta\in \ZZ^I$. Let $A_\kk\Mod^{\theta-ss}_\mu$ be the full subfunctor of $A_\kk\Mod$ whose objects are  families of $\theta$-semistable $A_\kk$-modules of slope $\mu\in (-\infty,+\infty]$, then $(A_\kk\Mod^{\theta-ss}_\mu,\omega,p)$ still satisfies our assumptions. We can even choose a generic Bridgeland stability condition given by $\zeta\in \{ r\exp(\sqrt{-1}\pi\phi)\mid r>0, \phi\in (0,1]\}^I$ and consider the $\zeta$-semistable modules of slope $\mu$. For $\zeta_i=-\theta_i+\sqrt{-1}$ we get back King stability.
\end{example}

\begin{example} \rm
 Let $X$ be a smooth projective variety over $\kk$ with polarization $\OO_X(D)$. Unless $\dim X=0$, there is no functor $\omega_\kk:\Coh(X)\to \Vect^I_\kk$ satisfying all of our assumptions, because then $\Coh(X)$ would  be artinian which is not the case. The next natural candidate is the category of Gieseker semistable sheaves of a given normalized Hilbert polynomial with respect to $\OO_X(D)$. As a fiber functor $\omega_S$ we should choose $\omega^{(m)}_S(E)=\pr_{S\, \ast} (E \otimes \pr_X^\ast \OO(m D))$ for $E$ being a flat family of Gieseker semistable sheaves. Ideally, this functor satisfies our assumptions for $m\gg 0$, and this is indeed true for curves. In general, it is not clear whether or not a lower bound for $m$ depends only on the normalized and not just on the absolute Hilbert polynomial. Nevertheless, one can exhaust the category $\Coh(X)^{ss}_P$ of Gieseker semistable sheaves of a fixed normalized Hilbert polynomial $P$ by Serre subcategories such that $\omega_S^{(m)}$ restricted to the subcategory satisfies all of our requirements for some $m$ depending only on the subcategory. It turns out that this is not a big restriction and we can pretend that a fiber functor on $\Coh(X)^{ss}_P$ satisfying all of our assumptions (1)--(5) exists. Recall that $\omega$ is just a technical tool, and only the category $\Ab$ along with its descendants $\Mst$ and $\Msp$ is of interest. The transformation $p$ satisfying (6) is again given by push-forward along the projection to $S$. In particular, all results of this section are also true in this example. Similarly to the quiver case, it is also possible to work with generic real polarizations $D$.  
\end{example}

Let us state our next assumption. \\
\\
\textbf{(7) Smoothness assumption:} The schemes $X_d$ are locally integral, and the quantity $\dim_\KK\Hom_{\AA_\KK}(E,F)- \dim_\KK \Ext^1(E,F)$ is locally constant on $\Mst\times \Mst$. \\
\\
This condition should be seen as our version of saying that $\Ab_\kk$ is of homological dimension at most one. Note that  none of the categories $\Ab_\KK$ might have enough projective and injective objects, and so it is not obvious how to define higher $\Ext$-groups. 
\begin{lemma}
 The abelian cone $\hat{\pi}_1\times\hat{\pi}_3:X_{d,d'} \to X_d\times X_{d'}$ is a vector bundle.
\end{lemma}
\begin{lemma}
 For every field extension $\KK\supset \kk$, the number $\dim_\KK\Hom_{\AA_\KK}(E,F)- \dim_\KK \Ext^1_{\AA_\KK}(E,F)$ defines a pairing $(-,-)$ on $\Ka_0(\Ab_\KK)$.
\end{lemma}
A special role is played by the pairing on $\Ka_0(\Ab_{\bar{\kk}})=\varinjlim_{\KK\supset \kk\mbox{ \scriptsize finite}}\Ka_0(\Ab_\KK)$ which descends to a non-degenerate pairing on $\Gamma:=\Ka_0(\Ab_{\bar{\kk}})/\rad (-,-)$. Indeed, we obtain a further decomposition $\Msp_d=\sqcup_{\gamma\in \Gamma} \Msp_{\gamma,d}$. The rank of the vector bundle mentioned in the previous Lemma over $\Msp_{\gamma,d}\times\Msp_{\gamma',d'}$ is $dd'-(\gamma',\gamma)$. In all of our examples, $(-,-)$ is just the Euler pairing.  
\begin{example} \rm This assumption is fulfilled for ``free'' $\kk$-linear categories $A=\kk Q$ but also for some finitely presented categories $A$ satisfying suitable relations. For example, fix a quiver $Q$ as before and a finite family of cycles $(C_\kappa)_{\kappa\in K}$ at vertex $i_\kappa\in Q_0$. We extend $Q$ by introducing more loops $l_\kappa$ at vertex $i_\kappa$ and impose the relations $l_\kappa C_\kappa-\id_{i_\kappa}$ and $C_\kappa l_\kappa - \id_{i_\kappa}$ for all $\kappa\in K$. A module for the quotient category is a $\kk Q$-module such that all the $C_\kappa$ act invertibly, which is an open condition.  
\end{example}
\begin{example} \rm
In the case of sheaves we have to restrict ourselves to curves or to surfaces with $K_X\cdot D<0$ if we are interested in 2-dimensional sheaves or 1-dimensional sheaves whose support has negative intersection number with $K_X$. Similarly to the quiver case, we can modify this example by imposing open conditions. Given a curve $C$ of genus $g$ and a line bundle $\Lin$ on $C$, we consider the ruled surface $\Proj_C(\Lin\oplus \OO_C)$ with its divisor $D_\infty$ at infinity. Sheaves of dimension 1 which do not meet $D_\infty$ can also be described as sheaves $E$ on $C$ with a $\OO_S$-linear map $\Lin \otimes_{\OO_C} E \to E$. Having homological dimension one requires $\deg(L)< 2-2g$.
\end{example}

Finally, we come to our last assumption on $\Ab$ which is very important in Donaldson--Thomas theory.\\
\\
\textbf{(8) Symmetry assumption:} For all field extensions $\KK\supset \kk$  the pairing $(-,-)$ on $\Ka_0(\AA_\KK)$ is symmetric, i.e.\ 
\[\dim_\KK \Hom_{\AA_\KK}(E,F)-\dim_\KK\Ext_{\AA_\KK}^1(E,F)=\dim_\KK\Hom_{\AA_\KK}(F,E)-\dim_\KK \Ext_{\AA_\KK}^1(F,E).\]
for all objects $E,F\in \AA_\KK$.

\begin{example} \rm
 In all of our examples above, this condition is fulfilled if we choose our stability condition, e.g.\ $\zeta$ or $D$, to be a bit less generic. Nevertheless, they can be chosen in the complement of countably many real codimension one walls inside some real or complex manifold. 
\end{example} 
\begin{example}\label{finite_direct_sum} \rm
Let $\PPP$ be a locally closed property which is closed under extensions and forming subquotients. This is equivalent to giving a locally closed subscheme $\Msp^\PPP\subset \Msp$ such that
\[ \xymatrix @C=2cm { \Msp^{\PPP}\times \Msp^\PPP \ar[r]^{\oplus} \ar@{^{(}->}[d] & \Msp^\PPP \ar@{^{(}->}[d] \\ \Msp\times\Msp \ar[r]^{\oplus} & \Msp }\]
is cartesian. We denote with $\Ab^{\PPP}\subset \Ab$ the full exact subfunctor such that $E\in \Ab_S$ is in $\Ab_S^{\PPP}$ if and only if every fiber of $E$ over $s\in S$ has property $\PPP$. The category $\Ab^{\PPP}$ satisfies all the properties of $\Ab$. If $X_d=\overline{X}_d^{ss}$ for all $d\in \NN^{\oplus I}$, then $\oplus:\Msp^{\PPP}\times \Msp^{\PPP}\to \Msp^{\PPP}$ is still a finite morphism as  being finite is stable under base change. This will be the case in the following two examples.
\end{example}
\begin{example} \rm An example for an open property $P$ in the quiver case is given by the requirement that a finite set of cycles act invertibly. An example for a closed property $P$ in the quiver case is given by the requirement that an arbitrary set of cycles act nilpotently. 
\end{example}
\begin{example} \rm In the case of 1-dimensional sheaves on a ruled surface $\Proj_C(\Lin\oplus \OO_C)$, such an open property $P$ is given by the requirement that the support of the sheaf does not meet the divisor at infinity, i.e.\ the sheaf is located on the vector bundle $\Spec_C \Sym \Lin$ associated to $\Lin$. Conversely, if we require that the support of the sheaves is set-theoretically contained in a closed subscheme, this defines a closed property $\PPP$ as above.
\end{example}

\subsection{Framings} 
\label{Sec_fr_rep}
As before, we just recall the notion and some facts about framings. Details can be found in \cite{Meinhardt4}. Let $(\Ab,\omega)$ be a pair as above satisfying properties (1)--(6). Given a framing vector $f\in \NN^I$, we define a new pair $(\Ab_f,\omega_f)$ as follows. Objects in $\Ab_{f,S}$ are triples $(E,W,h)$ with $E\in \Ab_S$, $W\in \Vect^I_S$ and a morphism $h:W\to \omega_S(E)$ of locally free sheaves on $S$. A morphism from $(E,W,h)$ to $(E',W',h')$ is a pair of morphisms $\phi:E\to E'$ and $\phi_\infty:W\to W'$ such that the obvious square commutes. We put $\hat{I}=I\sqcup \{\infty\}$ and $\omega_{f,S}(E,W,h):=\omega_S(E)\oplus W$ with $W$ corresponding to the index $\infty$.  
\begin{proposition}
 The pair $(\Ab_f,\omega_f)$ satisfies all the properties (1)--(6).
\end{proposition}
In fact, we are not interested in all framed objects, but in those with $W$ having rank one. 
\begin{proposition} There is a quasiprojective scheme $X_{f,d}$ whose $\KK$-points correspond to tuples $(E,\psi,W\cong \KK,h)$ (up to equivalence) with $E\in \Ab_\KK$, $\psi:\omega_\KK(E)\xrightarrow{\sim} \KK^d$ and $h:\KK\to \KK^d$ corresponding to a vector in $\KK^d$ such that there are no proper subobjects  $(E',\psi',W'\cong \KK,h')$ of $(E,\psi,W\cong \KK,h)$. The group $G_d$ acts freely on $X_{f,d}$ in the obvious way. There is a geometric quotient $X_{f,d}/G_d\cong \Msp_{f,d}$ in the category of quasiprojective schemes. Mapping  a tuple $(E,\psi,W\cong \KK,h)$ to $(E,\psi)$, we obtain a $G_d$-equivariant morphism $X_{f,d}\to X_d$ which induces a projective morphism $\pi_{f,d}:\Msp_{f,d}\to \Msp_d$ for every dimension vector $d$, called the (non-commutative) Hilbert--Chow morphism. The fiber over $\Msp^s_d$ is isomorphic to  $\PP^{fd-1}$. 
\end{proposition}
One of the main theorems of \cite{Meinhardt4} which is a straightforward generalization of the corresponding result for quivers proven in \cite{MeinhardtReineke} is the following result.
\begin{theorem}
 If, moreover, $\Ab$ satisfies (7) and (8), then $\Msp_{f,d}$ is smooth and the map $\pi_{f,d}:\Msp_{f,d}\to \Msp_d$ is virtually small, that is, there is a finite stratification $\Msp_d=\sqcup_\lambda S_\lambda$ with empty or dense stratum $S_0=\Msp^{s}_d=X^{st}_d/PG_d$ such that $\pi_{f,d}^{-1}(S_\lambda) \longrightarrow S_\lambda$ is \'etale locally trivial and 
 \[ \dim \pi_{f,d}^{-1}(x_\lambda) - \dim \PP^{f\cdot d-1} \le \frac{1}{2} \codim S_\lambda\] 
for every $x_\lambda\in S_\lambda$ with equality only for $S_\lambda=S_0\not=\emptyset$ with fiber $\pi_{f,d}^{-1}(x_0)\cong \PP^{f\cdot d-1}$. 
\end{theorem}

\subsection{Potentials} \label{potential}

We will now introduce two versions of potentials which are related to each other by means of the trace homomorphism. The first definition of a  potential depends only on $\Ab$ while the second  depends on a pair $(\Ab,\omega)$ as above. 
\begin{definition} Let $\Ab$ satisfy our assumptions (1)--(4) giving rise to a commutative monoid $(\Msp,\oplus,0)$ in the category of schemes over $\kk$. A potential for $\Ab$ is a monoid homomorphism $W:(\Msp,\oplus,0) \longrightarrow (\AA^1,+,0)$.  We denote with $\Hom(\Msp,\AA^1)$ the set  of potentials for $\Ab$. It has a natural structure of a $\kk$-vector space. 
\end{definition}
To avoid confusion, we will sometimes call potentials for $\Ab$ also ``generalized'' potentials. Thus, a (generalized) potential for $\Ab$ is just an additive character of $\Msp$. For the second definition of a potential  we need to fix a fiber functor $\omega:\Ab\to \Vect^I$.
\begin{definition} Given a pair $(\Ab,\omega)$ satisfying assumptions (1) and (2), we define  $\End_\kk(\omega)$ to be the $\kk$-algebra  of $\OO$-linear natural transformations $\theta:\omega\to \omega$ with multiplication given by composition. 
\end{definition}
Spelling out the definition, $\theta\in \End_\kk(\omega)$ is given by a family $\theta^S_E:\omega_S(E)\to \omega_S(E)$ of $\OO_S$-linear maps such that for morphisms $f:S'\to S$ and $\alpha:E\to E'$ in $\Ab_S$ the following two diagrams commute
\[ \xymatrix { \omega_{S'}(f^\ast (E)) \ar[r]^{\theta^{S'}_{f\ast(E)}} \ar[d]^\wr & \omega_{S'}(f^\ast(E)) \ar[d]^\wr &  & \omega_S(E)\ar[d]^{\omega_S(\alpha)} \ar[r]^{\theta^S_E} & \omega_S(E) \ar[d]^{\omega_S(\alpha)} \\ f^\ast\omega_S(E) \ar[r]^{f^\ast \theta^S_E} & f^\ast\omega_S(E) &  & \omega_S(E')\ar[r]^{\theta^S_{E'}} & \omega_S(E'). } \] 
Similar to the previous definition, we define a $\kk$-linear category $\Cat(\Ab,\omega)$ associated to $(\Ab,\omega)$ as follows.
\begin{definition} The set of objects of the category $\Cat(\Ab,\omega)$ is $I$. For $(i,j)\in I$ we define the $\kk$-vector space of morphisms $\Cat(\Ab,\omega)(i,j)$ from $i$ to $j$ as the $\kk$-vector space of $\OO$-linear natural transformations $\vartheta:\omega_i\to \omega_j$. Composition of morphisms is given by composition of natural transformations.
\end{definition}
\begin{example}\rm
Given a   $\kk$-linear category $A_\kk$ with $\obj(\Ab)=:I$ being a set. Let $\Ab=A_k\Mod$ be  the functor of $A_\kk$-modules mapping a $\kk$-scheme $S$ to the $\OO_S(S)$-module of $\kk$-linear functors $V:A_\kk\longrightarrow \Vect_S$. There is a natural $\kk$-linear functor $A_\kk\to \Cat(A_\kk\Mod,\omega)$ which is the identity on objects. By a standard argument using Zorn's lemma, one can show that this functor is faithful, i.e.\ $A_\kk \hookrightarrow \Cat(A_\kk\Mod,\omega)$ is an embedding. Conjecturally, this is  an isomorphism. Notice that the category of natural transformations between the graded components $(\omega_{\kk})_i$ of  $\omega_\kk$ is much larger. But requiring this for all $\omega_S=\oplus_{i\in I}(\omega_S)_i$ in a compatible way, puts severe restrictions.  
\end{example}
By construction $\End_\kk(\omega)\cong \prod_{i\in I}\End_{\Cat(\Ab,\omega)}(i)$ as $\kk$-algebras. Moreover, \\ $\bigoplus_{i,j\in I} \Hom_{\Cat(\Ab,\omega)}(i,j)$ is a $\kk$-algebra without unit if $|I|=\infty$. Abusing notation, we also denote this algebra with $\Cat(\Ab,\omega)$. Let $[\Cat(\Ab,\omega),\Cat(\Ab,\omega)]$ be the $\kk$-linear span of all commutators in $\bigoplus_{i,j\in I} \Hom_{\Cat(\Ab,\omega)}(i,j)$. Let $[\Cat(\Ab,\omega),\Cat(\Ab,\omega)]^\circ$ be the intersection of $[\Cat(\Ab,\omega),\Cat(\Ab,\omega)]$ with $\bigoplus_{i\in I} \End_{\Cat(\Ab,\omega)}(i)=:\End^\circ_\kk(\omega)$.
\begin{lemma}
 The inclusion $\End^\circ_\kk(\omega)\hookrightarrow \Cat(\Ab,\omega)$ induces an isomorphism \[\End^\circ_\kk(\omega)/[\Cat(\Ab,\omega),\Cat(\Ab,\omega)]^\circ \cong \Cat(\Ab,\omega)/[\Cat(\Ab,\omega),\Cat(\Ab,\omega)]  \] of $\kk$-vector spaces.
\end{lemma}
As $\End^\circ(\omega)\subset \End(\omega)$, there is an injective $\kk$-linear map 
\[\Cat(\Ab,\omega)/[\Cat(\Ab,\omega),\Cat(\Ab,\omega)] \hookrightarrow \End_\kk(\omega)/[\Cat(\Ab,\omega),\Cat(\Ab,\omega)]^\circ\]
due to the previous lemma which is an isomorphism if and only if  $|I|<\infty$. Note that $1:=\id_{\omega}$ represents an element in the left hand side which is not in the right hand side if $|I|=\infty$. 
\begin{definition}
 We define $\Pot(\Ab,\omega)$ to be the $\kk$-vector space
 \[ \End_\kk(\omega)/\big( [\Cat(\Ab,\omega),\Cat(\Ab,\omega)]^\circ + [\End_\kk(\omega),\End_\kk(\omega)]\big)\]
 whose elements are called potentials for $(\Ab,\omega)$ and denoted with $\WW$.
\end{definition}
If $|I|<\infty$, then $\Pot(\Ab,\omega)= \Cat(\Ab,\omega)/[\Cat(\Ab,\omega),\Cat(\Ab,\omega)]$. Given a $\kk$-linear category $A_\kk$ with object set $I$, there is a canonical $\kk$-linear morphism $A_\kk/[A_\kk,A_\kk]\longrightarrow \Pot(\Ab,\omega)$. We do not know if this is injective or surjective even if $I$ is finite.\\
There is a morphism $\Tr:\Pot(\Ab,\omega) \longrightarrow \Hom(\Msp,\AA^1)$ of $\kk$-vector spaces given as follows. We fix an element $\theta\in\End_\kk(\omega)$ as above and consider the universal trivialized object of dimension $d$ on $X_d$, i.e.\ an object $E\in \Ab_{X_d}$ along with an isomorphism $\psi:\omega_{X_d}(E)\xrightarrow{\sim}\OO_{X_d}^d$ having some universal property. Consider the function $\Tr(\psi \theta^{X_d}_E \psi^{-1})$ on $X_d$. It is certainly $\kk$-linear in $\theta$ and $G_d$-invariant. By the universal property of $\Msp_d$, it defines a morphism $\Msp\to \AA^1$ for fixed $\phi$.  As a function in $\phi$ it descends to a function on $\Pot(\Ab,\omega)$ by the properties of the trace. 
\begin{definition}
The $\kk$-linear morphism $\Tr:\Pot(\Ab,\omega)\to \Hom(\Msp,\AA^1)$ mapping a potential for $(\Ab,\omega)$ to a (generalized) potential for $\Ab$ is called the trace homomorphism for $(\Ab,\omega)$. 
\end{definition}
\begin{example}\rm
For $\WW=1$ we get $\Tr(\WW)(E)=|\dim( E)|:=\sum_{i \in I}\dim_\KK\omega_\KK(E)_i \in \KK$ for every $\KK$-point $E\in \Msp$. Similarly, if $\WW=e_i$ is the identity on $\omega(-)_i$ and zero else, then $\Tr(\WW)=\dim(E)_i=\dim_\KK \omega_\KK(E)_i$. 
\end{example}
Given a potential $W$ for $\Ab$, we define define $\Mst^W$ to be the closed substack in $\Mst$ such that $\Mst^W_d=\Crit(Wp_d\rho_d)/G_d$ for all $d\in \NN^{\oplus I}$, using the notation $X_d\xrightarrow{\rho_d}\Mst_d\xrightarrow{p_d}\Msp_d\xrightarrow{W}\AA^1$. Moreover,  the full subcategory $\Ab^W_S\subset \Ab_S$ is given by all objects $E\in \Ab_S$ such that the classifying map $S\to \Mst$ factorizes through $\Mst^W$. This subcategory is closed under subquotients once they exist, but not necessarily under extensions. Finally, $\Msp^W_d=\Crit(Wp_d\rho_d)/\!\!/G_d$ is the image of $\Mst^W$ under $p:\Mst\to \Msp$. It is a closed subscheme of $\Msp$. 
\begin{example}
 In the case of quivers let $\WW\in \kk Q/[\kk Q,\kk Q]$ be a potential for $(\kk Q\rep,\omega)$. Then $\Mst^{\Tr(\WW)}$ is the stack of $\kk Q$-modules such that the non-commutative derivatives of $\WW$ act by zero.  
\end{example}

\section{($\lambda$-)ring and ($\lambda$-)algebra theories} \label{theories}

The aim of this section is to introduce the formal setting in which Donaldson--Thomas theory will be formulated. We try to be as general as possible for two reasons. First of all, the formal structure of Donaldson--Thomas theory will be more transparent, and it is quite remarkable that moduli spaces of objects in $\Ab$ fit perfectly into this framework. Secondly, we wish to provide a framework which allows many applications and various realizations, and it would be very interesting to see new examples of this kind. \\

\subsection{($\lambda$-)ring theories}

Let $\Sch_\kk$ be the category of schemes $X=\sqcup_{X_j \in \pi_0(X)}X_j$ over $\Spec \kk$ with possibly infinitely many quasiprojective connected components $X_j$. In particular, a morphism  $u:Y\to X$ is proper if and only if its restriction $u^{-1}(X_j)\to X_j$ to every connected component $X_j$ of $X$ is projective in the sense of Hartshorne \cite{Hartshorne}, i.e.\ has a factorization $u^{-1}(X_j)\hookrightarrow X_j\times \PP^{n_j} \xrightarrow{\pr} X_j$. The restriction to quasiprojective components is not essential but this assumption is always fulfilled in our examples and allows more ``theories''. Similarly, given $\MM\in \Sch_\kk$, $\Sch_\MM$ denotes the category of morphisms $f:X\to M$ inside $\Sch_\kk$. \\
Let $\Sm, \Pro$ be two subcategories of $\Sch_\kk$ having the following properties:
\begin{enumerate}
\item $\obj(\Sm)=\obj(\Pro)=\obj(\Sch_\kk)$, in other words, $\Sm$ and $\Pro$ specify a class of morphisms closed under composition.
\item All isomorphisms are in $\Sm$ and $\Pro$.
\item For every family $(u_j:X_j\to Y_j)_{j \in J}$ of morphisms in $\Sm$ resp.\ $\Pro$, the direct sum $\sqcup_{j \in J}u_j:\sqcup_{j \in J}X_j \longrightarrow \sqcup_{j \in J}Y_j$ is in $\Sm$ resp.\ $\Pro$. The codiagonal $\nabla_X:X^{\sqcup J} \longrightarrow X$ is in $\Sm$ for every index set $J$.    
\item For morphisms $u,v\in\Pro$ between schemes over $\MM$, the fiber product $u\times_\MM v$ is a morphism in $\Pro$. Similarly, $u,v\in \Sm$ defined over $\MM$ implies $u\times_\MM v\in \Sm$.
\item For $u\in \Pro$ and $v\in \Sm$ we have $\tilde{u}\in \Pro$ and $\tilde{v}\in \Sm$ with $\tilde{u}$ and $\tilde{v}$ denoting the pull-backs
\[ \xymatrix { X\times_Z Y \ar[d]_{\tilde{u}} \ar[r]^{\tilde{v}} & X \ar[d]^u \\ Y \ar[r]_v & Z.} \]
\item  If the composition $vu:X\xrightarrow{u} Y\xrightarrow{v} Z$ is in $\Pro$ and $u$ is finite, then $v\in \Pro$. 
\item Every open embedding is in $\Sm$, and $\Sm$ is stable under pull-backs along open inclusions.
\item The class of finite morphisms is in $\Pro$, and $\Pro$ is stable under pull-backs along finite morphisms. Every morphism in $\Pro$ is of finite type.
\end{enumerate}
By the last property, every closed inclusion is in $\Pro$, and $\Pro$ is stable under pull-backs along closed embeddings. As a consequence of (6) and (8) we see that $\Sym^n(u):\Sym^n(X)\longrightarrow \Sym^n(Y)$ is in $\Pro$ if $X\xrightarrow{u}Y$ is in $\Pro$, where $\Sym^n(X)=X/\!\!/S_n$ denotes the $n$-th symmetric product of $X$. Since all schemes in $\Sch_\MM$ have connected components of finite type over $\kk$, the last condition in (8) is equivalent to the requirement that the map $\pi_0(u):\pi_0(X)\to \pi_0(Y)$ between the sets of connected components has finite fibers for every $u\in \Pro$.\\  
For a family $(u_j:X_j\to Y_j)_{j \in J}$ of morphisms in $\Sm$ with the same target $Y$, the morphism $\sqcup_{j \in J} X_j \longrightarrow Y$ is also in $\Sm$ by property (3). This is also true if $J$ is finite since then $X^{\sqcup J} \to X$ is finite and, therefore, also in $\Pro$.  \\ 
Given a scheme $\MM$ in $\Sch_\kk$, we denote by $\Sm_\MM$ and  $\Pro_\MM$ the preimages of $\Sm$ and $\Pro$ under the forgetful functor $\Sch_\MM\ni (X\to \MM)\longmapsto X\in \Sch_\kk$. In other words, $\Sm_\MM$ and $\Pro_\MM$ denote the class of morphisms in $\Sm$ and $\Pro$ respectively, which are defined over $\MM$. 
\begin{example} \rm
The first example is given by $\Sm=\Sch_\kk$ and $\Pro=ft$ being the class of all morphisms of finite type. By our assumptions, this is the biggest choice of $\Sm$ and $\Pro$. The restriction on $\Pro$ was made to obtain a ``sheafification'' functor for abelian group valued $(\Sm,\Pro)$-(pre)theories as we will see later. The smallest choice for $\Pro$ is certainly the class of all finite morphisms.    
\end{example}
\begin{example} \rm
The second example the reader should have in mind is the one where $\Sm$ is the class $Sm$ of smooth morphisms and $\Pro=prop$ denotes the class of proper morphisms. All the properties can be checked easily except maybe for property (6).  This can be checked with the help of  the valuation criterion using the following two facts. Firstly, finite morphisms are surjective on points. However, the residue field of a lift of a point $x$ might be a finite  extension of the residue field of $x$. But secondly, if $R$ is a valuation ring with quotient field $K$ and $L\supset K$ a finite field extension, then $R\subset L$ is dominated by a valuation ring $S$ for $L$.    
\end{example}

\begin{definition}  \label{theory}
A  (set valued) $(\Sm,\Pro)$-pretheory over $\MM$ is a rule $R$ associating to every object $X\xrightarrow{f}\MM$ in  $\Sch_\MM$ a set $R(f)$ which extends to a  contravariant functor on $\Sm_\MM$ and to a covariant functor on $\Pro_\MM$, i.e.\ for every commutative diagram 
\[ \xymatrix @C=0.5cm {Y \ar[dr]^g \ar[rr]^u && X \ar[dl]_f\\ & \MM }  \]
there is a map $u^\ast:R(f)\longrightarrow R(g)$ if $u\in \Sm$ and  a map $u_!:R(g)\longrightarrow R(f)$ if $u\in \Pro$. Moreover, we require that ``base change'' holds i.e.\ for every cartesian diagram 
\[ \xymatrix { X\times_Z Y \ar[rr]^{\tilde{v}} \ar[dd]_{\tilde{u}} \ar[dr] & & X \ar[dd]^u \ar[dl]_f \\ & \MM & \\ Y \ar[rr]_v \ar[ur]_g & & Z \ar[ul]_h } \]
with $u\in \Pro$ and $v\in \Sm$, we have $\tilde{u}_!\circ \tilde{v}^\ast = v^\ast \circ u_!$.\\ 
A ($\Sm,\Pro)$-pretheory $R$ is called a $(\Sm,\Pro)$-theory if in addition, considered as a contravariant functor to sets, $R$ commutes with all (not necessarily finite) ``products'', i.e.\ the morphism
\[ R(X\xrightarrow{f}\MM) \longrightarrow \prod_{X_j\in\pi_0(X)=\pi_0(f)} R(X_j\xrightarrow{f|_{X_j}}\MM) \]
given by restriction to connected components is an isomorphism for all $f\in \Sch_\MM$.\\
A morphism between two  $(\Sm,\Pro)$-(pre)theories $R_1,R_2$ over $\MM$ is a natural transformation $\eta:R_1\to R_2$, i.e.\ a family of maps $\eta_f:R_1(f)\to R_2(f)$ such that   $u_2^\ast \circ\eta_f=\eta_g \circ u_1^\ast$ holds for every morphism $u:g\to f$ in $\Sm_\MM$, and  $u_{2\,!}\circ \eta_g=\eta_f\circ u_{1\,!}$ holds for every  morphism $u:g\to f$ in $\Pro_\MM$. Thus, we obtain a category of $(\Sm,\Pro)$-(pre)theories over $\MM$. 
\end{definition}

\begin{example} \rm
If we had not required that $\Pro$ contains all finite morphisms, then we could have chosen $\Sm=\Sch_\kk$ and $\Pro\subset \Sch_\kk$ to be the subcategory of isomorphisms. In this case, a $(\Sm,\Pro)$-pretheory is  just  a presheaf on $\Sch_\kk$, and a $(\Sm,\Pro)$-theory is a sheaf on $\Sch_\kk$ with respect to a suitable rather coarse Grothendieck topology. The assumptions on $\Pro$ were made in order to have good algebraic structures at hand (cf.\ Lemma \ref{algebraic_structure}) which will not exist for ordinary (pre)sheaves.    
\end{example}
Let us give an alternative definition by introducing the category $\Sm\Pro^{op}_\MM$ of ``correspondences'' over $\MM$. The objects of $\Sm\Pro^{op}_\MM$ and of $\Sch_\MM$ are the same, but morphisms from $X\xrightarrow{f}\MM$ to $Y\xrightarrow{g}\MM$ are isomorphism classes of correspondences
\[\xymatrix @R=0.5cm { & Z \ar[dl]_p \ar[dr]^s \ar[dd] & \\ X \ar[dr]_f & & Y \ar[dl]^g \\ & \MM & }\]
with $p\in \Pro$ and $s\in \Sm$. The composition of two correspondences $X \xleftarrow{p}Z\xrightarrow{s}Y$ and $Y \xleftarrow{p'}Z'\xrightarrow{s'}V$ is defined in the usual way by forming fiber products $X \xleftarrow{p\tilde{p}'} Z\times_Y Z' \xrightarrow{s'\tilde{s}}V$. By construction, we have a faithful covariant functor $(-)_\sharp:\Sm_\MM \to \Sm\Pro^{op}_\MM$ and a faithful contravariant functor $(-)^\sharp:\Pro_\MM \to \Sm\Pro^{op}_\MM$. Moreover, $\Sm_\MM$ has a Grothendieck topology. A cover of a scheme $X\to \MM$ over $\MM$ is given by a family $(X_\kappa\to \MM)_{\kappa\in K}$ of disjoint unions of connected components of $X$ (over $\MM$) such that each point of $X$ is contained in at least one of these $X_\kappa$.

\begin{lemma}  \label{factorization}
 A $(\Sm,\Pro)$-pretheory over $\MM$ is nothing else than a presheaf on $\Sm\Pro^{op}_\MM$. A $(\Sm,\Pro)$-theory over $\MM$ is just a presheaf on $\Sm\Pro^{op}_\MM$ such that its restriction to $\Sm_\MM$ along $(-)_\sharp$ is a sheaf on $\Sm_\MM$. The category $\Sm\Pro^{op}_\MM$ has arbitrary direct sums computed as in $\Sch_\MM$. Moreover, finite sums and finite products agree.
\end{lemma}
\begin{proof}
 The first part is obvious. For the second part we use property (3) of $(\Sm,\Pro)$ and the fact that finite codiagonals are also in $\Pro$. 
\end{proof}
Using the fact that $\Pro$ contains finite morphisms, one can prove that there is no  Grothendieck topology on $\Sm\Pro^{op}_\MM$ such that sheaves correspond to $(\Sm,\Pro)$-theories and vice versa.  
Even more is true, every topology on $\Sm\Pro^{op}_\MM$ such that $\Sm_\MM\to \Sm\Pro^{op}_\MM$ is continuous, i.e.\ such that every sheaf on $\Sm\Pro^{op}_\MM$ is a $(\Sm,\Pro)$-theory, must be the discrete topology whose only sheaf is the trivial one $R:f\mapsto \{\star\}$. To see this, let us fix $\MM=\Spec\kk$ for simplicity and denote by $\underline{(-)}$ the sheafification functor.  Given an arbitrary sheaf $R$, the following must hold 
\[ R(\underline{X\sqcup X})=R(X\sqcup X)=R(X)\times R(X)=R(\underline{X})\times R(\underline{X})\] 
for every scheme $X$ by assumption on our topology. Hence, $\underline{X\sqcup X}$ must be the disjoint union of $\underline{X}$ with itself. Consider now the co-diagonal $\nabla:X\sqcup X\rightarrow X$ which is finite and, therefore, in $\Pro$. It induces a morphism $\nabla^\sharp:X\to X\sqcup X$ in $\Sm\Pro^{op}_\MM$. The reader should check  that the pull-back of the two inclusions $\iota_{1,2\,\sharp}:X\to X\sqcup X$ along $\nabla^\sharp$ in the category (of presheaves on) $\Sm\Pro^{op}_\MM$ is $\emptyset$. As sheafification commutes with colimits and finite limits, we see that the pull-back of the disjoint union $\underline{X}\sqcup\underline{X}=\underline{X\sqcup X}$ along $\underline{\nabla^\sharp}$ is $\emptyset\sqcup \emptyset=\underline{X}$. This implies $\underline{X}=\emptyset$, and $R(X)=R(\underline{X})=R(\emptyset)=\{\star\}$ follows for every sheaf $R$ proving the claim.
Nevertheless there is a ``sheafification functor'' for $(\Sm,\Pro)$-(pre)theories with values in abelian groups as we will see soon. 
\begin{remark} \rm 
From that perspective, the situation described here is very similar to the construction of the triangulated category of geometric motives by Voevodsky, where $\Sm$ is replaced with the category of smooth schemes over $\kk$ equipped with the Nisnevich topology and the role of $\Sm\Pro^{op}_\kk$ is also played by some category of correspondences between smooth schemes. The ``sheafification'' functor for (complexes of) abelian group valued presheaves with transfers by means of the Suslin complex does not work for set valued presheaves with transfer. See \cite{Voevodsky} for more details.  
\end{remark}
 
\begin{example}\rm
 Using the Yoneda embedding of $\Sm\Pro^{op}_\MM$, we obtain a $(\Sm,\Pro)$-pretheory for every scheme $X\xrightarrow{f} \MM$ over $\MM$. Applying  Lemma \ref{factorization}, one can show that the presheaf associated to $f$ is actually a $(\Sm,\Pro)$-theory giving rise to a fully faithful embedding of $\Sm\Pro^{op}_\MM$ into the category of $(\Sm,\Pro)$-theories over $\MM$.  
\end{example}
\begin{example} \rm Fixing $\MM$, we denote by $\Sm^{\Pro}$ the $(\Sm,\Pro)$-pretheory over $\MM$ such that  $\Sm^{\Pro}(X\xrightarrow{f}\MM):=\Sm^{\Pro}_X$ is the set of all morphisms $U\to X$ in $\Pro$ with $U\to \Spec\kk$ being in $\Sm$. As the definition does not involve $\MM$, we dropped it from notation, hoping not to confuse the reader. If $\MM=\Spec\kk$, $\Sm^{\Pro}$ is just the $(\Sm,\Pro)$-theory represented by $\Spec\kk$ via the Yoneda embedding. In general it is not representable but the pull-back of the $(\Sm,\Pro)$-theory represented by $\Spec\kk$ along the morphism $\MM\to \Spec\kk$, as we will see in section \ref{pushforwards_pull-backs}. We leave it to the reader to show that $\Sm^{\Pro}$ is in fact a $(\Sm,\Pro)$-theory over arbitrary $\MM$.   
\end{example} 
\begin{example} \rm
 The category of $(\Sm,\Pro)$-pretheories has all (small) limits and colimits computed objectwise. The full subcategory of $(\Sm,\Pro)$-theories is closed under limits, but not under direct sums taken in the category of pretheories. Nevertheless, the Yoneda embedding into the category of $(\Sm,\Pro)$-theories preserves not only limits but also arbitrary direct sums, which is not a contradiction to the previous assertion as the inclusion functor of the category of theories into the category of pretheories does not preserve direct sums. Recall that finite products and finite direct sums agree in $\Sm\Pro^{op}_\MM$.     
\end{example}

Following the general machinery, we can define abelian group valued $(\Sm,\Pro)$-(pre)theories either as corresponding functors to abelian groups, as abelian group  objects in the category of $(\Sm,\Pro)$-(pre)theories or as modules of the free pre-additive category $\ZZ\Sm\Pro^{op}_\MM$ generated by $\Sm\Pro^{op}_\MM$ (satisfying a ``sheaf condition''). The following lemma shows that abelian group valued $(\Sm,\Pro)$-theories behave much better than their set valued counterparts.
\begin{lemma}
There is a ``sheafification'' functor for abelian group valued $(\Sm,\Pro)$-theories over $\MM$, i.e.\ an exact functor $R\mapsto \underline{R}$ from the category of abelian group valued $(\Sm,\Pro)$-pretheories to the full subcategory of abelian group valued $(\Sm,\Pro)$-theories which is left adjoint to the inclusion functor. In particular, it preserves all colimits, and the category of abelian group valued $(\Sm,\Pro)$-theories has all (small) colimits obtained by sheafification of objectwise colimits. 
\end{lemma}
\begin{proof}
We define $\underline{R}(X\xrightarrow{f}\MM):=\prod_{X_j\in \pi_0(X)}R(f|_{X_j})$ and extend this to a contravariant functor on $\Sm_\MM$ in the obvious way. If $u:X\to Y$ is a morphism in $\Pro$, the fibers of $\pi_0(u):\pi_0(X)\to \pi_0(Y)$ are finite by assumption (8) on $\Pro$. Thus, given $Y\to \MM$ and $Y_i\in \pi_0(Y)$, the map
\[ \prod_{X_j\in \pi_0(u)^{-1}(Y_i)} R(X_j \to \MM)\ni (a_j)_j \to \sum_j (u|_{X_j})_!(a_j) \in R(Y_i\to \MM) \]
is well-defined using the additive structure on $R$. By taking the product over all $Y_i\in \pi_0(Y)$, we finally obtain the map $u_!:\underline{R}(X\xrightarrow{gu}\MM)\to \underline{R}(Y\xrightarrow{g}\MM)$. We leave it to the reader to show that $\underline{R}$ is a $(\Sm,\Pro)$-theory.
\end{proof}
\begin{example}\rm For a morphism $Y\xrightarrow{g}\MM$, we denote with $\ZZ g$ the representable abelian group valued $(\Sm,\Pro)$-pretheory given by the object $g\in \ZZ\Sm\Pro^{op}_\MM$. For $X\xrightarrow{f} \MM$, $\ZZ g(f)$ is just the free abelian group generated by all morphisms from $f$ to $g$, i.e.\ by all isomorphism classes of correspondences 
  \[ \xymatrix @R=0.4cm { & \ar[ld]_p Z \ar[dd]\ar[rd]^s & \\ X \ar[dr]_f  &  & Y \ar[dl]^g \\ & \MM } \]
 with $p\in \Pro$ and $s \in \Sm$. The sheafification of $\ZZ g$ will be denoted by $\underline{\ZZ}g$. Every other abelian group valued $(\Sm,\Pro)$-(pre)theory $R$ is a colimit of such representable $(\Sm,\Pro)$-(pre)theories. For $f:X\to \Spec\kk$ we simply write $\underline{\ZZ}X$. 
\end{example}
\begin{example} \rm Fixing $\MM$, we denote with $\underline{\ZZ}(\Sm^{\Pro})$ the sheafification of the $(\Sm,\Pro)$-pretheory $\ZZ(\Sm^{\Pro})$ over $\MM$ such that  $\ZZ(\Sm^{\Pro})(X\xrightarrow{f}\MM):=\ZZ(\Sm^{\Pro}_X)$ is the free abelian group  generated by isomorphism classes of morphisms $U\to X$ in $\Pro$ with $U\to \Spec\kk$ being in $\Sm$. The pull-back is given by fiber products and the push-forward linearly extends the composition of morphisms. If $\MM=\Spec\kk$, $\underline{\ZZ}(\Sm^{\Pro})$ is just $\underline{\ZZ}\kk:=\underline{\ZZ}\Spec\kk$. In general it is the pull-back of $\underline{\ZZ}\kk$ along the morphism $\MM\to \Spec\kk$ as we will see in section \ref{pushforwards_pull-backs}. 
\end{example}  

In a similar way we can define ($\lambda$-)ring valued $(\Sm,\Pro)$-(pre)theories, but as for set valued $(\Sm,\Pro)$-(pre)theories there will be no ``sheafification'' functor in general. Note that the sheafification functor defined above gives a $(\lambda$-)ring valued functor with respect to pull-backs but not with respect to push-forwards. For our purposes, we need the notion of a ($\lambda$-)ring $(\Sm,\Pro)$-(pre)theory which should not be confused with a ($\lambda$-)ring valued $(\Sm,\Pro)$-(pre)theory. For this  
let $(\MM,+,0)$ denote a (commutative) monoid in the category $\Sch_\kk$. 
\begin{example} \label{trivMonEx} \rm The commutative monoids which the reader should have in mind were defined in section 2 and shown in the following diagram along with (natural) monoid homomorphisms

\[ \xymatrix @C=1cm {  & & & \NN^{\oplus I} \ar[dr] & \\   \Msp \ar[rr]^{\dim\times W} \ar@/^/[urrr]^{\dim} \ar@/_/[drrr]_{W} & & \NN^{\oplus I}\times \AAA \ar[rr] \ar[ur]^{\pr} \ar[dr]_{\pr} & & \Spec\kk. \\ & & & \AAA \ar[ur] } \] 
The ``multiplication'' is always given by taking (direct) sums.
\end{example}

If $(\MM,+,0)$ is a (commutative) monoid in $\Sch_\kk$,  $\Sm\Pro^{op}_\MM$ comes with a (symmetric) monoidal structure. The product of two morphisms $X\xrightarrow{\;f\;}\MM$, $Y\xrightarrow{\;g\;}\MM$ is defined via
\[ f\boxtimes g: X\times Y \xrightarrow{f\times g} \MM\times \MM \xrightarrow{\;+\;} \MM. \]
The unit object is given by $\Spec \kk \xrightarrow{\;0\;} \MM$.
Having a (symmetric) monoidal structure at hand, we can talk about (commutative) monoids in $\Sch_\MM$. Note that a (commutative) monoid in $\Sch_\MM$ is just a (commutative) monoid $(X,+,0)$ in $\Sch_\kk$ together with a monoid homomorphism $f:(X,+,0)\longrightarrow (\MM,+,0)$. 
\begin{example} \rm
All monoid homomorphisms $N\rightarrow M$ in Example \ref{trivMonEx} provide commutative monoids in the  category $\Sch_\MM$ for the corresponding $\MM$.
\end{example}

\begin{definition} \label{ring_theory}
A  ring $(\Sm,\Pro)$-(pre)theory over a monoid $\MM$ is an abelian group valued $(\Sm,\Pro)$-(pre)theory $R$ with a functorial  ``exterior'' product
\[ \boxtimes: R(f)\otimes R(g) \longrightarrow R(f\boxtimes g) \]
defined for every object $(f,g)\in \Sm\Pro^{op}_\MM\times\Sm\Pro^{op}_\MM$ which is bilinear, associative and comes with a unit $1\in R(\Spec \kk \xrightarrow{0} \MM)$. 
If $M$ is commutative, a ring $(\Sm,\Pro)$-(pre)theory $R$ over $\MM$ is called commutative if $\tau_!(a\boxtimes b )=\tau^\ast(a\boxtimes b) = b\boxtimes a$, where $\tau:f\boxtimes g \stackrel{\sim}{\to} g\boxtimes f$ is the transposition.
A morphism between two ring $(\Sm,\Pro)$-(pre)theories $R_1,R_2$ over $\MM$ is a natural transformation $\eta:R_1\to R_2$, i.e.\ a family of group homomorphisms $\eta_f:R_1(f)\to R_2(f)$ commuting with pull-backs along $s\in \Sm$ and push-forwards along $p\in\Pro$, such that   $\eta$ commutes with exterior products, i.e.\ $\eta_{\Spec\kk\to \MM}(1_1)=1_2$ and $\eta_{f\boxtimes g}\circ \boxtimes_1=\boxtimes_2 \circ (\eta_f\otimes \eta_g)$. Thus, we obtain a category  of (commutative) ring $(\Sm,\Pro)$-(pre)theories over $\MM$. 
\end{definition}
Throughout this paper we are only interested in commutative monoids and commutative ring $(\Sm,\Pro)$-theories. We denote the category of the latter by $\Th(\MM)$ and unless $\MM$ is non-commutative, a ``ring $(\Sm,\Pro)$-theory'' should always be commutative.
It is easy to see that the sheafification of the underlying abelian group valued $(\Sm,\Pro)$-pretheory  of a ring $(\Sm,\Pro)$-pretheory is again a ring $(\Sm,\Pro)$-theory in big contrast to ring valued $(\Sm,\Pro)$-pretheories. Moreover, the category of ring $(\Sm,\Pro)$-(pre)theories has all limits taken objectwise, and the forgetful functor preserves limits. The ring $(\Sm,\Pro)$-theory $R\equiv 0$ is the terminal object.  We do not know if colimits or free ring $(\Sm,\Pro)$-(pre)theories exist. However, there is an initial ring $(\Sm,\Pro)$-(pre)theory for arbitrary $\MM$. 
\begin{example} \rm 
If $(g,+,0)$ is a (commutative) monoid in $\Sm_\MM$, the representable $(\Sm,\Pro)$-theory  $\underline{\ZZ} g$ is a (commutative) ring $(\Sm,\Pro)$-theory. For $g:Y\xrightarrow{0}\MM$, we simply write $\underline{\ZZ} Y$. This applies in particular to $\Spec\kk \xrightarrow{0}\MM$ with its trivial monoidal structure, and the resulting ring $(\Sm,\Pro)$-theory will be denoted with $\underline{\ZZ}\kk$. Note that $\ZZ\kk(X\xrightarrow{f}\MM)$ is the free abelian group generated by isomorphism classes of maps $U\to X_0=f^{-1}(0)$ such that $U\to X_0\hookrightarrow X$ is in  $\Pro$ and $U\to \Spec\kk$ is in $\Sm$. 
\end{example}
\begin{lemma} \label{initial_object}
The  ring $(\Sm,\Pro)$-theory  $\underline{\ZZ}\kk$ is an initial object in the category $\Th(\MM)$ of ring $(\Sm,\Pro)$-theories  over $\MM$. 
\end{lemma}
\begin{proof} We write $R(X)$ for $R(X\xrightarrow{0}\MM)$, $\unit_X:=c^\ast(1)\in R(X)$ if $c:X\to \Spec\kk$ is in $\Sm$, and $\int_X a:=c_!(a)$ for $a\in R(X)$ if $c:X\to \Spec\kk$ is in $\Pro$.
Given a morphism $\eta:R_1:=\underline{\ZZ}\kk\rightarrow R_2$ between  ring $(\Sm,\Pro)$-theories  over $\MM$ and a generator $[U\xrightarrow{\alpha} X_0\hookrightarrow X]$ of $\underline{\ZZ}\kk(X\to \MM)$ for connected $X$, we obtain using the constant map $c:U\to \Spec\kk$
\begin{eqnarray*}\eta_X([U\to X])&=& \eta_X \alpha_{1\,!}(\unit_{1,U})\;=\;\alpha_{2\,!}\eta_U(\unit_{1,U})\;=\;\alpha_{2\,!}\eta_U c_{1}^\ast (1_1)\\
&=&\alpha_{2\,!}c_{2}^\ast\eta_{\Spec\kk}(1_1)\;=\;\alpha_{2\,!}c_{2}^\ast(1_2)\;=\;\alpha_{2\,!}(\unit_{2,U}).\end{eqnarray*}
Hence, $\eta$ is uniquely determined for connected $X$ and by the sheaf property  also for every $X\in \Sch_\kk$. Conversely, we can define a morphism $\unit:R_1=\underline{\ZZ}\kk\longrightarrow R_2$ by $\ZZ$-linear extension of $\eta_X([U\xrightarrow{\alpha}X]):=\alpha_{2\,!}(\unit_{2,U})$ for connected $X$ and using the sheaf property again. 
\end{proof} 
\begin{example}\rm The $(\Sm,\Pro)$-theory $\underline{\ZZ}(\Sm^{\Pro})$ is a ring $(\Sm,\Pro)$-theory for every $M$. The product is defined via the cartesian product of its canonical generators. It agrees with $\underline{\ZZ}\kk$ if and only if $\MM=\Spec\kk$. \end{example} 
\begin{example} \rm
 For $\MM=\Spec\kk$, $\Sm=Sm$ being the class of smooth morphisms and $\Pro=prop$ being the class of proper morphisms, there is a ``quotient'' ring $(Sm,prop)$-theory   $\underline{\Ka}_0(Sm^{prop})$ of $\underline{\ZZ}(\Sm^{prop})=\underline{\ZZ}\kk$ obtained by modding out the following relations. Let $\pi:\Bl_V U \to U$ be the blow-up of a smooth scheme $U$ in a smooth center $V$ with exceptional divisor $E\hookrightarrow \Bl_V U$. If $\alpha:U\to X$ is a projective morphism, we define a relation by the linear combination \[ [\Bl_V U\xrightarrow{\alpha \pi} X]-[E\xrightarrow{\alpha\pi|_E} X] \stackrel{!}{=} [U\xrightarrow{\alpha}X] - [V\xrightarrow{\alpha|_V} X]. \]
 We write $\underline{\Ka}_0(Sm_X^{prop})$ for $\underline{\Ka}_0(Sm^{prop})(X)$.
\end{example}
Note that the functor from the category of commutative monoids in $\Sch_M$ to the category $\Sch_M$, given by forgetting the commutative monoid structure, has a left adjoint associating to every $X\xrightarrow{\;f\;}\MM$ the free commutative monoid 
\[ \Sym(f):= \sqcup_{n\in \NN} \Sym^n(f) \mbox{ with } \Sym^n(f)=\Sym^n(X) \longrightarrow \MM, \]
``generated'' by $X\xrightarrow{\;f\;}\MM$ with $\Sym^n(X)\to \MM$ being induced by $X^n\xrightarrow{f\circ \pr_1 + \ldots +f\circ \pr_n }\MM$. These symmetric powers should remind the reader of $\lambda$-rings, motivating the following definition.  
\begin{definition}
A $\lambda$-ring $(\Sm,\Pro)$-(pre)theory over $\MM$ is a ring $(\Sm,\Pro)$-(pre)theory $R$ over $\MM$ as defined above with certain operations 
\[ \sigma^n: R(f) \longrightarrow R(\Sym^n( f)) \]
for every  $n\in \NN$ and $f\in \Sch_\MM$ commuting with push-forwards and pull-backs, i.e.\ $\sigma^n(u_!(a))=\Sym^n (u)_!(\sigma^n(a))$ for all $u\in \Pro_\MM$ and $\sigma^n(u^\ast(a))=\Sym^n(u)^\ast (\sigma^n(a))$ for all $u$ such that $\Sym(u)\in \Sm_\MM$.
 Using the monoidal structure $\oplus:\Sym(f)\boxtimes \Sym(f)\longrightarrow \Sym(f)$ and $0:\Spec\kk \to \Sym(f)$ on $\Sym(f)$, we define  the universal convolution product $ab\in R(\Sym(f))$ of $a,b\in R(\Sym(f))$ by means of $\oplus_!(a\boxtimes b)$, and  it is easy to see that the universal convolution product makes $R(\Sym(f))$ into a commutative ring with unit $1_f:=0_!(1)$. We will also require that $R(\Sym(f))$ equipped with the convolution product and the operations\footnote{Note that $\Sym^n(\Sym(f)) \longrightarrow \Sym(f)$ is finite.}
\[ \sigma^n: R(\Sym(f)) \xrightarrow{\sigma^n} R(\Sym^n(\Sym(f))) \xrightarrow{\oplus_!} R(\Sym(f)) \] 
is a  $\lambda$-ring. For a $\lambda$-ring $(\Sm,\Pro)$-theory $R$, the $\lambda$-ring $R(\Sym (f))$ is complete with respect to the filtration 
\[F^iR(\Sym (f)):= R(\sqcup_{n\ge i} \Sym^n(X) \rightarrow \MM) \hookrightarrow R(\sqcup_{n\ge 0} \Sym^n(X) \rightarrow \MM)=R(\Sym (f)).\]
A morphism between $\lambda$-ring $(\Sm,\Pro)$-(pre)theories is a morphism $\eta:R_1 \to R_2$ of the underlying ring $(\Sm,\Pro)$-(pre)theories such that additionally $\eta_{\Sym^n(f)}\circ \sigma^n_1=\sigma^n_2\circ \eta_{f}$ holds. We denote by $\Th^\lambda(M)$ the category of $\lambda$-ring $(\Sm,\Pro)$-theories over $\MM$. 
\end{definition}
It is easy to see that the sheafification $\underline{R}$ of a $\lambda$-ring $(\Sm,\Pro)$-pretheory $R$ is a $\lambda$-ring $(\Sm,\Pro)$-theory.
\begin{example} \rm \label{darstellbar}
If $Y\xrightarrow{g}\MM$ is a commutative monoid in $\Sm_\MM$, the set $g(X\xrightarrow{f}\MM)$ of isomorphism classes $[p,s]$ of correspondences $X \xleftarrow{p} Z \xrightarrow{s} Y$ over $\MM$ can be made into a commutative monoid if $f$ is a commutative monoid in $\Pro_\MM$. Indeed, given two correspondences $(p,s)$ and $(p',s')$ we  define $[p,s][p',s']=[+_f\circ (p\times p'), +_g \circ (s\times s')]$. As explained in Lemma \ref{special_elements}, the free abelian group $\ZZ g(f)$ has a unique $\lambda$-ring structure $\sigma_t$ such that $[p,s]\in \Pic(\ZZ g(f),\sigma_t)$ for all $[p,s]\in g(f)$. This applies in particular to the free commutative monoid $\Sym(f)$ turning $\ZZ g(\Sym(f))$ into a $\lambda$-ring for every $f\in \Sch_\MM$. The $\lambda$-ring structure descends to the sheafification making $\underline{\ZZ} g$ into a $\lambda$-ring $(\Sm,\Pro)$-theory. 
\end{example}
\begin{example} \rm
The same construction can be done for $\underline{\ZZ}(\Sm^{\Pro})$ making it into a $\lambda$-ring $(\Sm,\Pro)$-theory.  
\end{example}

\begin{example} \rm \label{motives2}
For $\MM=\Spec\kk$, $\Sm=\Sch_\kk$ and $\Pro=ft$ being the class of morphisms of finite type, there is a ``quotient'' ring $(\Sch_\kk,ft)$-theory $\underline{\Ka}_0(\Sch^{ft})$ of $\underline{\ZZ}(\Sch^{ft})=\underline{\ZZ}\kk$ obtained by modding out the cut and paste relation 
\[ [U\xrightarrow{\alpha} X]=[V\xrightarrow{\alpha|_V} X]+[U\setminus Z\xrightarrow{\alpha_{U\setminus V}} X] \]
in $\underline{\ZZ}(\Sch_X^{ft})$ for every $X\in\Sch_\kk$ and every for every closed subscheme $V\subseteq U$. This turns out to be a $\lambda$-ring $(\Sch_\kk,ft)$-theory with $\lambda$-operations $\sigma^n$ given as in Example \ref{motives1}. We write $\underline{\Ka}_0(\Sch_X^{ft})$ for $\underline{\Ka}_0(\Sch^{ft})(X)$. It is not true that $\underline{\ZZ}(\Sch^{ft})\to \underline{\Ka}_0(\Sch^{ft})$ is a homomorphism of $\lambda$-ring theories. There is a morphism of ring $(Sm,prop)$-theories  $\underline{\Ka}_0(Sm^{prop})\to \underline{\Ka}_0(\Sch^{ft})$ which is an isomorphism if $\Char \kk=0$ (see \cite{Bittner04}). 
\end{example}
\begin{example} \rm Let $\kk=\CC$.  There is a $\lambda$-ring $(\Sch_\kk,ft)$-theory  $\Con$ over $\kk$ such that $\Con(X)$ is the abelian group of constructible $\ZZ$-valued functions on $X$ for every $X\in\Sch_\kk$. The pull-back is the usual pull-back of functions, and the (proper) push-forward is the fiberwise integral with respect to the Euler characteristic $\chi_c$ with compact support. The exterior product $\Con(X)\otimes\Con(Y) \longrightarrow \Con(X\times Y)$ is the usual product of functions after pulling them back to the product $X\times Y$ (cf.\ Proposition \ref{algebraic_structure}), and the unit is the function with value $1$ on $\Spec\kk$. The operation $\sigma^n$ is uniquely determined by the requirement that $\sigma^n$ applied to the characteristic function of a locally closed subset $Z\subseteq X$ is the characteristic function of $\Sym^n(Z)\subseteq \Sym^n(X)$.     
\end{example}
\begin{example} \rm
Given a connected scheme $X$ and a generator $[U\xrightarrow{\alpha} X]$ of $\Ka_0(\Sch^{ft}_X)$, the function $X\ni x \to \chi_c(u^{-1}(x))\in \ZZ$ is constructible. By extending this assignment linearly and by passing to the sheafification, we get a morphism $\chi_c:\underline{\Ka}_0(\Sch^{ft})\longrightarrow \Con$ of $\lambda$-ring $(\Sch_k,ft)$-theories  over $\kk$. 
\end{example}
\begin{example}\rm
Given a ($\lambda$-)ring $(\Sm,\Pro)$-theory $R$ over a commutative monoid $\MM$, we define a new ($\lambda$-)ring $(\Sm,\Pro)$-theory $R^{red}$ over $\MM$ by means of $R^{red}(f)=R(X\xrightarrow{0} \MM)$ for every $X\xrightarrow{f}\MM$ in $ \Sch_\MM$. For example, $\underline{\ZZ}\kk^{red}=\underline{\ZZ}(\Sm^{\Pro})$ and $\underline{\ZZ}(\Sm^{\Pro})^{red}=\underline{\ZZ}(\Sm^{\Pro})$. More general, $R^{red,red}=R^{red}$, in other words, $R\mapsto R^{red}$ is an idempotent endofunctor on $\Th^{(\lambda)}(\MM)$. A more conceptual characterization of the ``reduction'' will be given in section \ref{pushforwards_pull-backs}. 
\end{example}

\begin{example} \rm \label{fiber_example}
Given a ($\lambda$-)ring $(\Sm,\Pro)$-theory $R$ over $\MM$, we define a new ($\lambda$-)ring $(\Sm,\Pro)$-theory $R^{fib}$ by sheafification of the ($\lambda$-)ring $(\Sm,\Pro)$-pretheory 
\[ (X\xrightarrow{f}\MM)\longmapsto \oplus_{a\in \MM(\bar{\kk})} R(X_a\xrightarrow{c_a}\MM). \]
Here $c_a=f|_{X_a}:X_a\to \MM$ denotes the constant map of the fiber $f^{-1}(a):=\Spec \kk(a) \times_\MM X$ to $\MM$, where $\kk(a)$ denotes the residue field of $a$ which is a finite extension of $\kk$. Given $X\xrightarrow{f}\MM$ and $Y\xrightarrow{g}\MM$, we define the external product by means of 
\[ R(X_a\xrightarrow{c_a}\MM)\otimes R(Y_b\xrightarrow{c_b}\MM)\xrightarrow{\boxtimes}R(X_a\times Y_b\xrightarrow{c_a\boxtimes c_b} \MM)\to R((X\times Y)_{a+b}\xrightarrow{c_{a+b}}\MM)\]
and, if $R$ is a $\lambda$-ring $(\Sm,\Pro)$-theory, the $\lambda$-operations by 
\[ R(X_a\xrightarrow{c_a}\MM)\xrightarrow{\sigma^n} R(\Sym^n (X_a)\xrightarrow{\Sym^n(c_a)}\MM) \to R((\Sym^n X)_{na}\xrightarrow{c_{na}}\MM).\]
Details are left to the reader. By pushing forward along the inclusion $X_a\subseteq X$, we get a canonical morphism $R^{fib}\to R$ of ($\lambda$-)ring $(\Sm,\Pro)$-theories. If $\Sm$ contains closed embeddings, the pull-back along the inclusion $i_a:X_a\hookrightarrow X$ provides a left inverse of $i_{a!}$ and $R^{fib}$ is a subtheory of $R$. In this situation we may think of $R^{fib}(f)$ as  the subgroup of $R(f)$ containing elements supported on (finitely many over each connected component) fibers of $f$. In particular, if $f$ is locally constant, e.g.\ if $\MM$ is discrete, $R^{fib}=R$. The construction $R\mapsto R^{fib}$ is an idempotent endofunctor  on $\Th^{(\lambda)}(\MM)$. 
\end{example}

\subsection{Algebraic structures}

Given a ($\lambda$-)ring $(\Sm,\Pro)$-theory  $R$ over $\MM$ and a commutative monoid $(X\xrightarrow{f}\MM,+,0)$ in $\Pro_\MM$, i.e.\ a commutative monoid in $\Sch_\MM$  such that $f\boxtimes f \xrightarrow{+} f$ and $0:\Spec \kk \to f$ are in $\Pro_\MM$. Then $\Sym^n(f) \xrightarrow{+} f$ is also in $\Pro_\MM$ for every $n\in\NN$, and we can define a convolution product $\cdot$ on $R(f)$ and operations $\sigma_\cdot^n$, if $R$ is a $\lambda$-ring $(\Sm,\Pro)$-theory, as follows
\begin{eqnarray*}
& & \cdot : R(f)\otimes R(f) \xrightarrow{ \boxtimes } R(f\boxtimes f) \xrightarrow{+_!} R(f), \\
& & \sigma^n_\cdot: R(f)  \xrightarrow{\sigma^n} R(\Sym^n(f)) \xrightarrow{+_!} R(f).
\end{eqnarray*}
\begin{proposition} \label{algebraic_structure}
Let $R$ be a ($\lambda$-)ring $(\Sm,\Pro)$-theory  and let $(X\xrightarrow{f} \MM,+,0)$ be a commutative monoid in  $\Pro_\MM$.
\begin{enumerate}
\item  The convolution product and the operations $\sigma_\cdot^n$ defined above turn $R(f)$ into a ($\lambda$-)ring  with unit  $1_0=0_!(1)$.
\item If $X$ has a stratification $X = \cup_{n\in \NN}X^n$ with $\overline{X^n}=\cup_{m\ge n}X^m$ and such that the map $\Sym^n(X^1)\rightarrow X$ factors through the inclusion $X^n\hookrightarrow X$ for $n\in \NN$, the ($\lambda$-)ring $R(f)$ is complete with respect to the filtration $F^nR(f):=(\overline{X^n}\hookrightarrow X)_!R(f|_{\overline{X_n}})$.
\item If $u:(X\xrightarrow{f} \MM,+,0)\rightarrow (Y\xrightarrow{g} \MM,+',0)$ is a  homomorphism of commutative monoids in $\Pro_\MM$ and $u\in\Pro_\MM$, then $u_!:R(f)\to R(g)$ is a homomorphism of ($\lambda$-)rings.  
\item If $u:(X\xrightarrow{f} \MM,+,0)\rightarrow (Y\xrightarrow{g} \MM,+',0)$ is a homomorphism of commutative monoids  $\Pro_\MM$ and $u\in \Sm_\MM$ such that the diagram
\[
\xymatrix{
X\times X\ar[d]_+\ar[r]^{u\times u} &Y\times Y\ar[d]_{+'}\\
X\ar[r]^u&Y
}
\]
is cartesian, then $u^*:R(g)\to R(f)$ is a ($\lambda$-)ring homomorphism.
\end{enumerate}
\end{proposition}
\begin{proof}
Let us proof the final statement as the other ones follow directly from the properties of the universal $\lambda$-operations $\sigma^n:R(f)\to R(\Sym^n(f))$. The last statement will follow from the commutativity of the  universal $\lambda$-operations with pull-backs, once we know that 
\[ \xymatrix @C=1.5cm { \Sym^n(X) \ar[r]^{\Sym^n(u)} \ar[d]_+ & \Sym^n(Y) \ar[d]^{+'} \\ X \ar[r]^u & Y}\]
is cartesian for all $n\in \NN$. Indeed, as the vertical maps are in $\Pro$ by assumption, and $u\in \Sm$, the pull-back $\Sym^n(u)$ of $u$ along $+'$ is also in $\Sm$. We denote the fiber product $X\times_Y \Sym^n(Y)$ with $Z$ and obtain a morphism $\Sym^n(X)\to Z$. The morphism $\Sym^n(X)\to Z$ induces a bijection on geometric points as one can  check easily since the geometric points of $\Sym^n(X)$ are $n$-tuples of geometric points on $X$ up to ordering and similarly for $\Sym^n(Y)$. Moreover, the square 
\[ \xymatrix @R=0.5cm { X^n \ar[r]^{u^n} \ar[d] & Y^n \ar[dd] \\ \Sym^n(X) \ar[d] & \\ Z \ar[r] & \Sym^n(Y) }\]
is cartesian which can be  shown easily using that
\[ \xymatrix { X^n \ar[r]^{u^n} \ar[d]^+ & Y^n \ar[d]^{+'} \\ X \ar[r]^u & Y} \]
is cartesian because of our assumptions. In particular, being the pull-back of a finite morphism, $X^n\to \Sym^n(X) \to Z$ is a finite morphism. Thus, $\Sym^n(X)\to Z$ is a finite morphism, too. By Zariski's main theorem, a proper morphism between noetherian schemes which induces a bijection on geometric points is an isomorphism. Therefore, $\Sym^n(X)\to Z$ is an isomorphism and the proposition is proven. 
\end{proof}
\begin{remark} \rm
Given any morphism $u:(X\xrightarrow{f}\MM)\longrightarrow (Y\xrightarrow{g}\MM)$ such that $u\in \Pro$, the conditions of  Proposition \ref{algebraic_structure}(3) are satisfied for $\Sym(u):\Sym(f)\to \Sym(g)$. Indeed, $+:\Sym(X)\times \Sym(X)\to \Sym(X)$ and $\Spec \kk \to \Sym(X)$ are finite morphisms, hence in $\Pro$. Thus, $\Sym(f)$ and $\Sym(g)$ are commutative monoids in $\Pro_\MM$ and  $\Sym(u)\in \Pro_\MM$ by assumption (6) on the pair $(\Sm,\Pro)$. Thus, Proposition \ref{algebraic_structure}(3) is equivalent to our axiom that the universal $\lambda$-operations commute with push-forwards for $u\in\Pro_\MM$.\\
Moreover, using that $\Sym(X)\times \Sym(X)\longrightarrow\Sym(X)$ is proper, we conclude that $\Sym(X)\times \Sym(X)\longrightarrow \Sym(X)\times_{\Sym(Y)} (\Sym(Y)\times \Sym(Y))$ is a proper morphism which induces a bijection on geometric points, hence an isomorphism by Zariski's main theorem, and 
\[ \xymatrix @C=2.5cm { \Sym(f)\times\Sym(f) \ar[r]^{\Sym(u)\times\Sym(u)} \ar[d]_+ & \Sym(g)\times \Sym(g)\ar[d]^+ \\ \Sym(f)\ar[r]^{\Sym(u)} & \Sym(g) }\]
is cartesian. Thus, Proposition \ref{algebraic_structure}(4) is equivalent to our axiom that the universal $\lambda$-operations commute with pull-backs if $\Sym(u)\in \Sm_\MM$.  
\end{remark}

\begin{example} \rm \label{convention}
Let $R$ be a ($\lambda$-)ring $(\Sm,\Pro)$-theory. For $\Spec\kk\xrightarrow{0}\MM$ we arrive at two ($\lambda$-)ring structures on $R(\Spec\kk\xrightarrow{0}\MM)$, given by $(\boxtimes,\sigma^n)$ and $(\cdot,\sigma_\cdot^n)$, where we interpret $\Spec\kk\xrightarrow{0}\MM$ as a commutative monoid in $\Sch_\MM$. However, all structures are the same. Similarly, if $R$ is a $\lambda$-ring $(\Sm,\Pro)$-theory, the universal operations $\sigma^n:R(f) \to R(\Sym^n(f))\hookrightarrow R(\Sym(f))$ and the convolution $\lambda$-ring structure on $R(\Sym(f))$ restricted to $R(f)\hookrightarrow R(\Sym(f))$ agree. Moreover, thinking of $R(\Spec\kk)$ and $R(f)$ as being  subgroups of $R(\Sym(f))$ by means of the closed embeddings $\Spec\kk\xrightarrow{0} \Sym(f)$ and $f\xrightarrow{\Delta_1}\Sym(f)$, the exterior product $a\boxtimes b$ and the convolution product $ab$ in $R(\Sym(f))$  agree for $a\in R(\Spec\kk)$ and $b\in R(f)$, and we will suppress the symbol $\boxtimes$ in the following. Moreover, $0_!:R(\Spec\kk\xrightarrow{0}\MM)\to 
R(\Sym(f))$ is a $\lambda$-ring homomorphism.   
\end{example}
\begin{example} \rm
 For a commutative monoid $Y\xrightarrow{g}\MM$ in $\Sm_\MM$ and a commutative monoid $X\xrightarrow{f}\MM$ in $\Pro_\MM$, we defined a $\lambda$-ring structure on $\ZZ g(f)$ in Example \ref{darstellbar} which descends to $\underline{\ZZ} g(f)$. This structure is just the convolution $\lambda$-ring structure introduced above. 
\end{example}
\begin{example} \rm
 Similarly, for  a commutative monoid $X\xrightarrow{f}\MM$ in $\Pro_\MM$, we defined a $\lambda$-ring structure on $\ZZ(\Sm^{\Pro})(f)=\ZZ(\Sm^{\Pro}_X)$ which descends to $\underline{\ZZ}(\Sm^{\Pro}_X)$. This structure is just the convolution $\lambda$-ring structure introduced above. 
\end{example}
\begin{example} \rm
Let $R=\Con$ and let $(X,+,0)$ be a commutative monoid in $\Sch_\kk$. The convolution product is just the usual convolution of (constructible) functions.  We suggest to the reader to check all assertions of Proposition \ref{algebraic_structure} in this example to get an idea of what the Proposition is saying.   
\end{example}
\begin{example}\rm \label{lambda-ring_moduli_spaces}
Pick a family $\Ab$ of exact categories satisfying the properties (1)--(6) discussed in section 2. Using the notation of section 2, we assume that $\oplus:\Msp\times \Msp\longrightarrow \Msp$ is in $\Pro$ which is often the case due to Example \ref{finite_direct_sum}. Given a $\lambda$-ring $(\Sm,\Pro)$-theory  $R$ over $\Msp$, we get a complete $\lambda$-ring structure on $R(\id_{\Msp})$ since $\Msp$ has a stratification as in Proposition \ref{algebraic_structure} with $X^1=\Msp^s$. Similarly, if $\PPP$ is a locally closed property for objects in $\Ab$ closed under extension and subquotients, we obtain a complete $\lambda$-ring structure on $R(\Msp^\PPP\hookrightarrow \Msp)$ for every $\lambda$-ring $(\Sm,\Pro)$-theory $R$  on $\Msp$. If the inclusion $\bar{\tau}:\Msp^\PPP\hookrightarrow \Msp$ of the locally closed subspace of objects having property $\PPP$ is in $\Pro$, the push-forward along the inclusion is a $\lambda$-ring homomorphism $R(\Msp^\PPP\to \Msp)\longrightarrow R(\id_\Msp)$. The same is true for the pull-back if $\bar{\tau}\in \Sm$. Finally, if $\dim:\Msp\to \NN^{\oplus I}$ is in $\Pro$ and $R$ is a $\lambda$-ring $(\Sm,\Pro)$-theory on $\NN^{\oplus I}$, then $\dim_!:R(\dim)\to R(\id_{\NN^{\oplus I}})$ is a $\lambda$-ring homomorphism of complete $\lambda$-rings.
\end{example}

\begin{definition} \label{line_elements} Given a $\lambda$-ring $(\Sm,\Pro)$-theory  $R$ over $\MM$, we use the notation of appendix \ref{adjoining_roots} to define
 \begin{enumerate}
  \item[(i)] $R_{sp}:=\{ a\in R(\Spec\kk)\mid 0_!(a)\in R(\Sym(f))_{sp} \;\forall \,f\in \Sch_\MM\}\subseteq R(\Spec\kk)_{sp}$ with $R(\Sym(f))$ carrying the standard convolution structure of $(\Sym(f),\oplus,0)$,
  \item[(ii)] $\Pic(R):=R_{sp}\cap \Pic(R(\Spec\kk),\sigma_t)=\{a\in R_{sp}\mid   \sigma^n(a)=a^n \;\forall\, n\in \NN\}$.
\end{enumerate}
\end{definition}
The subset $\Pic(R)$ is a monoid under multiplication as one can see using Lemma \ref{special_elements}. Moreover, $R_{sp}$ is a special $\lambda$-subring of $R(\Spec\kk)$, but the inclusion into $R(\Spec\kk)$ can be strict as the following example shows. However, the authors are not aware of any example such that $R_{sp}\subsetneqq R(\Spec\kk)_{sp}$.  
\begin{example} \rm  
One can show $\LL\in \Pic(\underline{\Ka}_0(\Sch^{ft}))$ and $\underline{\Ka}_0(\Sch^{ft})_{sp}\subsetneqq \underline{\Ka}_0(\Sch^{ft}_\kk)=\Ka_0(\Sch^{ft}_\kk)$ (see \cite{LarsenLunts}). 
\end{example}

\begin{lemma}[cf.\ Lemma \ref{lambda_adjunction_3}] \rm \label{lambda_adjunction_2}
Given a $\lambda$-ring $(\Sm,\Pro)$-theory  $R$ and a (possibly infinite) family $(P_\alpha)_{\alpha\in A}$ of polynomials in $R_{sp}[T]$, there is a unique $\lambda$-ring $(\Sm,\Pro)$-theory  $R\langle x_\alpha \mid \alpha\in A \rangle$ such that
\[ R\langle x_\alpha \mid \alpha\in A \rangle(\Sym(f))=R(\Sym(f))\langle x_\alpha \mid \alpha\in A \rangle. \]
as $\lambda$-rings for all connected $f\in \Sch_\MM$. Using Lemma \ref{lambda_adjunction_5}, we get a group isomorphism
\[R\langle x_\alpha \mid \alpha\in A \rangle(f)\cong R(f)\otimes_{R_{sp}} R_{sp}\langle x_\alpha \mid \alpha\in A \rangle =:R(f)\langle x_\alpha \mid \alpha\in A \rangle\]
for every connected $f$. Clearly, there is a morphism $R\longrightarrow  R\langle x_\alpha \mid \alpha\in A \rangle$ and a family of elements $x_\alpha\in R\langle x_\alpha \mid \alpha\in A \rangle_{sp}$ for $\alpha\in A$ satisfying a universal property.  
\end{lemma}
\begin{example} \label{lambda_adjunction_4} \rm As $\LL\in \Pic(\underline{\Ka}_0(\Sch^{ft}))$, we can form $\underline{\Ka}_0(\Sch^{ft})\langle \LL^{-1}\rangle$ with $\LL^{-1}\in \Pic(\underline{\Ka}_0(\Sch^{ft})\langle \LL^{-1}\rangle)$ or $\underline{\Ka}_0(\Sch^{ft})\langle \LL^{-1/2}\rangle^-$ with $-\LL^{1/2}\in \Pic(\underline{\Ka}_0(\Sch^{ft})\langle \LL^{-1/2}\rangle^-)$. (cf.\ Example \ref{lambda_adjunction})
\end{example}
 
The following construction applies only to those pairs $(\Sm,\Pro)$ for which $\Sm$ contains all closed embeddings. This is for instance the case if $\Sm=\Sch_\kk$.
Given a ring $(\Sm,\Pro)$-theory  $R$ over $\MM$  and an object $X\xrightarrow{f}\MM$ in $\Sch_\MM$, we can defined the cap-product and, if $R\in \Th^{\lambda}(\MM)$, operations $\sigma_\cap^n$ on $R(X)=R^{red}(f)$, where we used the shorthand $X$ for $X\xrightarrow{0}\MM$, by means of the (small) diagonal embeddings $\Delta_X:X\hookrightarrow X\times_\kk X$ and $\Delta_n:X\hookrightarrow \Sym^n(X)$ as follows
\begin{eqnarray*} 
&& \cap:  R(X)\otimes R(X) \xrightarrow{\boxtimes} R(X\times_\kk X) \xrightarrow{\Delta_X^\ast} R(X),\\
&& \sigma_\cap^n: R(X) \xrightarrow{\sigma^n} R(\Sym^n X)\xrightarrow{\Delta_n^\ast} R(X). 
\end{eqnarray*}
More generally,  we can define a cap-product using the ``diagonal'' 
\[ \Delta_X:(X\xrightarrow{f}\MM) \longrightarrow (X\xrightarrow{0}\MM)\boxtimes (X\xrightarrow{f} \MM)=(X\times_\kk X \xrightarrow{f\pr_2} \MM) \]
as 
\[ R^{red}(f)\otimes R(f)=R(X)\otimes R(f) \xrightarrow{\boxtimes} R(X\boxtimes f) \xrightarrow{\Delta_X^\ast} R(f). \]

\begin{proposition} \label{algebraic_structure_ex}
Let $R$ be a ring $(\Sm,\Pro)$-theory  over $\MM$ with $\Sm$ containing all closed embeddings, and let $X\xrightarrow{f}\MM, Y\xrightarrow{g}\MM$ be two objects in $\Sch_\MM$. Then:
\begin{enumerate}
\item The cap-product  defines a ring structure on $R^{red}(f)=R(X)$ which is unital with unit $\unit_X=(X\to \Spec \kk)^\ast (1)$ if $X\to \Spec \kk$ is in $\Sm$. If the latter is true and  $R\in \Th^{(\lambda)}(\MM)$, the operations $\sigma^n_\cap$ provide a $\lambda$-ring structure on $R^{ref}(f)$. Moreover, the cap product $R^{red}(f)\otimes R(f)\to R(f)$ induces an $R^{red}(f)$-module structure on $R(f)$. 
\item For every morphism $u:f\to g$ in $\Sm$,  $u^\ast(a\cap b)=u^\ast(a)\cap u^\ast(b)$ for all $a\in R^{red}(g)$ and $b\in R^{red}(g)$ or $b\in R(g)$. If $R\in \Th^{\lambda}(\MM)$, $u^\ast(\sigma^n_\cap(a))=\sigma_\cap^n(u^\ast(a))$ for all $a\in R^{red}(f)$, i.e.\ $R^{red}$ is a contravariant functor to \\ ($\lambda$-)rings and $(R^{red},R)$ is a contravariant functor to modules.
\item For every morphism $u:f\to g$ in $\Pro$, $u_!$ satisfies the projection formula $u_!(a\cap u^\ast(b))=u_!(a)\cap b$  for $a\in R^{red}(f)$ and $b\in R(g)$ or $b\in R^{red}(g)$, and also the projection formula $u_!(u^\ast(a)\cap b)=a\cap u_!(b)$ for $a\in R^{red}(g)$ and $b\in R(f)$. The latter formula says that $u_!:R(f)\to R(g)$ is a $R^{red}(g)$-module homomorphism if $R^{red}(g)$ acts on $R(f)$ via $u^\ast:R^{red}(g)\to R^{red}(f)$. 
\item The exterior product of elements $a\in R(X)$ and $b\in R(g)$ or $b\in R(Y)$ satisfies $a\boxtimes b = pr_X^\ast(a)\cap pr_Y^\ast(b)$.
\item The subgroup $R^{fib}(f)\subseteq R(f)$ is an $R^{red}(f)$-submodule and $R^{red,fib}(f)\subset R^{red}(f)$ is a ($\lambda$-)ideal. 
\item If $(f,+,0)$ is a commutative monoid in $\Sch_\MM$, the map $R^{red}(f)\to R(f)$ given by $a\mapsto a\cap 1_0$ is a ($\lambda$-)ring homomorphism.
\end{enumerate}
\end{proposition}
\begin{example} \rm 
Let $R$ be a($\lambda$-)ring $(\Sm,\Pro)$-theory  over $\MM$ with $\Pro\subset \Sm$. For $X=\Spec\kk\xrightarrow{0}\MM$, the $\cap$-structure coincides with the convolution structure by the last part of the previous proposition. In particular, we will drop the $\cap$-symbol from notation.  
\end{example}
\begin{example} \rm
If $R=\Con$, we get $(\xi\cap \zeta)(x)=\xi(x)\zeta(x)$ and  $\sigma_\cap^n(\xi)(x)=\sigma^{st,n}(\xi(x))=(-1)^n{-\xi(x) \choose n}={\xi(x)+n-1\choose n}$ for every $\xi,\zeta\in \Con(X)$ and every $x\in X$. 
\end{example}
\begin{example}\rm  The $\cap$-product on $\underline{\ZZ}(\Sm^{\Pro}_X)$ is  $[U\to X]\cap[V\to X]=[U\times_X V \longrightarrow X]$, and $\sigma^n_\cap([U\to X])=[U\times_X\ldots \times_X U \longrightarrow X]$ which looks a bit weird as the reader would expect something like $[U\times_X\ldots \times_X/\!\!/S_n  \longrightarrow X]$. However, $U^n/\!\!/S_n\to \Spec\kk$ might not be in $\Sm$ and the generator $[U\sqcup V]$ of $\underline{\ZZ}(\Sm^{\Pro}_\kk)$ is not the sum $[U]+[V]$ of generators in general. 
\end{example}

\begin{definition} A pair $(\Sm,\Pro)$ is called motivic if $\Pro\cap \Sm$ contains all open and all closed embeddings. In particular, it also contains all locally closed embeddings.
\end{definition}
By the base change property, $e^\ast e_!=\id$ for every locally closed embedding. Hence, $e_!$ is injective and $e^\ast$ is surjective.
\begin{example}\rm Given a motivic pair $(\Sm,\Pro)$, we construct a quotient ring $(\Sm,\Pro)$-theory $\underline{\Ka}_0(\Sm^{\Pro})$ of $\underline{\ZZ}(\Sm^{\Pro})$ by imposing the relation $i_!i^\ast(a)+j_!j^\ast(a) = a$ for every element $a\in \underline{\ZZ}(\Sm^{\Pro}_X)$ and every closed embedding $i$ with open complement $j$. The $\cap$-product is  $[U\to X]\cap[V\to X]=[U\times_X V \longrightarrow X]$. As we have seen in Example \ref{motives2}, $\underline{\Ka}_0(\Sch^{ft})$ is even a $\lambda$-ring $(\Sch_\kk,ft)$-theory  and $\sigma^n_\cap([U\to X])=[U\times_X\ldots \times_X U/\!\!/S_n \longrightarrow X]$. However, the quotient map $\underline{\ZZ}(\Sch^{ft})\longrightarrow \underline{\Ka}_0(\Sch^{ft})$ is not a homomorphism of $\lambda$-ring theories and does also not preserve the $\lambda$-operations $\sigma_\cap$.  
\end{example}
\begin{definition}
Let $(\Sm,\Pro)$ be a motivic pair. A ring $(\Sm,\Pro)$-theory  $R$ over $\MM$ is called motivic if the canonical morphism $\underline{\ZZ}\kk^{red}=\underline{\ZZ}(\Sm^{\Pro})\to R^{red}$ of  ring $(\Sm,\Pro)$-theories   has a (unique) factorization $\underline{\ZZ}(\Sm^{\Pro})\twoheadrightarrow \underline{\Ka}_0(\Sm^{\Pro})\to R^{red}$. 
\end{definition}

\begin{example} \rm The ring $(\Sm,\Pro)$-theories  $\underline{\Ka}_0(\Sm^{\Pro})$ and $\Con$ are motivic as $\chi_c$ provides the factorization of $\underline{\ZZ}(\Sm^{\Pro})\to \Con$. More examples of (motivic) ring $(\Sm,\Pro)$-theories  will be given later.
\end{example}

As an application of the $\cap$-product, we prove the following alternative characterization of motivic $(\Sm,\Pro)$-theories. 
\begin{lemma} \label{motivic_ring theory }
A ring $(\Sm,\Pro)$-theory  $R$ over $\MM$ is motivic if and only if \[i_!i^\ast(a)+j_!j^\ast(a)=a\] holds for all $X\xrightarrow{f}\MM$ with $X\to \Spec\kk$ in $\Sm$, all $a\in R(X\xrightarrow{f}\MM)$, and all closed subschemes $i:Z\hookrightarrow X$ with open complement $j:X\setminus Z\hookrightarrow X$. 
\end{lemma}
\begin{proof} 
Assuming that the formula is true, we may choose  $f=0$ and $a=\unit_X$ to see that the defining relation of $\underline{\Ka}_0(\Sm^{\Pro}_X)$ holds in $R(X)$. Hence, $\underline{\ZZ}(\Sm^{\Pro}_X)\to R(X)$ factorizes through $\underline{\Ka}_0(\Sm^{\Pro}_X)$. Conversely, assuming the latter, the equation $i_!i^\ast(\unit_X)+j_!j^\ast(\unit_X)=i_!(\unit_Z)+j_!(\unit_{X\setminus Z})=\unit_X$ holds in $R(X)=R^{red}(f)$ for all $X\xrightarrow{f}\MM$ in $\Sch_\MM$. Given $a\in R(f)$, we use the the projection formula to conclude
\[ i_!i^\ast(a)+j_!j^\ast(a)=i_!(\unit_Z)\cap a + j_!(\unit_{X\setminus Z})\cap a = \unit_X\cap a=a. \]
\end{proof}


\subsection{Push-forwards and pull-backs of ($\lambda$-)ring theories } \label{pushforwards_pull-backs}

Given a homomorphism $u:\MM\to N$ of commutative monoids, we get two functors
\begin{eqnarray*}
&& u^!: \Sch_N \ni (X\xrightarrow{f}N) \longmapsto (\MM\times_N X\xrightarrow{\pr_\MM} \MM) \in \Sch_\MM, \\
&& u_!: \Sch_\MM \ni (X\xrightarrow{f} \MM) \longmapsto (X\xrightarrow{u\circ f} N) \in \Sch_N. 
\end{eqnarray*}
\begin{lemma} \label{reduction_3} The functor $u_!$ commutes with the monoidal structure and with $\Sym$. Moreover, the functor $u^!$ is the right adjoint of $u_!$. Furthermore, these functors lift to  functors $u_!:\Sm\Pro^{op}_\MM \longrightarrow \Sm\Pro^{op}_N$ and $u^\ast:\Sm\Pro^{op}_N\longrightarrow \Sm\Pro^{op}_\MM$ of categories. If $u$ is in $\Pro$, $u^\ast$ is left adjoint to $u_!$.  
\end{lemma}
\begin{proof} The first part is trivial. By assumption (4) on $(\Sm,\Pro)$, the functors $u^\ast:\Sm\Pro^{op}_N\leftrightarrows \Sm\Pro^{op}_\MM:u_!$ are well-defined. It remains to check that they are adjoint in the ``opposite'' order. Note that the unit $\eta_f:(X\xrightarrow{f} \MM) \longrightarrow (X\times_N\MM\to \MM)=u^!u_!(X\to \MM)$ is in general not in $\Sm$ and thus $\eta_{f\sharp}$ cannot be used to define a unit for $u_!:\Sm\Pro^{op}_M \leftrightarrows \Sm\Pro^{op}_N:u^\ast$. However, $\eta_f$ is a closed embedding, and $\eta_f^\sharp$ might serve as a counit for the converse adjunction $u^\ast:\Sm\Pro^{op}_N \leftrightarrows \Sm\Pro^{op}_M:u_!$. Moreover, the natural transformation $u^!u_!u^!u_!\to u^!u_!$ given by $(X\times_N\MM)\times_N \MM\xrightarrow{\id_X\times_N u\times_N \id_\MM} (X\times_N N)\times_N \MM=X\times_N \MM$ is in $\Pro_\MM$ if $u\in \Pro$ and induces the ``comultiplication'' $u^\ast u_! \to u^\ast u_!u^\ast u_!$ of the comonad $u^\ast u_!$. If $u\in \Pro$, the  counit $\epsilon_g:u_!u^!(Y\times_N\MM\to N) \xrightarrow{\id_Y\times_N u} (Y\times_N N\xrightarrow{\pr_N} N)=(Y\xrightarrow{g} N)$ is in $\Pro$ so that $\epsilon_g^\sharp$ can serve as a unit. Moreover, $u_!u^! \to u_!u^!u_!u^!$ applied to $Y\to N$ is given by $\id_Y\times_N \Delta_\MM$, hence in $\Pro_N$, and provides the ``multiplication'' $u_!u^\ast u_! u^\ast\to u_!u^\ast$ of the monad $u_!u^\ast$.
\end{proof}

\begin{lemma} By composition with $u_!$ and with $u^\ast$ respectively, we obtain two functors $u^!$ and $u_\ast$ with left adjoints $u_!$ and $u^\ast$ between (abelian group valued) $(\Sm,\Pro)$-pretheories on $\MM$ and on $N$ respectively, extending the previous functors using the Yoneda embedding. Moreover, $u^!$ lifts to a corresponding functor on ($\lambda$-)ring $(\Sm,\Pro)$-theories. If $u$ is in $\Pro$, $u_!=u_\ast$, and $u_!$ also lifts to ($\lambda$-)ring $(\Sm,\Pro)$-theories providing an adjoint pair $u_!,u^!$ of functors between categories of ($\lambda$-)ring $(\Sm,\Pro)$-theories.   
\end{lemma}
\begin{proof}
Using the previous lemma, general theory of presheaves and the fact that the sheafification functor for abelian group valued pretheories commutes with colimits, it remains to show that $u_!(R)=u_\ast(R)=R\circ u^\ast$ has a ($\lambda$-)ring structure. As $u^\ast$ does not preserve $\boxtimes$ and $\Sym$, we have to define the exterior product and the $\sigma^n$-operations more carefully. Since $u$ is in $\Pro$, the same holds for $(u\times u)\times +_\MM: \MM\times \MM\to (N\times N)\times_N\MM$ which is a closed embedding into $(\MM\times \MM)\times_N \MM$ followed by $(u\times u)\times_N \id_\MM$. Given $X\xrightarrow{f}N$ and $Y\xrightarrow{g} N$, we take the fiber product of $(u\times u)\times +_\MM$ with $\id_X\times \id_Y$ over $N\times N$. Thus, the map 
\[(\id_X\times \id_Y)\times +_\MM:(X\times Y)\times_{N\times N} (\MM\times\MM) \longrightarrow (X\times Y)\times_N \MM\]
must be in $\Pro$, too. We define the exterior product on $u_!(R)$ as the composition
\begin{eqnarray*} \lefteqn{ u_!(R)(f)\otimes u_!(R)(g)= R(X\times_N \MM \xrightarrow{\pr_\MM}\MM)\otimes R(Y\times_N\MM \xrightarrow{\pr_\MM} \MM) } \\ &&\xrightarrow{\qquad\quad\;\,\boxtimes\qquad\quad\;\,}R( (X\times Y)_{N\times N}(\MM\times\MM) \xrightarrow{+_\MM\circ\pr_{\MM\times\MM}} \MM) \\ 
&&\xrightarrow{\big((\id_X\times\id_Y)\times +_\MM\big)_!} R((X\times Y)\times_N\MM \xrightarrow{\pr_\MM}\MM)=u_!(R)(f\boxtimes g).
\end{eqnarray*}
The unit for $u_!(R)$ is given by $0_!(1)\in R(\ker(u)\hookrightarrow \MM)=u_!(R)(\Spec\kk)$. If $R$ is a $\lambda$-ring $(\Sm,\Pro)$-theory, we can also define $\sigma^n$-operations for $u_!(R)$.  As $u$ is in $\Pro$, the morphism \[\Sym^n(u)\times +_\MM:\Sym^n(\MM) \hookrightarrow \Sym^n(\MM)\times_N \MM \xrightarrow{\Sym^n(u)\times_N \id_M} \Sym^n(N) \times_N \MM\] is in $\Pro$ for all $n\in \NN$. Taking the fiber product with $\id_{\Sym^n(X)}$ over $\Sym^n(N)$ and composing it with the finite map \\ $\Sym^n(X\times_N\MM) \twoheadrightarrow \Sym^n(X)\times_{\Sym^n(N)}\Sym^n(\MM)$, we get another morphism
\[ \Sym^n(X\times_N\MM) \twoheadrightarrow \Sym^n(X)\times_{\Sym^n(N)}\Sym^n(\MM) \longrightarrow \Sym^n(X)\times_N \MM \] 
in $\Pro$. Using the push-forward along this map, we define our $\sigma^n$-operations as follows.
\begin{eqnarray*} \lefteqn{ u_!(R)(f)=R(X\times_N\MM\xrightarrow{\pr_\MM}\MM)\xrightarrow{\sigma^n} R(\Sym^n(X\times_N\MM) \xrightarrow{+_\MM} \MM)  } \\
 & & \longrightarrow R(\Sym^n(X)\times_N\MM \xrightarrow{\pr_\MM} \MM)=u_!(R)(\Sym^n(f)). 
\end{eqnarray*}
Hence, we obtain a functor $u_!:\Th^{(\lambda)}(\MM)\to \Th^{(\lambda)}(N)$. \\
The two adjunctions for $u^!,u_!$ are given as follows. For $X\xrightarrow{f}\MM$, let
\[\eta_f:R(X\xrightarrow{f}\MM)\xrightarrow{(\id_X\times f)_!}R(X\times_N\MM \xrightarrow{\pr_\MM}\MM)=u_!(R)(X\xrightarrow{uf}N)=u^! u_! (R)(f), \]
and for $Y\xrightarrow{g}N$ define
\[ \delta_f:u_! u^! (R')(g)=u^!(R')(Y\times_N\MM\xrightarrow{\pr_\MM}\MM)=R'(Y\times_N\MM \xrightarrow{g\pr_Y}N) \xrightarrow{(\pr_Y)_!} R'(g) \]
using $\pr_Y=\id_Y\times_N \,u$ and $u\in \Pro_M$ once more. 
\end{proof}

\begin{example} \rm
For every commutative monoid $(\MM,+,0)$ there are three distinguished examples of monoid homomorphisms, namely $u=0:\Spec \kk \to \MM$, $u=c:\MM\to \Spec \kk$ and the composition $c_0=0\circ c:\MM\to \MM$. Thus, we obtain the reduction functor $0^!:\Th^{(\lambda)}(\MM)\ni R \to R^{red}\in \Th^{(\lambda)}(\kk)$. Conversely, the constant map $c:\MM\to \Spec\kk$ induces a functor $c^!:\Th^{(\lambda)}(\kk)\to \Th^{(\lambda)}(\MM)$ which is faithful having the reduction as a left-inverse since the composition $\Spec\kk \xrightarrow{0} \MM \xrightarrow{c}\Spec\kk$ is the identity. Using this, we can always identify $\Th^{(\lambda)}(\kk)$ with the corresponding (non-full) subcategory of $\Th^{(\lambda)}(\MM)$. In particular, $R^{red}$ can be interpreted as a ``reduced'' ($\lambda$-)ring $(\Sm,\Pro)$-theory  over $\MM$ by means of $(\MM\xrightarrow{c_0}\MM)^!(R)$ which defines an idempotent endofunctor on $\Th^{(\lambda)}(\MM)$. Note that the reduction of the canonical morphism $R^{fib}\to R$ of ($\
\lambda$-)ring 
$(\Sm,\Pro)$-theories  is the identity, i.e.\ $R^{fib,red}=R^{red}$.  
\end{example}
\begin{example} \rm The push-forward of $\underline{\ZZ}\kk$ considered as a ring $(\Sm,\Pro)$-theory over $\Spec\kk$ is the representable ring $(\Sm,\Pro)$-theory associated to $\Spec\kk\xrightarrow{0}\MM$ which we also denoted with $\underline{\ZZ}\kk$ hoping not to confuse the reader. However, the pull-back of $\underline{\ZZ}\kk$ along $c:\MM\to \Spec\kk$ is $c^!\underline{\ZZ}\kk=\underline{\ZZ}(\Sm^{\Pro})$. Thus, $\underline{\ZZ}(\Sm^{\Pro})^{red}=\underline{\ZZ}(\Sm^{\Pro})$ is reduced. 
\end{example}
\begin{example}\rm
Use the notation of section 2 and consider the morphism $\dim:\Msp \to \NN^{\oplus I}$ which has a factorization $\Msp\xrightarrow{u} \Spec (\dim_\ast \OO_{\Msp}) \xrightarrow{v} \NN^{\oplus I}$, where $\dim_\ast\OO_{\Msp}$ is a sheaf of (co)commutative bialgebras on $\NN^{\oplus I}$. In particular, $\Spec(\dim_\ast\OO_{\Msp})$ is a commutative monoid in $\Sch_\kk$, affine over $\NN^{\oplus I}$, and $u,v$ are monoid homomorphisms. If $X_d=X_{d}^{ss}$ as in Lemma \ref{finite_sum}, $u$ is projective. Assuming $proj\subset \Pro$, we can apply the last Lemma to obtain an adjunction $u_!:\Th^{(\lambda)}(\Msp)\leftrightarrows \Th^{(\lambda)}(\Spec(\dim_\ast\OO_\Msp)):u^!$. 
\end{example}

\subsection{($\lambda$-)algebra theories and vanishing cycles}

\begin{definition}
 A ($\lambda$-)algebra $(\Sm,\Pro)$-theory over $\MM$ is just a homomorphism $\phi:R'\to R$ of ($\lambda$-)ring $(\Sm,\Pro)$-theories over $\MM$. With the obvious definition of morphisms, we obtain the category of ($\lambda$-)algebras over $\MM$, i.e.\ of morphisms in $\Th^{(\lambda)}(\MM)$. Fixing $R'$, we get the category of $R'$-($\lambda$-)algebra $(\Sm,\Pro)$-theories. Conversely, fixing $R$, we obtain the category of $R$-augmented $(\lambda$-)ring $(\Sm,\Pro)$-theories. The commutative group objects in this category can be interpreted as $R$-($\lambda$-)module $(\Sm,\Pro)$-theories. 
\end{definition}
\begin{example} \rm
 Every ($\lambda$-)ring $(\Sm,\Pro)$-theory $R$ over $\MM$ has a canonical \\ $R^{fib}$-($\lambda$-)algebra structure which is functorial in $R$.
\end{example}

As we have seen, every ring $(\Sm,\Pro)$-theory over $\MM$ is canonically a $\underline{\ZZ}\kk$-algebra. This is no longer true for $\underline{\ZZ}(\Sm^{\Pro})$ unless $\MM=\Spec\kk$ as $\underline{\ZZ}(\Sm^{\Pro})=\underline{\ZZ}\kk$ in this case. In particular, every reduced $(\Sm,\Pro)$-theory is canonically a $\underline{\ZZ}(\Sm^{\Pro})$-algebra by pulling back the canonical $\underline{\ZZ}\kk$-algebra structure over $\Spec\kk$. As $\underline{\ZZ}(\Sm^{\Pro})$-algebra $(\Sm,\Pro)$-theories over $\MM$ will play a special role, we  provide another equivalent characterization of $\underline{\ZZ}(\Sm^{\Pro})$-algebra $(\Sm,\Pro)$-theories which points already into the direction of vanishing cycles.
\begin{definition} An $\Sm$-scheme (over $\MM$) is a scheme $X$ (over $\MM$) such that $X\to \Spec\kk$ is in $\Sm$. \end{definition}
For example, if $\Sm=Sm$ is the class of smooth morphisms, an $\Sm$-scheme (over $\MM$) is just a smooth scheme (with a not necessarily smooth morphism to $\MM$).
\begin{lemma} \label{framing}
A $\underline{\ZZ}(\Sm^{\Pro})$-algebra $(\Sm,\Pro)$-theory  $R$ over $\MM$, i.e.\ a morphism $\phi:\underline{\ZZ}(\Sm^{\Pro})\to R$ of ring $(\Sm,\Pro)$-theories  over $\MM$, is uniquely given by a family of distinguished elements $\phi_f\in R(f)$ for every $\Sm$-scheme $X\xrightarrow{f}\MM$ over $\MM$, satisfying the following properties.
\begin{enumerate}
 \item For a morphism $u:Y\to X$ in $\Sm$ and $X\xrightarrow{f}\MM$ an $\Sm$-scheme over $\MM$, we obtain $\phi_{f\circ u}=u^\ast(\phi_f)$.
 \item We have $\phi_{\Spec\kk\xrightarrow{0}\MM}=1$, and if $X\xrightarrow{f}\MM$ and $Y\xrightarrow{g}\MM$ are $\Sm$-schemes, then $\phi_{f\boxtimes g}=\phi_f\boxtimes \phi_g$. 
\end{enumerate}
If $R$ is a $\lambda$-ring $(\Sm,\Pro)$-theory, the induced morphism $\underline{\ZZ}(\Sm^{\Pro})\to R$ is a morphism if $\lambda$-ring theories if and only if $\phi_f\in \Pic(\Sym(f),\sigma_t)$.
\end{lemma}
The distinguished element for a reduced ring $(\Sm,\Pro)$-theory  $R$ is given by $\unit_X$. Let last part of the lemma is a consequence of Lemma \ref{special_elements} and the fact that push-forwards along morphisms in $\Pro$ preserve line elements.
\begin{lemma} 
A $\underline{\Ka}_0(Sm^{prop})$-algebra theory is nothing else than a $\underline{\ZZ}(Sm^{prop})$-algebra theory $\phi$ for which the following blow-up formula holds in $R(X\xrightarrow{f}\MM)$. Let $\pi:\Bl_Y X \to X$ be the blow-up of a smooth scheme $X$  in a smooth center $i:Y\hookrightarrow X$ with exceptional divisor $\tilde{i}:E\hookrightarrow \Bl_Y X$. Then \begin{equation} \label{blow-up} \pi_!(\phi_{f\pi}-\tilde{i}_!(\phi_{f\pi|_E})) = \phi_f - i_!(\phi_{f|_Y}).\end{equation}
\end{lemma}
It should be clear from the definition that the pull-back $u^!(R)$ of a $\underline{\ZZ}(\Sm^{\Pro})$-algebra $(\Sm,\Pro)$-theory   $R$ over $N$ along a monoid homomorphism $\MM\to N$ is a  $\underline{\ZZ}(\Sm^{\Pro})$-algebra  $(\Sm,\Pro)$-theory over $\MM$.  \\ 
If $\phi:\underline{\ZZ}(\Sm^{\Pro})\to R'$ is a $\underline{\ZZ}(\Sm^{\Pro})$-algebra $(\Sm,\Pro)$-theory over $\MM$ and $\eta:R'\to R$ ring $(\Sm,\Pro)$-theory homomorphism, the composition $\eta\circ \phi:\underline{\ZZ}(\Sm^{\Pro})\to R$ is a $\underline{\ZZ}(\Sm^{\Pro})$-algebra structure  on $R$, and $\eta$ becomes a homomorphism of $\underline{\ZZ}(\Sm^{\Pro})$-algebra theories. In particular, if $Sm\subset \Sm$ and $prop\subset \Pro$, every $\underline{\ZZ}(\Sm^{\Pro})$-algebra $(\Sm,\Pro)$-theory has a canonical $\underline{\ZZ}(Sm^{prop})$-algebra structure when considered as a $(Sm,prop)$-theory.

Recall that $R^{fib}(X\to \MM)=\bigoplus_{a\in \MM(\bar{\kk})} R(X_a\xrightarrow{c_a} \MM)$ by definition. In general there is now way to relate $R(X_a\xrightarrow{c_a} \MM)$ to $R(X_a\xrightarrow{c_0} \MM)$ even though both maps involved are constant with the same source. The idea of a trivialization is to require an isomorphism between these two groups for all $a\in \MM(\bar{\kk})$ in way compatible with all algebraic and functorial structures. Note that $\bigoplus_{a\in \MM(\bar{\kk})} R(X_a\xrightarrow{c_0} \MM)$ is nothing else than $R^{red,fib}(X\to \MM)$.  
\begin{definition} \label{trivialization} 
Given a ($\lambda$-)ring theory  $R$, a trivialization for $R$ is a ($\lambda$-)isomorphism $t:R^{red,fib}\to R^{fib}$ of ($\lambda$-)ring $(\Sm,\Pro)$-theories  which becomes the identity after passing to the reduction. In particular, we get a morphism $R^{fib}\xrightarrow{t^{-1}}R^{red,fib}\longrightarrow R^{red}$ which is the identity after passing to the reduction.  
\end{definition}
The condition for the reduction is just saying that our isomorphism should be the identity for $a=0$. 
Given a trivialization $t$ of a ($\lambda$-)ring $(\Sm,\Pro)$-theory  $R$, we denote with   $t^a$ the image of $1$ under the map\footnote{For $a\in \MM(\bar{\kk})$ we denote with $\kk(a)$ the residue field of $a$ which is a finite algebraic extension of $\kk$.} $t_{\Spec \kk(a) \xrightarrow{c_a}\MM}:R(\Spec \kk(a))=R^{red,fib}(\Spec \kk(a)\xrightarrow{c_a}\MM) \longrightarrow R(\Spec \kk(a)\xrightarrow{c_a} \MM)$ for $a\in \MM(\bar{\kk})$. The family $(t^a)_{a\in\MM(\bar{\kk})}$ has the following properties
\begin{enumerate}
 \item $t^a\boxtimes t^b=t^{a+ b}$ for all $a,b\in\MM(\bar{\kk})$ and $t^0=1$,
 \item $t^a\in \Pic(R(\Sym(\Spec \kk(a)\xrightarrow{c_a}\MM)),\sigma_t)$ for all $a\in \MM(\bar{\kk})$,
 \item the map $-\boxtimes t^a:R(X)\to R(X\xrightarrow{c_a}\MM)$ is an isomorphism for all $X\in \Sch_\kk$.
\end{enumerate}
Note that the last condition is automatically fulfilled if $a\in \MM(\bar{\kk})$ is invertible as $-\boxtimes t^{-a}$ is an inverse. Conversely, given a family of elements $t^a\in R(\Spec \kk(a)\xrightarrow{a} \MM)$ satisfying these properties, we obtain a trivialization $t:R^{red,fib}\to R^{fib}$ using the composition
\[ R^{red,fib}(f)=\bigoplus_{a\in\MM(\bar{\kk})} R(X_a) \xrightarrow{\oplus_{a\in\MM(\bar{\kk})} (-\boxtimes t^a)} \bigoplus_{a\in\MM(\bar{\kk})}R(X_a\xrightarrow{c_a}\MM) = R^{fib}(f).\] 
Moreover, $t_f:R(X)\longrightarrow R(f)$ is an isomorphism for every locally constant map $f:X\to \MM$. In particular, $R(\Crit(f))\cong R(f|_{\Crit(f)})$ for every $f\in \Sch_\MM$. 
\begin{example}\rm Let $R$ be a reduced $(\Sm,\Pro)$-theory $R$. Then $R$ has a canonical trivialization as $R(X_a\xrightarrow{c_a} \MM)=R(X_a\xrightarrow{c_0} \MM)$ by definition. \end{example}

\begin{definition}
 Let $(\Sm,\Pro)\supset (Sm,proj)$ be a pair such that $\Sm$ contains all closed embeddings. A vanishing cycle is a $(\Sm,\Pro)$-theory $R$ over $\AA^1$ with  a $\underline{\ZZ}(Sm^{prop})$-algebra structure $\phi:\underline{\ZZ}(Sm^{prop})\to R|_{(Sm,proj)}$ and a trivialization. Moreover, $\phi_f$ should be supported on $\Crit(f)$ for every $X\xrightarrow{f}\MM$ with smooth $X$, i.e.\ $\phi_f$ is contained in the image of  $R(\Crit(f))\cong R(f|_{Crit(f)}) \hookrightarrow R^{fib}(f)\hookrightarrow R(f)$. 
\end{definition}

\begin{example} \rm
 For a trivialized ring $(\Sch,\Pro)$-theory $R$ over $\AA^1$, e.g.\ if $R$ is reduced, we can always construct the canonical vanishing cycle $\phi_f:=\iota_!(\unit_{\Crit(f)})\in R^{red,fib}(f)\cong R^{fib}(f)\hookrightarrow R(f)$.
\end{example}
\begin{example} \rm
There is also a vanishing cycle on $\Con$ over $\AAA$  such that $\phi^{con}_f=(-1)^{\dim X}\nu_{\Crit(f)}$ for $X\xrightarrow{f}\AAA$ with smooth, connected $X$, where $\nu_{\Crit(f)}$ is the Behrend function of $\Crit(f)$ considered as a constructible function on the fibers of $f$. 
\end{example}
\begin{example} \rm
 The motivic vanishing cycle $\phi^{mot}$ of Denef and Loeser \cite{DenefLoeser2} is another example of a vanishing cycle. It takes values in the motivic reduced $(\Sch_\kk,ft)$-theory $\underline{\Ka}^\muu(\Sch^{ft})$ over $\AAA$. (See \cite{DaMe1} for more details on this.)
\end{example}
\begin{example}\rm
 There are also vanishing cycles $\phi^{perv}$ and $\phi^{mhm}$ with values in the reduced $(\Sch_\kk,ft)$-theories $\underline{\Ka}_0(D^b(\Perv))$ and $\underline{\Ka}_0(D^b(\MHM)_{mon})$ over $\AAA$. More details will be given in the next subsection on categorification. 
\end{example}

\begin{definition} Let $\pi^\pm:V^\pm \to X$ be two vector bundles over a smooth scheme $X$, and let $f:V^+\oplus V^-\longrightarrow \MM$ be a $\GG_m$-invariant function, where $\GG_m$ acts by fiberwise multiplication with weight $\pm1$ on $V^\pm$. Let $(\Sm,\Pro)\supset (Sm,proj)$ and let $\Sm$ contain all closed embeddings. A ring $(\Sm,\Pro)$-theory  $R$ over $\MM$ with an $\underline{\ZZ}(Sm^{prop})$-algebra structure on $R|_{(Sm,proj)}$ satisfies the (linear) integral identity if 
\begin{equation} \label{integral_identity} \pi^+_! \phi_f|_{V^+}=\pi^+_! \phi_{f|_{V^+}}=\LL^{\rk V^+}\phi_{f|_X}\end{equation}
holds in $R(f|_X)$ using the shorthand $\LL=\int_\AAA \unit_\AAA\in R(\Spec\kk)$. 
\end{definition}
Note that the last equation in the linear integral identity is a simple consequence of the projection formula and the identity $\pi^{+\,\ast}(\phi_{f|X})=\phi_{f|_{V^+}}$ which must hold as $\pi^+\in \Sm$  and $f|_{V^+}=f|_X\circ \pi^+$ by equivariance. 
\begin{example} \rm It is easy to see that $\underline{\ZZ}(\Sm^{\Pro})$ with its canonical $\underline{\ZZ}(Sm^{prop})$-algebra structure satisfies the integral identity as soon as $\Sm$ contains all closed embeddings and $\pi^+\in \Pro\cap \Sm$. Similarly, every other $\underline{\ZZ}(\Sm^{\Pro})$-algebra $(\Sm,\Pro)$-theory  $R$ with the induced $\underline{\ZZ}(Sm^{prop})$-algebra structure will also do.  
\end{example}


\begin{theorem}[cf.\ Kontsevich, Soibelman \cite{KS2}] \label{integral_identity_2}
 Assume $\Char \kk=0$. Let $(\Sm,\Pro)\supset (Sm,proj)$ be a motivic pair, i.e.\ $\Pro\cap \Sm$ contains all locally closed embeddings, and let  $R$ be a motivic $(\Sm,\Pro)$-theory with  a $\underline{\Ka}_0(Sm^{proj})$-algebra structure on $R|_{(Sm,proj)}$. Then, the integral identity  
 \[ \pi^+_! \phi_f |_{V^+}=\pi^+_! \phi_{f|_{V^+}}=\LL^{\rk V^+}\phi_{f|_X}\]
 holds for every 2-graded vector bundle $\pi=\pi^+\oplus \pi^-:V^+\oplus V^-\longrightarrow X$ on a smooth scheme $X$ and every $\GG_m$-invariant morphism $f:V\to \MM$, where $\GG_m$ acts with weights $\pm 1$ on $V^\pm$.  
\end{theorem}
The proof of this Theorem goes back to M.\ Kontsevich and Y.\ Soibelman in the context of mixed Hodge modules. We give a slightly more detailed version of their proof in appendix A. In fact, we prove a categorification of this statement but passing to the Grothendieck group provides the arguments in our setting.  
\begin{example}\rm 
All of our examples $\phi^{mot},\phi^{mhm},\phi^{perv},\phi^{con}$ satisfy the requirements of the theorem, and the integral identity holds in this case. 
\end{example}

\subsection{Categorification}

More examples of $\lambda$-ring $(\Sm,\Pro)$-theories  can be obtained by decategorifying a categorified version of a  ring $(\Sm,\Pro)$-theory. 

\begin{definition}
A categorical  ring $(\Sm,\Pro)$-theory  is rule $\altT$ associating to every $f\in \Sch_\MM$ an essentially small Karoubian closed $\QQ$-linear\footnote{The restriction to $\QQ$-linear categories was made to have at least one decategorification functor to $\Th^\lambda(\MM)$. Moreover, it will provide us with an initial object.} triangulated category $\altT(f)$ which extends to a $\QQ$-linear covariant functor with respect to all morphisms in $\Pro_\MM$ and to a $\QQ$-linear contravariant functor with respect to all morphisms in $\Sm_\MM$. We require also the existence of a distinguished object $\unit\in \altT(\Spec \kk \xrightarrow{0}\MM)$ and a symmetric, associative exterior product $\boxtimes$ as before which is $\QQ$-bilinear and bi-exact. These data have to satisfy properties similar to the one  of Definition \ref{theory} and \ref{ring_theory}. In particular, pull-backs and push-forwards should be symmetric monoidal functors.\\
A morphism between two categorical  ring theories  $\altT_1,\altT_2$ is defined  by a family of $\QQ$-linear triangulated functors $\eta_f:\altT_1(f)\longrightarrow \altT_2(f)$ commuting with pull-backs along  morphisms in $\Sm_\MM$, push-forwards along  morphisms in $\Pro_\MM$, and exterior products in the sense of monoidal functors. Thus, we obtain a category $\CTh(\MM)$ of categorical  ring theories  over $\MM$. The subcategory $\CTh^\lambda(\MM)\subset \CTh(\MM)$ has the same objects but morphisms should be symmetric monoidal with respect to the $\boxtimes$-product. We write $\altT(X)$ for $\altT(X\xrightarrow{0}\MM)$.
\end{definition}
Note that we only require that $\altT(f)$ has finite sums which are also products as $\altT(f)$ is additive.
\begin{example} \rm
For $\kk=\mathbb{C}$, there is a categorical  ring $(\Sch_\kk,ft)$-theory  $D^b_{con}=D^b(\Perv)$ over $\kk$ such that $D^b_{con}(X)\cong D^b(\Perv(X))$ is the subcategory of the bounded derived category of sheaves of $\QQ$-vector spaces on $X$ consisting of complexes with constructible cohomology.  
\end{example}
\begin{example} \label{hodge_module} \rm
For $\kk=\mathbb{C}$, there is a categorical ring $(\Sch_\kk,ft)$-theory  $D^b(\MHM)$ over $\kk$ such that $D^b(\MHM)(X):=D^b(\MHM(X))$ is the  bounded derived category of mixed Hodge modules on $X$. By mapping a complex of mixed Hodge modules to the underlying complex of perverse sheaves, we end up with a morphism $\rat:D^b(\MHM)\longrightarrow D^b(\Perv)$ of  categorical  ring $(\Sch_\kk,ft)$-theories. The fact that $\boxtimes$ on $D^b(\MHM)$, pull-backs, push-forwards  and  $\rat$ are symmetric (monoidal) is not obvious and proven in \cite{Schuermann}. 
\end{example}
\begin{example} \rm
There is a categorical ring $(\Sch_\kk,ft)$-theory  $\DM_{gm}$ over $\kk$ such that $\DM_{gm}(X)$ is the  triangulated category of geometric mixed motives on $X$ with rational coefficients.  (See \cite{CisinskyDeglise} for more details.) If $\kk=\mathbb{C}$, there is a symmetric monoidal morphism of categorical ring $(\Sch_\kk,ft)$-theories  $\DM_{gm}\to D^b(\MHM)$.  
\end{example}

\begin{lemma} Given a categorical  ring $(\Sm,\Pro)$-theory, the functor $\Ka_0(\altT)(f):=\Ka_0(\altT(f))$ is a  ring $(\Sm,\Pro)$-pretheory, and by sheafification we obtain the  ``decategorification'' $\underline{\Ka}_0(\altT)$ of $\altT$, while $\altT$ is called a categorification of $\underline{\Ka}_0(\altT)$. Note that $\Ka_0(\altT(f))\longrightarrow \underline{\Ka}_0(\altT(f))$ is always surjective as $\altT$ satisfies the sheaf property. As every triangulated category is $\QQ$-linear, we can also define the additive Grothendieck groups $\Ka^0(\altT)(f):=\Ka^0(\altT(f))$ giving rise to a  $\lambda$-ring $(\Sm,\Pro)$-theory $\underline{\Ka}^0(\altT)$ after sheafification.   The operations $\sigma^n$ are induced by  $Sym^n|_{\altT(f)}:\altT(f)\longrightarrow \altT(\Sym^n(f))$, where $\Sym^n:\altT(\Sym(f))\longrightarrow \altT(\Sym(f))$ is the Schur functor (see section \ref{schurfunctors}) with respect to the convolution product $\oplus_!\circ \boxtimes: \altT(\Sym(f))\times \altT(\Sym(f))\longrightarrow \altT(\Sym(f))$ which is a symmetric monoidal $\QQ$-linear tensor product on $\altT(\Sym(f))$. Moreover, the $\lambda$-ring $\underline{\Ka}^0(\altT)(\Sym(f))$ is  special. In 
particular, $\underline{\Ka}^0(\altT)_{sp}=\Ka^0(\altT(\Spec\kk))$, and the class of every $\otimes$-invertible object in $\altT(\Spec\kk)$ is in $\Pic(\underline{\Ka}^0(\altT))$. \\
There is a canonical morphism $\underline{\Ka}^0(\altT)\to \underline{\Ka}_0(\altT)$ of ring $(\Sm,\Pro)$-theories.\\
Every morphism $\eta:\altT_1\to \altT_2$ in $\CTh(\MM)$  induces morphisms $\underline{\Ka}_0(\eta):\underline{\Ka}_0(\altT_1)\longrightarrow\underline{\Ka}_0(\altT_2)$ and $\underline{\Ka}^0(\eta):\underline{\Ka}^0(\altT_1)\longrightarrow\underline{\Ka}^0(\altT_2)$ of ring $(\Sm,\Pro)$-theories  compatible with $\underline{\Ka}^0(-) \to \underline{\Ka}_0(-)$. Thus the latter is  a morphism  functors $\underline{\Ka}_0,\underline{\Ka}^0:\CTh(\MM)\to \Th(\MM)$. If $\eta$ is even in $\CTh^\lambda(\MM)$, then $\underline{\Ka}^0(\eta)$ is also a morphism of $\lambda$-ring $(\Sm,\Pro)$-theories giving rise to a functor $\underline{\Ka}^0:\CTh^\lambda(\MM)\to \Th^\lambda(\MM)$.
\end{lemma}
\begin{remark} \rm
In general it is not true that the $\lambda$-ring structure on $\underline{\Ka}^0(\altT)$ descends to $\underline{\Ka}_0(\altT)$. However, if $\altT(f)=D^b(\mathcal{A}(f))$ is some sort of bounded derived category of a $\QQ$-linear abelian category $\Ab(\altT)$ stable under the convolution product $\oplus_!(-\boxtimes -)$ on $\altT(\Sym (f))$ for all $f\in \Sch_\MM$, then $\underline{\Ka}_0(\altT)$ is also a $\lambda$-ring $(\Sm,\Pro)$-theory and $\underline{\Ka}^0(\altT)\to \underline{\Ka}_0(\altT)$ is a morphism of $\lambda$-ring theories. Another special case is given if $\altT(f)=K^b(\Ab(f))$ is the triangulated homotopy category of bounded complexes in an additive Karoubian closed category $\Ab(f)$ preserved by the convolution product on $\altT(\Sym(f))$. See \cite{Biglari} for more details.  
\end{remark}
\begin{example} \label{perverse} \rm
The morphism $q:X^n\to X^n/\!\!/S_n$ is a finite morphism. In particular,  $q_!:D^b(\Perv(X^n))\to D^b(\Perv(X^n/\!\!/S_n))$ is exact and the convolution product preserves $\Perv(\Sym X)$. Hence, $\underline{\Ka}_0(D^b(\Perv))$ inherits the structure of a $\lambda$-ring $(\Sch_\kk,ft)$-theory over $\kk$. The same is true for Hodge modules and $\underline{\Ka}_0(D^b(\MHM))$ is a $\lambda$-ring theory. 
Given a complex of sheaves of $\QQ$-vector spaces on a scheme $X$ with constructible bounded cohomology, we obtain a constructible function on $X$ by taking the alternating sum of the dimensions of the stalks of the cohomology sheaves, i.e. $\sum_{i\in \ZZ} (-1)^i \dim_\QQ \mathcal{H}^i(\EE)_x$ for every $x\in X$ and $\EE\in D^b_{con}(X,\QQ)$. This defines a morphism of $\lambda$-ring $(\Sch_\kk,ft)$-theories  $\underline{\Ka}_0(D^b_{con})=\underline{\Ka}_0(D^b(\Perv)) \longrightarrow \Con$ over $\kk$. 
\end{example}
\begin{example}\rm 
 It has been shown in \cite{Biglari} that the $\lambda$-ring structure on $\Ka^0(DM_{gm}(\kk))$ descends to $\Ka_0(\DM_{gm}(\kk))$. We do not know if the proof can be extended to show that  $\underline{\Ka}_0(DM_{gm})$ inherits the structure of a $\lambda$-ring theory over $\kk$.  
\end{example}

\begin{example} \rm
Let $Y\xrightarrow{g}\MM$ be a monoid in the symmetric monoidal category $\Sm_\MM$. For $X\xrightarrow{f} \MM$ in $\Sch_\MM$ we denote with $g(f)$ the set of all morphisms in $\Sm\Pro^{op}_\MM$, i.e.\ the set of isomorphism classes of all commutative diagrams
\[ \xymatrix @R=0.5cm { & Z \ar[dl]_p \ar[dr]^s \ar[dd] &\\  X \ar[dr]_f & & Y \ar[dl]^g \\ & \MM & } \]
with $p\in \Pro$ and $s\in \Sm$. We denote with $D^b\QQ g(f)$ the bounded derived category of $g(f)$-graded $\QQ$-vector spaces with finite total dimension. We have  $D^b\QQ g(f)=\bigoplus_{g(f)} D^b(\Vect_\QQ)$, hence $\Ka^0(D^b\QQ g(f))=\Ka_0(D^b\QQ g(f))=\ZZ g(f)$ is the free abelian group generated by the set $g(f)$. Given a morphism $u:f'\to f$ in $\Sm\Pro^{op}_\MM$, we get a map $u^\ast:g(f)\to g(f')$ of ``indices'' inducing a corresponding functor between the  categories of (complexes of) graded vector spaces by reparametrization of the indices. As $g$ is a monoid in $\Sm_\MM$, we have a map $g(f)\times g(f')\longrightarrow g(f\boxtimes f')$. Using tensor products of (complexes of) graded vector spaces, we get a $\boxtimes$-product $D^b\QQ g(f) \times D^b\QQ g(f')\longrightarrow D^b\QQ g(f\boxtimes f')$. Therefore, $D^b\QQ g:f \to D^b\QQ g(f)$ is a categorical ring $(\Sm,\Pro)$-pretheory over $\MM$ which becomes a theory $\underline{D}^b\QQ g$, once we put $\underline{D}^b\QQ g(f):=\prod_{X_\kappa\in \pi_0(X)} D^b\QQ g(f|_{X_\kappa})$. Although the Grothendieck groups $\Ka_0(\underline{D}^b\QQ g(f))$ and $\Ka^0(\underline{D}^b\QQ g(f))$ of the possibly infinite product might be very complicated, their sheafification is just $\underline{\ZZ} g(f)$. Therefore $\underline{D}^b\QQ g$ is a categorification of $\underline{\ZZ} g$ and the induced $\lambda$-ring structure is the one of Example \ref{darstellbar}.  
\end{example}
\begin{example}\rm
We can modify the previous example by replacing $\QQ$ with any commutative $\QQ$-algebra $R$ which is a principal ideal domain, e.g.\ a field extension $K\supset \QQ$ or the ring of global functions on an affine smooth connected curve over $\QQ$. In this case $D^b R g(f)$ is just the bounded derived category of complexes of $g(f)$-graded $R$-modules with finite total rank. The theory of modules over principal ideal domains shows that $\Ka_0(D^b R g(f))=\ZZ g(f)$ and $\underline{D}^b R g$ provides another ``decategorification'' of $\underline{\ZZ}g$. However, the case $R=\QQ$ is somewhat exceptional as we can generalize Lemma \ref{initial_object}.   
\end{example}
\begin{lemma} \label{initial_object_category}
The categorical ring $(\Sm,\Pro)$-theory $\underline{D}^b \QQ\kk$ is the initial object in $\CTh(\MM)$. 
\end{lemma}
\begin{proof}
Let us mention that $D^b(\Vect_\QQ)$ is (equivalent to) the free $\QQ$-linear category generated by one element $\QQ$ and an autoequivalence $[1]$. It comes with a symmetric monoidal tensor product satisfying $V[1]\otimes W=V\otimes W[1]=(V\otimes W)[1]$. In particular, there is an associative action of $D^b(\Vect_\QQ)$ on every triangulated category $\altT(f)$ such that $\QQ^n\otimes A=A^{\oplus n}$ and $V[1]\otimes A=V\otimes A[1] = (V\otimes A)[1]$ for every $A\in \altT(f)$ and every vector space $V$. (See \cite{Hovey07}, section 4.1 for background on modules over categories.) Moreover, every $\QQ$-(bi)linear triangulated functor is $D^b(\Vect_\QQ)$-(bi)linear. Now, $D^b \QQ \kk(f)$ is the free $D^b(\Vect_\QQ)$-module category generated by $\kk(f)$. In other words, every additive $\QQ$-linear functor $\eta_f:D^b\QQ\kk(f)\to \altT(f)$ commuting with $[1]$ is uniquely determined by specifying the images of the one-dimensional vector space $\QQ_{[\alpha]}$ of index $[\alpha]\in \kk(f)$ for $\alpha:U\to X_0\hookrightarrow X$ in $\Pro$ with $U$ an $\Sm$-scheme. Such a functor is automatically triangulated as every triangle in $D^b\QQ\kk(f)$ splits. Functoriality in $f$ requires  $\eta_f(\QQ_{[\alpha]})=\alpha_! c_U^\ast(\unit)$ as in the proof of Lemma \ref{initial_object} with $c_U:U\to \Spec\kk$. Moreover, the definition of the $\boxtimes$-product on $\underline{D}^b\QQ \kk$ shows that the unique morphism to a categorical ring theory preserves the $\boxtimes$-product but might not be symmetric monoidal.        
\end{proof}

If $(X\xrightarrow{f}\MM,+,0)$ is a commutative monoid in $\Pro_\MM$, the symmetric monoidal convolution product on $\altT(f)$ is defined as before by means of 
\[ \altT(f)\times \altT(f)\xrightarrow{\boxtimes} \altT(f\boxtimes f)\xrightarrow{+_!} \altT(f) \]
with unit object $1_f:=0_!(1)$. The analogue of Proposition \ref{algebraic_structure} remains true in the categorical setting. The associated Schur functor $\Sym^n:\altT(f)\to \altT(f)$ induces the $\lambda$-ring structure on $\underline{\Ka}_0(\altT)(f)$ making it into a special $\lambda$-ring. \\

If $u:\MM\to N$ is a homomorphism of commutative monoids, the pull-back $u^\ast(\altT)$ of a categorical ring $(\Sm,\Pro)$-theory  over $N$ is defined as for ring $(\Sm,\Pro)$-theories  by composition with $u_!$. Similarly, for $u\in \Pro$ we can define push-forward $u_!(\altT)$ of a categorical ring $(\Sm,\Pro)$-theory  over $\MM$. The push-forward is left adjoint to the pull-back functor. A categorical theory on $\MM$ which is the pull-back of a categorical theory on $\Spec\kk$ is called reduced. Moreover, the decategorification $\underline{\Ka}_0$ commutes with pull-backs and push-forwards. In particular, the decategorification of a reduced categorical theory is reduced.\\

\begin{example}\rm
For the sake of completeness, let us describe the pull-back of $\underline{D}^b\QQ\kk$ along $\MM\to \kk$. For $X \in \Sch_\kk$ we denote  with $\Sm^{\Pro}_X$ as before the set of isomorphism classes of maps $Z\xrightarrow{p} X$ in $\Pro$ with $Z$ is an $\Sm$-scheme, i.e.\ the set of morphisms $X\to \Spec\kk$ in $\Sm\Pro^{op}_\kk$. Let $D^b\QQ(\Sm^{\Pro}_X)$ the bounded derived category of $\Sm^{\Pro}_X$-graded $\QQ$-vector spaces with finite total cohomology. We have $D^b\QQ(\Sm^{\Pro}_X)=\bigoplus_{\Sm^{\Pro}_X} D^b(\Vect_\QQ)$, hence $\Ka^0(D^b\QQ(\Sm^{\Pro}_X))=\Ka_0(D^b\QQ(\Sm^{\Pro}_X))=\ZZ(\Sm^{\Pro}_X)$ is the free abelian group generated by $\Sm^{\Pro}_X$. Given a morphism $u:Y\to X$ in $\Sm\Pro^{op}_\kk$, we get a map $u^\ast:\Sm^{\Pro}_X\to \Sm^{\Pro}_Y$ of ``indices'' inducing a corresponding functor between the  categories of (complexes of) graded vector spaces by reparametrization of the indices. In a similar way, we get a $\boxtimes$-product using tensor products of (complexes of) graded vector spaces. Therefore, \[ D^b\QQ(\Sm^{\Pro}):(X\xrightarrow{f}\MM)\longmapsto D^b\QQ(\Sm^{\Pro}_X)\] is a categorical ring $(\Sm,\Pro)$-pretheory which becomes a theory $\underline{D}^b\QQ(\Sm^{\Pro})$, once we put $\underline{D}^b\QQ(\Sm^{\Pro}_f):=\prod_{X_\kappa\in \pi_0(X)} D^b\QQ(\Sm^{\Pro}_X)$. Although the Grothendieck group of the possibly infinite product might be very complicated, it sheafification is just $\underline{\ZZ}(\Sm^{\Pro})$. 
\end{example}
As a consequence of Lemma \ref{initial_object_category}, there is a canonical morphism $\underline{D}^b\QQ(\Sm^{\Pro})\to \altT$ of categorical reduced ring $(\Sm,\Pro)$-theories.
\begin{definition}
 Given a reduced categorical ring $(\Sm,\Pro)$-theory $\altT$ and $X\in \Sch_\kk$, we denote with $\altT_{alg}(X)$ the Karoubian closed full triangulated subcategory of $\altT(X)$ generated by all elements $p_!(c_Z^\ast(1))$ for morphisms  $p:Z\to X$ in $\Pro$ with $Z$ being an $\Sm$-scheme. It is easy to see that $X\mapsto \altT_{alg}(X)$ is a reduced categorical ring $(\Sm,\Pro)$-theory. It is the smallest full subtheory containing the image of $\underline{D}^b\QQ(\Sm^{\Pro})\longrightarrow \altT$.  
\end{definition}
The notion of a categorical $(\lambda$-)algebra $(\Sm,\Pro)$-theory is defined as for algebra $(\Sm,\Pro)$-theories using morphisms in $\CTh^{(\lambda)}(\MM)$. The categorification of Lemma \ref{framing} is the following statement which is proven similar to Lemma \ref{initial_object_category}. 
\begin{lemma}
\label{framing_category}
A  $\underline{D}^b\QQ(\Sm^{\Pro})$-algebra structure on a categorical $(\Sm,\Pro)$-theory  $\altT$ over $\MM$ is (up to isomorphism) uniquely given by  a family of distinguished objects $\phi_f\in \altT(f)$ for $X\xrightarrow{f}\MM$ with $X$ being an $\Sm$-scheme, functorial with respect to isomorphism, satisfying the following properties.
\begin{enumerate}
 \item For every  morphism $u:Y\to X$ in $\Sm$, we obtain $\phi_{f\circ u}\cong u^\ast(\phi_f)$.
 \item We have $\phi_{\Spec\kk\xrightarrow{0}\MM}\cong 1$, and if $X\xrightarrow{f}\MM$ and $Y\xrightarrow{g}\MM$ are given with $\Sm$-schemes $X$ and $Y$, then $\phi_{f\boxtimes g}\cong\phi_f\boxtimes \phi_g$, where the isomorphism commute with associativity, symmetry and unit isomorphisms of $\boxtimes$. 
\end{enumerate}
The induced morphism $\underline{D}^b\QQ(\Sm^{\Pro})\longrightarrow \altT$ is in $\CTh^\lambda(\MM)$, i.e.\ symmetric monoidal, if and only if $\Sym^n(\phi_f)=(\phi_f)^{\otimes n}$ for the convolution product $\otimes:=\oplus_!(-\boxtimes -)$ on $\Sym(f)$ and all $f\in \Sch_\MM$.  
\end{lemma}
The decategorification functor $\underline{\Ka}_0$ lifts to a corresponding decategorification functor from  $\underline{D}^b\QQ(\Sm^{\Pro})$-algebra to $\underline{\ZZ}(\Sm^{\Pro})$-algebra $(\Sm,\Pro)$-theories. Moreover, $\underline{D}^b\QQ(\Sm^{\Pro})$-algebra structures are compatible with pull-backs along monoid morphisms $\MM\to N$. \\
\begin{example} \rm
 Given a reduced categorical $(\Sm,\Pro)$-theory $\altT$ on $\MM$, the canonical $\underline{D}^b\QQ(\Sm^{\Pro})$-algebra structure is given by $\phi_f:=\unit_X=c_X^\ast(\unit)$ for $X\xrightarrow{f}\MM$ and every $\Sm$-scheme $X \xrightarrow{c_X}\Spec\kk$.
\end{example}

The idempotent endofunctor $R\mapsto R^{fib}$ has an obvious idempotent categorification $\altT\mapsto \altT^{fib}$ together with a morphism $\altT^{fib} \longrightarrow \altT$ which is  the identity after reduction. There is also a categorified version of a trivialization $\altT^{red,fib}\cong\altT^{fib}$ which is uniquely given by a family of objects $t^a\in \altT(\Spec\kk\xrightarrow{c_a}\MM)$ satisfying 
\begin{enumerate}
 \item $t^0=1$ and $t^a\boxtimes t^b\cong t^{a+b}$ for all $a,b\in \MM(\bar{\kk})$,
 \item $-\boxtimes t^a: \altT(X)\to \altT(X\xrightarrow{c_a}\MM)$ is an equivalence of categories for all $a\in \MM(\bar{\kk})$ and all $X\in \Sch_\kk$.
\end{enumerate}
Given a pair $(\Sm,\Pro)\supset (Sm,proj)$ such that $\Sm$ contains all closed embeddings. A vanishing cycle is a trivialized categorical ring $(\Sm,\Pro)$-theory  $\altT$ over $\AA^1$ with a $\underline{D}^b\QQ(Sm^{proj})$-algebra structure on $\altT|_{Sm,proj}$ such that $\phi_f\in \altT(f)$ is supported on $\Crit(f)$, i.e.\ contained in the image of $\altT(\Crit(f))\cong \altT(f|_{\Crit(f)}) \hookrightarrow \altT^{fib}(f)\hookrightarrow \altT(f)$. The following lemma provides a construction of vanishing cycles.

\begin{lemma}
Suppose that $\altT$ is a categorical ring theory  over $\AA^1$, and let $\phi^0_f\in \altT(X_0)$ be a distinguished element for every $f:X\to \AAA$ with $X$ being smooth, where $X_0$ is the fiber over $0\in \AA^1$. Assume that
\begin{enumerate}
\item $u^\ast(\phi_f^0)=\phi_{f\circ u}^0$ for every smooth morphism $u:Y\to X$, 
\item $\phi^0_f$ is supported on $\Crit(f)\cap X_0$ (support property),
\item if $X\xrightarrow{f}\AAA$ and $Y\xrightarrow{g}\AAA$ are morphisms with smooth $X$ and $Y$, the  inclusion $i_{f,g}:(X_0\cap\Crit(f))\times (Y_0\cap\Crit(g))\hookrightarrow (X\times Y)_0\cap\Crit(f\boxtimes g)$ is open, and we require the existence of an isomorphism $ \phi^0_f\boxtimes \phi^0_g \longrightarrow i_{f,g}^\ast \phi^0_{f\boxtimes g} $ commuting with associativity, symmetry and unit isomorphisms. (Thom-Sebastiani)
\end{enumerate}
Then there is a vanishing cycle on $\altT$ given by 
\[ \phi_f=\oplus_{a\in\AAA(\bar{\kk})} \phi^0_{f-a}. \]
\end{lemma}
\begin{proof}
Note that the sum is finite for  smooth and connected $X$ as $\phi^0_{f-a}$ is supported on $\Crit(f)\cap X_a$ which is empty for all but for finitely many $a\in \AAA(\bar{\kk})$. For non-connected $X$ the sum is still finite after restriction to every connected component of $X$ and gives, therefore, a well-defined element  
\[ \phi_f:=\oplus_{a\in\AAA(\bar{\kk})} \phi^0_{f-a}=\oplus_{a\in\AAA(\bar{\kk})}\phi^a_f\] 
of $\altT^{fib}(f)\subseteq \altT(f)$, where we used the shorthand $\phi^a_f:=\phi^0_{f-a} \in \altT(X_a)$ for $f:X\to \AAA$ with smooth $X$.  Moreover, we get a natural morphism $\phi_f\boxtimes \phi_g \longrightarrow \phi_{f\boxtimes g}$ commuting with the associativity, commutativity and unit isomorphisms of the exterior product. \\
By applying Thom-Sebastiani to $f-a$ and $g-b$ for $a,b\in \AAA(\bar{\kk})$, we obtain 
\[ \phi^a_f\boxtimes \phi^b_g \xrightarrow{\;\sim\;} \phi^{a+b}_{f\boxtimes g}|_{(X_a\cap\Crit(f))\times (Y_b\cap\Crit(g))}  \]
if $X$ and $Y$ were smooth and connected. Since $(X\times Y)_c \cap \Crit(f\boxtimes g) =\sqcup_{a+b=c} (X_a\cap \Crit(f))\times (Y_b\cap\Crit(g))$ as schemes, we conclude using the support property of $\phi^0$ that 
\[ \phi_f\boxtimes \phi_f\xrightarrow{\;\sim\;}\phi_{f\boxtimes g}  \]
for all $X\xrightarrow{f}\AAA, Y\xrightarrow{g}\AAA$ with smooth connected $X,Y$. By taking the product over all connected components of $X$ and $Y$, we get the same statement for all morphisms $f,g$.  Also $\phi_f\in \altT(f)$ commutes with pull-backs along smooth morphisms. Hence, $\phi_f\in \altT(f)$ is a $\underline{D}^b\QQ(Sm^{proj})$-structure on $\altT$. The support property ensures that this structure is a vanishing cycle.
\end{proof}
\begin{example} \rm
When applied to the classical vanishing cycle functor $\phi^0_f\in D^b(\Perv(X_0))$
the previous lemma provides a vanishing cycle  $\phi^{perv}$  on $D^b(\Perv)$ considered as a categorical ring theory  over $\AAA$,  and also on the associated Grothendieck group $\underline{\Ka}_0(\Perv)$. The $\underline{D}^b\QQ(Sm^{proj})$-algebra structure has a factorization through the canonical $\underline{D}^b\QQ(Sm^{proj})$-algebra structure on $D^b(\Perv)_{alg}$, i.e.\ there is a factorization $\phi^{perv}:\underline{D}^b\QQ(Sm^{proj})\xrightarrow{1} D^b(\Perv)_{alg}\xrightarrow{\phi^{perv}} D^b(\Perv)$ of the underlying ring $(Sm,proj)$-theories over $\AAA$.  
\end{example}
\begin{example} \rm
Given a complex variety $X$ we define $D^b(\MHM_{mon}(X))$ to be the full subcategory of $D^b(\MHM(X\times\GG_m))$ containing those $\mathcal{F}$ such that for every $x\in X$ and for every $i\in\ZZ$ the pull-back $(x\times\GG_m\hookrightarrow X\times\GG_m)^*\mathcal{H}^i(\mathcal{F})$ is a smooth mixed Hodge module.  This category has a tensor structure described in detail in \cite[Sec.\ 7]{KS2}.  Given a function $f:X\rightarrow\AA^1$ we define as in \cite{KS2} $\phi^0_f:D^b(\MHM(X))\rightarrow D^b(\MHM_{mon}(X))$ via 
\[
\phi^0_f\mathcal{F}=\phi_{f/u}\pi_X^*\mathcal{F}
\]
where $\pi_X:X\times\GG_m\rightarrow X$ is the projection and $u$ is a coordinate on $\GG_m$.  When applied to the vanishing cycle functor $\phi^0_f$, the previous lemma provides a vanishing cycle on $D^b(\MHM_{mon})$ with a factorizing  $\phi^{mhm}:\underline{D}^b\QQ(Sm^{proj})\xrightarrow{1} D^b(\MHM)_{alg} \xrightarrow{\phi^{mhm}} D^b(\MHM_{mon})$, and also between the associated Grothendieck groups.  There is a full subcategory $D^b(\MHM_{ssimp}(X))\subset D^b(\MHM_{mon}(X))$ consisting of those $\mathcal{F}\in D^b(\MHM_{mon}(X))$ such that each pull-back $(x\times\GG_m\hookrightarrow X\times\GG_m)^*\mathcal{H}^i(\mathcal{F})$ has semisimple quasi-unipotent monodromy, such that the map $\Ka_0(\MHM_{ssimp}(X))\rightarrow \Ka_0(\MHM_{mon}(X))$ induced by the inclusion of categories is itself an inclusion.  It is a well-known fact that $\phi^{mhm}$ maps actually into $D^b(\MHM_{ssimp})$.  \\
There is also a natural map $\underline{\Ka}^{\muu}(\Sch)\rightarrow \Ka_0(D^b(\MHM_{ssimp}(X)))=\Ka_0(\MHM_{ssimp}(X))$ given by sending $[Y\xrightarrow{f} X\times \AA^1]\mapsto -[(j\times \id_X)_!f_!\QQ_Y]+[\pi_X^*(f|_{Y_0})_!\mathbb{Q}_{Y_0}]$, where $j:\GG_m\hookrightarrow\AA^1$ is the natural 
inclusion.  The following diagram commutes
\[
\xymatrix { \underline{\Ka}_0(\Sch^{ft})\ar[d]^{\phi^{mot}}\ar[r]& \underline{\Ka}_0(\MHM_{alg})\ar[d]^{\phi^{mhm}} \ar@{=}[r] &\underline{\Ka}_0(\MHM_{alg})\ar[d]^{\phi^{mhm}}\\
\underline{\Ka}^\muu(\Sch^{ft}) \ar[r] &\underline{\Ka}_0(\MHM_{ssimp})\ar[r] &\underline{\Ka}_0(\MHM_{mon}).
}\]
\end{example}

One can extend the morphisms of Examples \ref{perverse} and \ref{hodge_module} to obtain a commutative diagram
\[ \xymatrix { \underline{\Ka}_0(\Sch^{ft}) \ar[d]^{\phi^{mot}} \ar[r] & \underline{\Ka}_0(\MHM_{alg}) \ar[r] \ar[d]^{\phi^{mhm}} & \underline{\Ka}_0(\Perv_{alg}) \ar[r] \ar[d]^{\phi^{perv}} & \Con \ar[d]^{\phi^{con}} \\
\underline{\Ka}^\muu(\Sch^{ft}) \ar[r] & \underline{\Ka}_0(\MHM_{mon}) \ar[r] & \underline{\Ka}_0(\Perv) \ar[r] &  \Con. } \]
in the category of ring $(Sm,proj)$-theories over $\AAA$.


Recall that a pair $(\Sm,\Pro)$ was called motivic, if $\Pro\cap \Sm$ contains all open and all closed embeddings. As a result $e^\ast e_!\cong \id$ for all locally closed embeddings $e$. 
\begin{definition}
Let $(\Sm,\Pro)$ a motivic pair. A categorical ring $(\Sm,\Pro)$-theory $\altT$ is called motivic if for every closed embedding $i_!$ is the right adjoint functor of $i^\ast$, and for every open embedding $j_!$ is the left adjoint functor of $j^\ast$. Moreover, the sequence $j_! j^\ast(a) \to a\to  i_!i^\ast(a)$ should extend to a distinguished triangle in $\altT(f)$ functorial in $a\in \altT(f)$.
\end{definition}

The following result should be seen as a categorification of the integral identity. It is due to Kontsevich and Soibelman, at least in the context of Mixed Hodge modules. A slightly expanded version of their proof is contained in the appendix for the sake of completeness. 
\begin{theorem}[cf.\ Kontsevich, Soibelman \cite{KS2}]
 Let $(\Sm,\Pro)\supset(Sm,proj)$ be a motivic pair, i.e.\ $\Pro\cap \Sm$ contains all locally closed embeddings. Let  $\altT$ be a motivic categorical ring $(\Sm,\Pro)$-theory with a $\underline{D}^b\QQ(Sm^{proj})$-algebra structure over $\MM$ which has a factorization through the canonical $\underline{D}^b\QQ(Sm^{proj})$-algebra structure of a motivic reduced categorical ring $(\Sm,\Pro)$-theory $\mathcal{K}$, i.e.\ there is a morphism $\phi:\mathcal{K}\to \altT$ of categorical ring $(Sm,proj)$-theories over $\MM$, such that $\phi_f=\phi_f(\unit_X)$ for every $f:X\to \MM$ with smooth $X$. The $(\Sm,\Pro)$-theory $\mathcal{K}$ should also satisfy the following assumption: for every locally closed embedding $e$ the functors $e_!$ and $e^\ast$ have right adjoints $e^!$ and $e_\ast$ which implies $e_!=e_\ast$ for every closed embedding and $e^!=e^\ast$ for every open embedding as $\mathcal{K}$ is motivic. Moreover, given a closed embedding $i:Z\hookrightarrow X$ with open complement $j:U\hookrightarrow X$, the sequence
 \[ i_! i^!(a) \longrightarrow a \longrightarrow j_\ast j^\ast (a) \]
 of adjunction morphisms can be extended to a distinguished triangle functorial in $a\in \mathcal{K}(X)$. Finally, we assume  $\int_{\AA^1}(\GG_m\hookrightarrow \AA^1)_\ast(\unit_{\GG_m})=0$ in $\mathcal{K}(\kk)$ (``homotopy invariance''). Then 
 \[ \pi^+_! \phi_f |_{V^+}\cong\pi^+_! \phi_{f|_{V^+}}\cong\LL^{\rk V^+}\phi_{f|_X}.\]
 for every 2-graded vector bundle $\pi=\pi^+\oplus \pi^-:V^+\oplus V^-\longrightarrow X$ on a smooth scheme $X$ and every $\GG_m$-invariant morphism $f:V\to \MM$, where $\GG_m$ acts with weights $\pm 1$ on $V^\pm$.   
\end{theorem}
As $\underline{\Ka}_0(\Sch^{ft})|_{(Sm,proj)}=\underline{\Ka}_0(Sm^{proj})$ for $\Char\kk=0$ by a result of Bittner (see \cite{Bittner04}), the categorical ring $(\Sm,\Pro)$-theory $\mathcal{K}$ should be seen as a replacement of a categorification of $\underline{\Ka}_0(\Sch^{ft})|_{(\Sm,\Pro)}$ which does not exist (cf.\ Example \ref{motives_special}). 
\begin{example}\rm The categorical ring $(\Sch_\kk,ft)$-theories $D^b(\Perv)_{alg}$ and $D^b(\MHM)_{alg}$ satisfy the requirements on $\mathcal{K}$ and the Theorem applies to $\phi^{perv}$ and to $\phi^{mhm}$. 
\end{example}

\section{Donaldson--Thomas theory, framed version}
This section provides a first definition of Donaldson--Thomas functions and sheaves for abelian categories (with potential) of homological dimension at most one. It is more general than the one given by Kontsevich/Soibelman and Joyce as the latter requires Artin stacks which we consider in section \ref{theories_Artin_stacks}. On the other hand, if this extended structure exists, the definition given here is equivalent to the standard one by Kontsevich/Soibelman and Joyce which is the topic of section \ref{Donaldson_Thomas_general_version}. 
As the reader will see, the definition given here requires no stacks and works with moduli spaces only. Unfortunately, we could not prove a wall-crossing formula in this general setting. \\

We fix a pair $(\Ab,\omega)$ satisfying all assumptions (1)--(6) stated in section 2, a $\lambda$-ring $(Sm,proj)$-theory  $R$ over the moduli space $\Msp$ of objects in $\Ab$, and a $\underline{\ZZ}(Sm^{proj})$-algebra structure $\phi:\underline{\ZZ}(Sm^{proj})\to R$ over $\Msp$. Note that every $R'$-algebra structure  on $R$ with reduced $R'$ will induce a $\underline{\ZZ}(Sm^{proj})$-algebra structure by composition with the canonical $\underline{\ZZ}(Sm^{proj})$-algebra structure on $R'$. In addition to this we require the existence of some element $-\LL^{-1/2}\in \Pic(R)$ such that $(\LL^{1/2})^2=\LL:=[\PP^1]-1$. In particular, we can extend $\phi$ to a $\ZZ(Sm^{proj}_\kk)[\LL^{-1/2}]$-algebra structure on $R$. We do not require that $\phi$ commutes with the $\lambda$-ring structure on $\underline{\ZZ}(Sm^{proj})$ and on $R$.  

\begin{example} \label{example3} \rm
If $R$ is a reduced $\lambda$-ring theory  such that $1,\LL\in \Pic(R)$, we may replace $R$ with $R\otimes R_{sp}\langle \LL^{-1/2}\rangle^-$ as in Example \ref{lambda_adjunction} and Lemma \ref{lambda_adjunction_5}. Taking the canonical $\underline{\ZZ}(Sm^{proj})$-algebra structure, we obtain a large class of examples involving $\Con$, $\underline{\Ka}_0(\Perv)$ with $\LL^{-1/2}=-1$, $\underline{\Ka}_0(\MHM)\langle\LL^{-1/2}\rangle^-$ and $\underline{\Ka}_0(\Sch^{ft})\langle \LL^{-1/2}\rangle^-$.
\end{example}
\begin{example} \rm \label{example2}
We can also take the examples presented at the end of the last section
\[ \xymatrix { \underline{\Ka}_0(\Sch^{ft}) \ar[d]^{\phi^{mot}} \ar[r] & \underline{\Ka}_0(\MHM_{alg}) \ar[r] \ar[d]^{\phi^{mhm}} & \underline{\Ka}_0(\Perv_{alg}) \ar[r] \ar[d]^{\phi^{perv}} & \Con_{alg} \ar[d]^{\phi^{con}} \\
\underline{\Ka}^\muu(\Sch^{ft})\langle\LL^{-1/2}\rangle^- \ar[r] & \underline{\Ka}_0(\MHM_{mon})\langle \LL^{-1/2}\rangle^- \ar[r] & \underline{\Ka}_0(\Perv) \ar[r] &  \Con. } \]
\end{example}
Note that  in all of our examples $\phi$ is actually a vanishing cycle.

\subsection{Donaldson--Thomas functions and invariants} \label{Donaldson_Thomas_framed}

Let us start by defining the ``intersection complex $\ICS_X\in \underline{\ZZ}(Sm^{proj}_X)[\LL^{-1/2}]$ of a smooth scheme $X$ by the requirement that $\IC_X|_{X_i}=\LL^{-\dim X_i/2}\unit_{X_i}=L^{-\dim X_i/2}[X_i\xrightarrow{id}X_i]$ for each connected component $X_i$ of $X$. 
\begin{example} \label{example5} \rm Under the canonical $\underline{\ZZ}(Sm^{proj})$-algebra structure of $\underline{\Ka}_0(\Perv)$, the intersection complex $\ICS_X$ maps to the class of the classical intersection complex $\ICS^{perv}_X$ of $X$. In case of $\underline{\Ka}_0(\MHM)$ the situation is similar. However, $\LL^{-1/2}$ is not the class of an element in $\MHM(\kk)$. Instead, by applying the Tannakian formalism, we write $\MHM(\kk)\cong G^{mhm}\rep$ for some algebraic group $G^{mhm}$ which comes with a character $G^{mhm}\to \GG_m$ given by the inclusion of the abelian category of pure Hodge structures $\QQ(n)$ of weight $-2n$. We may define $\MHM_X[\QQ(1/2)]:=\MHM_X\otimes_{\MHM(\kk)}\tilde{G}^{mhm}\rep$, where $\tilde{G}^{mhm}$ is the twofold cover $G^{mhm}\times_{\GG_m}\GG_m$ of $G^{mhm}$ induced by $\GG_m\ni z\mapsto z^2\in \GG_m$. The one dimensional representation $\QQ(1/2)$ is given by the  character $\tilde{G}^{mhm}\xrightarrow{\pr_{\GG_m}} \GG_m$, and $\LL^{-1/2}$ is the class of $\QQ(1/2)[1]$ in $\Ka_0(\MHM(\kk)[\QQ(1/2)])\cong\Ka_0(\MHM(
\kk))\langle\LL^{-1/2}\rangle^-$. One can extend the weight filtration to $\MHM(X)[\QQ(1/2)]$ by requiring that  $\QQ(1/2)$ has weight $-1$. Using the canonical $\underline{\ZZ}(Sm,proj)$-algebra structure, the ``intersection complex'' maps to the class of the ``renormalized'' classical intersection complex which has now weight 0. 
\end{example}  
To simplify notation, let us introduce the shorthands $\id:=\id_{\Msp}$ and $\iota_d:=(\Msp_d\hookrightarrow \Msp)$ for $d\in \NN^{\oplus I}$. We use the notation of section \ref{Sec_fr_rep}. 
\begin{definition} 
A Donaldson--Thomas function is an element \[ \DTS(\Ab,\phi)=(\DTS(\Ab,\phi)_d)_{d\in \NN^{\oplus I}} \in R(\id)=\prod_{d\in \NN^{\oplus I}}R(\iota_d) \] with $\DTS(\Ab,\phi)_0=0$  such that for every non-zero framing vector $0\not=f\in \NN^I$ the following equation holds\footnote{Instead of infinite sums, we should better write products. However, multiplication and $\lambda$-operations look more natural using infinite sums as in the case of power series.} 
\begin{eqnarray} \label{defining_equation} \pi_{f\,!}\sum_{d\in \NN^{\oplus I}} \LL^{fd/2}\phi_{\pi_{f,d}}(\ICS_{\Msp_{f,d}})&=& \sum_{\gamma\in \Gamma} \LL^{(\gamma,\gamma)/2}\phi_{\id}\big([\Msp_{f,\gamma}\xrightarrow{\pi_{f,\gamma}}\Msp]\big) \nonumber \\ & = &\Sym\Big(\sum_{0\not= d\in \NN^{\oplus I}} \LL^{1/2}[\PP^{fd-1}] \DTS(\Ab,\phi)_d\Big)
\end{eqnarray}
in $R(\id)$ with $\Msp_{f,\gamma}=\pi_f^{-1}(\Msp_\gamma)=\sqcup_{d\in \NN^{\oplus I}} \Msp_{f,\gamma,d}$ and $\dim\Msp_{f,\gamma,d}=fd-(\gamma,\gamma)$ if nonempty. Notice that $\PP^{fd-1}$ is the fiber of $\pi_{f,d}$ over $\Msp^s_d$ if $\Msp^s_d\not=\emptyset$. If $\phi$ is the canonical $\underline{\ZZ}(Sm^{proj})$-algebra structure of a reduced ring theory  as in Example \ref{example3}, we simply write $\DTS(\Ab,R)$. 
\end{definition} 
If the framing vector $f$ is even, i.e.\ $f\in(2\NN)^I$, the map $(a_d)_{d\in\NN^{\oplus I}}\longmapsto (\LL^{-fd/2}a_d)_{d\in\NN^{\oplus I}}$ defines a $\lambda$-ring isomorphism on $R(\id)$, and our defining equation is equivalent to 
\begin{equation}\label{defining_equation_2} \pi_{f\,!}\phi_{\pi_f}(\ICS_{\Msp_{f}})=\phi_{\id}\big(\pi_{f\,!} \ICS_{\Msp_{f}}\big)=\Sym\Big(\sum_{0\not= d\in \NN^{\oplus I}}  \DTS(\Ab,\phi)_d\int_{\PP^{fd-1}}\ICS_{\PP^{fd-1}}\Big) \end{equation}
with
\[ \int_{\PP^{fd-1}}\ICS_{\PP^{fd-1}} = \frac{\LL^{df/2}-\LL^{-df/2}}{\LL^{1/2}-\LL^{-1/2}}=:[\PP^{fd-1}]_{vir}\] 
for the fiber $\PP^{fd-1}$ of $\pi_{f,d}$ over $\Msp^s_d$.\\

\begin{example}\rm If we use the reduced $\lambda$-ring theory  $\underline{\Ka}_0(\Sch^{ft})\langle \LL^{-1/2}\rangle^-$ with the canonical framing, we simple write $\DTS(\Ab,mot)$ for the Donaldson--Thomas function. Recall one of the  main results of \cite{Meinhardt4}.
\begin{theorem} Donaldson--Thomas functions $\DTS(\Ab,mot)_d$ exist for every $d\in \NN^{\oplus I}$. \end{theorem}
\end{example}

\begin{remark} \rm The following statements are obvious.
\begin{enumerate}
  \item[(i)] If $R,\phi$ is a $\underline{\ZZ}(Sm^{proj})$-algebra and $\eta:R\to R'$ is a morphism of  $\lambda$-ring $(Sm,proj)$-theories  over $\Msp$, then $(R',\eta\phi)$ is another $\underline{\ZZ}(Sm^{proj})$-algebra satisfying our assumptions and $\DTS(\Ab,\eta\phi)=\eta_{\id}(\DTS(\Ab,\phi))$.
  \item[(ii)]  If $\phi=W^\ast(\phi')$ for some potential $W:\Msp\to \AAA$ and some vanishing cycle $\phi':\underline{\ZZ}(Sm^{proj})\to R$ over $\AAA$ as in Example \ref{example2}, the left hand side of the defining equation is supported on $\pi(\Crit(W\circ \pi))=\Msp^W$, i.e.\ contained in $W^\ast(R)(\Msp^W\hookrightarrow\Msp)=R(\Msp^W\xrightarrow{W}\AAA)$. As $W:\Msp^W_d \to \AAA$ has only finitely many fibers, we can use the trivialization of $R^{fib}$ to embed $\DTS(\Ab,W^\ast(\phi'))_d$ into $R(\Msp^W_d)$. If $\phi'$ is clear from the context, we also write $\DTS(\Ab,W)$ for short. Moreover, assuming that $\Msp^W_d$ is projective or $R^{red}$ a ring $(\Sch_\kk,ft)$-theory, we can define Donaldson--Thomas invariants $\DT(Q,W^\ast(\phi'))_d\in R(\Spec\kk)$ by means of 
\[ \DT(Q,W^\ast(\phi'))_d:=\int_{\Msp^W_d} \DTS(\Ab,W^\ast(\phi'))_d.\] 
\end{enumerate}
\end{remark}
Let us prove one of the main result of the paper.
\begin{theorem}[Existence and Uniqueness] \label{main_theorem}
Assume that $\phi$ has a factorization $\phi:\underline{\ZZ}(Sm^{proj})\xrightarrow{\unit} \underline{\Ka}_0(\Sch^{ft})[\LL^{-1/2}] \xrightarrow{\phi} R$ for some morphism\footnote{Hopefully, our abuse of notation will not confuse the reader.} $\phi$ of $\lambda$-ring theories. Then, 
\[ \DTS(\Ab,\phi) = \phi_{\id} \big(\DTS(\Ab,mot)\big). \]
is a Donaldson--Thomas function and $\DTS(\Ab,\phi)_d$ is uniquely determined up to an element annihilated by $[\PP^{gcd(d)-1}]$. In particular, its image under the map $R(\id)\longrightarrow R(\id)[[\PP^n]^{-1}: n\in \NN]$ is unique. 
\end{theorem}
\begin{proof} The first statement about existence is given by the previous remark and the previous theorem. It remains to prove uniqueness. \\Let us introduce the shorthand $[n]_a=1+a+a^2 +\ldots+ a^{n-1}$ for any natural number $n\in \NN$ and any $a\in R(\Spec\kk)$. If $n=pm+r$ with $0\le r<m, 1\le p$, we get 
\begin{equation} \label{equation_2} [n]_\LL=\LL^r[p]_{\LL^m}[m]_\LL + [r]_\LL. \end{equation}
Note that $[\PP^{n-1}]=[n]_\LL$ in $R(\Spec\kk)$ as it already holds in $\Ka_0(\Sch^{ft})$. Given $r:=\gcd(d)$, we see $r|fd$ for all framing vectors $f\not=0$, and if $[\PP^{r-1}]\Delta_d=[r]_\LL\Delta_d=0$, the element $\DTS(\Ab,\phi)_d+\Delta_d$ would also solve the defining equation as $[\PP^{fd-1}]\Delta_d=[fd]_\LL\Delta_d=[fd/r]_{\LL^r}[r]_\LL\Delta_d=0$. To prove the converse  we choose finitely many integers $0\not=f\in \ZZ^{\oplus I}$ such that $fd=r$. Write $f=f_+-f_-$ with $0\not=f_\pm \in \NN^{\oplus I}$. If we have two solutions of the defining equation (\ref{defining_equation}), the difference $\Delta_d$ must satisfy $[\PP^{df_\pm-1}]\Delta_d=[df_\pm ]_\LL\Delta_d=0$ as $\LL^{1/2}$ is a unit. Since $df_+=df_-+r$, the equation $[df_+]_\LL= \LL^r[df_-]+[r]_\LL$ will prove  $\Delta_d[r]_\LL=\Delta_d[\PP^{r-1}]=0$. Hence, the converse is also proven.
\end{proof}
\begin{example} \rm 
The assumption of Theorem \ref{main_theorem} is fulfilled for the canonical  $\underline{\ZZ}(Sm^{proj})$-algebra structure on the reduced $\lambda$-ring theories  $\Con$, $\underline{\Ka}_0(\Perv)$ and  $\underline{\Ka}_0(\MHM)\langle\LL^{-1/2}\rangle^-$. The element $[\PP^n]$ is a nonzero divisor in all cases and the corresponding Donaldson--Thomas functions are unique. 
\end{example}
Using the second main result of \cite{Meinhardt4}, one can conclude the following.
\begin{theorem} \label{main_theorem_2}
 Assume that $\oplus:\Msp\times\Msp\longrightarrow \Msp$ is a finite morphism. If  $\phi$ has a factorization $\phi:\underline{\ZZ}(Sm^{proj})\xrightarrow{\unit} \underline{\Ka}_0(\MHM)[\LL^{-1/2}] \xrightarrow{\eta} R$ for some morphism $\phi$ of $\lambda$-ring theories. Then, 
\[ \DTS(\Ab,\phi)_d = \begin{cases} \phi_{\iota_d} \big(\cl(\ICS^{mhm}_{\Msp_d})\big) & \mbox{if } \Msp^s_d\not=\emptyset, \\ 0 &\mbox{else.} \end{cases} \]
Here, $\cl(\ICS^{mhm}_{\Msp_d})$ denotes the class of the intersection complex $\ICS^{mhm}_{\Msp_d}\in \MHM(\Msp_d)[\QQ(1/2)]$ of the singular space $\Msp_d$ normalized by multiplication with $\QQ(\dim \Msp^s_\gamma/2)=\QQ((1-(\gamma,\gamma))/2)$ on $\Msp_{\gamma,d}$.
\end{theorem}
The assumption on $\oplus$ is fulfilled in all examples. (cf.\ Example \ref{finite_direct_sum}) 
\begin{example}\rm 
The pull-backs of the vanishing cycles $\phi^{mhm}, \phi^{perv}$ and $\phi^{con}$ of Example \ref{example2} along a potential $W:\Msp\to \AAA$ satisfy the assumptions of the previous theorem, and we conclude for example $\DTS(\Ab,W^\ast(\phi^{mhm}))_d=\cl\big(\phi^{mhm}_{W_d}(\ICS^{mhm}_{\Msp_d})\big)\in \Ka_0(\MHM_{mon}(\Msp^{W}_d))[\LL^{-1/2}]$ if $\Msp^s_d\not=\emptyset$ and zero else. The constructible function $\DTS(\Ab,W^\ast(\phi^{con}))$ on $\Msp^{W}$ is a nontrivial extension of the Behrend function of $\Msp^{W,s}=\Msp^s\cap \Msp^{W}$ to $\overline{\Msp^s}\cap \Msp^{W}$.
\end{example}

\begin{example} \label{example4} \rm
 The authors tried to prove that the vanishing cycle $\underline{\ZZ}(Sm^{proj})\xrightarrow{\unit} \underline{\Ka}_0(\Sch^{ft})\xrightarrow{\phi^{mot}} \underline{\Ka}^\muu(\Sch^{ft})$ satisfies the assumptions of the main Theorem \ref{main_theorem}. Unfortunately, we could not show that $\phi^{mot}$ commutes with the $\lambda$-operations $\sigma^n$. Hence, the existence of $\DTS(\Ab,W^\ast(\phi^{mot}))$ remains an open problem for $W\not=0$. Nevertheless, the Donaldson--Thomas invariants can be computed in specific examples. (see \cite{BBS}, \cite{DaMe1},\cite{DaMe2},\cite{MMNS}) 
\end{example}
\begin{example}\rm 
Let us  consider the quiver $Q$ with one loop and arbitrary potential $\WW\in \kk[T]$.  Then $\Msp^s_1=\Msp_1=\AAA$ and $\emptyset=\Msp^s_d\subseteq \Msp_d=\Sym^d(\AAA)=\AA^d$ for all $d>1$. Hence $\DTS(\Ab,\Ka_0(\MHM)[\LL^{1/2}])_1=\cl(\ICS^{mhm}_\AAA)$ and $\DTS(\Ab,\Ka_0(\MHM)[\LL^{1/2}])_d=0$ for $d>0$. Thus, after identifying  $W_1$ with $\WW$,
 \[ \DTS(\Ab,W^\ast(\phi^{mhm}))_d=\begin{cases}
                        \cl\big(\phi^{mhm}_{\WW}(\ICS^{mhm}_\AAA)\big) & \mbox{for }d=1, \\ 0 & \mbox{for }d\not=1 
                       \end{cases} \]
holds in $\Ka_0(\MHM_{mon}(\AA^d))$. The same result  for the  ``motivic'' Donaldson--Thomas function $\DTS(\Ab,W^\ast(\phi^{mot}))_d\in\Ka^\muu(\Sch^{ft}_{\AA^d})\langle \LL^{-1/2}\rangle$ has been proven by the authors in \cite{DaMe1}.  
\end{example}

\begin{example} \rm \label{one_vertex_example}
Let us consider the case $\Ab=\Vect_\kk$ with $\Msp_d=\Spec\kk$ for all $d\in \NN$. Fix a $\underline{\ZZ}(Sm^{proj})$-algebra structure $\phi:\underline{\ZZ}(Sm^{proj})\to R$ factoring through $\underline{\Ka}_0(\Sch^{ft})$, and assume that  $t^d:=\phi_{\iota_d}$ is a trivialization of $R$ for $\iota_d:\Msp_d\hookrightarrow\Msp$. If he canonical $\underline{\ZZ}(Sm^{proj})$-algebra structure on a reduced ring theory  $R$ factorizes through $\underline{\Ka}_0(\Sch^{ft})$, its pull-back to $\Msp$ provides an example of such a situation. As $R(\id)=R^{fib}(\id)$, the $\underline{\ZZ}(Sm^{proj})$-algebra structure induces an isomorphism $R(\id)\cong R(\Spec\kk)[[t]]$. 
For $f\in \NN$, we can interpret $\Msp_{f,d}$ as the space of surjective $\kk$-linear maps $\kk^f\to \kk^d$ up to the action of $\Gl(d)$. Thus, $\Msp_{f,d}=\Gr(f-d,f)\cong\Gr(d,f)$ and we conclude for the defining equation (\ref{defining_equation})
\[ \sum_{d\ge 0} \LL^{(fd-\dim \Gr(d,f))/2}[\Gr(d,f)]t^d=\sum_{d\ge 0} \LL^{d^2/2}{f\choose d}_{\LL} t^d \stackrel{!}{=} \Sym\Big(\sum_{d>0} \LL^{1/2} [fd] \DTS(\Ab,\phi)_dt^d\Big) \]
using the shorthand  ${f \choose d}_\LL=\frac{[f]_\LL}{[d]_\LL[f-d]_\LL}=[\Gr(d,f)]$. This equation is solved by $\DTS(\Ab,\phi)_1=1$ and $\DTS(\Ab,\phi)_d=0$ else, giving rise to the beautiful ``binomial formula''
\[ \sum_{d\ge 0} \LL^{d^2/2} {f\choose d}_{\LL} t^d = \Sym\big(\LL^{1/2}[f]_\LL t\big)=:(1-t)^{-\LL^{1/2}[f]} \]
For $R=\Con$ this is just the well-know binomial identity
\[ \sum_{d\ge 0} {f\choose d}(-t)^d= \Sym(-ft)=(1-t)^f. \]
This example also shows that the choice $-\LL^{1/2}\in \Pic(R)$ was crucial. If we had taken $\LL^{1/2}\in \Pic(R)$, i.e.\ $\LL^{1/2}=1$ for $R=\Con$, the defining equation cannot be solved for odd $f$ as one can already see by computing the quadratic term for $R=\Con$ and $f=1$. A similar calculation shows that the alternative defining equation (\ref{defining_equation_2})
\[ \sum_{d\ge 0} \LL^{d(d-f)/2}[\Gr(d,f)]t^d=\sum_{d\ge 0} \LL^{d^2/2} {f\choose d}_\LL (\LL^{-f/2}t)^d=\Sym\big( \LL^{(1-f)/2}[f]t\big) \]
is only true for even $f$ and $-\LL^{1/2}\in \Pic(R)$, and has no solution otherwise.
\end{example}

\begin{definition}
A duality transformation on a ($\lambda$-)ring $(Sm,proj)$-theory  $R$ over some commutative monoid $\MM$ is an involutory natural transformation having the same properties as a morphism of ($\lambda$-)ring $(Sm,proj)$-theories  except for $D_f\circ u^\ast=\LL^{-d}u^\ast \circ D_g$  for every smooth morphism $u:f\to g$ of relative dimension $d$, where $\LL:=[\PP^1]-1$. 
\end{definition}
Mapping $[U\to X]$ to $\LL^{-\dim U}[U\to X]$ for every smooth, projective equidimensional scheme $U$, induces the unique duality transformation on  $\underline{\Ka}_0(Sm^{proj})$ (see \cite{Bittner04}) if $\Char \kk=0$. It can be pulled back along $\MM\to \Spec\kk$. 
\begin{corollary} Assume $D:R\to R$ is a duality transformations on the $\lambda$-ring theory  $R$. If $\Char\kk=0$ and $\phi$ has a selfdual factorization $\phi:\underline{\Ka}_0(Sm^{proj})\to R$, i.e.\ $D\circ\phi =\phi\circ D$, then $\DTS(\Ab,\phi)_d\in R(\iota_d)$ is also selfdual up to elements annihilated by $[\PP^{gcd(d)-1}]$.
\end{corollary}
\begin{proof} Using $u:\PP^1\to \Spec\kk$ one shows $D_{\Spec\kk}(\LL)=\LL^{-1}$, and since $D_{\Spec\kk}$ is a $\lambda$-ring-automorphism, $D_{\Spec\kk}(\LL^{1/2})=\LL^{-1/2}$ follows. Hence $\ICS_X\in \underline{\Ka}_0(Sm^{proj}_X)[\LL^{-1/2}]$ is selfdual for every smooth $X\in \Sch_{\Msp}$. Using this and equation (\ref{defining_equation}), we compute
 \begin{eqnarray*}
\lefteqn{ \Sym\Big(\sum_{0\not= d\in \NN^{\oplus I}} \LL^{-1/2}[\PP^{fd-1}]\LL^{1-fd} D_{\iota_d}\big(\DTS(\Ab,W)_d\big)\Big) } \\
 && = \Sym\Big(D_{\id}\sum_{0\not= d\in \NN^{\oplus I}} \LL^{1/2}[\PP^{fd-1}]\DTS(\Ab,W)_d\Big), \\
 && = D_{\id}\Big(\Sym\Big(\sum_{0\not= d\in \NN^{\oplus I}} \LL^{1/2}[\PP^{fd-1}]\DTS(\Ab,W)_d\Big) \Big), \\
 && = D_{\id}\Big(\pi_{f!}\sum_{d\in \NN^{\oplus I}} \LL^{fd/2}\phi_{\pi_{f,d}}(\ICS_{\Msp_{f,d}})\Big), \\
 && =  \pi_{f!}D_{\pi_f}\sum_{d\in \NN^{\oplus I}} \LL^{fd/2}\phi_{\pi_{f,d}}(\ICS_{\Msp_{f,d}}), \\
 && =  \pi_{f!}\sum_{d\in \NN^{\oplus I}} \LL^{-fd/2}\phi_{\pi_{f,d}}\big(D_{\pi_d}(\ICS_{\Msp_{f,d}})\big), \\
 && =  \pi_{f!}\sum_{d\in \NN^{\oplus I}} \LL^{-fd/2}\phi_{\pi_{f,d}}(\ICS_{\Msp_{f,d}}). \\  
\end{eqnarray*}
By applying the $\lambda$-ring automorphism $R(\id)\ni (a_d)_{d\in \NN^{\oplus I}} \longmapsto (\LL^{fd}a_d)_{d\in \NN^{\oplus I}}\in R(\id)$ to both sides of the equation and using equation (\ref{defining_equation}), we get the result. 
\end{proof}
If $W:\Msp \to \AAA$ is a  potential such that $\Msp^{W}_d$ is proper for all $d\in \NN^{\oplus I}$, and if $\phi=W^\ast(\phi')$ for some selfdual vanishing cycle $\phi':\underline{\Ka}_0(Sm^{proj})\to R'$ over $\AAA$, the same argument can be applied to the Donaldson--Thomas invariant $\DT(\Ab,W^\ast(\phi'))_d$ proving its selfduality. Properness of $\Msp^{W}_d$ is important to obtain selfduality. A counterexample is given by the Jordan quiver consisting of only one loop, $W=0$ and the reduced  $\lambda$-ring theory  $\underline{\Ka}(\Sch^{ft})$ which has a duality transformation if $\Char\kk=0$ as $\underline{\Ka}_0(Sm^{proj})\cong\underline{\Ka}_0(\Sch^{ft})$ in this case. One has $\Msp_1=\AAA$ and $\DT(Q,mot))_1=\LL^{1/2}$ which is not selfdual.   \\
Apart from the fact that defining the Donaldson--Thomas invariant requires additional properties, this is another reason why the Donaldson--Thomas function is more fundamental than the Donaldson--Thomas invariant.

\subsection{Donaldson--Thomas sheaves}

In this subsection we provide a categorification of the previous one. Let $\phi$ be a $\underline{D}^b\QQ(Sm^{proj})$-algebra structure on a categorical ring $(Sm,proj)$-theory  $\altT$ over $\Msp$, and assume the existence of an object $\LL^{-1/2}\in \altT(\kk)$ satisfying $\Sym^n(\LL^{-1/2}[-1])=(\LL^{-1/2}[-1])^{\otimes n}$, and an isomorphism $(\LL^{-1/2})^{\otimes 2}\otimes \LL\cong \unit$, where $\LL:=\cone(1\to c_{\PP^1\,!} c_{\PP^1}^\ast(1))$ denotes the Lefschetz object. Assume moreover that the $\lambda$-ring structure descends from $\underline{\Ka}^0(\altT)$ to $\underline{\Ka}_0(\altT)$. Then,  $-\cl(\LL^{-1/2})=\cl(\LL^{-1/2}[-1])$ is in $\Pic(\Ka_0(\altT(\kk),\sigma_t))$, and $\underline{\Ka}_0(\phi)$ is a $\underline{\ZZ}(Sm^{proj})$-algebra structure on the $\lambda$-ring theory  $\underline{\Ka}_0(\altT)$ satisfying the assumptions of section \ref{Donaldson_Thomas_framed}. 
\begin{example} \rm Generalizing Example \ref{example5}, we can adjoin such an object $\LL^{-1/2}$ to every framed categorical ring theory  if $\altT(\kk)=D^b(\altA)$ is the bounded derived category of a Tannakian category $\altA$ containing $\QQ(-1):=\LL[2]$ and a tensor inverse $\QQ(1)$. Indeed $\LL^{-1/2}:=\QQ(1/2)[1]\in D^b(\altA[\QQ(1/2)])$ is a good choice. Thus, we obtain for example the reduced categorical ring theories  $D^b(\Perv)$, $D^b(\MHM_{mon})$  and $D^b(\MHM)[\LL^{-1/2}]$ for $\kk=\mathbb{C}$. 
\end{example}
Let us use the same shorthands as before, namely $\id:=\id_{\Msp}$, $\iota_d:=(\Msp_d\hookrightarrow \Msp)$ for $d\in \NN^{\oplus I}$ and $[X]=c_{X\,!}c^\ast_X(1)=\int_X\unit_X$ for a smooth projective scheme $X$. For a smooth, equidimensional $X$ we also introduce $\ICS_X=\LL^{-\dim X/2}\unit_X$ and $[X]_{vir}=\int_X \ICS_X$ if $X$ is also projective.
\begin{definition} 
A Donaldson--Thomas sheaf is an element \[ \DTS(\Ab,\phi)=(\DTS(\Ab,\phi)_d)_{d\in \NN^{\oplus I}} \in \altT(\id)=\prod_{d\in \NN^{\oplus I}}\altT(\iota_d) \] with $\DTS(\Ab,\phi)_0=0$  such that for every non-zero framing vector $0\not=f\in \NN^I$ the following equation holds\footnote{As in the previous subsection, the use of infinite directs sums is just a nice way to do computations. Each infinite sum should be seen as an element in an infinite product of categories.}  
\begin{equation}\pi_{f\,!}\bigoplus_{\gamma\in \Gamma} \LL^{(\gamma,\gamma)/2}\phi_{\pi_{f,\gamma}}\cong\Sym\Big(\bigoplus_{0\not= d\in \NN^{\oplus I}} \LL^{1/2}[\PP^{fd-1}] \DTS(\Ab,\phi)_d\Big)
\end{equation}
in $\altT(\id)$. Notice that $\PP^{fd-1}$ is the fiber of $\pi_{f,d}$ over $\Msp^s_d$. If $\phi$ is the canonical $\underline{D}^b\QQ(Sm^{proj})$-algebra structure of a reduced categorical ring $(Sm,proj)$-theory  $\altT$, we simply write $\DTS(\Ab,\altT)$. 
\end{definition} 
Apparently,  $\cl(\DTS(\Ab,\phi))=\DTS(\Ab,\underline{\Ka}_0(\phi))$ in $\underline{\Ka}_0(\altT)(\id)=\prod_{d\in\NN^{\oplus I}}\Ka_0(\altT(\iota_d))$.
If the framing vector $f$ is even, i.e.\ $f\in(2\NN)^I$, the map $(a_d)_{d\in\NN^{\oplus I}}\longmapsto (\LL^{-fd/2}a_d)_{d\in\NN^{\oplus I}}$ defines an  isomorphism on the triangulated tensor category $\altT(\id)$, and our defining equation is equivalent to 
\begin{equation}\pi_{f\,!}\bigoplus_{d\in \NN^{\oplus I}} \LL^{-\dim \Msp_{f,d}/2}\phi_{\pi_{f,d}}\cong\Sym\Big(\bigoplus_{0\not= d\in \NN^{\oplus I}}  [\PP^{fd-1}]_{vir} \DTS(\Ab,\phi)_d\Big). \end{equation}
The following theorem will be proven in  \cite{DavisonMeinhardt4}. Recall that $D^b(\MHM)[\LL^{-1/2}]_{alg}=D^b(\MHM)_{alg}[\LL^{-1/2}]$ denotes the smallest full subring theory  of $D^b(\MHM)[\LL^{-1/2}]$ containing all direct images of intersection complexes of smooth varieties under proper maps. Complexes of mixed Hodge modules in $D^b(\MHM_X)_{alg}$ are called of ``geometric origin''.  
\begin{theorem} If $\oplus:\Msp\times \Msp\longrightarrow \Msp$ is a finite morphism, then the sheaf $\DTS(\Ab,D^b(\MHM)[\LL^{-1/2}]_{alg}):=\ICS_{\overline{\Msp^s}}$ given by the (normalized) intersection complex of the singular closure of $\Msp^s$ inside $\Msp$ is a Donaldson--Thomas sheaf.
\end{theorem}
The assumption on $\oplus$ is fulfilled in all examples. (cf.\ Example \ref{finite_direct_sum}) 
The proof of the following result is same as for Theorem \ref{main_theorem}. It is a consequence of the fact that every morphism $\phi\in \CTh^\lambda(\Msp)$ of categorical ring $(Sm,proj)$-theories commutes with proper push-forwards and all Schur functors.  
\begin{theorem} \label{main_theorem_3}
Assume that $\oplus:\Msp\times\Msp\longrightarrow \Msp$ is a finite morphism and that $\phi:\underline{D}^b\QQ(Sm^{proj})\to \altT$ has a factorization \[\phi:\underline{D}^b\QQ(Sm^{proj})\xrightarrow{1}D^b(\MHM)[\LL^{-1/2}]_{alg}\xrightarrow{\phi} \altT\] for a morphism $\phi\in \CTh^\lambda(\Msp)$ of categorical ring $(Sm,proj)$-theories  over $\Msp$. Then, the family of objects
\[ \DTS(\Ab,\phi)_d:=\begin{cases} \phi_{\iota_d}(\ICS^{mhm}_{\Msp_d}) & \mbox{if } \Msp^s_d\not=\emptyset, \\ 0 &\mbox{else} \end{cases} \]
is a Donaldson--Thomas sheaf. 
\end{theorem}
\begin{example}\rm
 If we apply the theorem to $\phi=\rat:D^b(\MHM)[\LL^{-1/2}] \to D^b(\Perv)$ mapping $\LL^{-1/2}$ to $\QQ[1]$, we get $\DTS(\Ab,perv)_d=\ICS^{perv}_{\Msp_d}$ if $\Msp^s_d\not=\emptyset$ and $0$ else. 
\end{example}
\begin{example} \rm
 Let $\phi$ be the pull-back of  \[\phi^{mhm}:D^b(\MHM)[\LL^{-1/2}]_{alg} \to D^b(\MHM_{mon})\] under a potential $W:\Msp\to \AAA$. Then the theorem applies, and we get the corresponding Donaldson--Thomas sheaves 
 \[ \DTS(\Ab,W)^{mhm}_d:=\begin{cases} \phi^{mhm}_{W_d}(\ICS^{mhm}_{\Msp_d}) & \mbox{if } \Msp^s_d\not=\emptyset, \\ 0 &\mbox{else.} \end{cases} \]
 in $D^b(\MHM_{mon}(\Msp^{W}_d))$. A similar formula holds for the pull-back of $\phi^{perv}\circ \rat: D^b(\MHM)[\LL^{-1/2}]_{alg} \to D^b(\Perv_{mon})$.
\end{example}

\section{($\lambda$-)ring theories on Artin stacks} \label{theories_Artin_stacks}

This section provides an extension of the theory developed in section \ref{theories} to (quotient) stacks which will be used in section \ref{Donaldson_Thomas_general_version}. Most of the definitions have a  straight forward generalization. We will see that certain $(\Sm,\Pro)$-theories behave very nicely with respect to restriction from stacks to schemes. \\
Let $\St_\kk$ denote a fiber product closed 2-category of Artin stacks $\Xst$ over $\kk$  having an atlas in $\Sch_\kk$ and such that every fiber product $Y\times_\Xst Z$ with $Y,Z\in \Sch_\kk$ is again in $\Sch_\kk$. The latter conditions ensures that we can avoid algebraic spaces and that $\Xst$ is the ``quotient stack'' of a groupoid $s,t:X_1\leftrightarrows X_0:\unit$ in $\Sch_\kk$ with smooth $s,t$. A special case of such a groupoid is given by group actions $m:X_1:=X\times G \to X=:X_0$ via $s=\pr_X, t=m, \unit=\id_X\times c_1$ giving rise to quotient stacks $X/G$. 
Similarly, given a commutative monoid $(\MM,+,0)$ in $\Sch_\kk$, the 2-category $\St_\MM$ consists of all morphisms $\fst:\Xst \to \MM$ in $\St_\kk$. Such an object can always be written as a ``quotient'' of a groupoid in $\Sch_\MM$, i.e.\ there is a commutative diagram
\begin{equation}\label{quotient_stack} \xymatrix { X_1 \ar[dd]_s \ar[rr]^t \ar[dr]^{f_1} & & X_0 \ar[dd]^\rho \ar[dl]_{f_0} \\ & \MM & \\ X_0 \ar[rr]_\rho \ar[ru]^{f_0} & & \Xst \ar[ul]^{\fst} } \end{equation}
with the outer square being cartesian and co-cartesian. A special case is given for a $G$-action on a scheme $X$ equipped with a $G$-invariant function $f:X\to \MM$.
The categories $\Sch_\kk$ and $\Sch_\MM$ are full subcategories of $\St_\kk$ and $\St_\MM$ respectively.
\begin{example}\rm \label{quotient_stack_3}
Let $\mathcal{G}$ be a class of algebraic groups which is closed under cartesian products (over $\kk$) and contains the trivial group. Let $\St_\kk:=\QSt_\kk^\mathcal{G}$ be the 2-category of  disjoint unions of quotient stacks $\Xst=\sqcup_{j\in J} X_j/G_i$ with $X_j\in \Sch_\kk$ and $G_j\in \mathcal{G}$ for all $j\in J$. This category satisfies our assumptions for every choice of $\mathcal{G}$. We write $\QSt^\mathcal{G}_\MM$ for $\St_\MM$ in this case. 
\end{example}
Given a homomorphism $u:\MM\to N$ of commutative monoids, we get an adjoint pair $u_!:\St_\MM \leftrightarrows \St_N:u^!$ of functors with $u_!$ preserving the $\boxtimes$-product which is defined as for schemes over monoids. As in section \ref{theories} we fix two subcategories $\SSm,\PPro\subseteq \St_k$ having the same properties (1) -- (8) as for schemes. The intersections $\Sm:=\SSm\cap \Sch_\kk$ and $\Pro:=\PPro\cup\Sch_\kk$ will satisfy (1) -- (8) as well. As before, $\SSm_\MM$ and $\PPro_\MM$ denote the corresponding classes of morphisms over $\MM$.
\begin{example} \rm The classes of smooth, proper and finite type morphisms have natural extensions to $\St_\kk$ by requiring that morphisms are representable and become smooth, proper or of finite type after every pull-back to a scheme. More generally, given a pair $(\Sm,\Pro)$ of subcategories in $\Sch_\kk$ satisfying (1) -- (8), we denote with $\SSm:=\Sm^{st}$ and $\PPro:=\Pro^{st}$ the class of representable morphisms whose pull-back to schemes are in $\Sm$ and $\Pro$ respectively. The pair $(\Sm^{st},\Pro^{st})$ has also the properties (1) -- (8). Notice that $ft^{st}$ is in general strictly contained in the class $\mathfrak{ft}$ of morphisms $\alpha:\Xst\to \Yst$ such that $\pi_0(\alpha):\pi_0(\Xst)\to \pi_0(\Yst)$ has finite fibers. 
\end{example}

\begin{definition}
A stacky  ring $(\SSm,\PPro)$-theory  over $\MM$ is a rule $R$ associating to every object $\fst:\Xst\to \MM$ in  $\St_\MM$ an abelian group $R(\fst)$, to every morphism $u:\gst\to \fst$ in $\St_\MM$ a pull-back $u^\ast:R(\fst)\to R(\gst)$ if $u\in \SSm_\MM$ and a push-forward $u_!:R(\gst)\to R(\fst)$ if $u\in \PPro_\MM$, and a $\boxtimes$-product for two objects $\fst,\gst\in \St_\MM$,  satisfying exactly the same properties as a ring $(\Sm,\Pro)$-theory  if we replace $\Sch_\MM$ with $\St_\MM$. It is clear that the restriction $R|_{\Sch_\kk}$ to $\Sch_\MM$ of a stacky ring $(\SSm,\PPro)$-theory $R$ is a ring $(\Sm=\SSm\cap\Sch_\kk,\Pro=\PPro\cap\Sch_\kk)$-theory, and $R$ is called a stacky $\lambda$-ring $(\SSm,\PPro)$-theory, if $R|_{\Sch_\MM}$ is actually a $\lambda$-ring $(\Sm,\Pro)$-theory. The definition of morphisms between stacky  ($\lambda$-)ring $(\SSm,\PPro)$-theories  is straight forward, giving rise to categories $\Th^{(\lambda)}(\MM)_{st}$. A stacky ring $(\St_k,ft)$-theory  is called motivic if $i_!i^\ast(a)+j_!j^\ast(a)=a$ 
holds for all $a\in R(\fst)$ and all closed substacks $i:\Zst\hookrightarrow \Xst$ with open complement $j:\Xst\setminus \Zst\hookrightarrow \Xst$ (cf.\ Lemma \ref{motivic_ring theory }).
\end{definition}
We point out that $\sigma^n$-operations of stacky $\lambda$-ring $(\SSm,\PPro)$-theories  are only required for morphisms $f:X\to \MM$ on schemes. This is related to the fact that the stacky ``generalization'' of $\Sym^n(X)$ would produce the Deligne--Mumford quotient stack $X^n/S_n$ which does not coincide with its coarse moduli space $\Sym^n(X)$ used in the definition for schemes. We avoid this issue by not considering stacky symmetric powers. \\ 
The stacky analogue of the $\cap$-product is well-defined for stacky ring $(\SSm,\PPro)$-theories if $\SSm$ contains closed embeddings, and Proposition \ref{algebraic_structure_ex} remains true. We can pull-back and for $u\in \Pro$ also push-forward stacky ring $(\SSm,\PPro)$-theories  along homomorphisms $u:\MM\to N$ of commutative monoids. In particular, the reduction is well-defined and the same is true for $R^{fib}$ producing  examples of non-reduced stacky ring $(\SSm,\PPro)$-theories. \\

In order to compare $(\SSm,\PPro)$-theories on $\St_\MM$ with their restrictions to $\Sch_\kk$, we need to add an assumption on $(\SSm,\PPro)$.\\
\\
\textbf{Property (0):} Every stack $\Xst$ has a non-empty distinguished class of diagrams (\ref{quotient_stack}) with $s,t$ and $\rho$ in $\SSm$. Such a groupoid is called an $\SSm$-groupoid for $\Xst$. The pull-back of an $\SSm$-groupoid for $\Xst$ along a representable morphism $\alpha:\Yst\to \Xst$  is an $\SSm$-groupoid for $\Yst$. The cartesian product (over $\kk$) of $\SSm$-groupoids for $\Xst$ and $\Yst$ is an $\SSm$-groupoid for $\Xst\times_\kk \Yst$. Moreover, every morphism in $\PPro$ should be representable.   

\begin{example} \rm
If $\St_\kk=\Sch_\kk$ and if $(\Sm,\Pro)$ satisfies the properties (1) -- (8), then property (0) is automatically true if the only $\Sm$-groupoid for $X$ is $(X,X,\id_X,\id_X,\id_X)$. 
\end{example}
\begin{example} \rm
To every $\Xst=\sqcup_{j\in J}X_j/G_j$ in $\QSt_\kk^\mathcal{G}$ we associate the groupoid 
\[ \xymatrix @C=1.5cm { \sqcup_{j\in J} \big(G_j\times_\kk X_j\big) \ar[d]_{\sqcup_{j\in J}\pr_{X_j}} \ar[r]^{\sqcup_{j\in J} m_j } & \sqcup_{j\in J} X_j \ar[d]^{\sqcup_{j\in J} \rho_j} \\ \sqcup_{j\in J} X_j \ar[r]_{\sqcup_{j\in J} \rho_j} & \sqcup_{j\in J} X_j/G_j } \]
with $m_j:G_j\times_\kk X_j \longrightarrow X_j$ denoting the action of $G_j$ on $X_j$. We require that all morphism of this diagram are in $\SSm$ making this groupoid into an $\SSm$-groupoid for $\Xst$. Notice that a stack might have different realizations as a quotient stack. If for example $\rho:P\to X$ is a principal $G$-bundle on $X$, then $X=P/G$, and $\rho$ must be in $\Sm=\SSm\cap \Sch_\kk$. Therefore, property (0) puts already severe restrictions on $\Sm=\SSm\cap \Sch_\kk$ if we are dealing with non-trivial stacks.  
\end{example}
Given a (stacky) ring theory  $R$ over $\MM$ and an $\SSm$-groupoid $f_\bullet=(f_0,f_1,s,t,\unit)$ as in diagram (\ref{quotient_stack}) above with $s,t,\rho\in \SSm$, we define $R(f_0)^{f_\bullet}$ to be the subgroup $\{a\in R(f_0)\mid s^\ast(a)=t^\ast(a)\}$ of $R(f_0)$. If the $\SSm$-groupoid $f_\bullet$ is associated to an action of an algebraic group $G$ on a scheme $X$ equipped with a $G$-invariant function $f$, we simply write $R(f)^G$ instead of $R(f)^{f_\bullet}$ for the subgroup of $G$-invariant elements. With this notation at hand, we can make the following definition which depends on the choice of $\SSm$-groupoids associated to every stack. 
\begin{definition}
 We denote with $\tilde{\Th}^{(\lambda)}(\MM)_{st}\subseteq \Th^{(\lambda)}(\MM)_{st}$ the subcategory of stacky  ($\lambda$-)ring $(\SSm,\PPro)$-theories  for which $\rho^\ast:R(\fst)\to R(f_0)^{f_\bullet}$ is an isomorphism for all $\SSm$-groupoids as in diagram (\ref{quotient_stack}). The subcategory $\tilde{\Th}^{(\lambda)}(\MM)\subset \Th^{(\lambda)}( \MM)$ is defined accordingly using schemes and $\Sm$-groupoids with quotients in $\Sch_\kk$. 
\end{definition}
Note that the definition of $\tilde{\Th}^{(\lambda)}(\MM)_{(st)}$ does not require the representability of morphisms in $\PPro$.
\begin{example} \rm The  $\lambda$-ring $(\Sch_\kk,ft)$-theory  $\Con$ is in $\tilde{\Th}^\lambda(\kk)$. Apparently, the defining property of $\tilde{\Th}^{(\lambda)}(\MM)_{(st)}$ is stable under base change $u:\MM\to N$. Using $\Con$, we get reduced $\lambda$-ring $(\Sch_\kk,ft)$-theories  in $\tilde{\Th}^{\lambda}(\MM)$ for all $\MM$.  
\end{example}
\begin{example} \rm The ring $(\SSm,\PPro)$-theory $\underline{\ZZ}(\SSm^{\PPro})$ has no reason to be in $\tilde{\Th}(\MM)_{st}$. For example, if $(X,G\times_\kk X, \pr_X,m,c_1\times\id_X)$ is an $\SSm$-groupoid for a quotient stack $X/G$ in $\St_\kk$, then $\pr_X^\ast(a)=m^\ast(a)$ applied to a generator $a=[Z\xrightarrow{p} X]$ of $\underline{\ZZ}(\SSm^{\PPro}_X)=\underline{\ZZ}(\Sm^{\Pro}_X)$ with representable $p\in\PPro$ and $Z$ an $\Sm$-scheme just means that $G\times_\kk Z$ is isomorphic to $(G\times_\kk X)\times_{m,X,p} Z$. If that isomorphism satisfies a certain cocycle condition, it induces a $G$-action on $Z$ making $p$ $G$-equivariant, and the generator descends to a generator $[Z/G \to X/G]$ of $\underline{\ZZ}(\SSm^{\PPro}_{X/G})$. However, there is no reason that the cocycle condition holds. A similar argument applies to $\underline{\Ka}_0(\SSm^{\PPro})$ but a modification of $\underline{\Ka}_0(\SSm^{\PPro})$ has better behavior. (See Example \ref{extension_4})   
\end{example}

\begin{theorem} \label{extension_2} Assume that the pair $(\SSm,\PPro)$ has the properties (0) -- (8). If $R,R'$ are ($\lambda$-)ring $(\SSm,\PPro)$-theories with $R$ being in $\tilde{\Th}^{(\lambda)}(\MM)_{st}$, then
\[ \Hom_{\Th^{(\lambda)}(\MM)_{st}}(R',R)\longrightarrow \Hom_{\Th^{(\lambda)}(\MM)}(R'|_{\Sch_\MM}, R|_{\Sch_\MM}) \]
given by restriction to schemes is an isomorphism. Moreover, the forgetful functor $\Th^{(\lambda)}(\MM)_{st} \ni R\longrightarrow R|_{\Sch_\MM} \in \Th^{(\lambda)}(\MM)$ is an equivalence between the categories $\tilde{\Th}^{(\lambda)}(\MM)_{st}$ and $\tilde{\Th}^{(\lambda)}(\MM)$.
\end{theorem}
\begin{proof} For convenience we assume $\MM=\Spec\kk$. The arguments for general $\MM$ are literally the same. Let us start with the proof of the first statement and assume that $\eta:R'\to R$ is a natural transformation with $\eta|_{\Sch_\kk}=0$. Let $\Xst$ be in $\St_\kk$ with $\SSm$-groupoid $X_\bullet=(X_0,X_1,s,t,\unit_X)$ and quotient map $\rho:X_0\to \Xst$. Thus $\rho^\ast(\eta_\Xst(a))=\eta_{X_0}(\rho^\ast(a))=0$ for every $a\in R'(\Xst)$. As $\rho^\ast:R(\Xst)\to R(X_0)$ is a monomorphism, we get $\eta_\Xst=0$ for all $\Xst\in \St_\kk$, and injectivity of the restriction map is proven.\\
To prove surjectivity, we start with a natural transformation $\eta:R'|_{\Sch_\kk}\longrightarrow R|_{\Sch_\kk}$. For arbitrary $\Xst\in \Sch_\kk$ we choose an $\SSm$-groupoid $X_\bullet$ as above and $a\in R'(\Xst)$. As $\rho^\ast(a)\in R'(X_0)^{X_\bullet}$, we conclude $\eta_{X_0}(\rho^\ast(a))\in R(X_0)^{X_\bullet}$ and define $\eta_\Xst(a)$ to be the unique element in $R(\Xst)$ such that $\rho^\ast(\eta_\Xst(a))=\eta_{X_0}(\rho^\ast(a))$. This definition does not depend on the choice of an $\SSm$-groupoid. Indeed, if $Y_\bullet=(Y_0,Y_1,u,v,\unit_Y)$ is another $\SSm$-groupoid for $\Xst$ with quotient map $\tau:Y_0\to \Xst$. We form the fiber product diagram
\begin{equation}\label{independence} \xymatrix @C=1.5cm @R=1.5cm { X_1\times_\Xst Y_1  \ar@/_1pc/[d]_{v''} \ar@/^1pc/[d]^{u''} \ar@/^1pc/[r]^{s''} \ar@/_1pc/[r]_{t''} &  X_0\times_\Xst Y_1  \ar@{.>}[l] \ar@/^1pc/[d]^{u'} \ar@/_1pc/[d]_{v'} \ar[r]^{\rho''} & Y_1 \ar@/^1pc/[d]^{u} \ar@/_1pc/[d]_{v} \\ 
X_1 \times_\Xst Y_0 \ar@{.>}[u] \ar[d]^{\tau''} \ar@/^1pc/[r]^{s'} \ar@/_1pc/[r]_{t'} & X_0\times_\Xst Y_0 \ar@{.>}[u] \ar@{.>}[l] \ar[d]^{\tau'} \ar[r]^{\rho'} & Y_0 \ar@{.>}[u] \ar[d]^\tau \\
X_1 \ar@/^1pc/[r]^s \ar@/_1pc/[r]_t & X_0 \ar@{.>}[l] \ar[r]^\rho & \Xst }
\end{equation}
with  all rows being the pull-back of the lower row and all columns being the pull-back of the right column. Dotted arrows indicate the pull-backs of $\unit_X$ and $\unit_Y$ respectively. In particular, $X_0,X_1,Y_0$ and $Y_1$ are quotients of the $\SSm$-groupoids shown in the diagram.
Thus, 
\begin{eqnarray*}\tau'^\ast\rho^\ast(\eta^{(X_\bullet)}_\Xst(a))&=&\tau'^\ast(\eta_{X_0}(\rho^\ast(a)))\;=\;\eta_{X_0\times_\Xst Y_0}(\tau'^\ast \rho^\ast(a)) \\&=&\eta_{X_0\times_\Xst Y_0}(\rho'^\ast\tau^\ast(a))\;=\;\rho'^\ast(\eta_{Y_0}(\tau^\ast(a))\\&=&\rho'^\ast\tau^\ast(\eta_\Xst^{(Y_\bullet)}(a)) \;=\;\tau'^\ast\rho^\ast(\eta_\Xst^{(Y_\bullet)}(a)),\end{eqnarray*}
and $\eta^{(X_\bullet)}_\Xst(a)=\eta^{(Y_\bullet)}_\Xst(a)$ follows for every $a\in R'(\Xst)$. Having this at hand, we can easily prove that $\eta$, extended to stacks, commutes with pull-backs and push-forwards. For this consider a morphism $\alpha:\Xst\to \Yst$ in $\SSm$ and  $\SSm$-groupoids $X_\bullet=(X_0,X_1,s,t,\unit_X)$ for $\Xst$ and $Y_\bullet=(Y_0,Y_1,u,v,\unit_Y)$ for $\Yst$. We form the following diagram
\begin{equation} \label{big_diagram} \xymatrix @C=1.5cm @R=1.5cm { X_1\times_\Yst Y_1  \ar@/_1pc/[d]_{v'''} \ar@/^1pc/[d]^{u'''} \ar@/^1pc/[r]^{s''} \ar@/_1pc/[r]_{t''} &  X_0\times_\Yst Y_1  \ar@{.>}[l] \ar@/^1pc/[d]^{u''} \ar@/_1pc/[d]_{v''} \ar[r]^{\rho''} & \Xst \times_\Yst Y_1 \ar@/^1pc/[d]^{u'} \ar@/_1pc/[d]_{v'} \ar[r]^{\alpha''} & Y_1 \ar@/^1pc/[d]^{u} \ar@/_/[d]_{v}\\ 
X_1 \times_\Yst Y_0 \ar@{.>}[u] \ar[d]^{\tau'''} \ar@/^1pc/[r]^{s'} \ar@/_1pc/[r]_{t'} & X_0\times_\Yst Y_0 \ar@{.>}[u] \ar@{.>}[l] \ar[d]^{\tau''} \ar[r]^{\rho'} & \Xst \times_\Yst Y_0 \ar@{.>}[u] \ar[d]^{\tau'} \ar[r]^{\alpha'} & Y_0 \ar@{.>}[u] \ar[d]^\tau \\
X_1 \ar@/^1pc/[r]^s \ar@/_1pc/[r]_t & X_0 \ar@{.>}[l] \ar[r]^\rho & \Xst \ar[r]^\alpha & \Yst }\end{equation} 
with  all rows being the pull-back of the lower row and all columns being the pull-back of the right column. Dotted arrows indicate the pull-backs of $\unit_X$ and $\unit_Y$ respectively. Thus,
\begin{eqnarray*}\tau''^\ast\rho^\ast\alpha^\ast(\eta_\Yst(a))&=&\rho'^\ast\alpha'^\ast\tau^\ast(\eta_\Yst(a))\;=\;\rho'^\ast\alpha'^\ast\eta_{Y_0}(\tau^\ast(a)) \\&=& \eta_{X_0\times_\Yst Y_0}(\rho'^\ast\alpha'^\ast\tau^\ast(a))\;=\;\eta_{X_0\times_\Yst Y_0}(\tau''^\ast\rho^\ast\alpha^\ast(a)) \\&=& \tau''^\ast(\eta_{X_0}(\rho^\ast\alpha^\ast(a)))\;=\;\tau''^\ast\rho^\ast(\eta_\Xst(\alpha^\ast(a))).\end{eqnarray*}
As $\tau''^\ast\rho^\ast$ is injective, $\alpha^\ast\eta_\Yst=\eta_\Xst\alpha^\ast$ follows. As $\alpha\in \PPro$ is  representable, $\Xst\times_\Yst Y_\bullet$ is an $\SSm$-groupoid for $\Xst$, and we get
\begin{eqnarray*} \tau^\ast\alpha_!(\eta_\Xst(a))&=&\alpha'_!\tau'^\ast(\eta_\Xst(a))\;=\;\alpha'_!\eta_{\Xst \times_\Yst Y_0}(\tau'^\ast(a))\\ &=&\eta_{Y_0}(\alpha'_!\tau'^\ast(a))\;=\;\eta_{Y_0}(\tau^\ast\alpha_!(a))\\ &=&\tau^\ast(\eta_\Yst(\alpha_!(a))) \end{eqnarray*}
which implies $\alpha_!\eta_\Xst=\eta_\Yst\alpha_!$ as $\tau^\ast$ is injective. This proves the first part of the theorem. For the second part it remains to show that every $R\in \tilde{\Th}^{(\lambda)}(\kk)$ is in the image of the restriction functor, i.e.\ has an extension $R^{st}$ to stacks. \\
If $\Xst$ is any stack in $\St_\kk$ and $X_\bullet=(X_0,X_1,s,t,\unit_X)$ an $\SSm$-groupoid for $\Xst$ with quotient morphism $\rho:X_0\to \Xst$, we define $R^{st}(\Xst):=R(X_0)^{X_\bullet}$. Let $Y_\bullet=(Y_0,Y_1,u,v,\unit_Y)$ be another $\SSm$-groupoid for $\Xst$. Consider the diagram (\ref{independence}) which by assumption on $R$ gives rise to the following diagram with exact rows and columns, where $K$ denotes the kernel of say $u^\ast-v^\ast$. 
\[ \xymatrix @C=1.5cm { & 0 \ar[d] & 0 \ar[d] & 0 \ar[d] \\ 
 & K \ar[r] \ar[d] &  R(X_0) \ar[r]^{s^\ast-t^\ast} \ar[d]^{\tau'^\ast} & R(X_1) \ar[d]^{\tau''^\ast} \\
0 \ar[r] & R(Y_0) \ar[r]^{\rho'^\ast} \ar[d]^{u^\ast-v^\ast} & R(X_0\times_\Xst Y_0) \ar[r]^{s'^\ast-t'^\ast} \ar[d]^{u'^\ast-v'^\ast} & R(X_1\times_\Xst Y_0) \ar[d]^{u''^\ast-v''^\ast} \\
0 \ar[r] & R(Y_1) \ar[r]_{\rho''^\ast} & R(X_0\times_\Xst Y_1) \ar[r]_{s''^\ast-t''^\ast} & R(X_1\times_\Xst Y_1) } 
\]
Hence, $R(X_0)^{X_\bullet}\cong K\cong R(Y_0)^{Y_\bullet}$ showing the independence of $R^{st}(\Xst)$ on the choice of an $\SSm$-groupoid for $\Xst$. To construct push-forwards and pull-backs along morphisms $\alpha:\Xst\to \Yst$, we consider diagram (\ref{big_diagram}) once more.

 As $\alpha\in \PPro$ is representable, $\Xst\times_\Yst Y_\bullet=(\Xst\times_\Yst Y_0, \Xst\times_\Yst Y_1, u',v',\unit')$ is an $\SSm$-groupoid for $\Xst$, and we put $\alpha_!(a):=(\tau^\ast)^{-1}\alpha'_!\tau'^\ast(a)\in R^{st}(\Yst)\cong R(Y_0)^{Y_\bullet}$ for $a\in R^{st}(\Xst)\cong R(\Xst\times_\Yst Y_0)^{\Xst\times_\Yst Y_\bullet}$.  The pull-back for  $\alpha\in \SSm$ is defined as follows. First of all, we form $b'=\rho'^\ast \alpha'^\ast\tau^\ast(b)\in R(X_0\times_\Yst Y_0)^{X_0\times_\Yst Y_\bullet}$ which can be written uniquely as $b'=\tau''^\ast(b'')$ with $b''
\in R(X_0)$ as $X_0\times_\Yst Y_\bullet$ is an $\SSm$-groupoid for $X_0$ and $R\in \tilde{\Th}^{(\lambda)}(\MM)$. Moreover, $s'^\ast(b')=t'^\ast(b')=\tau'''^\ast s^\ast(b'')=\tau'''^\ast t^\ast(b'')\in R(X_1\times_\Yst Y_0)^{X_1\times_\Yst Y_\bullet}$, and since $X_1\times_\Yst Y_\bullet$ is an $\SSm$-groupoid for $X_1$, we conclude $s^\ast(b'')=t^\ast(b'')$, i.e.\ $b''=\rho^\ast(b''')$ for some unique $b'''\in R^{st}(\Xst)$. Thus, we can define $\alpha^\ast(b):=b'''=(\rho^\ast)^{-1}(\tau''^\ast)^{-1}\rho'^\ast \alpha'^\ast\tau^\ast(b)\in R^{st}(\Xst)$. Finally, we define the $\boxtimes$-product via $((\rho\times \tau)^\ast)^{-1}(\rho^\ast(a)\boxtimes \tau^\ast(b)) \in R^{st}(\Xst\boxtimes \Yst)\cong R(X_0\boxtimes Y_0)^{X_\bullet\boxtimes Y_\bullet}$ with $X_\bullet\boxtimes Y_\bullet=(X_0\boxtimes Y_0,X_1\boxtimes Y_1, s\times u, t\times v, \unit_X\times \unit_Y)$. \\
It is a straight forward calculation to show that $R^{st}$  is indeed a ($\lambda$-)ring $(\SSm,\PPro)$-theory.  By construction, $R^{st}|_{\Sch_\MM}=R$, and we leave it to the reader to construct a canonical isomorphism $(R|_{\Sch_\MM})^{st}\cong R$ for every $R\in \tilde{\Th}^{(\lambda)}(\kk)_{st}$. \end{proof}

\begin{corollary} \label{extension_5}
Assume that the pair $(\SSm,\PPro)$ has the properties (0) -- (8). Then the category of $\underline{\ZZ}(\SSm^{\PPro})$-algebra $(\SSm,\PPro)$-theories in $\tilde{\Th}^{(\lambda)}(\MM)_{st}$ is equivalent to the category of $\underline{\ZZ}(\Sm^{\Pro})$-algebra $(\Sm,\Pro)$-theories in $\tilde{\Th}^{(\lambda)}(\MM)$ with $\Sm=\SSm\cap \Sch_\MM$ and $\Pro=\PPro\cap \Sch_\MM$.
\end{corollary}
\begin{proof} By assumption on $\PPro$ we have $\underline{\ZZ}(\SSm^{\PPro})|_{\Sch_\MM}=\underline{\ZZ}(\Sm^{\Pro})$ and the corollary follows from the previous theorem. \end{proof}
The corollary says in particular that a $\underline{\ZZ}(\SSm^\PPro)$-algebra structure on $R^{st}$ for $R\in \tilde{\Th}(\MM)$ is nothing else than a collection of elements $\phi_f\in R(f)$ for every $\Sm$-scheme $X\xrightarrow{f}\MM$ over $\MM$ satisfying the properties of Lemma \ref{framing} and $s^\ast(\phi_{f_0})=t^\ast(\phi_{f_0})$ for every $\SSm$-groupoid $f_\bullet=(X_0\xrightarrow{f_0}\MM,X_1\xrightarrow{f_1}\MM,s,t,\unit_f)$ with quotient in $\Sch_\MM$ and $X_0$ an $\Sm$-scheme.  For $\St_\MM$ being the category of quotient stacks over $\MM$ the last condition says that $\phi_f$ is $G$-invariant for every $G$-action on an $\Sm$-scheme $X\xrightarrow{f}\MM$ over $\MM$.  

\begin{example} \rm
 The stacky extension of $\Con$ as a $(\St_\kk,ft)$-theory has also been studied in \cite{JoyceCF}. A constructible function on an Artin stack $\Xst$ can be seen as a locally finite linear combination of closed reduced substacks.
\end{example}

In the remaining part of this subsection, we discuss a generalization of the previous results to classes $\PPro$ containing non-representable morphisms. As the reader might guess, this requires stronger conditions on $R\in \tilde{\Th}^{(\lambda)}(\MM)_{st}$, but also on $\PPro$.\\
\\
\textbf{Property (0'):} We say that $(\SSm,\PPro)$ satisfies property (0') if every $\Xst\in \St_\kk$ has a distinguished class of $\SSm$-groupoids as in property (0). Moreover, we require $\rho\in \PPro$ for every quotient map $X_0\to \Xst$ of  an $\SSm$-groupoid.\\ 

\begin{definition} \label{relation_3}
We denote with $\hat{\Th}^{(\lambda)}(\MM)_{st}\subseteq \Th^{(\lambda)}(\MM)_{st}$ the subcategory of \\ ($\lambda$-)ring $(\SSm,\PPro)$-theories  $R$ on $\St_\MM$ such that $\rho_!\rho^\ast$ is an automorphism on $R(\fst)$ for every $\fst\in\St_\MM$ and every  $\SSm$-groupoid $f_\bullet=(f_0,f_1,s,t,\unit_f)$ with quotient $\rho:f_0\to \fst$. Similarly, $\hat{\Th}^{(\lambda)}(\MM)$ denotes the corresponding subcategory of $(\Sm,\Pro)$-theories on $\Sch_\MM$.  
\end{definition}
As for $\tilde{\Th}^{(\lambda)}(\MM)_{(st)}$, this subcategory depends on the choice of $\SSm$-groupoids associated to every stack over $\MM$. 
\begin{lemma}
For every $R\in\hat{\Th}^{(\lambda)}(\MM)_{st}$ and every $\SSm$-groupoid $f_\bullet$ of $\fst\in \St_\MM$ with quotient map $\rho:f_0\to \fst$, the pull-back $\rho^\ast:R(\fst)\to R(f_0)^{f_\bullet}$ is an isomorphism. Hence, $\hat{\Th}^{(\lambda)}(\MM)_{(st)}\subset \tilde{\Th}^{(\lambda)}(\MM)_{(st)}$. 
\end{lemma}
\begin{proof}
Assume for convenience $\MM=\Spec\kk$, and let $X_\bullet=(X_0,X_1,s,t,\unit_X)$ be an $\SSm$-groupoid for $\Xst$ with quotient morphism $\rho:X_0\to \Xst$. As $\rho^\ast(R(\Xst))\subset R(X_0)^{X_\bullet}$, we show that $(\rho_!\rho^\ast)^{-1}\rho_!$ is the inverse of $\rho^\ast:R(\Xst)\to R(X_0)^{X_\bullet}$ under the assumption $R\in \hat{\Th}^{(\lambda)}(\MM)_{st}$. First of all $(\rho_!\rho^\ast)^{-1}\rho_!\rho^\ast=\id_{R(\Xst)}$. On the other hand 
\[ \rho^\ast(\rho_!\rho^\ast)=s_!t^\ast\rho^\ast=(s_!s^\ast)\rho^\ast \]
by base change, which implies $\rho^\ast(\rho_!\rho^\ast)^{-1}=(s_!s^\ast)^{-1}\rho^\ast$ as $s_!s^\ast$ is also invertible being the quotient map for $\rho^\ast(X_\bullet)$. Thus, for $b\in R(X_0)^{X_\bullet}$, i.e.\ $t^\ast(b)=s^\ast(b)$
\[ \rho^\ast\big((\rho_!\rho^\ast)^{-1}\rho_!(b)\big)=(s_!s^\ast)^{-1}\rho^\ast\rho_!(b)=(s_!s^\ast)^{-1}s_!t^\ast(b)=(s_!s^\ast)^{-1}(s_!s^\ast)(b)=b.\]
\end{proof}
\begin{corollary} \label{extension5}
Assume $Sm^{st}\subset \SSm$ and $prop^{st}\subset \PPro$. Then there is an equivalence between the category of $\underline{\ZZ}(Sm^{st,prop^{st}})$-algebra structures on objects in $\hat{\Th}(\MM)_{st}$ considered as  $(Sm^{st},prop^{st})$-theories and the category of $\underline{\ZZ}(Sm^{prop})$-algebra structures on objects in $\hat{\Th}(\MM)$ considered as $(Sm,prop)$-theories. 
\end{corollary}
\begin{proof} This is a consequence of the previous Lemma and Corollary \ref{extension_5}. \end{proof}
\begin{definition}
 A linear algebraic group $G$ is called special if every \'{e}tale locally trivial $G$-principal bundle is already Zariski locally trivial. This is the case if and only if $\Gl(n)\to \Gl(n)/G$ is Zariski locally trivial for some closed embedding $G\subseteq \Gl(n)$. We denote with $\QSt^{sp}_\MM$ the category $\QSt^\mathcal{G}_\MM$ for $\mathcal{G}$ being the class of special algebraic groups. (cf.\ Example \ref{quotient_stack_3})
\end{definition}
\begin{example} \rm
 The groups $\Gl(n)$ are special and also all finite products of special groups. In particular, the groups $G_d=\prod_{i\in I}\Gl(d_i)$ from section 2 are special. The symmetric groups $S_n$ are not special unless $n=1$.
\end{example}

\begin{example}\rm \label{quotient_stack_4}
Consider the category $\QSt_\kk^{sp}$ with $(\SSm,\PPro)$ satisfying properties (0') -- (8) and $\Sch_\kk\subset \SSm$, $\PPro\cap\Sch_\kk=ft$. For $R\in \hat{\Th}(\kk)$ and a special group $G$, the element $[G]=\int_G\unit_G\in R(\kk)$ is invertible because $\rho:G\to \Spec\kk$ is the quotient map of an $\SSm$-groupoid for $\Spec\kk$ and $\rho_!\rho^\ast$ is the multiplication with $[G]$. Conversely, if $R$ is a motivic $(\Sch_\kk,ft)$-theory such that $[G]\in R(\kk)$ is a unit for every special group $G$, then $R\in \hat{\Th}(\kk)$. Indeed, for an $\SSm$-groupoid $(P,G\times_\kk P,\pr_P,m,c_1\times \id_P)$ on a connected $X\in \Sch_\kk$ with $\rho:P\to X$ being a principal $G$-bundle, the composition $\rho_!\rho^\ast$ is the $\cap$-product with $[G]\unit_X$, i.e.\ the product with $[G]\in R(\kk)$, as $R$ is motivic and $G$ special. Hence, $\rho_!\rho^\ast$ is invertible. This example shows that $\Con$ is not in $\hat{\Th}^\lambda(\kk)$,  because $\chi_c(\Gl(n))=0$ for all $n>0$, even though it is in $\tilde{\Th}^\lambda(\kk)$. Thus, the inclusion of the previous lemma is strict. On the other hand, if $R$ is a motivic ($\lambda$-)ring $(\Sch_\kk,ft)$-theory over $\kk$, we can adjoin $[G]^{-1}$ by considering $R[[\Gl(n)]^{-1}: n\in \NN]$ which is again a motivic ($\lambda$-)ring $(\Sch_\kk,ft)$-theory containing $[G]^{-1}=[GL(n)]^{-1}[\Gl(n)/G]$ for every special algebraic group $G\subset \Gl(n)$. Note that  Example \ref{lambda_ring_4} and  Example \ref{lambda_adjunction} for $r=1$ ensure that $R[ [\Gl(n)]^{-1}: n\in \NN]$ is indeed a $\lambda$-ring $(\Sch_\kk,ft)$-theory  for  $R\in \Th^\lambda(\MM)$. Applied to $R=\Con$ this procedure gives the zero $(\Sch_\kk,ft)$-theory, but applied to for example $\Ka_0(\Sch^{ft})$, we end up with some non-trivial motivic $\lambda$-ring $(\Sch_\kk,ft)$-theory $\Ka_0(\Sch^{ft})[[\Gl(n)]^{-1}: n\in \NN]$ in $\hat{\Th}^\lambda(\kk)$.  
\end{example}

\begin{theorem} \label{extension_3} Assume that the pair $(\SSm,\PPro)$ has the properties (0') -- (8). If $R,R'$ are ($\lambda$-)ring $(\SSm,\PPro)$-theories with $R$ being in $\hat{\Th}^{(\lambda)}(\MM)_{st}$, then
\[ \Hom_{\Th^{(\lambda)}(\MM)_{st}}(R',R)\longrightarrow \Hom_{\Th^{(\lambda)}(\MM)}(R'|_{\Sch_\MM}, R|_{\Sch_\MM}) \]
given by restriction to schemes is an isomorphism. Moreover, the forgetful functor $\Th^{(\lambda)}(\MM)_{st} \ni R\longrightarrow R|_{\Sch_\MM} \in \Th^{(\lambda)}(\MM)$ is an equivalence between the categories $\hat{\Th}^{(\lambda)}(\MM)_{st}$ and $\hat{\Th}^{(\lambda)}(\MM)$.
\end{theorem}
\begin{proof}
The proof is an extension of the one for Theorem \ref{extension_2} since all the conditions of property (0) follow from property (0') except for the representability of morphisms in $\PPro$. Thus, we have to argue more carefully when it comes to push-forwards. As before we assume $\MM=\Spec\kk$ for convenience.\\
For the first part of the theorem, we only need to show that the extension of a natural transformation $\eta:R'|_{\Sch_\kk}\to R|_{\Sch_\kk}$ constructed in the proof of Theorem \ref{extension_2} commutes with non-representable push-forwards. Consider the diagram (\ref{big_diagram}) in which $\rho'\in \PPro$ is representable so that $\eta$, extended to stacks, commutes already with $\rho'^\ast, \rho'_!$ and, thus, also with $(\rho'_!\rho'^\ast)^{-1}$. As $(\alpha'\rho')_!$ commutes with $\eta$ by assumption on $\eta$, we finally obtain 
\begin{eqnarray*}
\tau^\ast\alpha_!(\eta_\Xst(a)) &=& \alpha'_!\tau'^\ast(\eta_\Xst(a))\;=\;\alpha'_!\eta_{\Xst\times_\Yst Y_0}(\tau'^\ast(a)) \\ &=& \alpha'_!\rho'_!\rho'^\ast(\rho'_!\rho'^\ast)^{-1}\eta_{\Xst\times_\Yst Y_0}(\tau'^\ast(a)) \;=\; \alpha'_!\rho'_!\eta_{X_0\times_\Yst Y_0}\big(\rho'^\ast(\rho'_!\rho'^\ast)^{-1}\tau'^\ast(a)\big) \\ &=& \eta_{Y_0}\big(\alpha'_!\rho'_!\rho'^\ast(\rho'_!\rho'^\ast)^{-1}\tau'^\ast(a)\big) \;=\;\eta_{Y_0}(\alpha'_!\tau'^\ast(a))\\ &=& \eta_{Y_0}(\tau^\ast\alpha_!(a))\;=\;\tau^\ast\eta_\Yst(\alpha_!(a))
\end{eqnarray*}
proving $\alpha_!\eta_\Xst=\eta_\Yst\alpha_!$ even for non-representable morphisms in $\PPro$.

For the second part of the theorem we need to extend a $(\SSm,\PPro)$-theory $R$ on $\Sch_\kk$. We do this as in the proof of Theorem \ref{extension_2} and denote this extension by $R^{st}$ again. It remains to construct push-forwards along morphisms $\alpha\in \PPro$. Assuming that $\alpha_!$ were already defined, we get using the notation of diagram (\ref{big_diagram})
\[ \tau^\ast\alpha_!(a)=\tau^\ast\alpha_!\rho_!\rho^\ast(\rho_!\rho^\ast)^{-1}(a)=(\alpha'\rho')_!\tau''^\ast\rho^\ast(\rho_!\rho^\ast)^{-1}(a). \]
Hence, we can use this formula to define $\alpha_!(a)$ via 
\[ \alpha_!(a):=(\tau^\ast)^{-1}(\alpha'\rho')_!\tau''^\ast\rho^\ast(\rho_!\rho^\ast)^{-1}(a)=(\tau_!\tau^\ast)^{-1}\tau_!(\alpha'\rho')_!\tau''^\ast\rho^\ast(\rho_!\rho^\ast)^{-1}(a).\] 
For representable $\alpha$ this agrees with our previous definition as 
 \[ (\tau^\ast)^{-1}(\alpha'\rho')_!\tau''^\ast\rho^\ast(\rho_!\rho^\ast)^{-1}(a)=(\tau^\ast)^{-1}\alpha'_!\rho'_!\rho'^\ast(\rho'_!\rho'^\ast)^{-1}\tau'^\ast(a)=(\tau^\ast)^{-1}\alpha'_!\tau'^\ast(a), \]
where we used $(\rho'_!\rho'^\ast)\tau'^\ast=\tau'^\ast(\rho_!\rho^\ast)$ which is a consequence of base change. We leave it to the reader to check that the properties of a $(\SSm,\PPro)$-theory  are fulfilled and that $R\mapsto R^{st}$ is indeed the inverse of the restriction (up to isomorphism). 
\end{proof}

\begin{example} \rm \label{extension_4}
If we apply the theorem to  any  motivic ($\lambda$-)ring $(\Sch_\kk,ft)$-theory $R$, we obtain a ($\lambda$-)ring $(\QSt^{sp}_\kk,\mathfrak{ft})$-theory  $R[[\Gl(n)]^{-1}\mid n\in \NN]^{st}$ on $\QSt^{sp}_\kk$. For $R=\underline{\Ka}_0(\Sch^{ft}_\kk)$ the result has an alternative description. Denote by $\underline{\Ka}_0(\QSt^{sp,\mathfrak{ft}})$ the sheafification of the quotient of $\underline{\ZZ}(\QSt^{sp,\mathfrak{ft}})$ by the usual cut and paste relation. Moreover, denote with $\underline{\Ka}(\QSt^{sp,\mathfrak{ft}})$ the sheafification of the quotient of $\underline{\Ka}_0(\QSt^{sp,\mathfrak{ft}})$ by the relation $\alpha_!(\unit_\Yst)=[G]\unit_\Xst$ for every principal $G$-bundle $\alpha:\Yst\to \Xst$ on $\Xst$, i.e.\ the pull-back of $\alpha$ to any scheme mapping to $\Xst$ is a principal $G$-bundle in the usual sense, and for every special group $G$. Applying this to $\Spec\kk \to \Spec\kk/G$, we obtain $1=[G]\int_{\Spec\kk/G}\unit_{\Spec\kk/G}$, and $[G]$ is invertible. Using the projection formula, $\rho_!\rho^\ast$ is invertible for every quotient map $\rho:X\to X/G$ proving that $\underline{\Ka}(\QSt^{sp,\mathfrak{ft}})$ is in $\hat{\Th}(\kk)_{st}$. One can show that its restriction to $\Sch_\kk$ is $\underline{\Ka}_0(\Sch^{ft})[[\Gl(n)]^{-1}\mid n\in \NN]$. (See \cite{Bridgeland10} for more details.) Thus, 
\[ \underline{\Ka}_0(\Sch^{ft})[[\Gl(n)]^{-1}\mid n\in \NN]^{st}=\underline{\Ka}(\QSt^{sp,\mathfrak{ft}}) \]
by  Theorem \ref{extension_3}.
\end{example}

\section{Donaldson--Thomas theory, standard version} \label{Donaldson_Thomas_general_version}

 In this section we provide the standard definition of the Donaldson--Thomas function generalizing ideas of M.\ Kontsevich and Y.\ Soibelman. We will also relate our approach to the work of D.\ Joyce. Moreover, we introduce the Hall algebra and various algebra homomorphisms, and prove the wall-crossing identity and the PT--DT correspondence which will relate this section to section \ref{Donaldson_Thomas_framed}. \\
 
We will always work in the setting of Example \ref{extension_4}. In particular, all Artin stacks will be disjoint unions of quotients by special groups, and we fix the pair $(\SSm,\PPro)=(\QSt^{sp},\mathfrak{ft})$.  Let $R\in \hat{\Th}^\lambda(\Msp)_{st}$ be a  stacky  $\lambda$-ring $(\QSt^{sp},\mathfrak{ft})$-theory  over the moduli space $\Msp$ of objects in $\Ab$ which is  determined by its restriction to $\Sch_\Msp$ due to Theorem \ref{extension_3}. Moreover, we fix a $\underline{\ZZ}(Sm,proj)$-algebra structure on $R|_{\Sch_\Msp}$ which by Corollary \ref{extension5} extends to a $\underline{\ZZ}(Sm^{st, proj^{st}})$-algebra structure on $R$. Moreover, assume the existence of some element $-\LL^{-1/2}\in \Pic(R)$ such that $(\LL^{-1/2})^2=\LL^{-1}$. 
\begin{example} \rm
 Given an algebra structure $\phi:\underline{\ZZ}(Sm^{proj})\to R'$ on a motivic $\lambda$-ring $(\Sch_\kk,ft)$-theory  $R'$ with $-\LL^{-1/2}\in \Pic(R')$ over $\Msp$, we can adjoin $[\Gl(n)]^{-1}$ for all $n\in \NN$ as in Example \ref{quotient_stack_4}, and obtain an algebra theory  $R=R'[[\Gl(n)]^{-1}:n\in \NN]^{st}$ satisfying our assumptions due to Example \ref{extension_4}. An important case is given by the canonical $\underline{\ZZ}(Sm^{proj})$-algebra structure on $\underline{\Ka}_0(\Sch^{ft})\langle \LL^{-1/2}\rangle^-$ giving rise to the reduced extended $\lambda$-ring theory  $\underline{\Ka}_0(\Sch^{ft})\langle \LL^{-1/2},(\LL-1)^{-1}\rangle^-$. If we apply this strategy to $R'=\Con$ or $R'=\underline{\Ka}_0(\Perv)$, we end up with the trivial ring theory  $R\equiv 0$ producing no interesting results. 
\end{example}

\subsection{Donaldson--Thomas functions and invariants}

We start by extending the definition of the ``intersection complex'' to smooth disjoint unions of quotient stacks. 
\begin{definition}
For a smooth stack $\Xst=\sqcup_{\Xst_i\in \pi_0(\Xst)} \Xst_i$ (over $\Msp$) with $\Xst_i= X_i/G_i$ we define the intersection complex $\ICS_\Xst\in \underline{\ZZ}(Sm^{st,proj^{st}}_\Mst)[\LL^{-1/2}]$ by the condition $\ICS_\Xst|_{\Xst_i}=\LL^{-\dim \Xst_i/2}\unit_{\Xst_i}$, where $\dim \Xst_i=\dim X_i-\dim G_i$ denotes the dimension of $\Xst_i$.  
\end{definition}
Let us  fix a functor $\Ab$ with values in exact categories as in section 2 satisfying properties (1)--(6). Using the convention and notation introduced in section 2, we want to apply the definition of the intersection complex  to the moduli stack $\Mst=\sqcup_{\gamma\in\Gamma} \Mst_\gamma$ of objects in $\Ab$ with $\dim \Mst_\gamma=-(\gamma,\gamma)$.
\begin{definition}
Using $\ICS_{\Mst}\in \underline{\ZZ}(Sm^{proj}_\Mst)[\LL^{-1/2}]$ and that $\Msp$ is a commutative monoid in $\Sch_\kk$ with respect to $\oplus$, we  define the Donaldson--Thomas function $\DTS(\Ab,\phi)=\big(\DTS(\Ab,\phi)_d\big)_{ d\in \NN^{\oplus I}}\in R(\id)$  by means of $\DTS(\Ab,\phi)_0=0$ and the equation
\begin{equation}
\label{stand_DTsheaf} p_{ !}\Big(\phi_{p}\big(\ICS_{\Mst}\big)\Big)= \Sym \Big( \frac{\DTS(\Ab,\phi)}{\LL^{1/2}-\LL^{-1/2}} \Big) 
\end{equation}
using Lemma \ref{complete_lambda_rings} and the shorthand $\DTS(\Ab,\phi)_d=\DTS(\Ab,\phi)|_{\Msp_d}$. If $R\in \hat{\Th}^\lambda(\Msp)$ is a reduced $(\Sch_\kk,ft)$-theory $R$, we write $\DTS(\Ab,R)$ for the canonical $\underline{\ZZ}(Sm^{proj})$-algebra structure on $R|_{Sm,proj}$.
\end{definition}
Note that the existence of the Donaldson--Thomas function follows already from the invertibility of $\LL^{1/2}-\LL^{-1/2}=\LL^{-1/2}[\Gl(1)]$. Moreover, $\DTS(\Ab,\phi)$ is uniquely defined by the same reason. This is an advantage of this approach compared to the one given in section \ref{Donaldson_Thomas_framed}. However, invertibility of $[\Gl(n)]$ puts severe restrictions on $R$. 
\begin{remark} \label{remark_2}
The following is true.
\begin{enumerate}
 \item[(i)] If $\eta:R\to R'$ is a morphism of stacky $\lambda$-ring theories, we have an algebra structure $\eta\phi$ on $R'$ and $\eta_{\id}(\DTS(\Ab,\phi))=\DTS(\Ab,\eta\phi)$. This applies in particular, if $\phi$ has a factorization $\underline{\ZZ}(Sm^{proj})\xrightarrow{\unit}\underline{\Ka}_0(\Sch^{ft})\langle \LL^{-1/2},(\LL-1)^{-1}\rangle^- \xrightarrow{\phi} R$ abusing notation.
 \item[(ii)] If $\phi=W^\ast(\phi')$ for some potential $W:\Msp\to \AAA$ and some vanishing cycle $\phi':\underline{\ZZ}(Sm^{proj})\to R$ over $\AAA$, the left hand side of the defining equation is supported on $\pi(\Crit(W\circ p))=\pi(\Mst^{W})=\Msp^{W}$, i.e.\ contained in $W^\ast(R)(\Msp^{W}\hookrightarrow\Msp)=R(\Msp^{W}\xrightarrow{W}\AAA)$. As $W:\Msp^{W}_d \to \AAA$ has only finitely many fibers, we can use the trivialization of $R$ to embed $\DTS(\Ab,W^\ast(\phi'))_d$ into $R(\Msp^{W}_d)$. If $\phi'$ is clear from the context, we also write $\DTS(\Ab,W)$ for short. We  define Donaldson--Thomas invariants $\DT(\Ab,W^\ast(\phi'))_d\in R(\Spec\kk)$ by means of 
\[ \DT(\Ab,W^\ast(\phi'))_d:=\int_{\Msp^{W}_d} \DTS(\Ab,W^\ast(\phi'))_d.\] 
\end{enumerate}
\end{remark}
The following conjecture is due to M.\ Kontsevich and Y.\ Soibelman and allows the realization of the Donaldson--Thomas functions in other non-stacky motivic ring theories. 
\begin{conjecture}[Motivic Integrality Conjecture] \label{integrality_conjecture_stack}
If $\Ab$ also satisfies (7) and (8), the motivic Donaldson--Thomas function $\DTS(\Ab,W^\ast(\phi^{mot}))_d$ is in the image of the map
\[ \Ka^\muu(\Sch^{ft}_{\Msp^{W}_d})\langle \LL^{-1/2}\rangle^- \longrightarrow  \Ka^\muu(\Sch^{ft}_{\Msp^{W}_d})\langle \LL^{-1/2},(\LL-1)^{-1}\rangle^-  \] for all dimension vectors $d\in \NN^{\oplus I}.$
\end{conjecture}
If the motivic integrality conjecture holds, and if $\eta:\underline{\Ka}^\muu(\Sch^{ft})\langle \LL^{-1/2}\rangle^- \longrightarrow R$ is a morphism in $\Th^\lambda(\Msp)$ of  motivic $\lambda$-ring theories  on $\Sch_{\Msp}$, we can define the Donaldson--Thomas function $\DTS(\Ab,\phi)_d$ in $R(\Msp_d\xrightarrow{\iota_d}\Msp)$ using the shorthand $\phi:=\eta\circ W^\ast(\phi^{mot})$ by putting \[\DTS(\Ab,\phi)_d:=\eta_{\iota_d}(\DTS^{mot}(\Ab,W)_d)\] for a lift $\DTS^{mot}(\Ab,W)_d\in \Ka^\muu(\Sch^{ft}_{\Msp_d})\langle \LL^{-1/2}\rangle^-$ of the the motivic Donaldson--Thomas function $\DTS(\Ab,W^\ast(\phi^{mot}))_d$. 
However, this definition is not unproblematic as it depends on the potential $W$. Indeed, if $\phi=\eta\circ W^\ast(\phi^{mot}):\underline{\Ka}_0(\Sch^{ft})\to R$ is even a morphism of $\lambda$-ring theories, we could also define $\DTS(\Ab,\phi)_d$ as $\phi_{\iota_d}(\DTS(\Ab,mot)_d)$. Thirdly, if $R\in \hat{Th}^\lambda(\Msp)$ and $-\LL^{-1/2}\in\Pic(R)$, we can define $\DTS(\Ab,\phi)_d$ directly using the framing  $\underline{\ZZ}(Sm^{proj})\xrightarrow{\unit} \underline{\Ka}_0(\Sch^{ft}) \xrightarrow{\eta\circ W^\ast(\phi^{mot})} R$ and extending all data to $\QSt^{sp}_\Msp$. Of course, these definitions should give the same answer. Using Remark \ref{remark_2}(i) and the first main result of \cite{Meinhardt4}, all of these problems disappear and the integrality conjecture would be true if the following conjecture holds.
\begin{conjecture}[$\lambda$-conjecture] \label{lambda_conjecture}
 The morphism $\phi^{mot}:\underline{\Ka}_0(\Sch^{ft}) \to \underline{\Ka}^\muu(\Sch^{ft})$ of motivic ring theories  over $\AAA$ is actually a morphism of $\lambda$-ring theories, i.e.\ $\phi^{mot}$ commutes with the $\sigma^n$-operations of the $\lambda$-ring theories.
\end{conjecture}
The authors tried hard to prove this conjecture but could only check a weaker version in \cite{DaMe1} involving the stacky symmetric powers $X^n/S_n$ which are Deligne--Mumford stacks. 

\subsection{Dimension reduction}
Given a pair $(\Ab,\omega)$ satisfying properties (1)--(6), one can define a new pair $(\Ab^l,\omega^l)$ such that $\Ab^l_S$ is the category of all pairs $(E,\beta)$ with $E\in \Ab_S$ and $\beta\in \End_{\OO_S}(\omega_S(E))$. Furthermore, $\omega^l(E,\theta)=\omega(E)$, and $(\Ab^l,\omega^l)$ also satisfies all properties (1)--(6) as one can see using \cite{Meinhardt4}, Proposition 3.27 since $\Ab^l=\Ab\times_{\Vect^I} \kk Q^l\rep$, where $Q^l$ is the quiver with $Q^l_0=I$ and one loop at every vertex. There is a ``forget'' functor $F:\Ab^l\to \Ab$ over $\Vect^I$, and it is not hard to see that the set of sections $\theta$ of $F$ over $\Vect^I$ is just $\End_\kk(\omega)$ (cf.\ section \ref{potential}). Moreover, $X^l_d=X_d\times_\kk \prod_{i\in I}\Mat(d_i,d_i)=X_d\times_\kk \AA^{d^2}$ with $G_d$ acting on the second factor by conjugation. 
Given a section $\theta\in \End_\kk(\omega)$, we define a potential $W_\theta:\Msp^l\to \AA^1$ for $\Ab^l$ as follows. Let $(\tilde{E},\tilde{\beta}, \tilde{\psi})$ be the universal trivialized object on $X^l_d$, i.e.\ $\tilde{E}\in \Ab_{X_d^l}$, $\tilde{\psi}:\omega_{X_d^l}(\tilde{E})\cong \OO_{X_d^l}^d$ and $\tilde{\beta}\in \End_{\OO_{X_d^l}}(\omega_{X_d^l}(\tilde{E}))$ giving rise to $\beta=\tilde{\psi}\tilde{\beta}\tilde{\psi}^{-1}\in \End_{\OO_{X_d^l}}(\OO_{X_d^l}^d)$. 
In fact, $(\tilde{E},\tilde{\psi})=(q_d^\ast(E),q_d^\ast(\psi))$ is just the pull-back of the universal trivialized object $(E,\psi)$ on $X_d$ along the projection $q_d:X_d^l\to X_d$. We consider the function 
\[f^\theta_d:=\Tr(\beta \tilde{\psi} \theta^{X_d^l}_{\tilde{E}}\tilde{\psi}^{-1})=\Tr\big(\beta q_d^\ast(\psi \theta^{X_d}_E\psi^{-1})\big)\]
on $X_d^l$ which is apparently $G_d$-invariant and descends to a function on $\Msp_d^l$. Using the properties of the trace, the induced function $W_\theta:\Msp^l\to \AA^1$ is indeed a monoid homomorphism. \\

Let us consider a vanishing cycle $\phi$ on a motivic reduced $\lambda$-ring $(\Sch_\kk,ft)$-theory $R$  (over $\AA^1$) satisfying the assumptions of this section to apply Donaldson--Thomas theory to $(W_\theta^\ast R,W_\theta^\ast\phi)$. We also assume the following property. Given a vector bundle $q:V\to X$ of rank $r$ on a smooth scheme $X\in \Sch_\kk$ and a section $s$ of $q$ inducing a regular function $\hat{s}:V^\ast \to \AA^1$ on the dual bundle. Then $\int_{V^\ast}\phi_{\hat{s}}=\LL^r\int_{\{s=0\}} \unit_{\{s=0\}}\in R(\Spec\kk)$, where $\{s=0\}\subset X$ is the vanishing locus of $s$. This property is fulfilled for $\phi^{mot}$ (see \cite{BBS}, \cite{DaMe1}) and, therefore, also for every vanishing cycle  factoring through $\phi^{mot}$, e.g.\ $\phi^{mhm}$. \\

We will apply this to the trivial vector bundle $q_d:X_d^l\to X_d$ which can be identified with its dual bundle using the trace pairing. The section $s$ is given by $\psi\theta^{X_d}_E\psi^{-1}$ and $\hat{s}$ is $f^\theta_d$. Using our assumption on $\phi$ and $\dim \Mst^l_{\gamma,d}=\dim X_{\gamma,d}=d^2-(\gamma,\gamma)$, we arrive at 
\begin{eqnarray*} \Sym\Big(\sum_{d\in \NN^{\oplus I}}t^d\frac{\DT(\Ab^l,W_\theta^\ast(\phi))_d}{\LL^{1/2}-\LL^{-1/2}}\Big)
& = & \sum_{d\in \NN^{\oplus I}} t^d\sum_{\gamma\in \Gamma} \frac{\LL^{((\gamma,\gamma)-d^2)/2}}{[G_d]}\int_{X^l_d} \phi_{f^\theta_d}  \\
& = & \sum_{d\in \NN^{\oplus I}} t^d\sum_{\gamma\in \Gamma} \frac{\LL^{((\gamma,\gamma)-d^2)/2+d^2}}{[G_d]}\int_{\{\theta^{X_d}_E=0\}} \unit_{\{\theta^{X_d}_E=0\}}  \\
& = & \sum_{d\in \NN^{\oplus I}} t^d\frac{\LL^{d^2}}{[G_d]}\int_{\{\theta^{X_d}_E=0\}} \ICS_{X_d}|_{\{\theta^{X_d}_E=0\}} \\
& = & \sum_{d\in \NN^{\oplus I}} t^d\,\LL^{d^2/2}\int_{\Mst_d^{\theta=0}} \ICS_{\Mst_d}|_{\Mst_d^{\theta=0}} 
\end{eqnarray*}
for the generating function of the Donaldson--Thomas invariants, where we used the notation $\Mst_d^{\theta=0}$ for the quotient stack $\{\theta^{X_d}_E=0\}/G_d$. This calculation is often called ``dimension reduction''.

\subsection{The Ringel--Hall algebra}

A very useful tool in Donaldson--Thomas theory  is the Ringel--Hall algebra which we are going to introduce now. Assume that $\Ab$ satisfies assumptions (1)--(6).\\
Recall that $\mathfrak{E} xact=\sqcup_{d,d'} \Mst_{d,d'}$ is the stack of short exact sequences in $\Ab$ with $\mathfrak{E}xact_{d,d'}=X_{d,d'}/G_{d,d'}$. Moreover, there are morphisms  
\[ \xymatrix { \mathfrak{E} xact \ar[d]_{\pi_1\times \pi_3} \ar[r]^{\pi_2} & \Mst \\ \Mst \times \Mst  & } \]
mapping a short exact sequence to its i-th entry. All morphisms are of finite type and $\pi_2$ is even representable and proper by assumption (5). In fact, $\mathfrak{E}xact \xrightarrow{\pi_2}\Mst$ is the universal Grassmannian  parameterizing subobjects.
\begin{definition}
The composition 
\[ \ast:R(\Mst)\otimes R(\Mst) \xrightarrow{\boxtimes}R(\Mst\times \Mst) \xrightarrow{(\pi_1\times \pi_3)^\ast} \mathfrak{E}xact \xrightarrow{\pi_{2\,!}} R(\Mst) \]
is called the Ringel--Hall product on $R(\Mst)$.  
\end{definition}
The proof of the following lemma is standard, see for instance \cite[Thm.\ 5.2]{JoyceII}.
\begin{lemma} \label{Ringel_Hall_product}
The Ringel--Hall product is associative and makes $R(\Mst)$ into an algebra with unit $1_0=0_!(1)$ for $\Spec\kk\xrightarrow{0} \Mst$. 
\end{lemma}
\begin{example} \rm \label{Hilbert_scheme_identity}
Assume that $R$ is motivic, fix a fiber functor $\omega$ for $\Ab$ and a framing vector $f\in \NN^{I}$ and use the notation of Section \ref{Sec_fr_rep}. Let $H\in R(\Mst)$ be the push-forward of $\unit_{\Mst_f}$ along the morphism $\tilde{\pi}:\Mst_{f}\longrightarrow \Mst$. Then $H\ast \unit_{\Mst}$ restricted to $\Mst_d$ is $\frac{\LL^{fd}}{\LL-1}\unit_{\Mst_d}$. A proof  of this formula can be found in \cite{Meinhardt4}, section 6.4. 
\end{example}

\subsection{The integration map}
\begin{definition}
Using the notation introduced before, we define the integration map $I^{\phi}=\prod_{d\in\NN^{\oplus I}}I^{\phi}_d:  R(\Mst) \xrightarrow{-\cap \phi_{p}(\ICS_{\Mst})} R(p) \xrightarrow{\;p_!\;} R(\id)$, i.e.\ $I^{\phi}(a)=p_!\big(a\cap\phi_{p}(\ICS_{\Mst})\big)$.
\end{definition}
From now on we will also assume that $\Ab$ satisfies the smoothness condition (7). Hence, $X=\sqcup_{\gamma\in \Gamma} X_{\gamma}$ with $\Gamma=\Ka_0(\Ab_{\bar{\kk}})/\rad (-,-)$, and 
\[(E,E')=\dim_\KK \Hom_{\Ab_\KK}(E,E')-\dim_\KK\Ext^1_{\Ab_\KK}(E,E')\] 
as well as $(E',E)$ are constant for all $(E,\psi)\in X_{\gamma}(\KK)$, all $(E',\psi')\in X_{\gamma'}(\KK)$ and all field extensions $\KK\supset \kk$. We write $(\gamma,\gamma')$ and $(\gamma',\gamma)$ in this case. Combining this with the decomposition with respect to dimension vectors, we get $X=\sqcup_{\gamma\in \Gamma,d\in \NN^{\oplus I}} X_{\gamma,d}$.
\begin{definition} \label{star_product_2}
 The $\ast$-product on $R(\id)=\prod_{\gamma\in \Gamma}R(\Msp_{\gamma}\xrightarrow{\iota_{\gamma}}\Mst)$  is defined by ``continuous'' linear extension of 
 \[ a\ast b:= \LL^{\langle \gamma,\gamma'\rangle /2} a\cdot b \]
 for $a\in R(\iota_{\gamma}), b\in R(\iota_{\gamma'})$. Here, $a\cdot b$ is the usual convolution product $\oplus_!(a\boxtimes b)$ of $a$ and $b$, and $\langle \gamma,\gamma'\rangle=(\gamma,\gamma')-(\gamma',\gamma)$.  
\end{definition}
Note that the symmetry condition (8) for $\Ab$ implies that this $\ast$-product is just the usual convolution product.\\
We fix a fiber functor $\omega$ for $\Ab$ as in section 2 and consider the following commutative diagram.
\begin{equation} \label{convolution_diagram}
 \xymatrix { & X_{d,d'} \ar[dl]_{\hat{\pi}_1\times\hat{\pi}_3} \ar@{^{(}->}[dr]^{\hat{\pi}_2} \ar[dd]^(0.3){\rho_{d,d'}}& \\
 X_d\times X_{d'} \ar[dd]_{\rho_d\times \rho_{d'}} & & X_{d+d'} \ar[dd]^{\rho_{d+d'}}  \\
 & X_{d,d'}/G_{d,d'} \ar[dl]^{\pi_1\times \pi_3} \ar[dr]_{\pi_2} & \\
 X_d/G_d \times X_{d'}/G_{d'} \ar[dd]_{p_d\times p_{d'}} & & X_{d+d'}/G_{d+d'} \ar[dd]^{p_{d+d'}}\\ & & \\
 \Msp_d\times \Msp_{d'} \ar[rr]^\oplus & & \Msp_{d+d'} }
\end{equation}
\begin{definition}
 We say that $\phi$ satisfies the non-linear integral identity for $(\Ab,\omega)$ satisfying the assumptions (1)--(7) if 
 \[ (\hat{\pi}_1\times\hat{\pi_3})_!\hat{\pi}_2^\ast\phi_{p_{d+d'}\rho_{d+d'}} = (\hat{\pi}_1\times \hat{\pi}_3)_!\phi_{p_{d+d'}\rho_{d+d'}\hat{\pi}_2}= \LL^{dd'-(\gamma',\gamma)}\phi_{p_d\rho_d}\boxtimes \phi_{p_{d'}\rho_{d'}} \]
 holds in $R(X_{\gamma,d}\times X_{\gamma',d'}\longrightarrow\Msp)$.
\end{definition}
As for the linear integral identity (\ref{integral_identity}), the last equation is a consequence of the projection formula, of the fact that  $\hat{\pi}_1\times\hat{\pi}_3$ is a vector bundle of rank $dd'-(\gamma',\gamma)$ and the property of $\phi$ to commute with smooth pull-backs. Using the intersection complex $\ICS_{X_{\gamma,d}}=\LL^{((\gamma,\gamma)-d^2)/2}\unit_{X_{\gamma,d}}$ and similarly for $X_{\gamma',d'}$ and $X_{\gamma+\gamma',d+d'}$, we can rewrite the integral identity in a form independent of the $\Gamma$-decomposition of $\Msp$
\begin{eqnarray} \label{non-linear_integral_identity}
 \nonumber (\hat{\pi}_1\times\hat{\pi}_3)_!\hat{\pi}_2^\ast\phi_{p_{d+d'}\rho_{d+d'}}(\ICS_{X_{d+d'}}) & =& (\hat{\pi}_1\times \hat{\pi}_3)_!\phi_{p_{d+d'}\rho_{d+d'}\hat{\pi}_2}(\hat{\pi}_2^\ast \ICS_{X_{d+d'}})\\ &=& \LL^{\langle \gamma,\gamma'\rangle/2 }\phi_{p_d\rho_d}(\ICS_{X_d})\boxtimes \phi_{p_{d'}\rho_{d'}}(\ICS_{X_d'}). 
\end{eqnarray}

\begin{proposition} \label{integral_identity_3}
The following is true.
 \begin{enumerate}
  \item The non-linear integral identity hold for all canonical $\underline{\ZZ}(Sm^{proj})$-algebra structures on motivic reduced ring $(\Sch_\kk,ft)$-theories and all pairs $(\Ab,\omega)$.
  \item If the identity holds for $(\Ab,\omega)$ and $\PPP$ is an open property closed under extensions and subquotients, then it also holds for $(\Ab^\PPP,\omega^\PPP)$. 
  \item Assume $\Char{\kk}=0$. The non-linear integral identity holds in the case of quiver representation for any $\underline{\ZZ}(Sm^{proj})$-algebra $R$ whose structure homomorphism has a factorization through $\underline{\Ka}_0(\Sch^{ft})$.
  \item Assume $\kk=\CC$. The non-linear integral identity holds  for every pair $(\Ab,\omega)$ and the pull-back of $\phi^{con},\phi^{perv}$ or $\phi^{mhm}$ with respect to   any generalized potential $W:\Msp\to \AAA$.  
 \end{enumerate}
\end{proposition}
\begin{example}   \rm
 Combining the second and the third statement of the proposition, the non-linear integral identity also holds in the case $\Ab=\kk Q\rep^\zeta_\mu$ of  $\zeta$-semistable quiver representations of slope $\mu$. 
\end{example}
\begin{proof}
\begin{enumerate}
 \item The first case is trivial as $\phi_{p_{d+d'}\rho_{d+d'}}=\unit_{X_{d+d'}}$ and similarly for the other spaces.
 \item The second case follows from the assumption that $\PPP$ is extension closed which implies that with $(E,\psi)$ and $(E',\psi')$ having property $\PPP$, the hole fiber of $\hat{\pi}_1\times \hat{\pi}_3$ over $((E,\psi),(E',\psi'))$ has property $\PPP$. Moreover, $\phi$ commutes with open pull-backs.
 \item The case of quiver representations is a straight forward application of the linear integral identity proven in Theorem \ref{integral_identity_2} with respect  to the 2-graded vector bundle $V=X_{d+d'}\longrightarrow X=X_d\times X_{d'}$ with $V^+=X_{d,d'}$ and $V^-=X_{d',d}$.
 \item The proof of the last case is a bit more involved and will be given after Lemma \ref{analytic_neighbourhood}. 
\end{enumerate}
\end{proof}

The main result of this subsection is the following statement due to Kontsevich and Soibelman (see \cite{KS1},\cite{KS2}).
\begin{theorem} \label{algebra_homomorphism}
If $\phi$ satisfies the non-linear integral identity, the integration map $I^{\phi}$ is an algebra homomorphism with respect to the $\ast$-product.
\end{theorem}
\begin{proof}
We prove the theorem in several steps. 

\begin{definition}
 We define the Ringel--Hall product on $R(X)=\prod_{d\in \NN^{\oplus I}}R(X_d)$ by continuous linear extension of
 \[ a\ast b:= \LL^{-dd'}\hat{\pi}_{2\, !}(\hat{\pi}_1\times \hat{\pi}_3)^\ast (a\boxtimes b) \]
 for $a\in R(X_d)$, $b\in R(X_{d'})$, and $dd'=\sum_{i\in I}d_id'_i$.
\end{definition}
\begin{lemma} \label{star_product}
 The Ringel--Hall product on $R(X)$ is associative with unit $1_0=(\Spec\kk=X_0\hookrightarrow X)_!(1)$. Moreover, given $a\in R(X_d/G_d)=R(\Mst_d)$ and $b\in R(X_{d'}/G_{d'})=R(\Mst_{d'})$ we define $\bar{a}:=\rho^\ast_d(a)/[G_d]$ and $\bar{b}:=\rho^\ast_{d'}(b)/[G_{d'}]$. Then
 \begin{enumerate} 
  \item[(i)] $a=\rho_{d\, !}(\bar{a})$ and $b=\rho_{d'\,!}(\bar{b})$, 
  \item[(ii)] $a\ast b=\rho_{d\, !}(\bar{a}) \ast \rho_{d'\,!}(\bar{b}) = \rho_{(d+d')\,!}(\bar{a}\ast\bar{b})$.
 \end{enumerate}
\end{lemma}
\begin{proof}
Proving the associativity law and the unit law is done in exactly the same way as for the usual Ringel--Hall product. Using that $\rho_d:X_d \longrightarrow X_d/G_d$ is a principal $G_d$-bundle and Example \ref{quotient_stack_4}, we get $\rho_{d\,!}(\bar{a})=a\cap\rho_{d\,!}(\unit_{X_d}) /[G_d]=a\cap\unit_{\Mst_d}[G_d]/[G_d]=a$ and similar for $b$ proving (i). Moreover,
\begin{eqnarray*}
 \rho_{(d+d')\,!}(\bar{a}\ast \bar{b}) & = & \LL^{-dd'}\rho_{(d+d')\,!}\hat{\pi}_{2\,!} (\hat{\pi}_1\times\hat{\pi}_3)^\ast (\bar{a}\boxtimes \bar{b}) \\
 & = & \LL^{-dd'}\pi_{2\,!}\rho_{d,d' \,!} (\hat{\pi}_1\times\hat{\pi}_3)^\ast\Big(\frac{\rho_d^\ast(a)}{[G_d]}\boxtimes \frac{\rho_{d'}^\ast(b)}{[G_{d'}]}\Big) \\
 & = & \frac{1}{\LL^{dd'}[G_d][G_{d'}]} \pi_{2\,!}\rho_{d,d' \,!} \rho_{d,d'}^\ast (\pi_1\times \pi_3)^\ast (a\boxtimes b) \\
 & = & \frac{[G_{d,d'}]}{[G_{d,d'}]} \pi_{2\,!} (\pi_1\times \pi_3)^\ast (a\boxtimes b) \\
 & = & a\ast b.
\end{eqnarray*}
\end{proof}
\begin{definition}
We define the integration  map $\hat{I}^{\phi}=\prod_{d\in \NN^{\oplus I}}\hat{I}^{\phi}_d:R(X)\longrightarrow R(\id)$ as follows
\[ \hat{I}^{\phi}_d:R(X_d) \ni \hat{a} \longmapsto \LL^{d^2/2}(p_d  \rho_d)_!\big(\hat{a}\cap \phi_{p_d\rho_d}(\ICS_{X_d})\big) \in R(\iota_d). \] 
\end{definition}
\begin{proposition} \label{algebra_homomorphism_2}
Assuming the non-linear integral identity, the following  holds
\begin{enumerate}
 \item[(i)] $I^{\phi}(\rho_!(\hat{a}))=\hat{I}^{\phi}(\hat{a})$ for all $\hat{a}\in R(X)$,
 \item[(ii)] the map $\hat{I}^{\phi}:(R(X),\ast)\longrightarrow (R(\id),\ast)$ is an algebra homomorphism.
\end{enumerate}
\end{proposition}
\begin{proof}
(i) Note that $\rho_d^\ast\ICS_{\Mst_d}=\LL^{d^2/2}\ICS_{X_d}$, and we obtain for $\hat{a}\in R(X_d)$ 
\begin{eqnarray*} I^{\phi}_d(\rho_{d\,!}(\hat{a})) &=&p_{d}\big(\rho_{d\,!}(\hat{a})\cap \phi_{p_d}(\ICS_{\Mst_d})\big), \\
&=& (p_{d}\circ\rho_{d})_!\big(\hat{a}\cap\rho_d^\ast \phi_{p_d}(\ICS_{\Mst_d})\big), \\
&=& (p_{d}\circ\rho_{d})_!\big(\hat{a}\cap \phi_{p\rho_d}(\rho_d^\ast\ICS_{\Mst_d})\big), \\
&=&\LL^{d^2/2} (p_{d}\circ\rho_{d})_!\big(\hat{a}\cap \phi_{p\rho_d}(\ICS_{X_d})\big),\\
&=&\hat{I}^{\phi}_d(\hat{a}). 
\end{eqnarray*}
(ii) For the second part, pick $\hat{a}\in R(X_{\gamma,d})$ and $\hat{b}\in R(X_{\gamma',d'})$. We use diagram (\ref{convolution_diagram}) and the non-linear integral identity (\ref{non-linear_integral_identity}) to conclude
\begin{eqnarray*}
 \lefteqn{ \hat{I}^{\phi}_{d+d'}(\hat{a}\ast \hat{b}) } \\
 & = & \LL^{(d+d')^2/2}(p_{d+d'}\rho_{d+d'})_!\big((\hat{a}\ast\hat{b})\cap \phi_{p_{d+d'}\rho_{d+d'}}(\ICS_{X_{d+d'}})\big)    \\
 &=& \LL^{(d+d')^2/2}\LL^{-dd'}(p_{d+d'}\rho_{d+d'})_!\big(\hat{\pi}_{2\,!}(\hat{\pi}_1\times\hat{\pi}_3)^\ast(\hat{a}\boxtimes \hat{b})\cap \phi_{p_{d+d'}\rho_{d+d'}}(\ICS_{X_{d+d'}})\big) \\
 &=& \LL^{d^2/2+d'^2/2}(p_{d+d'}\rho_{d+d'}\hat{\pi}_{2})_!\big((\hat{\pi}_1\times\hat{\pi}_3)^\ast(\hat{a}\boxtimes \hat{b})\cap \hat{\pi}_2^\ast\phi_{p_{d+d'}\rho_{d+d'}}(\ICS_{X_{d+d'}})\big) \\
 &=& \LL^{d^2/2+d'^2/2}\big(\oplus (p_d\rho_d\times p_{d'}\rho_{d'})(\hat{\pi}_1\times\hat{\pi}_3)\big)_!\big((\hat{\pi}_1\times\hat{\pi}_3)^\ast(\hat{a}\boxtimes \hat{b})\cap \hat{\pi}_2^\ast\phi_{p_{d+d'}\rho_{d+d'}}(\ICS_{X_{d+d'}})\big) \\
 &=& \LL^{d^2/2+d'^2/2}\oplus_!(p_d\rho_d\times p_{d'}\rho_{d'})_!\big(\hat{a}\boxtimes \hat{b}\cap (\hat{\pi}_1\times\hat{\pi}_3)_! \hat{\pi}_2^\ast \phi_{p_{d+d'}\rho_{d+d'}}(\ICS_{X_{d+d'}})\big) \\
 &=& \LL^{\langle \gamma,\gamma'\rangle/2} \LL^{d^2/2+d'^2/2}\oplus_!(p_d\rho_d\times p_{d'}\rho_{d'})_!\big( (\hat{a}\boxtimes \hat{b})\cap (\phi_{p_d\rho_d}(\ICS_{X_d})\boxtimes \phi_{p_{d'}\rho_{d'}}(\ICS_{X_{d'}}))\big) \\
 & = & \LL^{\langle \gamma,\gamma'\rangle/2}\oplus_! \big( \LL^{d^2/2}(p_{d}\circ\rho_{d})_!(\hat{a}\cap \phi_{p\rho_d}(\ICS_{X_d})) \boxtimes \LL^{d'^2/2}(p_{d'}\circ\rho_{d'})_!(\hat{b}\cap \phi_{p\rho_{d'}}(\ICS_{X_{d'}})) \big) \\
 &= & \hat{I}^{\phi}_d(\hat{a})\ast \hat{I}^{\phi}_{d'}(\hat{b}) 
\end{eqnarray*}
\end{proof}
If we now apply Lemma \ref{star_product} and Proposition \ref{algebra_homomorphism_2} to $\bar{a}:=\rho^\ast_d(a)/[G_d]$ and $\bar{b}:=\rho^\ast_{d'}(b)/[G_{d'}]$, we finally get
\begin{eqnarray*} I^{\phi}_{d+d'}(a\ast b) &=& I^{\phi}_{d+d'}(\rho_{(d+d')\,!}(\bar{a}\ast\bar{b})) \;=\;\hat{I}^{\phi}_{d+d'}(\bar{a}\ast\bar{b}) \\
&=&\hat{I}^{\phi}_d(\bar{a})\ast\hat{I}^{\phi}_{d'}(\bar{b})\;=\;I^{\phi}_d(a)\ast I^{\phi}_{d'}(b)
\end{eqnarray*}
proving Theorem \ref{algebra_homomorphism}. 
\end{proof}

The following technical lemma will be used for the proof of the remaining fourth case in Proposition \ref{integral_identity_3}. Hence, we will assume $\kk=\CC$ for the remaining part of this subsection.

\begin{lemma} \label{analytic_neighbourhood}
 For every closed point $(x,x')\in X_d\times X_{d'}\subset X_{d+d'}$ corresponding to a pair $((E,\psi),(E',\psi'))$ of ``trivialized'' semisimple objects in $\Ab_\CC$ with reductive stabilizer $H:=\Aut_{\Ab_\CC}(E)\times \Aut_{\Ab_\CC}(E)\subset G_{d+d'}$ given by a product of general linear groups, there is an analytic $H$-invariant neighborhood $U\subset X_{d+d'}$, an analytic $H$-invariant neighborhood $\tilde{U}\subset \Ta_{(x,x')}X_{d+d'}$ of $0$ and an $H$-equivariant biholomorphic map $\Theta:U \xrightarrow{\sim} \tilde{U}$ mapping $(x,x')$ to $0$, $U\cap (X_d\times X_{d'})$ to $\tilde{U}\cap (\Ta_x X_d\oplus \Ta_{x'}X_{d'})$, and  $U\cap X_{d,d'}$ to $\tilde{U}\cap (\Ta_x X_d\oplus \Ta_{x'}X_{d'}\oplus \Ext_{\Ab_\CC}^1(E',E))=\tilde{U}\cap \Ta_{(x,x')}X_{d,d'}$. 
\end{lemma}
\begin{proof}
 We pick an affine $H$-invariant neighborhood $\Spec A$ of $(x,x')$ in $X_{d+d'}$ which exists due to our assumption (4) on $\Ab$. Let $\mm$ be the maximal ideal of $(x,x')$ such that $A/\mm=\CC$. We denote with $I\subset \mm$ and $J\subset I$ the $H$-invariant vanishing ideals of $(X_d\times X_{d'})\cap \Spec A$ and $X_{d,d'}\cap \Spec A$ respectively. Note that $\Ta_{(x,x')}X_{d+d'}=\Spec \Sym_\CC \mm/\mm^2$ as an affine scheme. First of all, we chose  $H$-equivariant splittings of  
 \begin{eqnarray*}
  \mm &\longrightarrow & \mm/(I+\mm^2)\cong \Ta^\ast_x X_d \oplus \Ta^\ast_{x'}X_{d'},\\  
  I &\longrightarrow&  I/(J+\mm I) \cong N^\ast_{(x,x');X_d\times X_{d'}| X_{d,d'}} \cong \Ext^1_{\Ab_\CC}(E',E)^\ast,\\
  J &\longrightarrow & J /\mm J \cong N^\ast_{(x,x');X_{d,d'}\mid X_{d+d'}}
 \end{eqnarray*}
 which can be done as $H$ is reductive.  Since \[ \Ta^\ast_{(x,x')}X_{d+d'}=\Ta^\ast_xX_d \oplus \Ta_{x'}X_{d'}\oplus N^\ast_{(x,x');X_d\times X_{d'}| X_{d,d'}} \oplus N_{(x,x');X_{d,d'}\mid X_{d+d'}} \] 
 as algebraic $H$-representations, we get an $H$-equivariant $\CC$-algebra homomorphism $\Theta^\sharp:\Sym_\CC \mm/\mm^2\longrightarrow A$ such that the pull-back of $\mm$ is the maximal ideal corresponding to $0\in \Ta_{(x,x')}X_{d+d'}$. Hence, we obtain an $H$-equivariant morphism $\Theta:\Spec A \to \Ta_{(x,x')}X_{d+d'}$ mapping $\Spec A \cap (X_d\times X_{d'})$ to $\Ta_x X_d\oplus  \Ta_{x'}X_{d'}$ and $\Spec A\cap X_{d,d'}$ to $\Ta_{x}X_d\oplus \Ta_{x'}X_{d'}\oplus \Ext^1_{\Ab_\CC}(E',E)$ by construction. Moreover, $\Omega^1_{\Spec A / \Ta_{(x,x')}X_{d+d'}}$ has vanishing fiber over $(x,x')$. Hence $\Omega^1_{\Spec A_f / \Ta_{x,x'} X_{d+d'}}=0$ for some $H$-invariant $0\not= f\in A$, in other words, after replacing $\Spec A$ with $\Spec A_f\subset \Spec A$, $\Theta:\Spec A\to \Ta_{(x,x')}X_{d+d'}$ is \'{e}tale. We pick a Riemannian metric on the (real) manifold $\Spec A$, and by Weyl's trick we can even assume that the metric is invariant under the maximal connected compact subgroup $K$ of $H$ which is a product of unitary groups. As the differential of $\Theta$ at $(x,x')$  is the identity, the inverse function theorem of complex analysis tells us that $\theta$ is biholomorphic on an open $K$-invariant ball $B_\epsilon(x,x')$ centered at $(x,x')$. Then, it is not difficult to show that $\Theta$ remains biholomorphic when restricted to  the open $H$-invariant subset $U:=H\cdot B_\epsilon(x,x')$ of $\Spec A$. Indeed, it is still locally biholomorphic in the analytic topology as $\Theta$ is \'{e}tale. Moreover, it is injective on $U$. If $\Theta(gy)=\Theta(g'y')$ for $y,y'\in B_\epsilon(x,x')$ and $g,g'\in H$ then, $\Theta((g')^{-1}gy)=\Theta(y')$, and we may assume $g'=1$. The complex analytic function $H\ni g \longmapsto \Theta(gy)-\Theta(y)$ vanishes on the maximal connected compact subgroup $K$ by construction of $U$. Hence, it vanishes on $H$ by Hartog's theorem. Therefore, $\Theta(y)=\Theta(gy)=\Theta(y')$ and $y=y'$ by assumption $y,y'\in B_\epsilon(x,x')$ and construction of $B_\epsilon(x,x')$. The subset $\tilde{U}=\Theta(U)$ in $\Ta_{(x,x')}X_{d+d'}$ is open in the analytic topology as $\Theta$ is \'{e}tale, and the proof of the lemma is finished.     
\end{proof}

\begin{proof}[Proof of Proposition \ref{integral_identity_3}(4)]
 It is enough to show the statement of the non-linear integral identity for the pull-back of $\phi^{perv}$. Passing to fiberwise Euler characteristics commutes with pull-backs and is a $\underline{\ZZ}(Sm^{proj})$-algebra homomorphism. Hence, the result for $\phi^{con}$ will follow. \\
 Let us denote the function $X_{d+d'}\twoheadrightarrow \Msp_{d+d'}\xrightarrow{W} \AA^1$ with $f$. For the mixed Hodge case, we apply $\phi^{mhm}_f$ to the adjunction morphism $\QQ_{X_{d+d'}} \longrightarrow i_! \QQ_{X_{d,d'}}$,  where $i:X_{d,d'}\hookrightarrow X_{d+d'}$ is the closed embedding. As $\phi^{mhm}$ commutes with proper push-forwards, we obtain $\phi_f|_{X_{d,d'}} \longrightarrow \phi_{f|_{X_{d,d'}}}$. Pushing  this down to $X_d\times X_{d'}$ along $\hat{\pi}_1\times \hat{\pi}_3$, we denote the cone of the resulting morphism with $C$, and we have to show $C=0$. As $\rat:\MHM_{mon}(X_d\times X_{d'})\longrightarrow \Perv(X_d\times X_{d'})$ is exact and faithful, it remains to show $\rat(C)=0$, in other words, the non-linear integral identity for $\phi^{perv}$. Using that $D^b(\Perv(X_d\times X_{d'}))$ is motivic even in the analytic topology, we can replace $X_d\times X_{d'}$ with any $H$-invariant analytic neighborhood of a closed point $(x,x')$ with reductive stabilizer $H$. Indeed, if $(x_0,x_0')$ is any closed point in $X_d\times X_{d'}$, we can find a closed point $(x,x')$ with reductive stabilizer in the closure of the $G_d\times G_{d'}$-orbit of $(x_0,x'_0)$. An open neighborhood of $(x,x')$ must meet this orbit, and the non-linear integral identity holds for all points in an open neighborhood of a point $(gx_0,g'x'_0)$ of the $G_d\times G_{d'}$-orbit of $(x_0,x'_0)$. Multiplying with $(g,g')^{-1}$ and using the functoriality of $\phi^{perv}$ and equivariance of the setting, the non-linear integral identity holds in a neighborhood of $(x_0,x'_0)$. To show the identity for a closed point $(x,x')=((E,\psi),(E',\psi'))$ with reductive stabilizer $H=\Aut_{\Ab_\CC}(E)\times \Aut_{\Ab_\CC}(E')$, we use the previous Lemma \ref{analytic_neighbourhood}. As $\phi^{perv}$ depends only on the analytic structure and commutes with open pull-backs, we can reduce the problem to the linear case with $X=\tilde{U}\cap(\Ta_x X_d\times  \Ta_{x'}X_{d'})$ and the open $H$-invariant subset $\hat{U}:=\tilde{U}\cap \big(X \times \Ext^1_{\Ab_\CC}(E',E)\times \Ext^1_{\Ab_\CC}(E,E')\big)$ of the 2-graded vector bundle $V\to X$ with $V^+=X\times \Ext^1_{\Ab_\CC}(E',E)$ and similarly for $V^-$. Notice that the action of $\GG_m$ of weight $\pm 2$ on $V^\pm$ can be realized by the ``anti-diagonal embedding'' $\GG_m\hookrightarrow H$. Invariance under this action will certainly imply invariance under the weight $\pm 1$ action mentioned in Theorem \ref{integral_identity_2}. The modification of the latter to our case is discussed in Remark \ref{integral_identity_4}. Thus, the (non-)linear integral identity holds and the Proposition is finally proven.       
\end{proof}
\begin{remark}\rm
 One could hope to prove Proposition \ref{integral_identity_3}(4) even for the motivic vanishing cycle $\phi^{mot}$ using some formal neighborhood of $X_{d,d'}$ to reduce the non-linear integral identity to the linear one. The relevant tool one needs to prove is Conjecture 4 in \cite{KS1}.  
\end{remark}

\subsection{Wall-crossing}
Our first application of Theorem \ref{algebra_homomorphism} is the wall-crossing formula. For the remaining part of this section we will assume that the non-linear integral identity holds allowing us to make use of the integration map $I^\phi$. For this we fix weak stability condition $(\zeta,T,\le)$  on $\Ab$ which is permissible in the sense of \cite{JoyceIII}, Def. 4.7. Roughly speaking this means the following: $(T,\le)$ is a totally ordered set and every object in $\Ab_\KK$ has a unique Harder--Narasimhan filtration with $\zeta$-semistable subquotients indexed by elements of $T$ in decreasing order, where $\KK\supset\kk$ is any algebraic extension. Moreover, the substack $\Mst^{\zeta-ss}_\mu$ of semistable objects of ``slope'' $\mu\in T$ is locally closed. Gieseker stability or any Bridgeland stability is a good example. We make the following definition using the shorthand $\delta_\mu:=(\Mst^{\zeta-ss}_\mu\hookrightarrow \Mst)_!(\unit_{\Mst})$
\begin{definition}
 The Donaldson--Thomas function \[ \DTS(\Ab,\phi)_\mu^\zeta=\sum_{0\not=d\in \NN^{\oplus I},\zeta(d)=\mu}\DTS(\Ab,\phi)^\zeta_d\in  R(\id)\] is defined by means of the equation
 \[ I^{\phi}(\delta_\mu)=p|_{\Mst^{\zeta-ss}_\mu\,!}\big(\phi_{p}(\ICS_{\Mst})|_{\Mst^{\zeta-ss}_\mu}\big)=\Sym\Big(\frac{\DTS(\Ab,\phi)^\zeta_\mu}{\LL^{1/2}-\LL^{-1/2}}\Big). \]
 
\end{definition}
The existence and uniqueness of the Harder--Narasimhan filtration can  be written as 
\[ \unit_{\Mst}=\prod_\ast^\curvearrowright \delta_{\mu}, \]
where the Ringel--Hall product is taken in decreasing order of the slopes. Applying $I^{\phi}$ under the assumption that the non-linear integral identity holds proves the following proposition.
\begin{proposition}[Wall-crossing formula]
Given any permissible weak stability condition $(\zeta,T,\le)$, the following formula holds
\[ I^{\phi}(\unit_{\Mst}) = \prod_\ast^\curvearrowright I^{\phi}(\delta_\mu)= \prod_\ast^\curvearrowright \Sym\Big(\frac{\DTS(\Ab,\phi)^\zeta_\mu}{\LL^{1/2}-\LL^{-1/2}}\Big), \]
where the product is taken in decreasing order of the slopes. In particular, 
\[ \prod_\ast^\curvearrowright \Sym\Big(\frac{\DTS(\Ab,\phi)^{\zeta}_\mu}{\LL^{1/2}-\LL^{-1/2}}\Big)= \prod_\ast^\curvearrowright \Sym\Big(\frac{\DTS(\Ab,\phi)^{\zeta'}_{\mu'}}{\LL^{1/2}-\LL^{-1/2}}\Big) \]
for every pair $(\zeta,T,\le),(\zeta',T',\le)$ of permissible weak stability conditions. 
\end{proposition}
A word of warning is in order. In many situations the subfunctor $\Ab^{\zeta-ss}_\mu$ obtained by restricting ourselves to families of semistable objects of slope $\mu$ does also satisfy our assumptions (1)--(7). However, in contrast to the moduli stack $\Mst^{\zeta-ss}_\mu$, the moduli space $\Msp^{\zeta-ss}_\mu$ is not a locally closed subscheme of $\Msp$. Instead, there is a monoid homomorphism $q^\zeta_\mu:\Msp^{\zeta-ss}_\mu\to \Msp$ of finite type. Pulling back $(R,\phi)$ to $\Msp^{\zeta-ss}_\mu$ along $q^\zeta_\mu$ allows us to apply our machinery to $\Ab^{\zeta-ss}_\mu$. The corresponding Donaldson--Thomas function $\DTS(\Ab^{\zeta-ss}_\mu,q^{\zeta\ast}_\mu \phi)$ will be an element in $(q^{\zeta\ast}_\mu R)(\id_{\Msp^{\zeta-ss}_\mu})=R(q^\zeta_\mu)$. It particular, it is not $\DTS(\Ab,\phi)^\zeta_\mu$. However, if $\Mst^{\zeta-ss}_\mu\subset \Mst$ is open, a simple computation shows $q^\zeta_{\mu!}\DTS(\Ab^{\zeta-ss}_\mu,q^{\zeta\ast}_\mu\phi)=\DTS(\Ab,\phi)^\zeta_\mu$.

\subsection{The PT--DT correspondence}

As second application of Theorem \ref{algebra_homomorphism} we  will show in this section that the Donaldson--Thomas function defined in this section agrees with the one of section \ref{Donaldson_Thomas_framed}. We fix a pair $(\Ab,\omega)$ satisfying all assumptions (1)--(8) and $(R,\phi)$ with $\phi$ satisfying the non-linear integral identity.\\

We apply the ``integration map'' $I^{\phi}=\prod_{d\in \NN^{\oplus I}} I^{\phi}_d$ to the identity proven in Example \ref{Hilbert_scheme_identity} and  use that $\Mst_{f,d}\to \Mst_d$ and $\Mst_{f,d}\to \Msp_{f,d}$ are smooth of relative dimension $fd-1$ and $-1$. Since $\Sym(\LL^i a)=\sum_{n\ge 0} \LL^{ni}\Sym^n(a)$, we obtain
\begin{eqnarray*}
\lefteqn{\frac{1}{\LL-1}\Sym\Big(\sum_{0\not= d\in \NN^{\oplus I}} \frac{\LL^{fd}}{\LL^{1/2}-\LL^{-1/2}} \DTS(\Ab,\phi)_d\Big) }\\
&=&\sum_{d\in \NN^{\oplus I}} \frac{\LL^{fd}}{\LL-1}p_{d\,!}\phi_{p_d}(\ICS_{\Mst_d}) \\
&=&  I^{\phi}\Big( \sum_{d\in \NN^{\oplus I}} \frac{\LL^{fd}}{\LL-1} \unit_{\Mst_d}\Big) \\
&=& I^{\phi}(H)I^{\phi}(\unit_{\Mst})\\
&=&\Big(p_!\big(  (\Mst_{f} \to \Mst)_!(\unit_{\Mst_{f}})\cap \phi_{p}(\ICS_{\Mst})\big)\Big)I^{\phi}(\unit_{\Mst})\\
&=& \Big(\pi_{f !}p_{f !} (\Mst_{f} \to \Mst)^\ast\phi_{p}(\ICS_{\Mst})\Big)I^{\phi}(\unit_{\Mst}) \\
&=& \Big(\pi_{f !}p_{f !}\sum_{d\in \NN^{\oplus I}} \LL^{(fd-1)/2}\phi_{\pi_{f,d}p_{f,d}}(\ICS_{\Mst_{f,d}})\Big)\Sym\Big(\frac{\DTS(\Ab,\phi) }{\LL^{1/2}-\LL^{-1/2}}\Big) \\
&=& \Big(\pi_{f !}p_{f !}\sum_{d\in \NN^{\oplus I}} \LL^{fd/2}p_{f,d}^\ast\phi_{\pi_{f,d}}(\ICS_{\Msp_{f,d}})\Big)\Sym\Big(\frac{\DTS(\Ab,\phi) }{\LL^{1/2}-\LL^{-1/2}}\Big) \\
&=& \frac{1}{\LL^{1/2}-\LL^{-1/2}}\Big(\pi_{f!}\sum_{d\in \NN^{\oplus I}} \LL^{(fd-1)/2}\phi_{\pi_{f,d}}(\ICS_{\Msp_{f,d}})\Big)\Sym\Big(\frac{\DTS(\Ab,\phi) }{\LL^{1/2}-\LL^{-1/2}}\Big) 
\end{eqnarray*}
Using the properties of $\Sym$ and $\frac{\LL^{fd}-1}{\LL^{1/2}-\LL^{-1/2}}=\LL^{1/2}[\PP^{fd-1}]$, we get the so-called PT--DT correspondence.
\begin{proposition}[PT--DT correspondence] \label{PT-DT_stack} Let $(\Ab,\omega)$ be a pair as in section 2 satisfying the conditions (1)--(8). Assuming the non-linear integral identity, we have
\[ \pi_{f!}\sum_{d\in \NN^{\oplus I}} \LL^{df/2}\phi_{\pi_{f,d}}(\ICS_{\Msp_{f,d}})= \Sym\Big(\sum_{0\not= d\in \NN^{\oplus I}} \LL^{1/2}[\PP^{fd-1}] \DTS(\Ab,\phi)_d\Big)
\]
for all framing vectors $f\in \NN^{I}$. 
\end{proposition}
\begin{corollary} Under the assumptions made at the beginning of this section and if the non-linear integral identity holds, the definition of the Donaldson--Thomas function given in this section is equivalent the one given in section \ref{Donaldson_Thomas_framed}. In particular, the integrality conjecture holds under the assumptions of Theorem \ref{main_theorem}.
\end{corollary}
If $f\in (2\NN)^{I}$, we have $fd/2\in \NN$ and the map 
\[ R(\id)\ni (a_d)_{d\in \NN^{\oplus I}} \longmapsto (\LL^{-fd/2}a_d)_{d\in \NN^{\oplus I}}\in R(\id) \]
is an isomorphism of $\lambda$-rings as $\Sym^n(\LL^{-fd/2}a_d)=\LL^{-nfd/2}\Sym^n(a_d)$ in this case. Applying this isomorphism to the PT--DT correspondence yields the alternative form.
\begin{corollary}[PT--DT correspondence, alternative form] Let $(\Ab,\omega)$ be a pair as in section 2 satisfying the conditions (1)--(8). Assuming the non-linear integral identity, we have
\[ \pi_{f!}\phi_{\pi_{f}}(\ICS_{\Msp_{f}})=\Sym\Big(\sum_{0\not= d\in \NN^{\oplus I}} [\PP^{fd-1}]_{vir} \DTS(\Ab,\phi)_d\Big) \]
for all framing vectors $f\in (2\NN)^{I}$ with $[\PP^{fd-1}]_{vir}=\int_{\PP^{fd-1}}\ICS_{\PP^{fd-1}}=\frac{\LL^{fd/2}-\LL^{-fd/2}}{\LL^{1/2}-\LL^{-1/2}}$. 
\end{corollary}
Notice that $\PP^{fd-1}$ is the fiber of $\pi_d$ over $\Mst^s_d$.  The arguments given above show already that the motivic Donaldson--Thomas function $\DTS(\Ab,W^\ast(\phi^{mot}))_d$ must be contained in the image of
\[ \Ka^\muu(\Sch_{\Msp_d})[\LL^{-1/2},[\PP^n]^{-1}: n\in \NN] \longrightarrow \Ka^\muu(\Sch_{\Msp_d})\langle \LL^{-1/2},(\LL-1)^{-1}\rangle^- \;\forall \, d\in \NN^{\oplus I}.\]

If Conjecture \ref{lambda_conjecture} holds, $\DTS(\Ab,W^\ast(\phi^{mot}))$ is  even contained in the image of 
\[ \Ka^\muu(\Sch_{\Msp_d})\langle \LL^{-1/2}\rangle^- \longrightarrow \Ka^\muu(\Sch_{\Msp_d})\langle \LL^{-1/2},(\LL-1)^{-1}\rangle^- \;\forall \, d\in \NN^{\oplus I}\]
due to Theorem \ref{main_theorem}.

Let us discuss once more the case of a functor $\Ab$ satisfying (1)--(8) together with a locally closed property $\PPP$ which is closed under subquotients and extensions. As seen before, $\PPP$ is uniquely characterized by a locally closed subscheme $\bar{\tau}:\Msp^{\PPP}\hookrightarrow \Msp$ such that  
\[
\xymatrix @C=2cm {
\Msp^{\PPP}\times\Msp^{\PPP}\ar[d]_{\oplus}\ar@{^(->}[r]^{\bar{\tau}\times \bar{\tau}}&\Msp\ar[d]_{\oplus}\times\Msp\ar[d]_{\oplus}\\
\Msp^{\PPP}\ar@{^(->}[r]^{\bar{\tau}} &\Msp
}
\]
is cartesian. Let us denote the corresponding substack with $\tau:\Mst^{\PPP}\hookrightarrow \Mst$ and $p^{\PPP}:\Mst^{\PPP}\to \Msp^{\PPP}$ is the restriction of $p:\Mst\to \Msp$. 
Using the pull-back along $\bar{\tau}$, every $\underline{\ZZ}(Sm^{proj})$-algebra $R$ satisfying the assumptions made at the beginning of this section defines a corresponding structure on $\Msp^{\PPP}$ and we could apply our machinery. However, in some situations, it is more desirable to defined restricted Donaldson--Thomas functions   $\DTS^\PPP(\Ab,\phi)$ via
\[
p_{ !}\Big(\phi_{p}\big(\ICS_{\Mst}\big)\cap \tau_!(\unit_{\Mst^{\PPP}})\Big)= \Sym \Big( \frac{\DTS^\PPP(\Ab,\phi)}{\LL^{1/2}-\LL^{-1/2}} \Big).
\]
Then
\begin{eqnarray*}
\lefteqn{ p_{!}\Big(\phi_{p}\big(\ICS_{\Mst}\big)\cap \tau_!(\unit_{\Mst^{\PPP}})\Big)} \\
&=& p_{ !}\Big(\phi_{p}\big(\ICS_{\Mst}\big)\cap \tau_!\big(p^{\PPP\ast}(\unit_{\Msp^{\PPP}})\big)\Big) \\
&=& p_{ !}\Big(\phi_{p}\big(\ICS_{\Mst}\big)\cap p^{\ast}\big(\bar{\tau}_!(\unit_{\Msp^{\PPP}})\big)\Big) \\
&=&p_{ !}\Big(\phi_{p}\big(\ICS_{\Mst}\big)\Big)\cap \bar{\tau}_!(\unit_{\Msp^{\PPP}})\\
&=&\Sym \Big( \frac{\DTS^\PPP(\Ab,\phi)}{\LL^{1/2}-\LL^{-1/2}}\Big)\cap\bar{\tau}_!(\unit_{\Msp^{\PPP}})\\
&=&\bar{\tau}_!\bar{\tau}^{ \ast} \Sym \Big( \frac{\DTS^\PPP(\Ab,\phi)}{\LL^{1/2}-\LL^{-1/2}} \Big)\\
&=& \Sym \Big( \frac{\bar{\tau}_!\bar{\tau}^{\ast}\DTS(\Ab,\phi)}{\LL^{1/2}-\LL^{-1/2}} \Big)\\
\end{eqnarray*}
using Example \ref{lambda-ring_moduli_spaces} in the last step. Hence, \[\DTS^\PPP(\Ab,\phi)=\bar{\tau}_!\bar{\tau}^{ \ast}\DTS(\Ab,\phi)=\bar{\tau}_{!}(\DTS(\Ab,\phi)|_{\Msp^{\PPP}})\] is basically the restriction of $\DTS(\Ab,\phi)$ to $\Msp^{\PPP}$. 
\begin{corollary}
Choosing $\phi=W^\ast(\phi^{mhm})$ as in Example \ref{example2}, we get
 \[ \DTS^\PPP(\Ab,\phi)_d= \begin{cases} \bar{\tau}_{d\,!}\big(\phi^{mhm}_{W}(\ICS_{\Msp_d})|_{\Msp^{\PPP}_d}\big) &\mbox{ if }\Msp^s_d\not=\emptyset, \\ 0 &\mbox{ otherwise.} \end{cases}\]
 in $\Ka_0(\MHM_{mon}(\Msp^{W}_d))$ under the assumption that $\oplus:\Msp\times\Msp\longrightarrow \Msp$ is a finite morphism, and a similar statement holds in $\Ka_0(\Perv(\Msp^{W}_d))$ and also in $\Con(\Msp^{W}_d)$.  
 \end{corollary}

\subsection{Relation to the work of D.\ Joyce}

Let us close this section by showing that our definition of a Donaldson--Thomas function agrees with the one given by D.\ Joyce and Y.\ Song in \cite{JoyceDT}. Their strategy is as follows. Pick a functor $\Ab$ satisfying (1)--(8) as before and a monoid homomorphism $W:\Msp\to \AA^1$. Joyce and Song take the logarithm $\epsilon:=\log_\ast \delta$ with respect to the Ringel--Hall product.\footnote{In \cite{JoyceDT} the authors put a bar on top of $\delta$ and $\epsilon$.} Here, $\delta$ is without loss of generality just $\id_\Mst$. This computation is done in $\SF_{al}(\Mst)$, some $\QQ$-linear version of $\underline{\Ka}_0(\St_{\Mst})$. Moreover, they prove in a deep theorem (see \cite{JoyceDT}, Theorem 3.11 or \cite{JoyceIII}, Theorem 8.7) that $\epsilon\in \SF_{al}^{ind}(\Mst)\subseteq \SF_{al}(\Mst)$. Now, they apply a map $\tilde{\Psi}:\SF_{al}^{ind}(\Mst) \longrightarrow \Con(\Msp)_\QQ$ to $\epsilon$ and call the result $-\bar{\DTS}(\Ab,W)$ (mind the sign!). To connect their approach with the one given here, we need to translate the objects involved. This is done by the following commutative diagram using the shorthand $\nu^{mot}=W^\ast(\phi^{mot})_{p}(\ICS_{\Mst}^{mot})=\phi^{mot}_{Wp}(\ICS_{\Mst}^{mot})$. The definition of the Behrend function $\nu^{con}$ on $\Mst$ will be given below. Note that $\Ka^\muu(\Sch^{ft}_\CC)$ contains already a square root of $\LL$ and we will drop the adjunction of  $\LL^{-1}$ for better readability.
\[ \xymatrix @C=0.5cm { \SF_{al}(\Mst)  \ar[r] &  \underline{\Ka}(\St^{ft}_{\Mst})_\QQ  \ar@{=}[r] &  \underline{\Ka}(\St^{ft}_{\Mst})_\QQ  \ar@{=}[r]&  \underline{\Ka}(\St^{ft}_{\Mst})_\QQ \ar[d]^{\cap \nu^{mot}} \\ 
\SF_{al}^{ind}(\Mst) \ar[dddd]^{\tilde{\Psi}} \ar@{^{(}->}[u] \ar[r] &   \frac{1}{\LL-1}\underline{\Ka}_0(\Sch^{ft}_{\Mst})_\QQ \ar@{^{(}->}[u] \ar[d]^{\cdot(\LL-1)} \ar@{=}[r] &  \frac{1}{\LL-1}\underline{\Ka}_0(\Sch^{ft}_{\Mst})_\QQ \ar@{^{(}->}[u] \ar[d]^{\cdot(\LL-1)} & \underline{\Ka}^\muu(\St^{ft}_{\Mst})_\QQ \ar[d]^{p_{!}}\\
& \underline{\Ka}_0(\Sch^{ft}_{\Mst})_\QQ \ar[d]^{\chi^{na}_c} \ar@{=}[r] &  \underline{\Ka}_0(\Sch^{ft}_{\Mst})_\QQ  \ar[d]^{\cap \nu^{mot}}&  \underline{\Ka}^\muu(\St^{ft}_{\Msp})_\QQ\\
 &   \Con(\Mst)_\QQ \ar[d]^{\cap \nu^{con}} & \underline{\Ka}^\muu(\Sch^{ft}_{\Mst})_\QQ \ar[d]^{p_{ !}} & \frac{1}{\LL-1}\underline{\Ka}^\muu(\Sch^{ft}_{\Msp})_\QQ \ar[d]^{\cdot (\LL-1)} \ar@{^{(}->}[u] \\ 
 &  \Con(\Mst)_\QQ \ar[d]^{p^{na}_{!}} & \underline{\Ka}^\muu(\Sch^{ft}_{\Msp})_\QQ \ar[d]^{\chi^{na}_c} \ar@{=}[r] & \underline{\Ka}^\muu(\Sch^{ft}_{\Msp})_\QQ \ar[d]^{\chi_c}\\
 \Con(\Msp)_\QQ  \ar@{=}[r] &   \Con(\Msp)_\QQ \ar@{=}[r] &   \Con(\Msp)_\QQ \ar@{=}[r] &   \Con(\Msp)_\QQ }
\]
Here, $\chi_c^{na}$ and $p^{na}_!$ are the naive extensions of $\chi_c$ and $p_!$ to ``stack functions'' and non-representable morphism. See \cite{JoyceCF} for more details. The commutativity of the square in the upper right corner is obvious. The commutativity of the big square in the lower left corner follows from the definition of $\tilde{\Psi}$ (see \cite{JoyceDT}, equation (5.6)). Because of $p^{na}_!\chi^{na}_c=\chi^{na}_c p_!=\chi_c p_!$, the commutativity of the square in the middle is just saying that $\chi^{na}_c$ commutes with the $\cap$-product. We will see that this is a consequence of the fact that $\chi_c$ is a morphism $\phi^{mot}\to \phi^{con}$ of vanishing cycles.  As we have not defined $\chi_c^{na}$ and $p^{na}_!$, let us give an alternative description of $p^{na}_!(\chi^{na}_c(-)\cap \nu^{con})$ according to equation (5.5) in \cite{JoyceDT}. Given a morphism $Y\xrightarrow{g} X_d/G_d$ with $Y$ being a scheme, we get as usual a principal $G_d$-bundle $P\xrightarrow{\alpha} Y$ by taking the pull-back of $X_d\to X_d/G_d$ and a $G_d$-equivariant morphism $\beta:P\to X_d$. 
\[ \xymatrix @C=1.5cm { P \ar[r]^\beta \ar[d]_\alpha & X_d \ar[d] \ar[drr]^{W} & &  \\ Y \ar[r]_g & X_d/G_d \ar[r]_p & \Msp_d \ar[r]_{W} &\AA^1 }\]
Then, $p^{na}_{!}(\chi^{na}_c(g)\cap \nu^{con})=(pg)_!(g^\ast \nu^{con})$, were $g^\ast\nu^{con}$ is the unique constructible function on $Y$ such that $\alpha^\ast g^\ast \nu^{con}= (-1)^{d^2}\beta^\ast \phi^{con}_{W}(\ICS_{X_d}^{con})$ by definition of $\nu^{con}$. On the other hand $\chi_{c}p_{!}(g\cap \nu^{mot})=\chi_{c}(pg)_!(g^\ast \nu^{mot})=(pg)_!\chi_c(g^\ast \nu^{mot})$ by the projection formula. But \[\alpha^\ast \chi_cg^\ast\nu^{mot}=\chi_c \alpha^\ast g^\ast\nu^{mot}=\chi_c\beta^\ast \LL^{d^2/2}\phi^{mot}_{W}(\ICS_{X_d}^{mot})=(-1)^{d^2}\beta^\ast\phi^{con}_{W}(\ICS^{con}_{X_d})=\alpha^\ast g^\ast\nu^{con}\] by definition of $W^\ast(\phi^{mot})$ on stacks and the fact that $\chi_c$ is a morphism $\phi^{mot}\to \phi^{con}$ of vanishing cycles. Thus, the lower middle rectangle does also commute. There are two technical difficulties to mention. First of all $\underline{\Ka}_0(\Sch^{ft}_{\Mst})$ should be replaced by its image in $\underline{\Ka}(\St^{ft}_{\Mst})$ as the  ``embedding'' map does not need to be injective and similarly for $\underline{\Ka}^\muu(\Sch^{ft}_{\Msp})$. The kernel is given by motives annihilated by $[G_d]$ for some $d\in \NN^{\oplus I}$. However, $\chi_c$ will vanish on the kernel  as the E-polynomial of the motives in the kernel must be zero. Secondly, one has to check that $g^\ast\nu^{mot}$ is in $\Ka^\muu(\Sch^{ft}_Y)$, in other words that $-\cap \phi^{mot}_{p}(\ICS_{\Mst})$ maps  $\underline{\Ka}_0(\Sch^{ft}_{\Mst})$ into $\underline{\Ka}^\muu(\Sch^{ft}_{\Mst})$. Since everything is motivic, we can assume that  $\alpha:P\to Y$ has a section $s:Y\to P$.  Then, $ g^\ast\nu^{mot}=s^\ast\alpha^\ast g^\ast \nu^{mot}=\LL^{d^2/2}s^\ast \beta^\ast \phi^{mot}_{W}(\ICS_{X_d}^{mot}) $ is indeed in $\Ka^\muu(\Sch^{ft}_Y)$.\\ 
By the commutativity of the diagram above, we can do the computation using the  right column instead of the left one. Note that the composition of the first two vertical arrows is just $I^{\phi}$. Therefore, $I^{\phi}((\LL^{1/2}-\LL^{-1/2})\epsilon)=:\bar{\DTS}^{mot}(\Ab,W)$ will satisfy $\chi_c(\bar{\DTS}^{mot}(\Ab,W))=\bar{\DTS}(\Ab,W)$ as $\chi_c(\LL^{-1/2})=-1$.
Since $I^{\phi}$ is an algebra homomorphism, we obtain under the symmetry assumption (8) the following equation 
\begin{eqnarray*} \lefteqn{ \exp\Big(\frac{\bar{\DTS}^{mot}(\Ab,W)}{\LL^{1/2}-\LL^{-1/2}}\Big)} \\
&=& \exp(I^{\phi}(\epsilon)) \\
&=& I^{\phi}(\exp_\ast \epsilon)\\
&=& I^{\phi}(\delta) \\
&=&\Sym\Big(\frac{\DTS^{mot}(\Ab,W)}{\LL^{1/2}-\LL^{-1/2}}\Big)\\
&=&\exp\Big( \sum_{k\ge 1}\frac{(-1)^{k-1}}{k} \frac{\psi^k\big(\DTS^{mot}(\Ab,W)\big)}{\LL^{k/2}-\LL^{-k/2}}\Big), 
\end{eqnarray*}
where we used $\psi^k(\LL^{\pm 1/2})=-\psi^k(-\LL^{\pm 1/2})=(-1)^{k+1}\LL^{\pm k/2}$. Thus, we conclude
\[ \bar{\DTS}^{mot}(\Ab,W)=\sum_{k\ge 1} \frac{(-\LL^{1/2})^{k-1}}{k[\PP^{k-1}]}\psi^k\big(\DTS^{mot}(\Ab,W)\big), \]
and after specializing to Euler characteristics, this becomes
\[ \bar{\DTS}(\Ab,W)=\sum_{k\ge 1} \frac{1}{k^2} P_{k\,!}\DTS(\Ab,\phi), \]
where $P_k:\Msp\ni x\longmapsto x^{\oplus k} \in \Msp$ sends a point to its k-fold direct sum. The constructible function $\DTS(\Ab,W)_d$ is called $F_d(\zeta)$ in \cite{JoyceDT} and its integral $\int_{\Msp_d}F_d(\zeta)=:\hat{\DT}(Q,W)_d$ coincides with our Donaldson--Thomas invariant. In particular, using Theorem \ref{main_theorem_2}, we can prove Conjecture 6.12, 6.13, and 6.14 in \cite{JoyceDT} for quivers with potential and generic, geometric stability condition  with $\mathcal{Q}$, the sheaf on $\Msp_{d}$ categorifying the DT/BPS invariant, being the perverse sheaf $\phi^{perv}_{W}(\ICS_{\Msp_d})$ if $\Msp^s_d\not=\emptyset$ and $\mathcal{Q}=0$ else. \\

\begin{appendix}

\section{Proof of the integral identity}

The aim of this section is to prove the following theorem.

\begin{theorem}[cf.\ Kontsevich, Soibelman \cite{KS2}]
 Let $(\Sm,\Pro)\supset(Sm,proj)$ be a motivic pair, i.e.\ $\Pro\cap \Sm$ contains all locally closed embeddings. Let  $\altT$ be a motivic categorical ring $(\Sm,\Pro)$-theory with a $\underline{D}^b\QQ(Sm^{proj})$-algebra structure over $\MM$ which has a factorization through the canonical $\underline{D}^b\QQ(Sm^{proj})$-algebra structure of a motivic reduced categorical ring $(\Sm,\Pro)$-theory $\mathcal{K}$, i.e.\ there is a morphism $\phi:\mathcal{K}\to \altT$ of categorical ring $(Sm,proj)$-theories over $\MM$, such that $\phi_f=\phi_f(\unit_X)$ for every $f:X\to \MM$ with smooth $X$. The $(\Sm,\Pro)$-theory $\mathcal{K}$ should also satisfy the following assumption: for every locally closed embedding $e$ the functors $e_!$ and $e^\ast$ have right adjoints $e^!$ and $e_\ast$ which implies $e_!=e_\ast$ for every closed embedding and $e^!=e^\ast$ for every open embedding as $\mathcal{K}$ is motivic. Moreover, given a closed embedding $i:Z\hookrightarrow X$ with open complement $j:U\hookrightarrow X$, the sequence
 \[ i_! i^!(a) \longrightarrow a \longrightarrow j_\ast j^\ast (a) \]
 of adjunction morphisms can be extended to a distinguished triangle functorial in $a\in \mathcal{K}(X)$. Finally, we assume that $\int_{\AA^1}(\GG_m\hookrightarrow \AA^1)_\ast(\unit_{\GG_m})=0$ in $\mathcal{K}(\kk)$ (``homotopy invariance''). Then 
 \[ \pi^+_! \phi_f |_{V^+}=\pi^+_! \phi_{f|_{V^+}}=\LL^{\rk V^+}\phi_{f|_X}.\]
 for every 2-graded vector bundle $\pi=\pi^+\oplus \pi^-:V^+\oplus V^-\longrightarrow X$ on a smooth scheme $X$ and every $\GG_m$-invariant morphism $f:V\to \MM$, where $\GG_m$ acts with weights $\pm 1$ on $V^\pm$.     
\end{theorem}
\begin{remark} \rm \label{remark3}
\begin{enumerate}
\item As all open inclusions are in $\Pro$, $\Pro$ contains all quasiprojective morphisms.  
\item Using the fact that every line bundle $\pi:L\to X$ on a smooth $X$ is locally trivial, the homotopy invariance of $\mathcal{K}$ implies 
\begin{equation} \label{homotopy} \pi_!(L\setminus 0_L \hookrightarrow L)_\ast (\unit_{L\setminus 0_L})=0\mbox{ in }\mathcal{K}(X),\end{equation} where we identify $X$ with the zero section of $L$.
\item Even more is true under the assumption that $(-)_\ast$ extends to a functor on quasiprojective morphism $g$  such that $g_\ast=g_!$ if $g$ is projective. In other words, given a factorization $g:Y \xrightarrow{j} Z   \xrightarrow{q} X$ of a quasiprojective morphism $g$ into an open embedding $j$ followed by a projective morphism $q$, then $g_\ast:=q_!j_\ast$ is independent of the factorization. Consider a vector bundle $\pi:V\to X$ on a smooth $X$ and  the blow-up of its zero section which can be identified with the total space of $\bar{\pi}:L=\OO_{\PP(V)}(-1) \to \PP(V)$. If we denote the projective blow-up morphism with $q$, then $q:L\setminus \PP(V) \longrightarrow E\setminus X$ is an isomorphism, and 
\begin{eqnarray*}  \pi_!(V\setminus X \hookrightarrow V)_\ast (\unit_{V\setminus X})&=&\pi_!q_!(L\setminus \PP(V)\hookrightarrow L)_\ast(\unit_{L\setminus \PP(V)})\\ &=&(\PP(V)\to X)_!\bar{\pi}_!(L\setminus \PP(V)\hookrightarrow L)_\ast(\unit_{L\setminus \PP(V)})\\ &=& 0,\end{eqnarray*}
generalizing the second remark to all vector bundles. 
\end{enumerate} 
\end{remark}
The following lemma provides an alternative and more familiar characterization of the homotopy invariance.
\begin{lemma}
For a vector bundle $\pi:V\to X$ on a smooth $X$ we define $\pi_\ast(\unit_V)$ by means of $\bar{\pi}_!(V\hookrightarrow \bar{V})_\ast(\unit_V)$, where $\bar{\pi}:\bar{V}=\PP(V\oplus \AA^1)\longrightarrow X$ is the projective closure of $V$. Then, homotopy invariance holds if and only if $\unit_V\to 0_{V\ast}\unit_X$ induces an isomorphism $\pi_\ast(\unit_V)\cong \unit_X$. 
\end{lemma}
\begin{proof}
Let us consider the line bundle $\hat{\pi}:\bar{V}\setminus 0_V \to \PP(V)$ on the divisor $\PP(V)\hookrightarrow \PP(V\oplus \AA^1)$ at infinity which can be identified with $\OO_{\PP(V)}(-1)$. Apply $\bar{\pi}_!$ to the following commutative diagram with exact rows
\[ \xymatrix { 0_{L\,!}0_L^! (\unit_{\bar{V}}) \ar[r] \ar[d] & \unit_{\bar{V}} \ar[r] \ar@{=}[d] & (V\hookrightarrow \bar{V})_\ast \unit_V \ar[d] \\ (L\hookrightarrow \bar{V})_!(\unit_L) \ar[r] & \unit_{\bar{V}} \ar[r] & 0_{V\ast}(\unit_X) } \] 
and take the vertical cones. As a result we get the following diagram \[\PP(\pi)_!\hat{\pi}_!(L\setminus 0_L\hookrightarrow L)_\ast(\unit_{L\setminus 0_L}) \longrightarrow 0 \longrightarrow \cone(\pi_\ast(\unit_V)\to \unit_X)\]  
proving $\pi_\ast(\unit_V)\cong \unit_X$ if homotopy invariance in form of Remark \ref{remark3}(2) holds. Conversely, if we assume $\pi_\ast(\unit_V)\cong \unit_X$ for all line bundles $V\to X$, then $L$ constructed above is the dual line bundle $V^\vee$ and $\PP(\pi)=\id_X$ which implies $\hat{\pi}_!(L\setminus 0_L\hookrightarrow L)_\ast(\unit_{L\setminus 0_L})=0$ for $L=V^\vee$.    
\end{proof}

Let us come back the Theorem. The proof presented here is just an expanded version of the original proof given by Kontsevich and Soibelman.  Before we start proving the theorem, let us begin with a lemmas.
\begin{lemma} \label{lemma_3}
 Given a scheme $X$ over $\MM$ and two open subsets $U,V\hookrightarrow X$ with disjoint complements. Then 
 \[ (U \hookrightarrow X)_!(U\cap V \hookrightarrow U)_\ast \EE \cong (V \hookrightarrow X)_\ast (U\cap V\hookrightarrow V)_! \EE \]
 for every $\EE\in \mathcal{K}(U\cap V)$.
\end{lemma}
\begin{proof}
Let us introduce the shorthand $\FF:=(U \hookrightarrow X)_!(U\cap V \hookrightarrow U)_\ast \EE$ and $\altG:= (V \hookrightarrow X)_\ast (U\cap V\hookrightarrow V)_! \EE $ for the left respectively right hand side of the equation. Using base change, we conclude for $Y:=X\setminus U$ and $Z:=X\setminus V$.
\begin{eqnarray*}
 j_V^\ast \FF=(U\cap V\to V)_! \EE & \mbox{and} & i_Y^\ast \FF=0, \\ 
 j_U^! \altG=(U\cap V \to U)_\ast \EE & \mbox{and} & i_Z^! \altG =0.
\end{eqnarray*}
Writing down the distinguished adjunction  triangles, we get
\begin{eqnarray*}
 && i_{Z\, !}i_Z^! \FF \longrightarrow F \longrightarrow j_{V\ast}j_V^\ast \FF = \altG , \\
 && \FF=j_{U\, ! } j_U^! \altG \longrightarrow \altG \longrightarrow i_{Y \ast} i_Y^\ast \altG.
\end{eqnarray*}
Applying $i_Y^\ast$ to the first triangle and using base change together with $Y\cap Z=\emptyset$ yields $i_Y^\ast \altG=0$. Hence, $\FF\to \altG$ from the second triangle must be an isomorphism.
  
\end{proof}

\begin{proof}[Proof of the Theorem]
Notice that the last equality is a simple consequence of the commutativity of the vanishing cycle functor with  pull-backs along morphisms in $\Sm$ and of the (projection) formula $\pi^+_!(\pi^+)^\ast \mathcal{G} =\LL^{\rk V^+} \mathcal{G} $ for all objects $\mathcal{G}$ in $\altT(X)$. Note that $f|_{V^+}=f|_X\circ\pi^+$ by $\GG_m$-invariance. \\
Let us mention that the first equality of the theorem is equivalent to the equation $\pi^+_!(\phi_f j_!\unit_U)|_{V^+}=0$ for $j:U=V \setminus V^+ \hookrightarrow V$ as one can see by applying $\phi_{f}$ and commutativity with proper maps  to the triangle $j_!\unit_U \rightarrow \unit_V \rightarrow (V^+\hookrightarrow V)_\ast \unit_{V^+}$ in $\mathcal{K}(V)$. We will prove this equation in six steps (see below) after introducing some geometrical objects.\\
Write $Z$ for the naive algebraic quotient $V/\mathbb{G}_m$ which can be described as the relative spectrum of the sheaf $(\Sym \mathcal{V^\vee})^{\mathbb{G}_m}$ of algebras on $X$. Here, $\mathcal{V^\vee}$ is the sheaf of sections of the dual bundle $V^\vee$. There is an obvious morphism $q:V \rightarrow Z$ of $X$-schemes. Any $\GG_m$-invariant regular function $f:V\rightarrow \MM$ induces a unique regular function $g:Z\rightarrow \MM$ such that $f=g\circ q$ giving rise to the following commutative diagram
\[ \xymatrix { V^+ \ar@{^{(}->}[r] \ar[dr]_{\pi^+} & V \ar[r]_q \ar[d]^\pi \ar@/^1pc/[rr]^f & Z \ar[r]_g \ar[dl] & \MM \\ & X. } \]
Notice that $Z\rightarrow X$ has a canonical section $0_Z:X\rightarrow Z$ given by the image of the zero section $0_V:X\rightarrow V$. \\

\textbf{Step 1:} The first thing we wish to do is to compactify $q$ to obtain a proper morphism $\bar{q}:Y \rightarrow Z$. By looking at the limits of 
$\lambda\cdot (v^+,v^-)=(\lambda v^+,\lambda^{-1}v^-)\in V^+\oplus V^-$ for $\lambda\in \GG_m$ going to zero respectively infinity, we see that we have to add something at infinity of $V^+\times \{0\}$ and $\{0\}\times V^-$. More precisely, let us compactify $E$ by forming the projective closures of $V^+$ and $V^-$ 
\begin{eqnarray*} \bar{V} & := & \PP(V^+\oplus \AA^1_X)\times \PP(V^-\oplus \AA^1_X) \\ & = & \big(V^+\oplus V^-\big) \;\cup\; \big(\PP(V^+) \times V^- \big) \;\cup\; \big( V^+\times \PP(V^-) \big)\;\cup\; \big(\PP(V^+)\times \PP(V^-)\big). \end{eqnarray*}
The point $\lambda\cdot(v^+,v^-)$ converges to well defined limits in the subsets $\PP(V^+)\times \{0\}$ and $\{0\}\times \PP(V^-)$. However, the map $q$ does not extend as several orbits have a common boundary point. Hence we need to blow up $\PP(V^+\oplus \AA^1_X)\times \PP(V^-\oplus \AA^1_X)$ in the closed subvarieties $\PP(V^+)\times \{0\}$ and $\{0\}\times \PP(V^-)$ to get separate limit points. After blowing up we remove superfluous strata, namely the strict transform of the divisors $\PP(V^+)\times \PP(V^-\oplus \AA^1_X)$ and $\PP(V^+\oplus \AA^1_X)\times \PP(V^-)$. We denote the resulting space by $Y$ and remark that $q:V\rightarrow Z$ does extend to a proper morphism $\bar{q}:Y\rightarrow Z$ as the reader can check easily. 
\[ \xymatrix { V^+ \ar@{^{(}->}[r] \ar[drr]_{\pi^+} & V \ar[dr]^\pi \ar@{^{(}->}[r] & Y \ar[r]_(0.4){\bar{q}} \ar[d]^(0.4){\bar{\pi}} \ar@/^1pc/[rr]^{\bar{f}} & Z \ar[r]_g \ar[dl] & \MM \\ & & X } \]
For the sake of completeness let us describe the complement of $V$ in $Y$. For this let $p^\pm:\PP(V^\pm) \rightarrow X$ denote the projections. The normal bundle of $\PP(V^+)\times\{0\}$ respectively $\{0\}\times \PP(V^-)$ in $\bar{V}$ is given by $\mathcal{O}_{\PP(V^\pm)}(1)\oplus p^{\pm \,\ast} V^\mp$ and, thus, the two exceptional divisors 
\[\PP\big( \mathcal{O}_{\PP(V^\pm)}(1)\oplus p^{\pm \,\ast} V^\mp\big)\cong \PP\big(\AA^1_{\PP(V^\pm)} \oplus p^{\pm\,\ast} V^\mp\otimes \mathcal{O}_{\PP(V^\pm)}(-1)\big) \] of the blow up are just the projective closures of $p^{\pm\,\ast}V^\mp\otimes \mathcal{O}_{\PP(V^\pm)}(-1)$. In order to get the complement of $V$ in $Y$ we have to remove the strict transform of the divisors $\PP(V^+)\times \PP(V^-\oplus \AA^1_X)$ and $\PP(V^+\oplus \AA^1_X)\times \PP(V^-)$ as mentioned above. Their intersections with the exceptional divisors are just the hyperplanes $\PP(p^{\pm\,\ast}V^\mp)$ at infinity and, thus, the complement $Y\setminus V$ consists of the total space of the two vector bundles $D^\pm:=p^{\pm\,\ast}V^\mp \otimes \mathcal{O}_{\PP(V^\pm)}(-1)$ over $\PP(V^\pm)$. The projection to $\PP(V^\pm)$ can be identified with the blow up map. The following diagram shows our main strata of $Y$, where we used $V^\pm_{\neq 0}:= V^\pm\setminus\{0\}$.
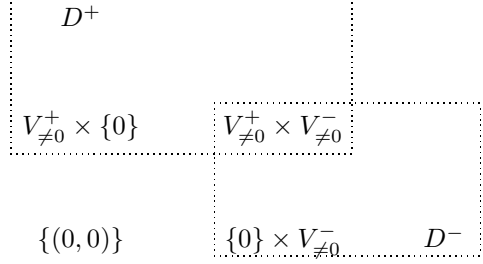
\[ \xymatrix { \quad\;\; D^+\quad\;\; & & \\ V^+_{\neq 0}\times\{0\} & V^+_{\neq 0}\times V^-_{\neq 0} & \\ \{(0,0)\} &  \{0\}\times V^-_{\neq 0} & D^-\; \save "1,1"."2,2"*[F.]\frm{} \restore \save "2,2"."3,3"*[F.]\frm{} \restore } \]
The open subsets $U^+:=D^+ \cup V^+_{\neq 0 }\times V^-$ and $U^-:=D^-\cup V^+\times V^-_{\neq 0}$ indicated by the dotted boxes form the total space of a line bundle over $D^+$ respectively $D^-$. To see this, we consider for example  the limit of $\lambda\cdot (v^+,v^-) \in V^+\times V^-_{\neq 0}$ as $\lambda$ goes to zero. This limit is just $v^+\otimes v^-\in D^-=p^{-\,\ast}V^+ \otimes \mathcal{O}_{\PP(V^-)}(-1)$ if we identify $v^-$ with its associated element in $\mathcal{O}_{\PP(V^-)}(-1)$. By a theorem of Bialynicki--Birula  the open subset $U^-$ is the total space of a line bundle over $D^-$, and similarly for $U^+$. The fibers are given by the corresponding compactification of the $\GG_m$-orbits and the zero sections are just $D^\pm\subset U^\pm$.\\

\textbf{Step 2:} Remember that we defined the open subset $U=V^+\times V^-_{\neq 0}=U^-_{\neq 0}$ with inclusion $j:U\hookrightarrow V$ and we had to show the identity $\pi^+_!(\phi_f j_!\unit_U)|_{V^+}=0$. For this we extend $j_!\unit_U$ on $V$ in two different ways over the divisors $D^+$ and $D^-$, namely as
\[ \FF:= (Y\setminus D^- \hookrightarrow Y)_\ast (V \hookrightarrow Y\setminus D^-)_! j_! \unit_U\stackrel{\ref{lemma_3}}{=}(U^-\hookrightarrow Y)_!(U^-_{\neq 0} \hookrightarrow U^-)_\ast \unit_{U^-_{\neq 0}} \in \mathcal{K}(Y).\]
Clearly $\FF|_V=j_! \unit_U$. Using this, we obtain the equation
\[ (\phi_f j_!\unit_U)|_{V^+} = (\phi_{\bar{f}} \FF)|_{V^+} \]
since $(\phi_{\bar{f}} \FF)|_{V^+}=(\phi_{\bar{f}} \FF)|_{V}|_{V^+}=(\phi_f j_!\unit_U)|_{V^+}$ by commutativity with pull-backs applied to the open inclusion $V\hookrightarrow Y$. Hence, it is sufficient to show $\pi^+_! (\phi_{\bar{f}} \FF)|_{V^+}=0$ or $\int_{V^+} \phi_{\bar{f}}\FF=0$ for short.\footnote{For a morphism $\pi:Y \rightarrow X$, an object $\mathcal{G}$ on $Y$ and a locally closed subset $Z\subset Y$ we denote $\pi|_{Z\; !}\mathcal{G}|_Z$ on $X$ by $\int_Z \mathcal{G}$.}  \\

\textbf{Step 3:} Notice that $\bar{q}^{-1}(0_Z)$ consists of the open subset $U^-|_{\PP(V^-)}$ and the closed complement $\overline{V^+\times \{0\}}$. The latter contains $V^+$ as an open subset with closed complement $\PP(V^+)$. Here we identify $\PP(V^\pm)$ with the zero section of $D^\pm\rightarrow \PP(V^\pm)$. As a result of this we obtain the following two distinguished triangles
\begin{eqnarray}
 \int_{U^-|_{\PP(V^-)}} \phi_{\bar{f}} \FF \quad\longrightarrow & \int_{\bar{q}^{-1}(0_Z)} \phi_{\bar{f}} \FF &\longrightarrow\quad \int_{\overline{V^+\times \{0\}}} \phi_{\bar{f}}\FF \label{firsttriangle}\\
 \int_{V^+} \phi_{\bar{f}} \FF \quad\longrightarrow& \int_{\overline{V^+\times \{0\}}} \phi_{\bar{f}} \FF &\longrightarrow\quad \int_{\PP(V^+)} \phi_{\bar{f}}\FF \label{secondtriangle}
\end{eqnarray}
To prove $\int_{V^+} \phi_{\bar{f}}\FF=0$ we have to show the vanishing of three integrals which is done in the remaining 3 steps. \\

\textbf{Step 4:} Let us start with the  integral $\int_{\bar{q}^{-1}(0_Z)} \phi_{\bar{f}}\FF$. At this point we use the projectivity of $\bar{q}$ and, hence, its commutativity with $\phi$ to obtain
\[ \int_{\bar{q}^{-1}(0_Z)} \phi_{\bar{f}}\FF=(\bar{q}_! \phi_{\bar{f}}\FF)|_{0_Z}=(\phi_g \bar{q}_! \FF)|_{0_Z}. \]
The vanishing of the integral will be a consequence of $\bar{q}_!\FF=0 \in \mathcal{K}(Z)$ which we are going to prove now. Let us use the stratification of  $Y$ given by the open subset $U^-$ and its closed complement $Y\setminus U^-=V^+\cup D^+$ as well as $\FF|_{Y\setminus U^-}=0$ which is a consequence of the construction of $\FF$. Using $\bar{q}|_{U^-}=\bar{q}|_{D^-}\circ \alpha^-$ for $\alpha^-:U^- \rightarrow D^-$ being the projection of the line bundle and equation (\ref{homotopy}) we finally get
\[ \bar{q}_!\FF=\int_{Y} \FF = \int_{U^-} \FF = \int_{D^-} \alpha^-_!(U^-_{\neq 0}\hookrightarrow U^-)_\ast \unit_{U^-_{\neq 0}} =0.\]

\textbf{Step 5:} We now prove the vanishing of the integral $\int_{U^-|_{\PP(V^-)}} \phi_{\bar{f}} \FF$. As $U^-|_{\PP(V^-)}=\mathcal{O}_{\PP(V^-)}(1)$ is contained in the open subset $U^-\subset Y$, we obtain using commutativity with pull-backs along open inclusions 
\[ \int_{U^-|_{\PP(V^-)}} \phi_{\bar{f}}\FF= \int_{U^-|_{\PP(V^-)}} \phi_{\bar{f}|_{U^-}}\FF|_{U^-}=\int_{\PP(V^-)} \alpha^-_! \phi_{\bar{f}|_{U^-}}\FF|_{U^-}\]  
where $\alpha^-:U^- \rightarrow D^-$ is the projection. The vanishing of the integral will be a consequence of $\alpha^-_!\phi_{\bar{f}|_{U^-}} \FF|_{U^-}=0$. To prove the latter equation we can assume for a moment that $\alpha^-$ is trivial, i.e.\ $U^-\cong D^-\times \AA^1$. By construction $\FF|_{U^-}$ is then of the form $\unit_{D^-} \boxtimes (\GG_m\hookrightarrow \AA^1)_\ast \unit_{\GG_m}$. Moreover, $\bar{f}|_{U^-}$ does not depend on the fiber coordinate, i.e.\ $\bar{f}|_{U^-}=\bar{f}|_{D^-} \oplus 0$. Using the Thom--Sebastiani theorem, homotopy invariance of $\mathcal{K}$ and the motivic property we  even get globally $ \alpha^-_!\phi_{\bar{f}|_{U^-}} \FF|_{U^-} =0$.\\

\textbf{Step 6:} By triangle (\ref{firsttriangle}) and the previous two steps $\int_{\overline{V^+\times \{0\}}} \phi_{\bar{f}}\FF=0$. Hence, we can conclude $\int_{V^+} \phi_{\bar{f}} \FF=0$ by using triangle (\ref{secondtriangle}) once we have shown $\int_{\PP(V^+)} \phi_{\bar{f}}\FF=0$ which we are going to do now. As $\PP(V^+)$ is contained in the open subset $U^+\subset Y$, we obtain as in the previous step $\int_{\PP(V^+)} \phi_{\bar{f}}\FF= \int_{\PP(V^+)} \phi_{\bar{f}|_{U^+}}\FF|_{U^+}$ using commutativity with pull-backs again. On the other hand $\alpha^+:U^+\rightarrow D^+$ is a line bundle with zero section $0^+:D^+ \hookrightarrow U^+$ and the restriction of the line bundle to $\PP(V^+)\subset D^+$ is just the tautological bundle $\mathcal{O}_{\PP(V^+)}(1)$. Using the object $\mathcal{G}:=(D^+\setminus \PP(V^+) \hookrightarrow D^+)_! \unit_{D^+\setminus \PP(V^+)}$ we get the triangle
\[ \FF|_{U^+} \longrightarrow \alpha^{+\, \ast} \mathcal{G} \longrightarrow 0^+_\ast \mathcal{G} \]
by construction of $\FF$. We now apply commutativity with the  pull-back along $\alpha^+\in \Sm$ and the commutativity of $\phi$ with push-forward along $0^+$ to get the triangle
\[ \int_{\PP(V^+)} \phi_{\bar{f}|_{U^+}}\FF|_{U^+} \longrightarrow \int_{\PP(V^+)} \alpha^{+\, \ast} \phi_{\bar{f}|_{D^+}} \mathcal{G} \longrightarrow \int_{\PP(V^+)} 0^+_\ast \phi_{\bar{f}|_{D^+}} \mathcal{G}.\] 
Notice that the last morphism is an isomorphism as the restriction of the integrands to $\PP(V^+)\subset D^+\subset U^+$ agree. This proves $\int_{\PP(V^+)} \phi_{\bar{f}}\FF=\int_{\PP(V^+)} \phi_{\bar{f}|_{U^+}}\FF|_{U^+}=0$ and we are done.
\end{proof}

\begin{remark} \rm
The proofs of all previous results remain true and even simplify in the context of  motivic ring $(\Sm,\Pro)$-theories $R$ if we make the same assumptions on $(\Sm,\Pro)$. Formally, we just apply $\underline{\Ka}_0(-)$ to all of the arguments and equations and replace $\underline{\Ka}_0(\mathcal{K})$ with $\underline{\Ka}_0(\Sch^{ft})|_{(\Sm,\Pro)}$. As $\underline{\Ka}_0(Sm^{proj})=\underline{\Ka}_0(\Sch^{ft})|_{Sm,proj}$ for $\Char\kk=0$ by \cite{Bittner04}, it is enough to mention the existence of maps $e^!:\underline{\Ka}_0(\Sch^{ft}_X)\longrightarrow \underline{\Ka}_0(\Sch^{ft}_Y)$ and $e_\ast:\underline{\Ka}_0(\Sch^{ft}_Y)\longrightarrow \underline{\Ka}_0(\Sch^{ft}_X)$ functorial in $e$ for every locally closed embedding $e:Y\hookrightarrow X$ such that 
\begin{enumerate}
 \item $j^!=j^\ast$ for every open embedding $j:U\hookrightarrow X$,
 \item $i_!=i_\ast$ for every closed embedding $i:Z\hookrightarrow X$,
 \item base change holds, i.e.\[(Y_2\to X)^!(Y_1\to X)_\ast=(Y_1\cap Y_2 \to Y_2)_\ast(Y_1\cap Y_2 \to Y_1)^!,\]
 \item the formula $a=i_\ast i^!(a) +j_\ast j^!(a)=i_!i^!(a)+j_\ast j^\ast(a)$ holds for every $a\in \underline{\Ka}_0(\Sch^{ft}_X)$,
 \item $\int_{\AA^1}(\GG_m\hookrightarrow \AA^1)_\ast (\unit_{\GG_m})=0$ in $\Ka_0(\Sch^{ft}_\kk)$.
\end{enumerate}
For the construction of $e_\ast, e^!$ and the proof of their properties, we refer the reader to \cite{Bittner04}, section 6.
\end{remark}

\begin{remark}\label{integral_identity_4} \rm
 There are two possible modifications of the proof which are important in the proof of Proposition \ref{integral_identity_3}(4). First of all, we can replace $V$ with any $\GG_m$-invariant open subset $\hat{U}$ containing the zero section. The latter is the limit of $V^+$ and $V^-$ under the $\GG_m$-action if $z\in \GG_m$ goes to $0$ and $\infty$ respectively. Thus, $V^+,V^-\subset \hat{U}$. Denoting the closed complement of $\hat{U}$ in $V$ with $C$, the proof of the integral identity is literally the same if we replace $Y$ with $Y \setminus \overline{C}$ which is the partial compactification of $U$ with respect to $Y\xrightarrow{q}Z$. Also $D^\pm$ has to be replaced by the ``boundary'' $D^\pm \setminus \overline{C}$ of $\hat{U}$ which is open in $D^\pm$. We advise the reader to check the arguments once more. \\
 The second modification is obtained by assuming that $\hat{U}$ is only open in the analytic topology of $V$, defined over $\kk=\CC$. This of course requires a framework of $\phi$ which works in the analytic topology as well. This is for instance true for  (the pull-back of)  $\phi^{perv}$. 
\end{remark}

\section{$\lambda$-rings}

\subsection{Examples of $\lambda$-rings}
\begin{definition}
 A $\lambda$-ring is a commutative ring $R$ with unit $1$ and a map $\sigma_t:R \ni a \longmapsto \sigma_t(a)=\sum_{n \in \NN} \sigma^n(a)t^n \in 1+tR[[t]] \subseteq R[[t]]$ such that
 \begin{itemize}
  \item[(i)] $\sigma_t(0)=1, \sigma_t(a+b)=\sigma_t(a)\cdot \sigma_t(b)$, i.e.\ $\sigma_t:(R,+,0)\longrightarrow (1+tR[[t]],\cdot,1)$ is a group homomorphism,
  \item[(ii)] $\sigma_t(a)=1+at \mod t^2$ (normalization).
 \end{itemize}
 A homomorphism $f:(R,\sigma_t) \longrightarrow (R',\sigma'_t)$ of $\lambda$-rings is a ring homomorphism $f:R\to R'$ such that $\sigma_t(f(a))=f(\sigma_t(a))$ for all $a\in R$. A $\lambda$-ideal is an ideal $I\subseteq R$ of the ring $R$ such that $\sigma_t(a)\in 1+tI[[t]]$ for all $a\in I$. In that case, there is an obvious (universal) quotient homomorphism $(R,\sigma_t) \longrightarrow (R/I,\tilde{\sigma}_t)$.   
\end{definition}

\begin{example} \rm
 If $(R,\sigma_t)$ is a $\lambda$-ring, $(R,\sigma^{op}_t(a):=\sigma_{-t}(a)^{-1})$ defines a $\lambda$-ring structure, too, the so-called opposite $\lambda$-ring. In all the examples below, the operation $\sigma^n:R\rightarrow R$ is induced by taking symmetric powers. The opposite operation $\lambda^n:=\sigma^{op,n}:R\rightarrow R$ is then related to exterior powers, whenever they make sense. Since mathematicians first considered these operations, our rings are called $\lambda$-rings and not $\sigma$-rings.  
\end{example}
\begin{example} \rm
 Any $\lambda$-ring structure on $\ZZ$ is uniquely determined by the power series $\sigma_t(1)=1+ t+\ldots$. Conversely every such power series defines a $\lambda$-ring structure on $\ZZ$ giving rise to a bijection between $\lambda$-ring structures on $\ZZ$ and power series of the form $1+t+\ldots$.  The standard $\lambda$-ring structure is $\sigma^{st}_t(a)=1/(1-t)^a$ determined by the geometric series $1+t+t^2+t^3+\ldots$. The opposite $\lambda$-ring structure is $\lambda_t(a)=(1+t)^a$ associated to the power series $1+t$.  
\end{example}
\begin{example} \rm
 Since $\ZZ=\Ka_0(\Vect_\kk)$ and $\sigma^{st}_t([V])=\sum_{n\ge 0} [\Sym^n V]t^n$ for the class of a vector space $V$, we can generalize this example by considering $R=\Ka_0(\mathcal{A})$ with $\sigma_t([V])=\sum_{n\ge 0} [\Sym^n V]t^n$ for $\mathcal{A}$ being an essentially small abelian symmetric tensor category and $V\in \mathcal{A}$. This will make $R$ into a $\lambda$-ring.
\end{example}
\begin{example} \rm
Since $\Ka_0(\mathcal{A})=\Ka_0(D^b(\mathcal{A}))$, we can generalize the previous example even further by considering Karoubian closed triangulated symmetric tensor categories $\altT$. This class of examples will be discussed in more detail in the next subsection, but let us quickly mention the outcome. We put $R=\Ka^0(\altT)$ and $\sigma_t([V])=\sum_{n\ge 0} [\Sym^n V]t^n$ for $V\in \altT$. Again, this is a $\lambda$-ring. Under good conditions, this $\lambda$-ring structure descends to $\Ka_0(\altT)$.
\end{example}
\begin{example} \rm \label{motives1}
 Let $M$ be a scheme over $\kk$ and $\Sym M:=\sqcup_{n\in \NN} M^n/\!\!/S_n$ be the commutative monoid in the category $\Sch_\kk$ freely generated by $M$. let $\Ka_0(\Sch^{ft}_{\Sym M})$ be the Grothendieck group of morphisms $u:X\longrightarrow \Sym M $ of finite type, i.e.\ the free abelian group generated by isomorphism classes $[u:X \longrightarrow \Sym M]$ of morphisms $u$ of finite type modulo the cut and paste relation $[u:X \longrightarrow \Sym M]=[u|_Z:Z\longrightarrow \Sym M]+[u|_{X\setminus Z}:X\setminus Z \longrightarrow \Sym M]$ for every closed subscheme $Z\subseteq X$. It becomes a ring by bilinear extension of $[X\longrightarrow \Sym M][Y\longrightarrow \Sym M]=[X\times Y \longrightarrow \Sym M\times \Sym M \xrightarrow{\;\oplus\;} \Sym M]$ with unit $1=[\Spec \kk \xrightarrow{\;0\;} \Sym M]$. Moreover, the operations $\sigma^n([X\longrightarrow \Sym M]) =[X^n/\!\!/S_n \longrightarrow (\Sym M)^n/\!\!/S_n \xrightarrow{\;\oplus\;} \Sym M]$ can be extended making $K_0(\Sch^{ft}_{\Sym M})$ into a 
$\lambda$-ring. 
\end{example}
\begin{definition}
 Given a $\lambda$-ring $(R,\sigma_t)$, we define the Adams operations $\psi_t:R\longrightarrow tR[[t]]$ by means of the logarithmic derivative
 \[ \psi_t(a)=\sum_{n\ge 1}\psi^n(a)t^n:=\frac{d \log \sigma_t(a)}{d \log t}=\frac{t \sigma_t'(a)}{\sigma_t(a)}.\]
\end{definition}
Using the properties of the logarithmic derivative, we immediately prove the following lemma.
\begin{lemma}
 Given a $\lambda$-ring $(R,\sigma_t)$, the Adams operations satisfy the following properties
 \begin{enumerate}
  \item $\psi_t:R\longrightarrow tR[[t]]$ is a group homomorphism with respect to ``$+$'', i.e.\ $\psi_t(0)=0$ and $\psi_t(a+b)=\psi_t(a)+\psi_t(b)$.
  \item If  $R$ is a $\QQ$-algebra, the $\lambda$-operations $\sigma_t$ can be expressed by means of the Adams operations
  \[ \sigma_t(a)=\exp\Big(\int \psi_t(a)\frac{dt}{t}\Big),\quad \mbox {i.e.} \quad\sigma^n(a)=\sum_{k|n} \frac{\psi^{n/k}(a)}{k} \quad\forall n>0.\]
 \end{enumerate}
\end{lemma}
 \begin{definition}
 If $R$ is a $\lambda$-ring, an $R$-$\lambda$-algebra $A$ is given by a unital $\lambda$-ring $A$, with an $R$-algebra structure such that $R\xrightarrow{r\mapsto r\cdot 1} A$ is a $\lambda$-ring homomorphism.
\end{definition}

\subsection{Schur functors} \label{schurfunctors}
Let $\altT$ be a Karoubian closed triangulated $\QQ$-linear symmetric tensor category, i.e.\ a Karoubian closed triangulated $\QQ$-linear category with a bi-exact tensor product $\otimes$ and a unit object $\unit$ together with the usual associativity, commutativity and unit isomorphisms satisfying the usual identities. Using the fact that $\altT$ is Karoubian closed, general arguments show that for $\EE\in \altT$
\[ \EE^{\otimes n} =  \bigoplus\limits_{\lambda \dashv n} W_\lambda \otimes_\QQ S^\lambda(\EE) \]
for certain objects $S^\lambda(\EE)$, where $W_\lambda$ denotes the irreducible representation of $S_n$ associated to the partition $\lambda$ of $n$. The decomposition is functorial,  giving rise to Schur functors $S^\lambda:\altT \longrightarrow \altT$ for every partition $\lambda$. If $F:\altT \to \altT'$ is a symmetric tensor functor, it commutes with the Schur functors on $\altT$ and on $\altT'$. 
\begin{example} \rm \quad 
\begin{enumerate}
 \item For $\lambda=(n)$, the representation $W_\lambda$ is the trivial representation of $S_n$ and $S^\lambda(\EE)=:\Sym^n(\EE)$.  
  \item For $\lambda=(1,\ldots,1)$, the representation $W_\lambda$ is the sign representation of $S_n$ and $S^\lambda(\EE)=:\Alt^n(\EE)$. 
\end{enumerate}
\end{example}
The following proposition is a standard result. 
\begin{proposition}
 Let $\EE,\FF$ be in $\altT$. Then
 \begin{eqnarray}
   \Sym^n(0)&=& \begin{cases}
                \unit \mbox{ for }n=0, \\ 0 \mbox{ else,} \label{eqn0}
                \end{cases} \\
  \Sym^n(\EE \oplus \FF) & \cong & \bigoplus\limits_{i+j=n} \Sym^i(\EE)\otimes \Sym^j(\FF),\label{eqn1} \\
    \Sym^n(\EE\otimes \FF) & \cong & \bigoplus\limits_{\lambda \dashv n} S^\lambda(\EE)\otimes S^\lambda(\FF). \label{eqn2}
 \end{eqnarray}
\end{proposition}
where (\ref{eqn1}) and (\ref{eqn2}) are natural equivalences of bifunctors.  If $\Sym(\EE):=\oplus_{n\in\NN}\Sym^n(\EE)$ is well defined in $\altT$, these equations imply 
\begin{eqnarray*}
\Sym(0)& = &\unit,\\
\Sym(\EE \oplus \FF) & \cong & \Sym(E) \otimes \Sym(F), \\
\Sym (\EE \otimes \FF) & \cong & \bigoplus\limits_{\lambda \in \Part } S^\lambda(\EE)\otimes S^\lambda(\FF),
\end{eqnarray*}
where $\Part$ denotes the set of all partitions. 
\begin{proposition}[\cite{Biglari}, \cite{Deligne1}]\label{lambda-ring}
 The Schur functors $S^\lambda$ induce well defined operations $\sigma^n$ on the additive Grothendieck group $\Ka^0(\altT)$,  satisfying the analogues of equation (\ref{eqn0}), (\ref{eqn1}) and (\ref{eqn2}). In particular, $\Ka^0(\altT)$ carries the structure of a $\lambda$-ring. If $\altT$ is the homotopy category $K^b(\Ab)$ of bounded complexes in a Karoubian closed $\QQ$-linear category $\Ab$ which is preserved by $\otimes$, then the $\lambda$-ring structure descends to $\Ka_0(\altT)$, the Grothendieck group with respect to distinguished triangles. The same holds for $\altT=D^b(\Ab)$ with $\Ab$  an abelian category preserved by the tensor product.
\end{proposition}

\subsection{Complete $\lambda$-rings}
As we have seen in the previous subsection, it is sometimes desirable to form the infinite sum $\sum_{n\in \NN} \sigma^n(a)$. To ensure convergence, we make the following definition. 
\begin{definition} 
A filtered $\lambda$-ring is a $\lambda$-ring $(R,\sigma_t)$ together with a descending filtration 
$ R=F^0R\supseteq F^1R \supseteq F^2R \supseteq \ldots$ such that 
\begin{itemize}
\item[(i)] $F^iR\cdot F^jR \subseteq F^{i+j}R$ for all $i,j\ge 0$,
\item[(ii)] $\sigma^n(F^iR)\subseteq F^{in}R$ for all $i,n\ge 0$.
\end{itemize}
In particular, $F^iR$ is a $\lambda$-ideal and we call the inverse limit $\hat{R}:=\varprojlim_i R/F^iR$ the completion of $R$ with respect to the topology induced by the filtration. We call $R$ a complete filtered $\lambda$-ring if the canonical morphism $R\to \hat{R}$ is an isomorphism.
\end{definition}
\begin{lemma} \label{complete_lambda_rings}
Let $(R,\sigma_t,F^\bullet R)$ be a complete filtered $\lambda$-ring. For $a\in F^1R$ the series $\Sym(a):=\sigma_1(a)=\sum_{n\ge 0}\sigma^n(a)$ is convergent and provides an isomorphism $\Sym: (F^1R,+,0) \xrightarrow{\,\sim\,} (1+F^1R,\cdot,1)$ of groups.  
\end{lemma}

\subsection{Adjoining roots of polynomials} \label{adjoining_roots}

The remaining part of this section is rather technical, but simplifies a lot if one is only interested in $\lambda$-rings arising from Karoubian closed triangulated $\QQ$-linear symmetric tensor categories as seen before. The problem we will address is how to  adjoin roots of polynomials in the category of $\lambda$-rings. \\
Let us start by defining some polynomials $P^{m,n}$ depending on $m,n\in \NN$. Choose variables $x_1, \ldots, x_{mn}, y_1,\ldots,y_m$, and consider the $S_m\times S_{mn}$-invariant polynomial $h_m(\overline{y},\overline{x}^{\overline{n}})$ in $y_1,\ldots,y_m$ and all monomials $\overline{x}^{\overline{n}}=\prod_{i=1}^{mn}x_i^{n_i}$ of total degree $n=\sum_{i=1}^{mn} n_i$, where $h_d$ denotes the d-th totally symmetric polynomial. As $\kk[y_1,\ldots,y_m]^{S_m}\otimes\kk[x_1,\ldots,x_{mn}]^{S_{mn}}$ is generated by totally symmetric polynomials $h_1(\overline{y}),\ldots,h_m(\overline{y}),h_1(\overline{x}),\ldots, h_{mn}(\overline{x})$, there must be a polynomial $P^{m,n}$ such that
\[ h_m(\overline{y},\overline{x}^{\overline{n}})=P^{m,n}(h_1(\overline{y}),\ldots,h_m(\overline{y}),h_1(\overline{x}),\ldots,h_{mn}(\overline{x})).\]

\begin{definition} \label{Picard_group}
 For a $\lambda$-ring $(R,\sigma_t)$ we define the following subset
 \[ R_{sp}:=\{ a\in R \mid \sigma^m(b\sigma^n(a))=P^{m,n}(\sigma^1(b),\ldots,\sigma^m(b), \sigma^1(a),\ldots, \sigma^{mn}(a)) \forall b\in R, n\in \NN\},\]
and call $(R,\sigma_t)$ special if $R_{sp}=R$. We also define $\Pic(R,\sigma_t):=\{a\in R_{sp}\mid  \sigma^n(a)=a^n \;\forall\,n\in \NN \}$ and call elements in $\Pic(R,\sigma_t)$ line elements. 
\end{definition}
One can show that the relations used to define $R_{sp}$ imply 
\[\psi^m(\psi^n(a)b)=\psi^{mn}(a)\psi^m(b) \mbox{ for all } a\in R_{sp}, b\in R,\] 
and if $R$ is a $\QQ$-algebra, these equations are equivalent to the defining equations for $R_{sp}$.  
\begin{lemma} \label{special_elements} Given a $\lambda$-ring $(R,\sigma_t)$, the following assertions are true.
 \begin{itemize}
   \item[(i)] The set $\Pic(R,\sigma_t)$ is a commutative semigroup under multiplication. If moreover $\sigma^n(1)=1$ for all $n\in\NN$, then $1\in \Pic(R,\sigma_t)$, and $\Pic(R,\sigma_t)$ is a commutative monoid. The functor $\Pic(-)$ from the category of $\lambda$-rings $(R,\sigma_t)$ with $1\in \Pic(R,\sigma_t)$ to the category of commutative monoids has a left adjoint associating to a given monoid $M$ the monoid ring $\ZZ M$ with the unique $\lambda$-ring structure $\sigma_t$ such that $\sigma^n(a)=a^n$ for all $a\in M$ and $n\in \NN$. This left adjoint functor is a fully faithful embedding (see \cite{Betley}, Lemma 2.2).
  
   \item[(ii)] The subset $R_{sp}$ is a special $\lambda$-subring of $(R,\sigma_t)$. The inclusion of the full subcategory of special $\lambda$-rings has a left adjoint with adjunction $(R,\sigma_t)\longrightarrow (R^{sp},\tilde{\sigma}_t)$, where $R^{sp}$ is the quotient of $R$ by the ideal 
  \[ I=\Big( \sigma^m(b\sigma^n(a))- P^{m,n}(\sigma^{mn}(a),\ldots,\sigma^1(a),\sigma^m(b),\ldots, \sigma^1(b)) \mid \forall a,b\in R\Big) \]
  which turns out to be a $\lambda$-ideal. In particular $R\to R^{sp}$ has a universal property.

\end{itemize}

\end{lemma}
\begin{remark} \rm
 For a special $\lambda$-ring $(R,\sigma_t)$ the composition $R_{sp}\hookrightarrow R\twoheadrightarrow R^{sp}$ is an isomorphism. It would be interesting to find out whether or not this composition is always an isomorphism.   
\end{remark}
\begin{example} \rm
 We have $1\in \ZZ_{sp}$ if and only if $\ZZ$ is equipped with the standard $\lambda$-ring structure. Then, $\Pic(\ZZ,\sigma_t^{st})=\{1\}$. Moreover, $(\ZZ,\sigma^{st}_t)$ is the initial object in the full subcategory of special $\lambda$-rings and in the full subcategory of $\lambda$-rings $(R,\sigma_t)$ satisfying $1\in R_{sp}$. The category of all $\lambda$-rings has no initial object.
\end{example}
\begin{example} \rm
 The $\lambda$-ring structures constructed on $\Ka_0(\mathcal{A})$ and on $\Ka_0(\altT)$ for an abelian, respectively triangulated, essentially small symmetric tensor category $\mathcal{A}$, respectively $\altT$, are special. In other words, any $\lambda$-ring which is a ``decategorification'' is special. Moreover, the class of any object which is invertible with respect to the tensor product is in $\Pic(\Ka_0(\mathcal{A}),\sigma_t)$, respectively in $\Pic(\Ka_0(\altT),\sigma_t)$.  
\end{example}
\begin{example} \rm \label{motives_special}
 One can show (see \cite{LarsenLunts}) that $\Ka_0(\Sch^{ft}_{\Sym M})$ is never special but $1,\LL\in \Pic(\Ka_0(\Sch^{ft}_{\Sym M}),\sigma_t)$, where $\LL$ denotes the class of $\AAA\xrightarrow{\;0\;}\Sym M$.
\end{example}
\begin{example} \label{lambda_ring_1} \rm Given a $\lambda$-ring $R$ with $1\in \Pic(R,\sigma_t)$, there is a unique $\lambda$-ring structure $\sigma_t$ on $R[T]$ such that $\sigma^n(aT^k)=\sigma^n(a)T^{kn}$ for all $a\in R$ and $k\in \NN$. Moreover, $1,T\in \Pic(R[T],\sigma_t)$. Given a fixed natural number $r>0$ and an element $\LL\in R_{sp}$ such that $\sigma^n(\LL)=\LL^n$ for all $n\in \NN$, the principal ideal $I:=(\LL T^r-1)\subset R[T]$ is a $\lambda$-ideal. Indeed, as $\LL T^r-1\in R[T]_{sp}$, it suffices to show that $\sigma^n(\LL T^r-1)\in I$ for all $n>0$. But
 \[ \sigma^n(\LL T^r-1)=\LL^n T^{nr}-\LL^{n-1}T^{(n-1)r}=\LL^{n-1}T^{(n-1)r}(\LL T^r-1)\in I \]
 for all $n>0$. Hence $R[\LL^{-1/r}]:=R[T]/I$, with $\LL^{-1/r}$ denoting the residue class of $T$, is a $\lambda$-ring with $\LL^{1/r}=(\LL^{-1/r})^{r-1}\LL\in \Pic(R[\LL^{-1/r}],\sigma_t)$ being the inverse of $\LL^{-1/r}$. 
\end{example}
\begin{example} \label{lambda_ring_2} \rm As in the previous example, we consider a $\lambda$-ring $R$ with $1\in \Pic(R,\sigma_t)$ and an element $\LL\in R_{sp}$ satisfying $\sigma^n(\LL)=\LL^{n}$ for all $n\in \NN$. Consider the polynomial ring $R[T_m \mid m>0]$ in infinitely many variables $T_1,T_2,\ldots $, and define elements $\sigma^n_m$ for $n\in \NN$ inductively with respect to $n$ by 
\[ \sigma^n_m:= T_{mn}(\sigma^{n-1}_m + \ldots + \sigma^1_m + \sigma^0_m) \]
starting with $\sigma^0_m=1$ for all $m>0$. There is a unique $\lambda$-ring structure on $R[T_m \mid m > 0]$ restricting to the given structure on $R$ such that $T_m\in R[T_m \mid m > 0]_{sp}$ and $\sigma^n(T_m)=\sigma^n_m$ for all $m>0$ and all $n\in \NN$. Using the shorthand $r_m:=(\LL^m-1)T_m-1$, we therefore have
\begin{equation} \label{eq1} (\LL^{mn}-1)\sigma^n(T_m)=(r_{mn}+1)\big(\sigma^{n-1}(T_m)+ \ldots+ \sigma^1(T_m)+ \sigma^0(T_m)\big)=(r_{mn}+1)\sigma^{n-1}(T_m+1). \end{equation}
Applying $\sigma^n$ to $\LL^mT_m=r_m+T_m+1$ yields
\[ \LL^{mn}\sigma^n(T_m)=\sigma^n(r_m) + \sigma^{m-1}(r_m)\sigma^1(T_m+1) +\dots + \sigma^1(r_m)\sigma^{n-1}(T_m+1)+ \sigma^n(T_m+1) \]
which can be written using $\sigma^n(T_m+1)=\sigma^n(T_m)+ \sigma^{n-1}(T_m+1)$ and equation (\ref{eq1}) as
\[ r_{mn}\sigma^{n-1}(T_m+1)= \sigma^n(r_m) + \sigma^{n-1}(r_m)(T_m+1) + \ldots + r_m\sigma^{n-1}(T_m+1). \]
By induction on $n$, we see that $\sigma^n(r_m)$ is contained in the ideal $I:=(r_{m'}\mid m'>0)$ for all $m,n>0$. As $\LL,r_m\in R[T_m \mid m > 0]_{sp}$, this implies that $I$ is a $\lambda$-ideal. Thus, $R[(\LL^m-1)^{-1}\mid m>0]:=R[T_m \mid m > 0]/I$ is a $\lambda$-ring containing $R$ as a $\lambda$-subring. The elements $(\LL^m-1)^{-1}$ are given by the residue classes of $T_m$ and $(\LL^m-1)^{-1}\in R[(\LL^m-1)^{-1}\mid m>0]_{sp}$ follows. Note that $(\LL^m-1)^{-1}=(\psi^m(\LL-1))^{-1}$ for all $m>0$. 
\end{example}

\begin{lemma} \label{lambda_adjunction_3}
 Let $(R,\sigma_t)$ be a $\lambda$ ring and let $(P_\alpha)_{\alpha\in A}$ be a (possibly infinite) set of polynomials in $R_{sp}[T]$. There is a (unique up to isomorphism) universal $\lambda$-ring homomorphism $(R,\sigma_t)\longrightarrow (R',\sigma'_t)$ and a distinguished family $(x_\alpha)_{\alpha\in A}$ of elements in $R'$ such that
 \begin{itemize}
  \item[(i)] $x_\alpha\in R'_{sp}$ for all $\alpha\in A$,
  \item[(ii)] $P_\alpha(x_\alpha)=0$ for all $\alpha\in A$.
 \end{itemize}
We call $(R',\sigma'_t)$ the $\lambda$-ring obtained from $R$ by adjoining roots of $P_\alpha$ and use the notation $R\langle x_\alpha \mid \alpha\in A\rangle$.  
\end{lemma}
\begin{proof}
 We consider the polynomial ring $R'':=R[s_n^\alpha\mid n \ge 1, \alpha\in A]$ and define operations $\sigma^m$ recursively on monomials as follows. We fix a total order on $A$ and consider a monomial $z=c\prod_{k\ge 1, \alpha\in A} (s_k^\alpha)^{i_k^\alpha}$ which can be written as $z=s_n^\beta z'$, where $(\beta,n)\in A\times \NN$ is the biggest index $(\alpha,k)$ in the lexicographic order on $A\times \NN$ such that $i_k^\alpha\not=0$. Now we put   
 \[ \sigma^m(z) = P^{m,n}(s_{mn}^\beta,\ldots,s_1^\beta,\sigma^m(z'),\ldots, \sigma^1(z')) \]
 which defines $\sigma^m(z)$ recursively. Finally, we extend $\sigma^m$ to the free additive group $R''$ such that $\sigma_t$ is a group homomorphism. By the properties of the $P^{m,n}$, this extension defines a $\lambda$-ring structure and is independent of the order on $A$. Moreover, it is the unique extension such that $s_n^\alpha=\sigma^n(s_1^\alpha)\in R''_{sp}$. Finally, we mod out the $\lambda$-ideal generated by the elements $P_\alpha(s_1^\alpha)$ and denote the residue class of $s_1^\alpha$ with $x_\alpha$.
\end{proof}

\begin{example} \rm \label{lambda_ring_3}
 Suppose for simplicity $|A|=1$ and $P(T)=aT^r-1$ with $1,a\in R_{sp}$ and $r\ge 1$. Then, $\psi^n(a)^{-1/r}$ exists in $R\langle x\rangle$ for all $n>0$. Indeed, by applying $\psi^n$ to $P$, we obtain $\psi^n(a)\psi^n(x)^r=1$ in $R\langle x\rangle$. Hence, $\psi^n(a)$ is a unit, and $\psi^n(x)$ is an $r$-th root of its inverse. Thus, we get a ring homomorphism \[ R[(\psi^n(a))^{-1/r} \mid n \ge 1] \longrightarrow R\langle x\rangle\]
 which is surjective if $R$ is a $\QQ$-algebra.
\end{example}
\begin{example} \rm \label{lambda_ring_4} As a special case of the previous example consider the case $a=\LL-1, r=1$ with $1,\LL\in \Pic(R,\sigma_t)$. In this particular case, the ring homomorphism
 \[ R[(\psi^n(\LL-1))^{-1}\mid n\ge 1]=R[(\LL^n-1)^{-1}\mid n\ge 1] \longrightarrow R\langle x\rangle  \]
is surjective. Indeed, it contains the generators $\sigma^n(x)$ of the ring on the right hand side, as one can see by induction on $n\in \NN$ using the formula $\sigma^n(x)(\LL^n-1)=\sigma^{m-1}(x+1)=\sigma^{m-1}(x)+\ldots+ 1$. On the other hand, we can use the $\lambda$-ring structure on $R[(\LL^n-1)^{-1}\mid n\ge 1]$ constructed in Example \ref{lambda_ring_2} to obtain a $\lambda$-ring homomorphism
\[ R\langle x \rangle \longrightarrow R[(\LL^n-1)^{-1}\mid n\ge 1]\]
mapping $x$ to $(\LL-1)^{-1}$. As the first one was surjective and the composition is the identity on $R[(\LL^n-1)^{-1}\mid n\ge 1]$, we finally get
\[ R\langle (\LL-1)^{-1}\rangle \cong R[(\LL^n-1)^{-1}\mid n\ge 1] \]
as $\lambda$-rings. 
\end{example}
\begin{example} \rm \label{lambda_adjunction} Another special case is given for $a=\LL\in \Pic(R,\sigma_t)$. Fix any $r\in \NN$, and consider the $\lambda$-ring constructed in Example \ref{lambda_ring_1}. Thus we get a surjective $\lambda$-ring homomorphism 
\[ R\langle x\rangle \longrightarrow R[\LL^{-1/r}]\]
mapping $x$ to $\LL^{-1/r}$. For $r=1$, this is an isomorphism with the morphisms of Example \ref{lambda_ring_3} as its inverse because $(\psi^{n}(\LL))^{-1}=\LL^{-n} \in R[(\psi^n(\LL))^{-1} \mid n \ge 1]$ must already be contained in $R[\LL^{-1}]$. For $r>1$ the situation is different. Assuming that $r>1$ is even, then $R\langle x\rangle \longrightarrow R[\LL^{-1/r}]$ can never be an isomorphism. Otherwise, $x\in \Pic(R\langle x \rangle, \sigma_t)$ and $-x\in \Pic(R\langle x \rangle, \sigma_t)$ as there is an automorphism of $R\langle x\rangle $ mapping $x$ to $-x$ by the universal property, but this is impossible. However, we can form the two quotient $\lambda$-rings $R\langle x\rangle^\pm$ by modding out the lambda ideal generated by $\sigma^n(\pm x)- (\pm x)^n$ for all $n\in \NN$. Then, $\pm x \in \Pic(R\langle x\rangle^\pm,\sigma_t)$, and the pair $(R\langle x\rangle^\pm, x)$ satisfies a corresponding universal property. In particular, by mapping $x\to -x$, we get a $\lambda$-ring isomorphism $R\langle 
x\rangle^+\cong R\langle x \rangle^-$. If we map $x$ to $\pm \LL^{-1/r}$, we get a well-defined surjective $\lambda$-ring homomorphism
\[ R\langle x \rangle^\pm \longrightarrow R[\LL^{-1/r}]. \] Composed with the ring homomorphism
\[ R[\LL^{-1/r}] \longrightarrow R[(\LL^n)^{-1/r} \mid n \ge 1] \longrightarrow R\langle x\rangle \longrightarrow R\langle x\rangle^\pm \]
of Example \ref{lambda_ring_3}, we get the unique automorphism of $R[\LL^{-1/r}]$ fixing $R$ and mapping $\LL^{-1/r}$ to $\pm \LL^{-1/r}$. Moreover, the morphism $R[\LL^{-1/r}] \to R\langle x\rangle^\pm$ is surjective as $R\langle x\rangle^\pm$ is generated as a ring by $\sigma^n(\pm x)=(\pm x)^n$ which is the image of $\LL^{-n/r}$ for all $n\in \NN$. Thus,
\[ R\langle \LL^{-1/r}\rangle^\pm \cong R[\LL^{-1/r}] \]
as $\lambda$-rings with $\LL^{-1/r}$ on the left hand side mapping to $\pm \LL^{-1/r}$ on the right hand side. For odd $r>1$ the situation is similar due to the presence of nontrivial Galois $\lambda$-automorphisms of $R\langle \LL^{-1/r} \rangle$.
\end{example}

\subsection{Tensor product of $\lambda$-rings}

Using the notation of the previous subsection, we can introduce the tensor product of two $\lambda$-rings.

\begin{proposition}
Given three $\lambda$-rings $(R,\sigma_t), (R',\sigma'_t), (R'',\sigma''_t)$ and two $\lambda$-ring homomorphisms $\eta':R\to R'_{sp}\subseteq R'$ and $\eta'':R\to R''_{sp}\subseteq R''$, there is a well-defined $\lambda$-ring structure $(\sigma'\otimes\sigma'')_t$ on $R'\otimes_R R''$ such that 
\[\sigma^m(a\otimes b)=P^{m,1}(\sigma'^1(a),\ldots,\sigma'^m(a),\sigma''^1(b),\sigma''^m(b))\] 
for all $a\in R', b\in R''$. If $1'\in \Pic(R',\sigma'_t)$ and $1''\in \Pic(R'',\sigma''_t)$, then $1'\otimes 1''\in \Pic(R'\otimes R'',(\sigma'\otimes \sigma''_t)$. Moreover, $R'\ni a\mapsto a\otimes 1''\in R'\otimes R''$ is a $\lambda$-ring homomorphism and similarly for $R''\to R'\otimes R''$. The lambda ring $(R'\otimes_R R'',(\sigma'\otimes\sigma'')_t)$ together with the bilinear map $\otimes:R'\times R''\to R'\otimes_R R''$ is called the tensor product of $R'$ and $R''$ over $R$.
\end{proposition}
\begin{proof} The proof is a straight forward generalization of the corresponding proof for special $\lambda$-rings. One starts by constructing a $\lambda$-ring structure on $\oplus_{(a,b)\in R'\times R''}\ZZ e_{(a,b)}$ using $P^{m,1}$ and shows that the defining relations of the tensor product form a $\lambda$-ideal. For this we need that $\eta'$ and $\eta''$ have images in $R'_{sp}$ and $R''_{sp}$ respectively.
\end{proof}
The proof of the following lemma is also straightforward and left to the reader. 
\begin{lemma} \label{lambda_adjunction_5}
Say we are given a homomorphism $\eta:R=R_{sp}\longrightarrow R'_{sp}\subseteq R'$ from a special $\lambda$-ring $(R,\sigma_t)$ to a  $\lambda$-ring $(R',\sigma'_t)$ with $1'\in \Pic(R',\sigma'_t)$, and a family $(P_\alpha)_{\alpha\in A}$ of polynomials in $R[T]$. By applying $\eta$, we obtain a family $(\eta(P_\alpha))_{\alpha\in A}$ of polynomials in $R'_{sp}[T]$ and can adjoin roots $x'_\alpha$ to $R'$. Then there is an isomorphism
\[ R'\langle x'_\alpha \mid \alpha\in A \rangle \cong R'\otimes_{R} R\langle x_\alpha\mid \alpha\in A\rangle \]
such that $x'_\alpha$ maps to $1'\otimes x_\alpha$. 
\end{lemma}

\end{appendix}

\bibliographystyle{plain}
\bibliography{Literatur}

\vfill

\textsc{\small B. Davison: MA B1 437 (B\^atiment MA), Station 8, CH-1015 Lausanne, Switzerland}\\
\textit{\small E-mail address:} \texttt{\small nicholas.davison@epfl.ch}\\
\\

\textsc{\small S. Meinhardt: Fachbereich C, Bergische Universit\"at Wuppertal, Gau{\ss}stra{\ss}e 20, 42119 Wuppertal, Germany}\\
\textit{\small E-mail address:} \texttt{\small meinhardt@uni-wuppertal.de}\\

\end{document}